\documentclass{article}

\usepackage[linesnumbered,ruled,vlined]{algorithm2e}
\usepackage{amsfonts}
\usepackage{booktabs}
\usepackage{amsmath}
\allowdisplaybreaks
\usepackage{amssymb}
\usepackage{amsthm}
\usepackage{bbm}
\usepackage{cases}
\usepackage{bm}
\usepackage{float}
\usepackage{braket}
\usepackage{graphicx}
\usepackage{subcaption}
\usepackage{caption}
\usepackage[title]{appendix}
\usepackage{dsfont}
\usepackage{booktabs}
\usepackage{enumitem}
\usepackage[margin=1.25in]{geometry}
\usepackage{graphicx}
\usepackage{hyperref}
\usepackage{lineno}
\usepackage{mathrsfs}
\usepackage{mathtools}
\usepackage{multirow}
\usepackage{multicol}
\usepackage{cleveref}
\usepackage{ulem}
\usepackage{xcolor}
\allowdisplaybreaks[3]
\numberwithin{equation}{section}
\theoremstyle{plain}
\newtheorem{The}{Theorem}[section]  %自定义定理
\newtheorem{lem}{Lemma}[section]
\newtheorem{cor}{Corollary}[section]

\newtheorem{assumption}{Assumption}[section]

\theoremstyle{plain}

\theoremstyle{plain}
\newtheorem{remark}{Remark}[section]

\crefname{The}{Theorem}{Theorems}
\crefname{lem}{Lemma}{Lemmas}
\crefname{cor}{Corollary}{Corollaries}
\crefname{prop}{Proposition}{Propositions}
\crefname{assumption}{Assumption}{Assumptions}
\crefname{Def}{Definition}{Definitions}
\crefname{remark}{Remark}{Remarks}
\crefname{section}{Section}{Sections}

\hypersetup{colorlinks,breaklinks}
\hypersetup{
	colorlinks=true,
	allcolors={green!50!black},
	urlcolor={red!50!black}
}

\newcommand{\dd}{\mathrm{\,d}}

\newcommand{\Expect}{\mathbb{E}}

\newcommand{\one}{\mathbf{1}}

\newcommand{\gs}{\text{GS}}
\newcommand{\calB}{\mathcal{B}}

\newcommand{\calF}{\mathcal{F}}

\newcommand{\calI}{\mathcal{I}}
\newcommand{\calK}{\mathcal{K}}
\newcommand{\calL}{\mathcal{L}}
\newcommand{\calN}{\mathcal{N}}

\newcommand{\calO}{\mathcal{O}}

\newcommand{\C}{\mathbb{C}}
\newcommand{\N}{\mathbb{N}}

\newcommand{\R}{\mathbb{R}}

\newcommand{\RR}{\boldsymbol{R}}

\newcommand{\bx}{\boldsymbol{x}}

\newcommand{\bfx}{\mathbf{x}}

\newcommand{\bu}{\boldsymbol{u}}
\newcommand{\bv}{\boldsymbol{v}}
\newcommand{\bw}{\boldsymbol{w}}
\newcommand{\by}{\boldsymbol{y}}
\newcommand{\bz}{\boldsymbol{z}}

\newcommand{\btheta}{\boldsymbol{\theta}}

\renewcommand{\i}{\mathbf{i}}

\renewcommand{\Re}{\mathrm{Re}}
\newcommand{\scrH}{\mathscr{H}}

\newcommand{\scrF}{\mathscr{F}}

\newcommand{\lrbrace}[1]{\left\{#1\right\}}
\newcommand{\lrbracket}[1]{\left(#1\right)}
\newcommand{\lrsquare}[1]{\left[#1\right]}

\newcommand{\abs}[1]{\left|#1\right|}

\newcommand{\norm}[1]{\left\Vert#1\right\Vert}

\newcommand{\rank}{\mathrm{rank}}

\newcommand{\Eloc}{E_{\text{loc}}}
\definecolor{lightblue}{rgb}{0.957,0.963,0.975}
\definecolor{midblue}{rgb}{0.937,0.943,0.965}
\definecolor{deepblue}{rgb}{0.325,0.427,0.569}
\definecolor{blocktitleblue}{rgb}{0.225,0.427,0.669}
\definecolor{lightred}{rgb}{0.996,0.969,0.969}
\definecolor{midred}{rgb}{0.976,0.949,0.949}
\definecolor{deepred}{rgb}{0.686,0.133,0.098}
\definecolor{deepgreen}{rgb}{0,0.5,0}
\definecolor{halfgray}{gray}{0.55}

\newcommand{\blue}[1]{\textcolor[rgb]{0,0,.8}{#1}}

\definecolor{lightpurple}{rgb}{0.978,0.978,1.0}
\definecolor{deeppurple}{rgb}{0.353,0.275,0.478}

\title{
   Momentum Stability and Adaptive Control in Stochastic Reconfiguration
}
\date{\today}

\author{Yuyang Wang\thanks{State Key Laboratory of Mathematical Sciences
, Academy of Mathematics and Systems Science, Chinese Academy of Sciences, Beijing 100190, China and School of Mathematical Sciences, University of Chinese Academy of Sciences, Beijing 100049, China (email: \href{mailto:wangyuyang@lsec.cc.ac.cn}{\blue{wangyuyang@lsec.cc.ac.cn}}).}  \and Xin Liu\thanks{State Key Laboratory of Mathematical Sciences, Academy of Mathematics and Systems Science, Chinese Academy of Sciences, Beijing 100190, China and School of Mathematical Sciences, University of Chinese Academy of Sciences, Beijing 100049, China (e-mail: \href{mailto:liuxin@lsec.cc.ac.cn}{\blue{liuxin@lsec.cc.ac.cn}}).}}

\begin{document}
	
	\maketitle

    \begin{abstract}

Variational Monte Carlo (VMC) combined with expressive neural network wavefunctions has become a powerful route to high-accuracy ground-state calculations, yet its practical success hinges on efficient and stable wavefunction optimization. While stochastic reconfiguration (SR) provides a geometry-aware preconditioner motivated by imaginary-time evolution, its Kaczmarz-inspired variant, subsampled projected-increment natural gradient descent (SPRING), achieves state-of-the-art empirical performance. However, the effectiveness of SPRING is highly sensitive to the choice of a momentum-like parameter $\mu$. The original sensitivity of $\mu$ and the instability observed at $\mu=1$, have remained unclear. In this work, we clarify the distinct mechanisms governing the regimes $\mu<1$ and $\mu=1$. We establish convergence guarantees for $0\le\mu<1$ under mild assumptions, and construct counterexamples showing that $\mu=1$ can induce divergence via uncontrolled growth along kernel-related directions when the step-size is not summable. Motivated by these theoretical insights and numerical observations, we further propose \textit{Principal Range Informed MomEntum SR} (PRIME-SR), a tuning-free momentum-adaptive SR method based on effective spectral dimension and subspace overlap. PRIME-SR achieves performance comparable to optimally tuned SPRING while significantly improving robustness in VMC optimization.

    \end{abstract}

    \section{Introduction}
\label{sec:introduction}

\par Computing the ground-state of a quantum many-body system is a fundamental problem in computational physics, materials science, and quantum chemistry. Such a system is governed by a self-adjoint Hamiltonian operator $\scrH$, and the goal is to determine its lowest eigenvalue $E_{\gs}$ called ground-state energy, and the associated eigenfunction $\psi_{\gs}$ called ground-state wavefunction,
\begin{equation*}
\mathscr{H} \psi_{\gs} = E_{\gs} \psi_{\gs},
\qquad
\psi_{\gs}\in\calF :=\{ \psi:\Omega^N \to \mathbb{C}, \psi\neq 0 \text{ satisfies certain physical symmetry.}\}
\end{equation*}
where $N$ denotes the number of particles and $\Omega$ is the single-particle configuration space. In realistic settings, accurately approximating ground-state remains challenging due to the curse of dimensionality, the strong correlations encoded in $\mathscr{H}$, and additional physical constraints that generally prevent $\psi$ from factorizing \cite{wang2025complexitytensorproductfunctions}.

\par Variational methods based on the Rayleigh-Ritz principle provide a flexible route to this problem by recasting the eigenvalue computation as minimization of the energy functional
\begin{equation*}
    E_{\gs}=\inf_{\psi\in\calF} \dfrac{\braket{\psi|\scrH|\psi}}{\braket{\psi|\psi}}.
\end{equation*}
From this viewpoint, Variational Monte Carlo (VMC) \cite{mcmillan1965ground,foulkes2001quantum} considers a parametrized family of trial wavefunctions $\psi_{\btheta}$ with $\btheta \in \mathbb{R}^{N_p}$ and turns the optimization over wavefunctions into an optimization over parameters in $\mathbb{R}^{N_p}$. By rewriting the objective as an expectation and estimating it using Monte Carlo sampling, VMC evaluates high-dimensional expectation values without explicitly representing $\psi_{\btheta}$ over the full configuration space (see Section~\ref{sec:vmc} for details). This feature enables VMC to handle system sizes far beyond those accessible to deterministic discretization-based methods such as configuration interaction \cite{shavitt1977method} and multi-configuration self-consistent field \cite{hinze1973mc}.

\par In recent years, neural networks have substantially increased the expressiveness of trial wavefunctions. Following the seminal work of Carleo and Troyer \cite{carleo2017solving}, a broad class of neural network-based and physics-inspired wavefunctions has been developed \cite{pfau2020ab,hermann2020deep,vonself,li2024computational,zhou2024multilevel}, together with theoretical studies of their representational capability \cite{abrahamsen2025anti,ye2024widetilde,wang2025complexitytensorproductfunctions}. These advances have substantially enhanced the approximation power of wavefunctions, thereby enabling near-exact ground-state energy computations in both electronic-structure and lattice settings \cite{hermann2023ab}, and opening the era of neural network-based VMC (NN-VMC).

\par Despite these successes, optimizing neural network wavefunctions remains challenging \cite{wang2025optimization}. Standard optimizers in deep learning such as Adam \cite{kingma2014adam} and AdamW \cite{loshchilov2017decoupled} often perform poorly in this context because of the complex landscape of neural network wavefunctions \cite{pfau2020ab,vonself}. In practice, a variety of optimization methods based on the geometry and physics of VMC have been developed. These include stochastic reconfiguration (SR) and related natural gradient methods \cite{pfau2020ab,sorella1998green,sorella2001generalized,stokes2020quantum}, variants of the linear method \cite{umrigar2005energy,zhao2017blocked,webber2022rayleigh}, and more recent manifold-based approaches \cite{neklyudov2023wasserstein,armegioiu2025functional,li2025acceleratednaturalgradientmethod}. While these methods have demonstrated strong empirical performance, many of them require solving large linear systems whose size scales with the number of model parameters $N_p$, resulting in a per-iteration computational cost of $\calO(N_p^3)$ and consequently limiting their scalability.

\par Among these approaches, SR-based methods are the most widely used in large scale NN-VMC applications (see Section~\ref{sec:sr} for a detailed introduction to SR). In contrast to linear method and manifold-based approaches, which often require additional backpropagation steps or lack effective cost-reduction structures, SR admits a natural reformulation that enables scalable approximations \cite{martens2015optimizing,pfau2020ab,chen2024empowering,
goldshlager2024kaczmarz,rende2024simple,zhou2025wssr,gu2025solvinghubbardmodelneural}.
This has motivated a range of scalable SR variants that approximate or reformulate the SR update in a computationally tractable way for large neural network wavefunctions.

\par Existing scalable SR methods can be broadly categorized by the approximation viewpoint they adopt. Structure exploiting approaches approximate the SR matrix using block-diagonal and Kronecker-factorized forms, leading to K-FAC-type methods \cite{martens2015optimizing,pfau2020ab}. Low-rank approaches exploit the intrinsic rank deficiency of the SR matrix, for instance via truncated or warm-started singular value decomposition (SVD) techniques
\cite{zhou2025wssr}. Finally, least-squares-based approaches interpret the SR update as an underdetermined linear system and select specific solutions, such as Minimum-step SR (MinSR) \cite{chen2024empowering,rende2024simple} and the Subsampled Projected-Increment Natural Gradient (SPRING) method \cite{goldshlager2024kaczmarz}. Notably, these least-squares formulations allow the use of the Sherman-Morrison-Woodbury identity, reducing the per-iteration cost to nearly linear in the number of parameters.

\par Among these SR variants, SPRING has attracted particular attention due to its empirical state-of-the-art performance in NN-VMC \cite{goldshlager2024kaczmarz,gu2025solvinghubbardmodelneural}.
Motivated by randomized Kaczmarz projection methods \cite{karczmarz1937angenaherte,needell2014paved}, SPRING reuses the previous update direction by projecting it onto the subspace identified by the current batch of samples, and then combines this projected direction with the current gradient information to form the next update (see Section~\ref{sec:spring} for details). 

\par However, a central practical issue is that the stability and performance of SPRING depend critically on how strongly the previous direction is reused. The original work \cite{goldshlager2024kaczmarz} reported that naively using the projected previous direction can lead to unstable optimization trajectories. To control the contribution of historical information, SPRING introduces a parameter $\mu \in [0,1)$ that weights the projected previous update direction. Although $\mu=1$ can lead to severe instability or even divergence, empirically, values of $\mu$ close to 1 often accelerate convergence. Moreover, SPRING can be highly sensitive to $\mu$: a value that performs well on one problem may perform poorly on another, as illustrated in Fig.~\ref{fig:spring_sensetive} (see Section~\ref{sec:experiment} for experimental details). These observations motivate the two main goals of this paper: (i) to understand the mechanisms distinguishing the regimes $\mu<1$ and $\mu=1$, and (ii) to develop a principled momentum-adaptive SR method that avoids manual tuning while retaining the practical benefits of momentum reuse.

\begin{figure}
    \centering
    \begin{subfigure}[t]{0.45\textwidth}
        \centering
        \includegraphics[width=\linewidth]{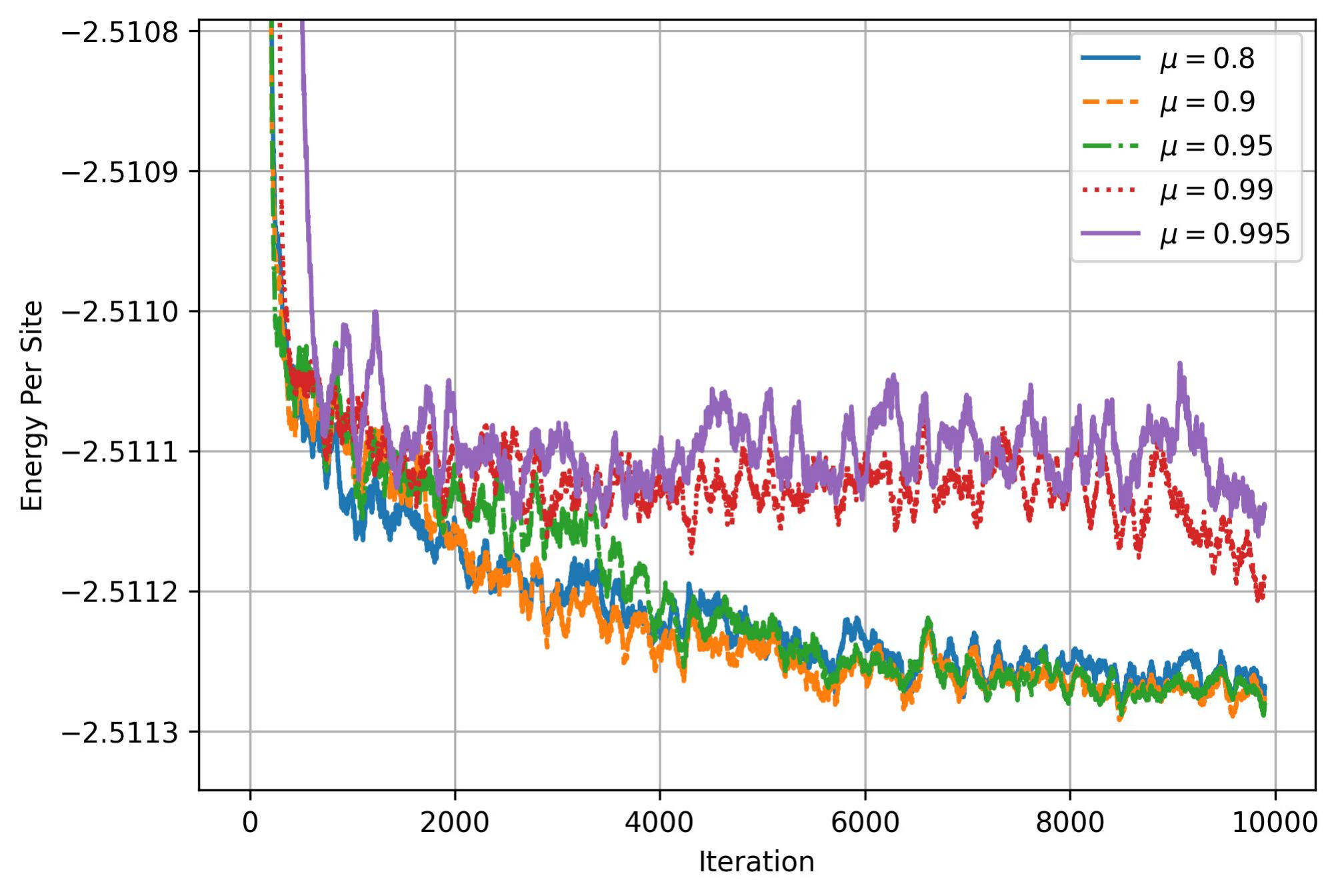}
    \end{subfigure}
    \hfill
    \begin{subfigure}[t]{0.45\textwidth}
        \centering
        \includegraphics[width=\linewidth]{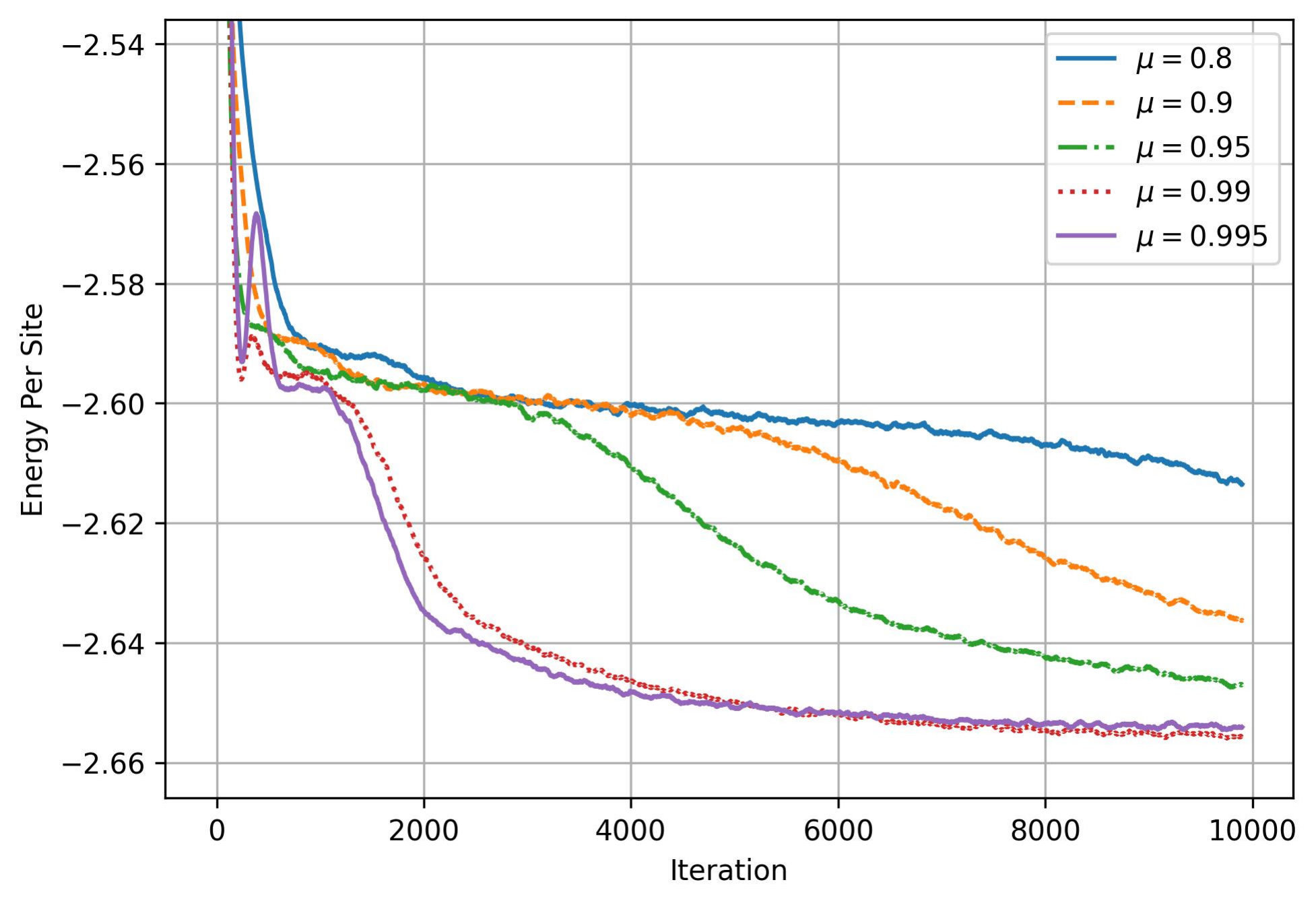}
    \end{subfigure}
    \caption{Relative energy error of SPRING for different choices of $\mu$ on the two-dimensional transverse-field Ising (2D-TFI) model with $N=10\times10$ sites and transverse-field strength $h=2$ (left) and the two-dimensional Heisenberg (2D-Heisenberg) model with $N=10\times10$ sites (right).}
    \label{fig:spring_sensetive}
\end{figure}

\par \textbf{Contributions.} Our main contributions are as follows:
    \begin{enumerate}
        \item[(1)] \textbf{Convergence guarantees for SPRING with $0\le \mu<1$.} We establish convergence of noiseless SPRING to a first-order stationary point. In the stochastic setting with independent and identically distributed (i.i.d.) Monte Carlo samples, we show that the expected norm of the gradient converges to zero up to a sampling error.

        \item[(2)] \textbf{Divergence counterexamples for SPRING with $\mu=1$.} Exploiting the fact that the VMC gradient lies in the range of the SR matrix, we construct explicit counterexamples showing that SPRING with $\mu = 1$ can exhibit uncontrolled growth along kernel-related directions when the step-size is not summable, while the same constructions remain convergent for $0\le \mu<1$.

        \item[(3)] \textbf{PRIME-SR: a tuning-free momentum-adaptive SR method.} Motivated by the theoretical insights and numerical observations, we propose \textit{Principal Range Informed MomEntum SR} (PRIME-SR), a new SR variant that adaptively controls momentum reuse through effective spectral dimension and subspace overlap indicators derived from the SR matrix. PRIME-SR achieves performance comparable to optimally tuned SPRING. In particular, in electronic-structure problems where SPRING is sensitive to initialization and can be unstable, PRIME-SR remains consistently stable across different initializations.
    \end{enumerate}

\vskip 0.2cm

\par \textbf{Organization.} The rest of this paper is organized as follows. In Section~\ref{sec:background}, we review the VMC framework, SR, and the SPRING algorithm. In Section~\ref{sec:theory}, we present the theoretical analysis of SPRING, including convergence results for $0 \le \mu < 1$ and divergence mechanisms at $\mu = 1$. In Section~\ref{sec:adaptive}, we introduce PRIME-SR, a momentum-adaptive SR method motivated by kernel-range decompositions and numerical observations. In Section~\ref{sec:experiment}, we report numerical results on lattice spin systems as well as atomic and molecular systems. Finally, in Section~\ref{sec:conclusion}, we summarize our findings and discuss future directions.

\vskip 0.2cm

\par \textbf{Notations.} Throughout this paper, $\R$ and $\C$ denote the real and complex number fields, and $\i$ denotes the imaginary unit. Bold letters (e.g., $\boldsymbol{x}$) represent vectors. The configuration space of particles is denoted by $\Omega$, and $\Omega^N$ denotes the $N$-fold Cartesian product of $\Omega$. We write $\mathbf{x}\in\Omega^N$ for particle configuration and $\boldsymbol{x}\in\Omega$ for a single-particle configuration. The Hamiltonian is denoted by $\mathscr{H}$, the wavefunction by $\psi$, and the parameters of a parameterized wavefunction by $\boldsymbol{\theta}$. The symbol $\mathbf{1}$ denotes the all-ones vector of appropriate dimension, and $\sigma^{x}$, $\sigma^{y}$, and $\sigma^{z}$ denote the Pauli matrices. We use the standard asymptotic notation $f_k=\Omega(g_k)$ to mean that there exist constants $c>0$ and $k_0\ge 0$ such that $f_k \ge c\, g_k$ for all $k\ge k_0$.

\par The inner product is written in Dirac notation $\braket{\cdot|\cdot}$: it corresponds to the $\calL^2$ inner product in continuous settings and the $\ell^2$ inner product for discrete vectors. The norm $\|\cdot\|$ denotes the $2$-norm unless stated otherwise (matrix spectral norm, vector Euclidean norm, and function $\calL^2$ norm). We use $\Expect$ for expectation, ``i.i.d.'' for independent and identically distributed, and ``a.s.'' for almost surely. We use $\mathscr{F}_k$ to denote a $\sigma$-algebra filtration for discrete iterations, and $\sigma(X)$ for the $\sigma$-algebra generated by the random variable $X$.
    \section{Background and Preliminaries}
\label{sec:background}

\par In this section, we briefly review the VMC framework, the SR method, and the SPRING algorithm. The exposition is intentionally concise and focuses on aspects that are directly relevant to the subsequent sections.

\subsection{Variational Monte Carlo}
\label{sec:vmc}

\par Consider a quantum system of $N$ particles with configuration $\bfx=(\bx_1,\dots,\bx_N)\in\Omega^N$ and wavefunction $\psi:\Omega^N\to\C$. Here $\Omega$ denotes the single-particle configuration space, which is either discrete (e.g., lattice spin systems) or continuous (e.g., electronic systems). The system is governed by a self-adjoint Hamiltonian operator $\scrH$, and the goal is to compute its ground-state energy and wavefunction. Moreover, for the class of Hamiltonians considered in this work, including the electronic Hamiltonian and the TFI and Heisenberg models (see Appendix~\ref{sec:experiment_setting} for details), the real part of a ground-state wavefunction is also a ground-state wavefunction. Therefore, unless stated otherwise, we restrict attention to real-valued wavefunctions of the form $\psi:\Omega^N\mapsto\R$.

\par By the Rayleigh-Ritz variational principle, the ground-state energy can be written as
\begin{equation}
    E_{\gs} = \inf_{\psi\in\calF} L[\psi]:=\dfrac{\braket{\psi|\scrH|\psi}}{\braket{\psi|\psi}}= \dfrac{\int_{\Omega^N} \psi(\bfx) \braket{\bfx|\scrH|\psi}\dd\bfx}{\int_{\Omega^N} \abs{\psi(\bfx)}^2\dd\bfx},\quad \text{s. t. }\psi \text{ satisfies certain constraints.}
    \label{eq:rr_principle}
\end{equation}
When the configuration space is discrete, the integral denotes summation over $\Omega^N$:
\begin{equation*}
    \int_{\Omega^N} \cdot \dd\bfx:=\sum_{\bx_1\in\Omega}\cdots\sum_{\bx_N\in\Omega}\cdot.
\end{equation*}
In VMC, one introduces a parametrized ans\"atz $\psi_{\btheta}:\Omega^N\mapsto\R$ with parameters $\btheta\in\R^{N_p}$ satisfying $\psi_{\btheta}\in\calF$. The ground-state energy $E_{\gs}$ is then approximated by solving
\begin{equation*}
    \tilde{E}_{\gs} := \min_{\btheta\in \R^{N_p}} L(\btheta):=L[\psi_{\btheta}].
\end{equation*}
Since the physical constraints (e.g., antisymmetry) are typically enforced by the parameterization, $\tilde{E}_{\gs}$ provides an upper bound on the true ground-state energy, i.e., $\tilde{E}_{\gs}\ge E_{\gs}$.

\par Direct evaluation of $L(\btheta)$ involves high-dimensional integrals or exponentially large sums. In VMC, the objective can be rewritten as an expectation:
\begin{equation*}
    L(\btheta)=\int \underbrace{\dfrac{\abs{\psi_{\btheta}(\bfx)}^2}{\int\abs{\psi_{\btheta}}^2\dd\bfx'}}_{=:\pi_{\btheta}(\bfx)}\cdot \underbrace{\dfrac{\braket{\bfx|\scrH| \psi_{\btheta}}}{\psi_{\btheta}(\bfx)}}_{=:\Eloc(\btheta;\bfx)}\dd\bfx = \Expect_{X\sim\pi_{\btheta}}\lrsquare{\Eloc(\btheta;X)}.
\end{equation*}
Here $\pi_{\btheta}\propto \abs{\psi_{\btheta}}^2$ is the probability distribution induced by the wavefunction: for continuous systems it is a probability density, while for discrete systems it is a probability mass function on $\Omega^N$. The quantity $\Eloc(\btheta;\bfx)$ is called the local energy \cite{becca2017quantum}.

\begin{remark}
For continuous systems, $\scrH$ acts on functions and $\braket{\bfx|\scrH|\psi}=\scrH\psi(\bfx)$. For discrete systems, $\psi$ can be viewed as a vector $\ket{\psi}$ indexed by $\Omega^N$, and $\scrH$ is represented by a matrix, $\braket{\bfx|\scrH|\psi} = \sum_{\bfx'\in\Omega^N} \psi(\bfx')\braket{\bfx|\scrH|\bfx'}$, where $\braket{\bfx|\scrH|\bfx'}$ is the $(\bfx,\bfx')$-th element of matrix $\scrH$. Although this sum formally runs over an exponentially large configuration space, many Hamiltonians of interest in this work, such as the TFI and Heisenberg models, are sparse. Consequently, only a small number of terms are nonzero, and the cost of evaluating the local energy is typically $\calO(N)$.
\end{remark}

\par Therefore, within the VMC framework, computing the ground-state reduces to the stochastic optimization problem
\begin{equation}
    \min_{\btheta\in\R^{N_p}}L(\btheta) = \Expect_{X\sim\pi_{\btheta}}\lrsquare{\Eloc(\btheta;X)}.
    \label{eq:vmc_opt_problem}
\end{equation}
Moreover, the gradient of $L(\btheta)$ admits the expression \cite{lin2023explicitly},
\begin{equation}
    g(\btheta):=\nabla_{\btheta}L(\btheta)=2\Expect_{X\sim\pi_{\btheta}}\lrsquare{ \bar{E}(\btheta;X)O(\btheta;X)},
    \label{eq:vmc_grad}
\end{equation}
where 
\begin{equation}
    \begin{split}
        \bar{E}(\btheta;\bfx)&:= \Eloc(\btheta;\bfx) - L(\btheta) \in \R,\\
        O(\btheta;\bfx)&:= \nabla_{\btheta}\log\abs{\psi_{\btheta}(\bfx)} - \Expect_{X\sim\pi_{\btheta}}\lrsquare{\nabla_{\btheta}\log\abs{\psi_{\btheta}(X)}} \in \R^{N_p}.
    \end{split}
    \label{eq:O_and_bar_E}
\end{equation}
In practice, both $L(\btheta)$ and $g(\btheta)$ are estimated using Markov chain Monte Carlo (MCMC) samples \cite{foulkes2001quantum}. Let $N_s$ denote the sample size, and let $\calB:=\{X_1,\dots,X_{N_s}\}$ be a batch of samples. The standard Monte Carlo estimators of $L(\btheta)$ and $g(\btheta)$ are
\begin{equation}
    \begin{split}
        L(\btheta)&\approx L(\btheta;\calB):=\dfrac{1}{N_s}\sum_{i=1}^{N_s}\Eloc(\btheta;X_i)\\
        g(\btheta)&\approx g(\btheta;\calB):=\dfrac{2}{N_s-1}\sum_{i=1}^{N_s}\bar{E}(\btheta;X_i,\calB)O(\btheta;X_i,\calB),
    \end{split}
    \label{eq:grad_mc}
\end{equation}
where
\begin{equation}
    \begin{split}
        \bar{E}(\btheta;X_i,\calB)&:= \Eloc(\btheta;X_i)-L(\btheta;\calB),\\
        O(\btheta;X_i,\calB)&:= \nabla_{\btheta}\log\abs{\psi_{\btheta}(X_i)} - \dfrac{1}{N_s}\sum_{j=1}^{N_s} \nabla_{\btheta}\log\abs{\psi_{\btheta}(X_j)}.
    \end{split}
    \label{eq:O_E_single}
\end{equation}
For a more compact representation, define
\begin{equation}
    \begin{split}
        \bar{E}(\btheta;\calB)&:=\dfrac{1}{\sqrt{N_s-1}}\lrbracket{\bar{E}(\btheta;X_1,\calB),\dots,\bar{E}(\btheta;X_{N_s},\calB)}^\top \in \R^{N_s}\\
        O(\btheta;\calB)&:= \dfrac{1}{\sqrt{N_s-1}}\lrbracket{O(\btheta;X_1,\calB),\dots,O(\btheta;X_{N_s},\calB)}\in \R^{N_p\times N_s}.
    \end{split}
    \label{eq:O_E_mc_form}
\end{equation}
Under these definitions, the gradient estimator can be written compactly as $g(\btheta;\calB)=2O(\btheta;\calB)\bar{E}(\btheta;\calB)$. Moreover, it is unbiased under i.i.d. sampling:
\begin{lem}[\cite{abrahamsen2024convergence}]
    Let $\calB:=\{X_1,\dots,X_{N_s}\}$ and $X_1,\dots,X_{N_s}\stackrel{\text{i.i.d}}{\sim}\pi_{\btheta}$. Then $\Expect_{\calB}[g(\btheta;\calB)] = g(\btheta)$.
    \label{lem:grad_unbiased}
\end{lem}

\subsection{Stochastic Reconfiguration}
\label{sec:sr}

\par SR and its variants remain among the most widely used optimization methods in VMC, originally motivated by imaginary-time evolution \cite{sorella1998green,sorella2001generalized}. In essence, SR chooses the parameter update so that the variation of the parameterized wavefunction matches an infinitesimal imaginary-time evolution step (see \cite{becca2017quantum} for details). The resulting SR update can be written as
\begin{equation*}
     \begin{split}
                 \btheta_{k+1} &:= \btheta_k + \eta_k \Delta\btheta_k,\\
        \Delta\btheta_k&\in\mathop{\arg\min}_{\Delta\btheta\in\R^{N_p}} ~\Delta\btheta^\top S(\btheta_k)\Delta\btheta+g(\btheta_k)^\top \Delta\btheta,
     \end{split}
\end{equation*}
where $\eta_k$ denotes the step-size, and $S(\btheta):=\Expect_{X\sim \pi_{\btheta}}\lrsquare{O(\btheta;X)O(\btheta;X)^\top} \in R^{N_p\times N_p}$ is the SR matrix, which is also known as the Fisher information matrix \cite{pfau2020ab}. Using the expression of $g(\btheta)$ in Eq.~\eqref{eq:vmc_grad}, the SR direction can equivalently be characterized as
\begin{equation}
    \Delta\btheta_k= \mathop{\arg\min}_{\Delta\btheta\in\R^{N_p}} \Expect_{X\sim \pi_{\btheta_k}}\lrsquare{ \abs{O(\btheta_k;X)^\top\Delta\btheta + \bar{E}(\btheta_k;X)}^2}.
    \label{eq:sr_update_full_least}
\end{equation}

\par In practice, let $\calB_k:=\{X_{k,1},\dots,X_{k,N_s}\}$ be the sample batch drawn from $\pi_{\btheta_k}$. Using the notation in Eq.~\eqref{eq:O_E_mc_form}, the SR update becomes
\begin{equation}
    \begin{split}
        \btheta_{k+1} &= \btheta_k + \eta_k\Delta\btheta_k,\\
        \Delta\btheta_k & \in \mathop{\arg\min}_{\Delta\btheta \in \R^{N_p}} \norm{O(\btheta_k;\calB_k)^\top \Delta\btheta + \bar{E}(\btheta_k;\calB_k)}^2.
    \end{split}
    \label{eq:sr_update_discrete}
\end{equation}
Equivalently, $\Delta\btheta_k$ solves the linear least-squares system $O(\btheta_k;\calB_k)^\top\Delta\btheta=-\bar{E}(\btheta_k;\calB_k)$ with $N_p$ unknowns and $N_s$ equations.

\par In the NN-VMC era, the number of parameters can be extremely large, and one typically has $N_p \gg N_s$. On the one hand, directly solving an $N_p$-dimensional linear system is impractical; on the other hand, Eq.~\eqref{eq:sr_update_discrete} becomes an underdetermined least-squares problem with infinitely many solutions. In \cite{chen2024empowering}, the authors exploited this underdetermined structure and proposed selecting the minimum-norm solution, leading to the MinSR method whose update can be computed at cost $\calO(N_s^3+N_pN_s^2)$. It highlights the role of the underdetermined structure in large scale SR. This perspective also motivates the SPRING method.

\subsection{SPRING}
\label{sec:spring}

\par SPRING, introduced in \cite{goldshlager2024kaczmarz}, is a recent SR variant motivated by the randomized Kaczmarz method for overdetermined least-squares problems \cite{karczmarz1937angenaherte,needell2014paved}. Since the parameter $\btheta$ usually changes gradually, the SR equations often vary only mildly from one iteration to the next. Based on randomized Kaczmarz method, the SPRING uses the previous update direction as a reference and corrects it with the current SR system. This leads to the following projection-inspired update:
\begin{equation}
    \Delta\btheta_k :=\mathop{\arg\min}_{\Delta\btheta\in\R^{N_p}} \norm{\Delta\btheta-\Delta\btheta_{k-1}}^2 + \dfrac{1}{\lambda}\norm{O(\btheta_k;\calB_k)^\top \Delta\btheta + \bar{E}(\btheta_k;\calB_k)}^2,
    \label{eq:spring_update_unstable}
\end{equation}
where $\lambda>0$ is a regularization parameter.

\par However, as demonstrated in \cite{goldshlager2024kaczmarz}, the update in Eq.~\eqref{eq:spring_update_unstable} can yield unstable optimization trajectories, and the origin of this instability was left unexplained. To improve stability, the authors introduced a decay factor, which is also called a momentum parameter, $\mu\in[0,1)$ applied to the previous update direction, resulting in the practical SPRING update
\begin{equation}
\begin{split}
    \Delta\btheta_k &:=\mathop{\arg\min}_{\Delta\btheta\in\R^{N_p}} \norm{\Delta\btheta-\mu\Delta\btheta_{k-1}}^2 + \dfrac{1}{\lambda}\norm{O(\btheta_k;\calB_k)^\top \Delta\btheta + \bar{E}(\btheta_k;\calB_k)}^2\\
    &=\lrbracket{\lambda I+O(\btheta_k;\calB_k)O(\btheta_k;\calB_k)^\top}^{-1}\lrbracket{\lambda\mu \Delta\btheta_{k-1}-O(\btheta_k;\calB_k)\bar{E}(\btheta_k;\calB_k)}.
\end{split}
    \label{eq:spring_update_argmin}
\end{equation}
Using the Sherman-Morrison-Woodbury identity, the computational cost of SPRING can be reduced to $\calO(N_s^3+N_s^2N_p)$, and the update can be written as
\begin{equation}
    \Delta\btheta_k = \mu\Delta\btheta_{k-1}-O(\btheta_k;\calB_k)\lrbracket{\lambda I + O(\btheta_k;\calB_k)^\top O(\btheta_k;\calB_k)}^{-1}\zeta_k,
    \label{eq:spring_update_SMW}
\end{equation}
where
\begin{equation}
    \zeta_k := \mu O(\btheta_k;\calB_k)^\top\Delta\btheta_{k-1}+\bar{E}(\btheta_k;\calB_k)\in\R^{N_s}.
    \label{eq:spring_zeta}
\end{equation}

\par To further stabilize the optimization, \cite{goldshlager2024kaczmarz} introduced two additional techniques. The first addresses the potential singularity of $O(\btheta_k;\calB_k)^\top O(\btheta_k;\calB_k)$ due to finite sampling. Let $\one\in \R^{N_s}$ denote the all-ones vector. By the definition of $O(\btheta_k;\calB_k)$ in Eqs.~\eqref{eq:O_E_single} and \eqref{eq:O_E_mc_form}, we have $O(\btheta_k;\calB_k)\one = 0$. The authors proposed replacing Eq.~\eqref{eq:spring_update_SMW} by
\begin{equation}
    \Delta\btheta_k = \mu\Delta\btheta_{k-1}-O(\btheta_k;\calB_k)\lrbracket{\lambda I + O(\btheta_k;\calB_k)^\top O(\btheta_k;\calB_k)+\frac{1}{N_s}\boldsymbol{1}\boldsymbol{1}^\top}^{-1}\zeta_k,
    \label{eq:spring_update_stable}
\end{equation}
since this modification does not change the update in exact arithmetic (see the derivation in \cite{goldshlager2024kaczmarz}):
\begin{equation*}
    O(\btheta_k;\calB_k)\lrbracket{\lambda I + O(\btheta_k;\calB_k)^\top O(\btheta_k;\calB_k)+\frac{1}{N_s}\boldsymbol{1}\boldsymbol{1}^\top}^{-1} = O(\btheta_k;\calB_k)\lrbracket{\lambda I + O(\btheta_k;\calB_k)^\top O(\btheta_k;\calB_k)}^{-1}.
\end{equation*}
The second technique is a norm constraint inspired by K-FAC \cite{martens2015optimizing}, which controls the update magnitude at each iteration,
\begin{equation}
    \btheta_{k+1} := \btheta_k + \Delta\btheta_k\cdot \min\lrbrace{\eta_k, \frac{\sqrt{C}}{\norm{\Delta\btheta_k}}},
    \label{eq:spring_norm_constraint}
\end{equation}
where $C>0$ is a prescribed constant. The resulting SPRING procedure is summarized in Algorithm~\ref{alg:spring}.

\begin{algorithm}[htbp]
        \caption{SPRING}
        \label{alg:spring}
    \KwIn{Initial parameter $\btheta_0\in\R^{N_p}$, sample size $N_s\in\N$, step-size $\eta_k$, regularization parameter $\lambda>0$, momentum parameter $\mu\in [0,1)$, and norm-constraint parameter $C>0$.}

    $\Delta\btheta_{-1}:=\boldsymbol{0}$;

    \While{the stopping criterion is not met}{
Sample $\calB_k:=\{X_{k,1},\dots,X_{k,N_s}\}$ from $\pi_{\btheta_k}$;
    
Use $\calB_k$ to construct $O(\btheta_k;\calB_k)\in\R^{N_p\times N_s}$ and $\bar{E}(\btheta_k;\calB_k)\in\R^{N_s}$;

Compute $\zeta_k$ by Eq.~\eqref{eq:spring_zeta}, and compute $\Delta\btheta_k$ by Eq.~\eqref{eq:spring_update_stable};

Update $\btheta_{k+1}$ by Eq.~\eqref{eq:spring_norm_constraint};

Set $k:=k+1$ and update $\eta_k$.

    }
    \KwOut{Optimized parameters and energy.}
    \end{algorithm}
    \section{Theoretical Analysis of SPRING}
\label{sec:theory}

\par In this section, we develop a theoretical analysis of the SPRING iteration, with particular emphasis on the role of the momentum parameter $\mu$. We show that the regimes $0 \le \mu < 1$ and $\mu = 1$ exhibit fundamentally different behaviors, both theoretically and numerically.

\par We distinguish between two settings. First, we analyze a full-batch (deterministic) variant of SPRING, where the SR matrix and gradient are evaluated in expectation form. This setting allows us to isolate the intrinsic dynamics of the algorithm without stochastic effects. Second, we consider the practical stochastic setting induced by Monte Carlo sampling, where both the SR matrix and the gradient are estimated from finite samples.

\par Our main conclusions are as follows. When $0 \le \mu < 1$, SPRING admits convergence guarantees under mild conditions. In contrast, when $\mu = 1$, SPRING can diverge even in the full-batch setting, driven by an uncontrolled accumulation along kernel-related directions of the SR matrix.

\par Recall that the practical SPRING update can be written in compact form as
\begin{equation}
        \begin{cases}
            \btheta_{k+1} = \btheta_k + \eta_k \Delta\btheta_k,\\
            \Delta\btheta_k = \lrbracket{\lambda I +S(\btheta_k;\calB_k)}^{-1} \lrbracket{\lambda \mu \Delta\btheta_{k-1} - \dfrac{1}{2}g(\btheta_k;\calB_k)},
        \end{cases}
        \label{eq:p-spring}
        \tag{SPRING}
\end{equation}
and its full-batch counterpart is
        \begin{equation}
        \begin{cases}
            \btheta_{k+1} = \btheta_k + \eta_k \Delta\btheta_k,\\
            \Delta\btheta_k = \lrbracket{\lambda I +S(\btheta_k)}^{-1} \lrbracket{\lambda \mu \Delta\btheta_{k-1} - \dfrac{1}{2}g(\btheta_k)}.
        \end{cases}
        \label{eq:full-spring}
        \tag{F-SPRING}
    \end{equation}

\begin{remark}
The full-batch setting serves only as a noise-free baseline for analyzing the optimization behavior of SPRING. In this regime, all quantities are evaluated exactly over the full configuration space, so the residual error is not affected by Monte Carlo sampling noise. However, for lattice models, the deterministic cost grows exponentially with the system size, and exact evaluation is therefore feasible only for small systems. Hence, the full-batch setting is adopted here only for analysis and diagnostic comparison on small lattice models, and should not be regarded as a practical optimization setting.
\end{remark}

\subsection{Convergence Analysis of SPRING with $0\le \mu<1$}
\label{sec:convergence_mu_0_1}

\par We begin with the full-batch iteration \eqref{eq:full-spring}. This deterministic setting serves as a baseline for understanding SPRING and forms the foundation for the subsequent stochastic analysis.

\par We impose the following assumptions, which provide sufficient smoothness and moment control to carry out the convergence analysis.
\begin{assumption}
\label{assume:1}
The following hold:
\begin{enumerate}
    \item[(1)]    There exists $C_m>0$, such that for any $\btheta\in\R^{N_p}$,
    \begin{equation*}
\Expect_{X\sim\pi_{\btheta}}\lrsquare{\norm{\nabla_{\btheta}\log\abs{\psi_{\btheta}(X)}}^4}\le C_m,\quad \Expect_{X\sim\pi_{\btheta}}\lrsquare{\Eloc(\btheta;X)^4}\le C_m.
    \end{equation*}
    \item[(2)]  The gradient $g(\btheta)$ is $C_g$-Lipschitz continuous, i.e., there exists $C_g>0$, such that
    \begin{equation*}
        \norm{g(\btheta) - g(\btheta')}\le C_g \norm{\btheta-\btheta'},\quad \forall~\btheta,\btheta'\in \R^{N_p}.
    \end{equation*}
    \item[(3)] For any $k \ge 0$, the step-size $\eta_k$ satisfies
    \begin{equation*}
        0<\eta_k\le \eta_{k-1},\quad \eta_0\le \dfrac{2\lambda(1-\mu)}{C_g}.
    \end{equation*}
\end{enumerate}
\end{assumption}

\par The fourth-order moment bounds in Assumption~\ref{assume:1} control higher-order fluctuations of the stochastic quantities and are standard in analyses of Monte Carlo estimators.

\begin{remark}
We note that related moment bounds and smoothness have also appeared in the existing works that study the convergence of optimization methods for VMC, for example in the analysis of stochastic gradient descent (SGD) \cite{abrahamsen2024convergence}. In that work, Lipschitz continuity of the gradient is obtained by imposing additional regularity assumptions on second-order derivatives of $\psi_{\btheta}$. Here, for simplicity and clarity, we directly assume Lipschitz continuity of $g(\btheta)$, which is a common and convenient condition in optimization.
\end{remark}

\par Lipschitz continuity of the gradient implies a quadratic upper bound for the objective $L(\btheta)$.
\begin{lem}[\cite{nocedal2006numerical}]
    If $g(\btheta)=\nabla_{\btheta}L(\btheta)$ is $C_g$-Lipschitz continuous, then
    \begin{equation*}
        L(\btheta')-L(\btheta)\le g(\btheta)^\top (\btheta'-\btheta) +\dfrac{C_g}{2}\norm{\btheta'-\btheta}^2,\quad \forall~\btheta,\btheta'\in \R^{N_p}.
    \end{equation*}
    \label{lem:2_upper_bound}
\end{lem}
Using this bound, we obtain the following convergence result for the full-batch SPRING iteration.

\begin{The}
    Under \cref{assume:1}, let $\{\btheta_k\}$ be generated by \eqref{eq:full-spring}. Then for any $0\le \mu<1$, there exist $C>0$, such that, for any $K\ge 1$,
    \begin{equation*}
        \sum_{k=1}^{K} \eta_k \norm{g(\btheta_k)}^2 \le C.
    \end{equation*}
    \label{the:full_spring_convergence}
\end{The}
\begin{proof}
    See Appendix \ref{sec:proof_full_convergence}.
\end{proof}

\par \cref{the:full_spring_convergence} implies that, when $0 \le \mu < 1$, the accumulated squared gradient norm is finite. As a direct consequence, we obtain the following corollary.
   \begin{cor}
        Under the assumptions of \cref{the:full_spring_convergence}:
        \begin{enumerate}
            \item[(1)] If $\sum_k \eta_k = \infty$, then
            \begin{equation*}
                \min_{1\le k \le K}\norm{g(\btheta_k)}^2 = \calO\lrbracket{\dfrac{1}{\sum_{k=1}^K \eta_k}}\quad \text{and} \quad \liminf_{k\to \infty} \norm{g(\btheta_k)}=0.
            \end{equation*}
            \item[(2)] If $\eta_k \equiv \eta>0$, then $\displaystyle\lim_{k\to\infty}g(\btheta_k) = 0$.
        \end{enumerate}
        \label{cor:full_spring}
    \end{cor}
    
\par These results establish convergence to first-order stationary points in the deterministic setting. Notably, the restriction $\mu < 1$ is essential: as $\mu \to 1$, the admissible step-size bound will vanish.

\par While \cref{the:full_spring_convergence} establishes convergence guarantees in the idealized full-batch setting, we now turn to the practical SPRING update with Monte Carlo sampling, given by \eqref{eq:p-spring}. In this setting, both the gradient and the SR matrix are estimated from finite samples, which introduces additional stochastic noise. To avoid complications arising from non-stationary Markov chains, we impose the following simplifying assumption on the sampled batches.

\begin{assumption}
    For $k\ge 0$, the sampled configurations $\calB_k:=\{X_{k,i}\}_{i=1}^{N_s}$ satisfy:
    \begin{equation*}
        X_{k,1},\dots,X_{k,N_s} \stackrel{\text{i.i.d.}}{\sim} \pi_{\btheta_k}.
    \end{equation*}
    \label{assume:sample}
\end{assumption}

\par Under \cref{assume:sample}, \cref{lem:grad_unbiased} shows that the gradient estimator $g(\btheta_k;\calB_k)$ is unbiased. The same assumption also implies that the SR matrix is an unbiased estimator of $S(\btheta_k)$.

\begin{lem}
    Under \cref{assume:sample}, for any $k\ge 0$, $\Expect_{\calB_k}[S(\btheta_k;\calB_k)] = S(\btheta_k)$.
    \label{lem:sr_matrix_unbiased}
\end{lem}
\begin{proof}
    See Appendix \ref{sec:proof_sr_unbiased}.
\end{proof}

\par Despite the unbiasedness of both $g(\btheta_k;\calB_k)$ and $S(\btheta_k;\calB_k)$, the SPRING update direction in \eqref{eq:p-spring} is conditionally biased due to the matrix inversion. Specifically, conditioning on the $\sigma$-algebra generated by previous samples $\scrF_{k-1}:=\sigma\lrbracket{X_{i,j},\,0\le i\le k-1,\,1\le j\le N_s}$, we generally have
\begin{equation*}
    \Expect\lrsquare{\lrbracket{\lambda I +S(\btheta_k;\calB_k)}^{-1} \lrbracket{\lambda \mu \Delta\btheta_{k-1} - \dfrac{1}{2}g(\btheta_k;\calB_k)}|\scrF_{k-1}}\neq \lrbracket{\lambda I +S(\btheta_k)}^{-1} \lrbracket{\lambda \mu \Delta\btheta_{k-1} - \dfrac{1}{2}g(\btheta_k)},
\end{equation*}
This conditional bias introduces an additional error term in the convergence analysis and prevents convergence to an exact first-order stationary point.

    \begin{The}
    Under \cref{assume:1,assume:sample}, let $\{\btheta_k\}$ be generated by \eqref{eq:p-spring}. Then for any $0\le \mu<1$, there exist $C_1,~C_2,~C_3>0$, such that, for any $K\ge 1$,
    \begin{equation}
    \sum_{k=1}^K \eta_k\Expect[\norm{g(\btheta_k)}^2] \le C_1 + C_2\dfrac{1}{N_s}\sum_{k=1}^K\eta_k + C_3\dfrac{1}{N_s^2}\sum_{k=1}^{K}\eta_k.
        \label{eq:p_spring_result}
    \end{equation}
    Here, $\Expect[\cdot]$ denotes the total expectation.
        \label{the:p_spring_convergence}
    \end{The}
    \begin{proof}
        See Appendix \ref{sec:proof_p_spring_convergence}.
    \end{proof}

\par The bound in Eq.~\eqref{eq:p_spring_result} shows that Monte Carlo noise yields a residual term of order $\mathcal{O}(1/N_s)$. Consequently, the parameter sequence $\{\btheta_k\}$ converges to an $\mathcal{O}(1/N_s)$-neighborhood of a first-order stationary point.
\begin{cor}
    Under the assumptions of \cref{the:p_spring_convergence}, we have
    \begin{equation*}
        \min_{1\le k\le K}\Expect\lrsquare{\norm{g(\btheta_k)}^2} = \calO\lrbracket{\dfrac{1}{\sum_{k=1}^K\eta_k}} + \calO\lrbracket{\dfrac{1}{N_s}}.
    \end{equation*}
\end{cor}

\begin{remark}
In contrast to SPRING, the SGD update in VMC admits a conditionally unbiased update direction under \cref{assume:sample}, enabling first-order convergence results \cite{abrahamsen2024convergence}. From this viewpoint, our analysis is closer in spirit to studies that explicitly account for Monte Carlo/MCMC-induced bias in the update direction, such as \cite{li2023convergence}. The resulting $\mathcal{O}(1/N_s)$ residual term is of the same type. We emphasize that \eqref{eq:p_spring_result} is an upper bound; in practice, SR preconditioning can yield substantially faster convergence.
\end{remark}

\par A natural question is whether the $\mathcal{O}(1/N_s)$ term is merely an artifact of the proof or an intrinsic consequence of stochastic sampling. To investigate this, we conduct numerical experiments on the one-dimensional transverse-field Ising (1D-TFI) model with $N=10$ sites and transverse-field strength $h=1$, using a restricted Boltzmann machine (RBM) ans\"atz \cite{carleo2017solving} with $D=5N$ hidden units (see Appendix~\ref{sec:experiment_setting} for details). All experiments are implemented using NetKet \cite{netket3:2022} and JAX \cite{jax2018github}.

\par For SPRING, we fix the regularization parameter $\lambda=10^{-3}$ and do not apply the norm-constraint stabilization. Unless otherwise specified, all other algorithmic components follow Algorithm \ref{alg:spring}. We run SPRING with $\mu=0$ and $\mu=0.99$ under both full-batch evaluation and Monte Carlo sampling with varying sample sizes $N_s$. To eliminate MCMC non-stationarity in this experiment, we use direct sampling throughout. For each run, we record the full-batch gradient norm over $K=100000$ iterations. The results are reported in Fig.~\ref{fig:convergence_test}. 

\par In this experiment, the full-batch size is $2^N=1024$. The results show that, even when $N_s$ exceeds the full-batch size, stochastic sampling prevents convergence to an exact stationary point. Moreover, the empirical results are well fitted by $\sqrt{b/N_s + c/N_s^2}$, consistent with the bound in Eq.~\eqref{eq:p_spring_result}. These observations indicate that the $\mathcal{O}(1/N_s)$ residual term reflects an intrinsic limitation of stochastic sampling rather than a proof artifact.

\begin{figure}
    \centering

    \begin{subfigure}[t]{0.48\textwidth}
        \centering
        \includegraphics[width=0.48\linewidth]{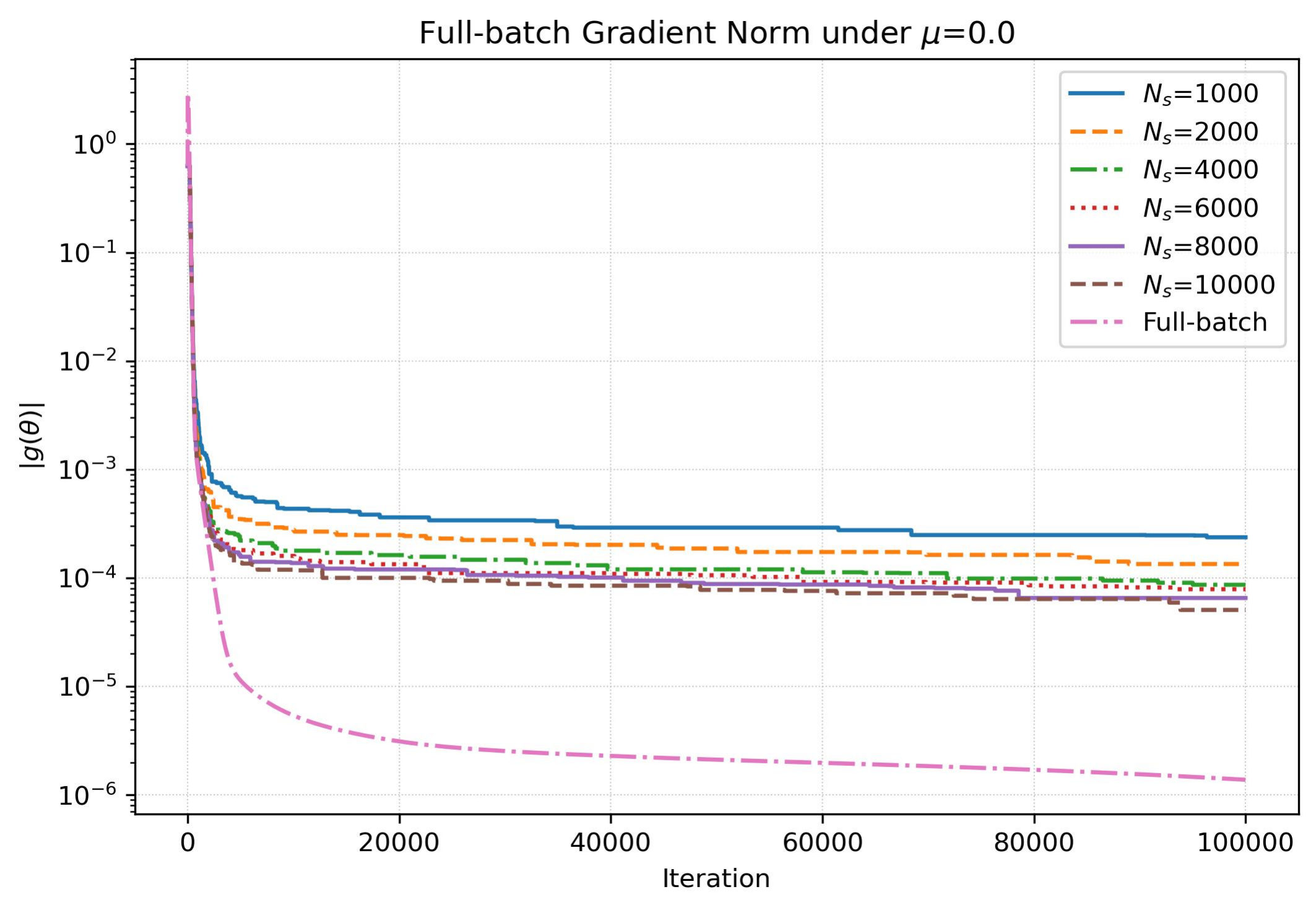}
        \includegraphics[width=0.48\linewidth]{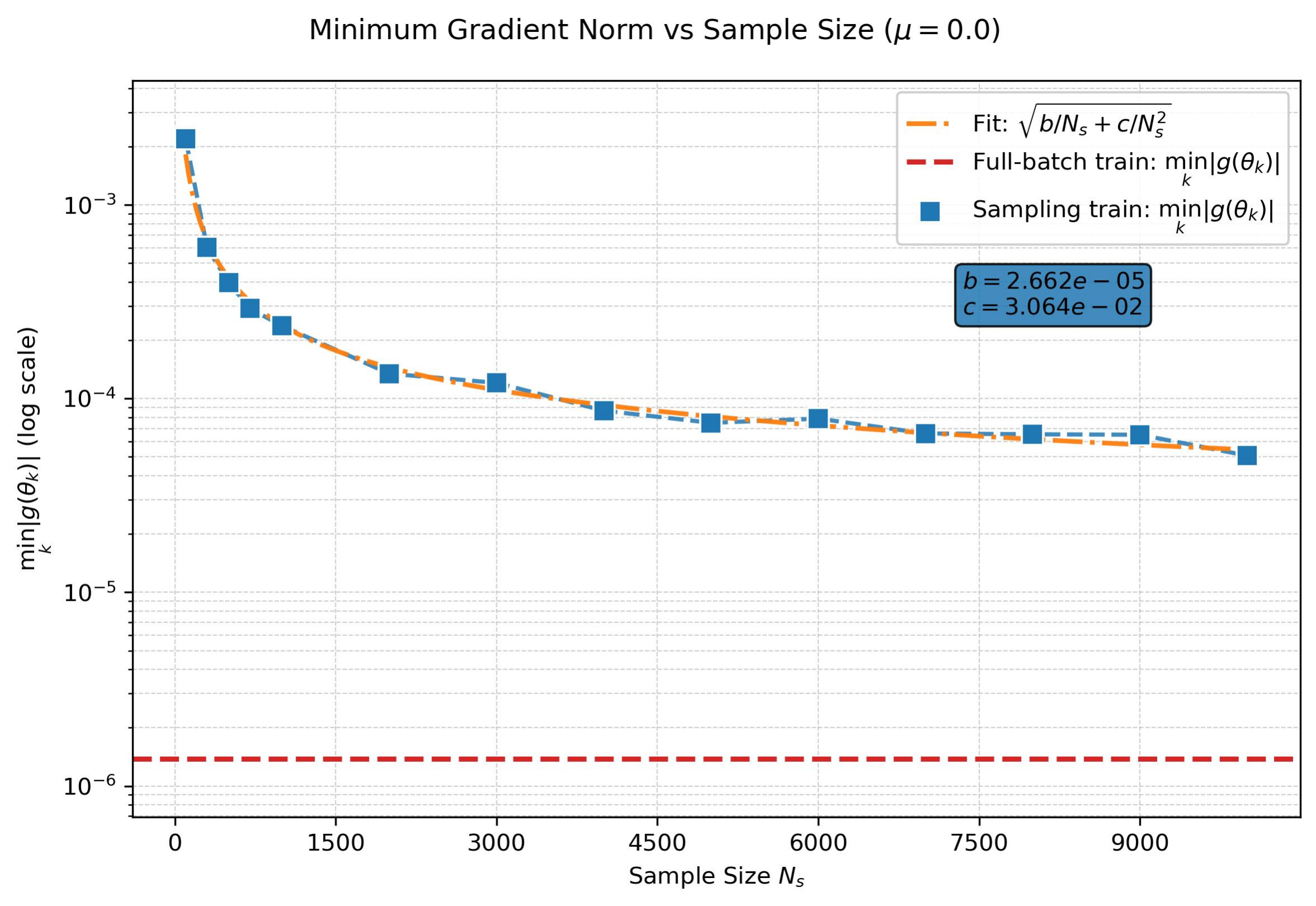}
        \caption{$\mu=0$, $\eta_k=0.01$.}
    \end{subfigure}
    \hfill
    \begin{subfigure}[t]{0.48\textwidth}
        \centering
        \includegraphics[width=0.48\linewidth]{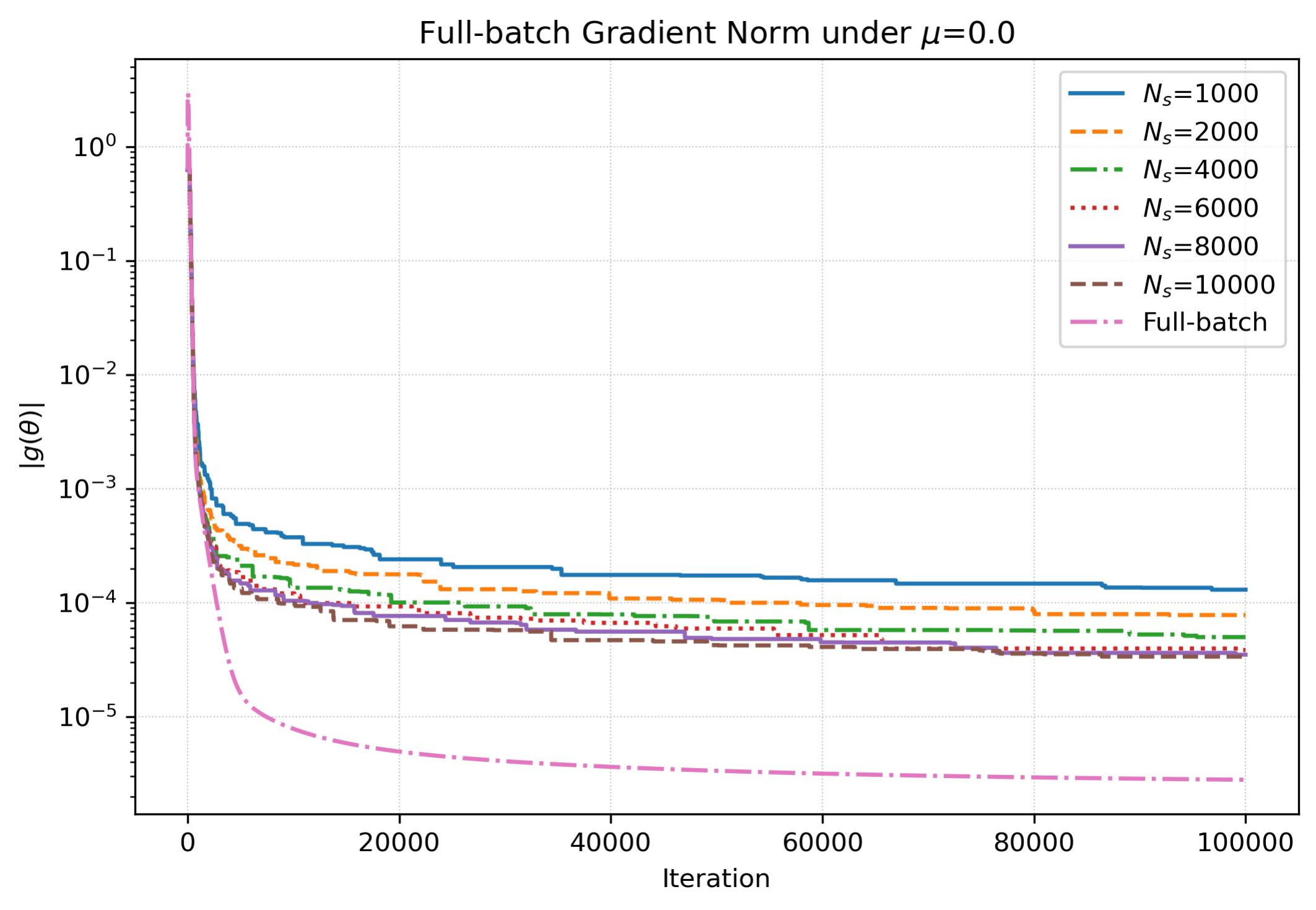}
        \includegraphics[width=0.48\linewidth]{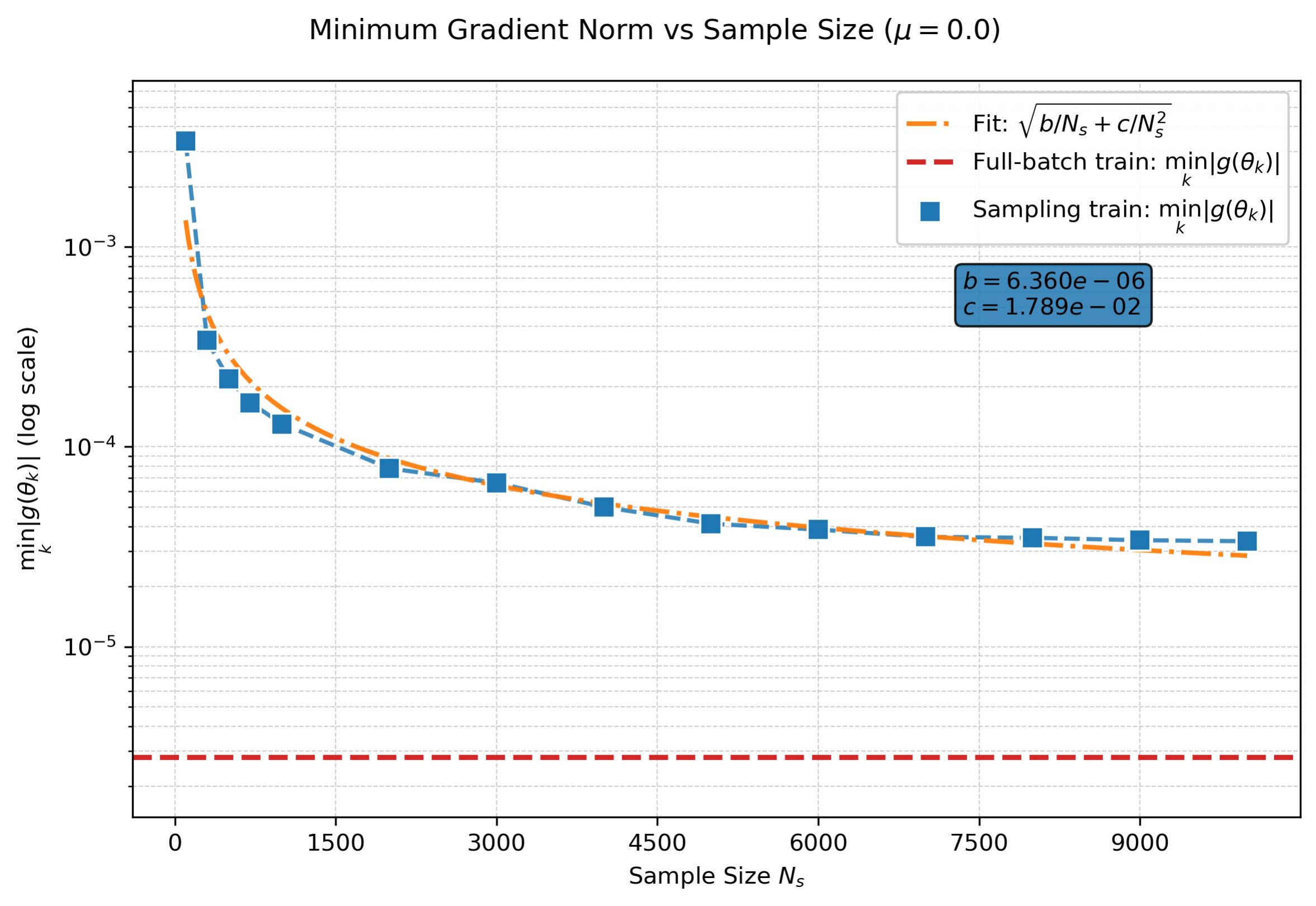}
        \caption{$\mu=0$, $\eta_k=0.01/(1+10^{-4}k)$.}
    \end{subfigure}

    \vspace{0.8em}

    \begin{subfigure}[t]{0.48\textwidth}
        \centering
        \includegraphics[width=0.48\linewidth]{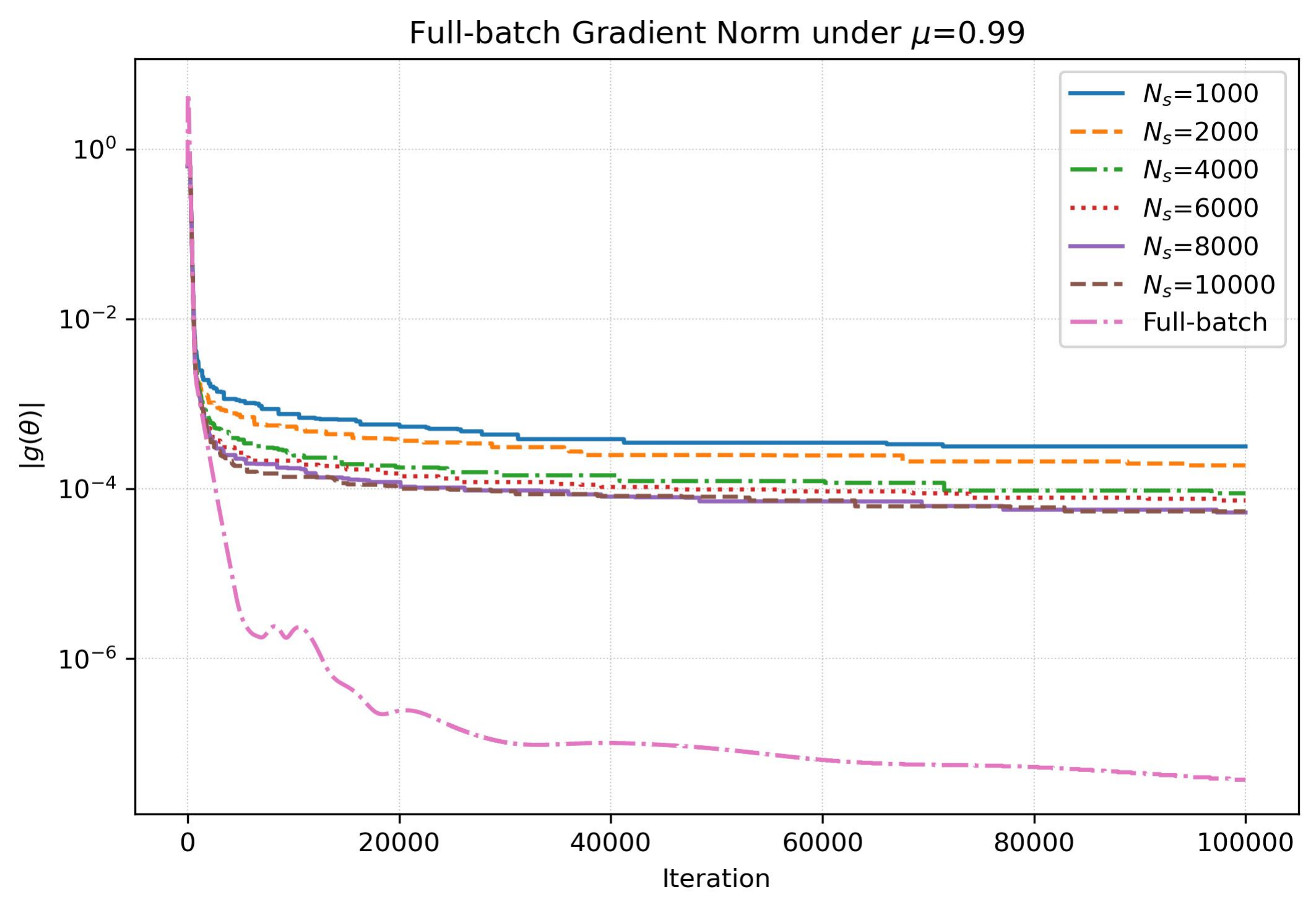}
        \includegraphics[width=0.48\linewidth]{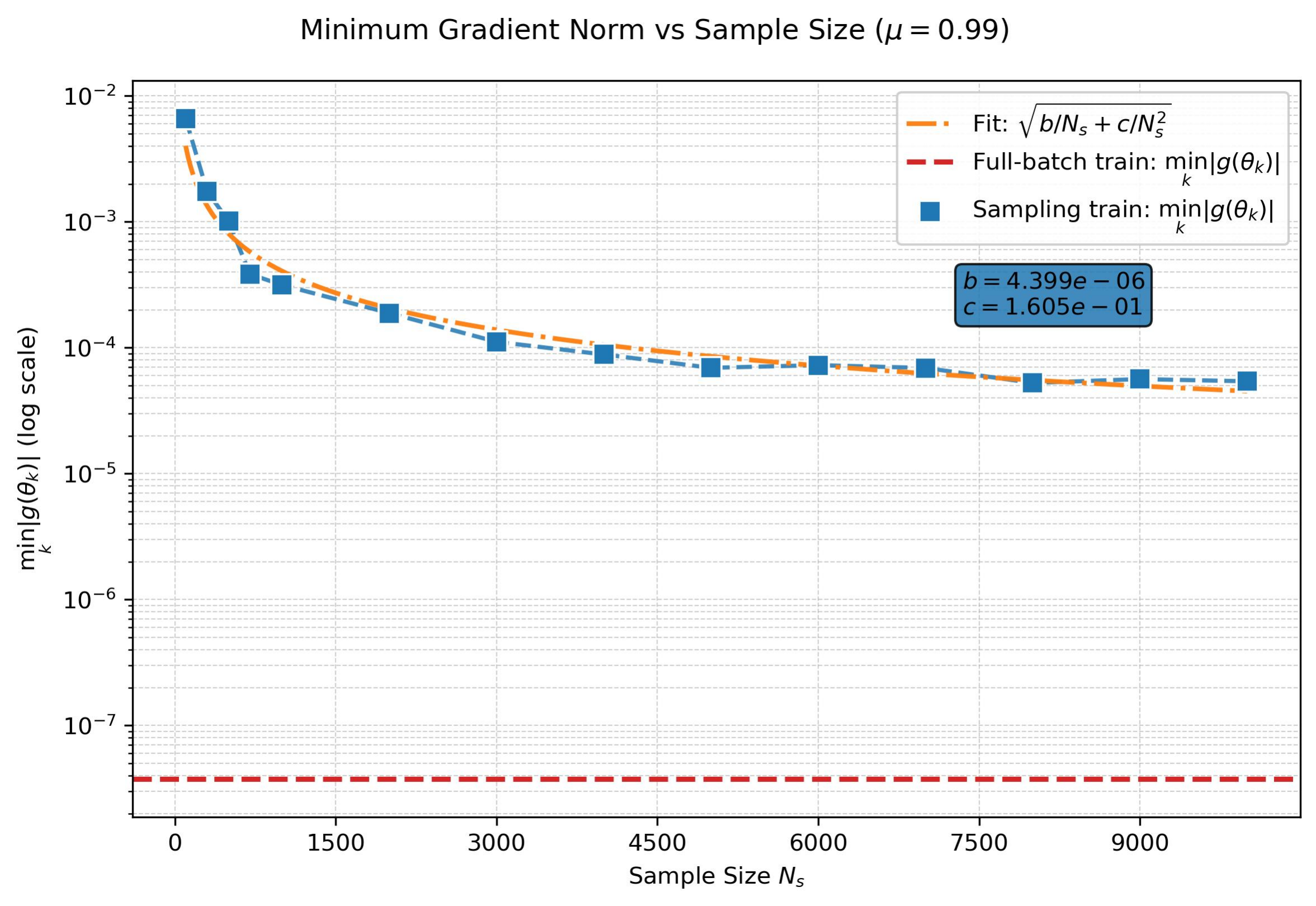}
        \caption{$\mu=0.99$, $\eta_k=0.01$.}
    \end{subfigure}
    \hfill
    \begin{subfigure}[t]{0.48\textwidth}
        \centering
        \includegraphics[width=0.48\linewidth]{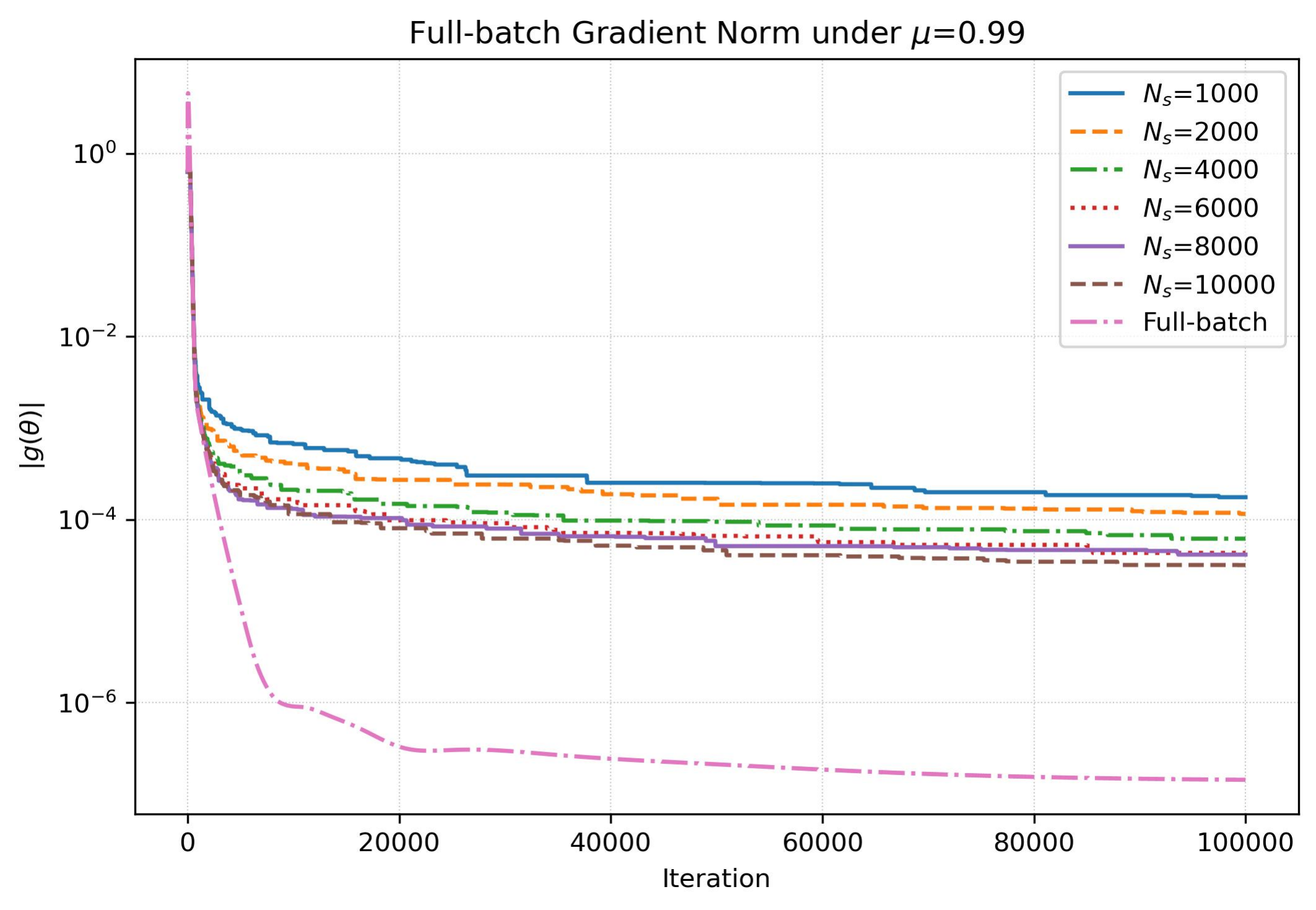}
        \includegraphics[width=0.48\linewidth]{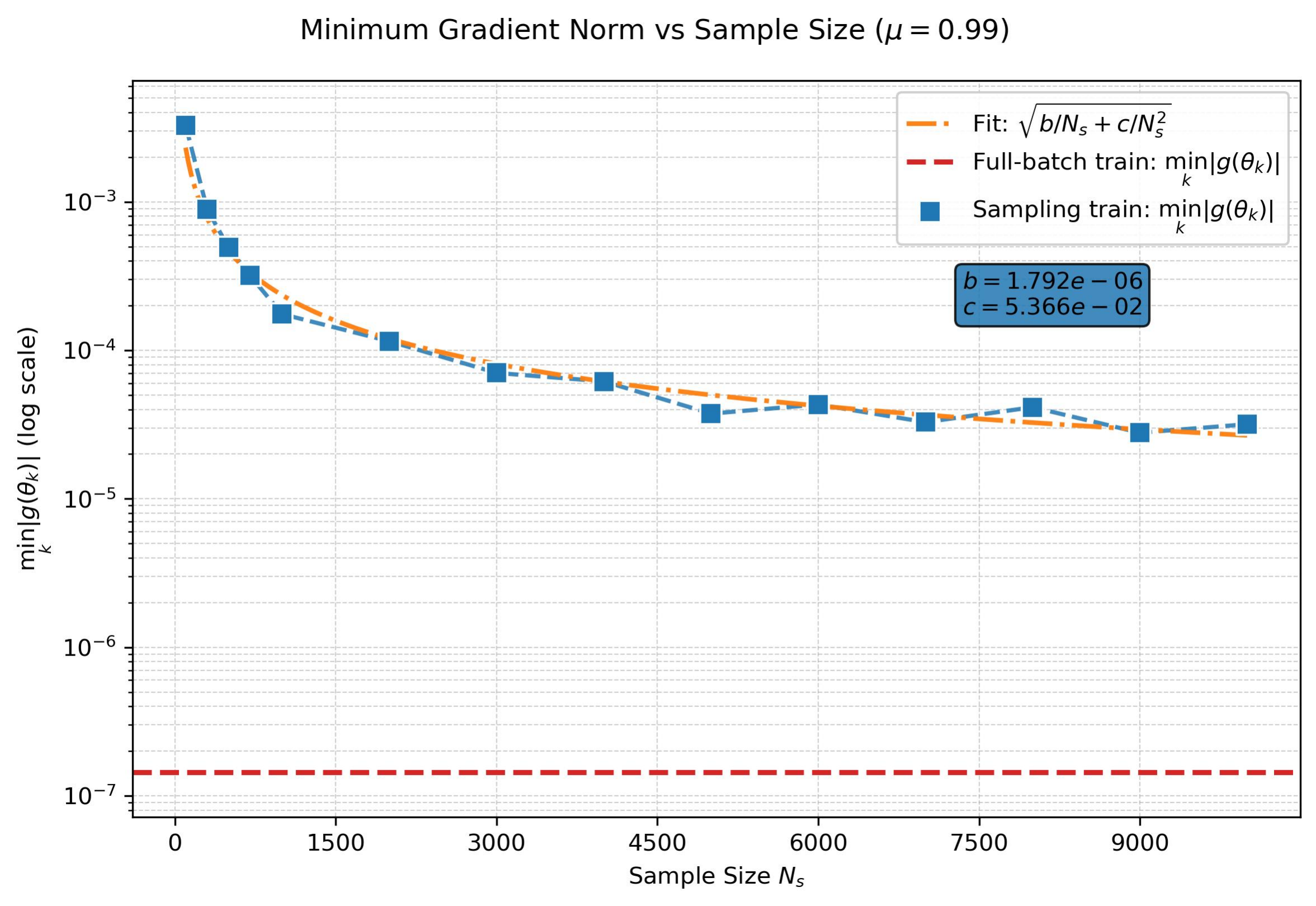}
        \caption{$\mu=0.99$, $\eta_k=0.01/(1+10^{-4}k)$.}
    \end{subfigure}

    \caption{Convergence test on the 1D-TFI model with $N=10$ sites. Each block corresponds to a fixed momentum parameter $\mu$ and a step-size schedule. Left panels: evolution of the minimum gradient norm $\min_{0\le k\le K}\norm{g(\btheta_k)}$ over iterations, for full-batch evaluation and for different sample sizes $N_s$. Right panels: the minimum gradient norm attained over $K=100000$ iterations as a function of $N_s$. The red dashed line indicates the full-batch result, and the orange curve shows the fit $\sqrt{b/N_s + c/N_s^2}$.}

    \label{fig:convergence_test}
\end{figure}

\subsection{Divergence Counterexample of SPRING with $\mu=1$}
\label{sec:divergence_mu_1}

\par We now analyze the regime $\mu=1$, which exhibits dynamics that are qualitatively different from the case $0\le \mu<1$. Throughout this subsection, we focus on the full-batch update \eqref{eq:full-spring}. A key structural property of VMC plays a central role: the gradient $g(\btheta)$ always lies in the range of the SR matrix $S(\btheta)$.
\begin{lem}
    Let $\mathcal{R}(S(\btheta))$ denote the range space of $S(\btheta)$. Then $g(\btheta) \in \mathcal{R}(S(\btheta))$.
\end{lem}
\begin{proof}
    Denote $\calK(S(\btheta))$ the kernel space of $S(\btheta)$. For any $w\in\calK(S(\btheta))$, we have
    \begin{equation*}
        w^\top S(\btheta)w= \Expect_{X\sim\pi_{\btheta}}\lrsquare{\abs{O(\btheta;X)^\top w}^2}=0,
    \end{equation*}
    which implies $O(\btheta;X)^\top w = 0$ $\pi_{\btheta}$-a.s. Consequently,
    \begin{equation*}
        g(\btheta)^\top w = \Expect_{X\sim\pi_{\btheta}}[\bar{E}(\btheta;X)O(\btheta;X)^\top w]=0,\quad \forall~w\in\calK(S(\btheta)).
    \end{equation*}
    Therefore, $g(\btheta)\perp \calK(S(\btheta))$ and hence $g(\btheta)\in\mathcal{R}(S(\btheta))$.
\end{proof}

\begin{remark}
    The same property holds for the Monte Carlo estimators. Indeed, with the usual constructions
    $g(\btheta;\calB)=2O(\btheta;\calB)\bar{E}(\btheta;\calB)$ and $S(\btheta;\calB)=O(\btheta;\calB)O(\btheta;\calB)^\top$, we also have $g(\btheta;\calB)\in \mathcal{R}(S(\btheta;\calB))$.
\end{remark}

\par Let $P^{\calK}(\btheta)\in\R^{N_p\times N_p}$ denote the orthogonal projector onto the kernel space $\calK(S(\btheta))$. Projecting the full-batch update \eqref{eq:full-spring} onto $\calK(S(\btheta_k))$ yields
\begin{equation}
    P^{\calK}(\btheta_k)\Delta\btheta_k=\mu P^{\calK}(\btheta_k)\Delta\btheta_{k-1}.
    \label{eq:kernel_project_change}
\end{equation}
This identity shows that the kernel component of the update is driven purely by the momentum term. Since both the kernel space and the projector $P^{\calK}(\btheta_k)$ depend on $\btheta_k$, directly analyzing \eqref{eq:kernel_project_change} is challenging. To gain intuition, we consider a simplified setting where the SR matrix is fixed.

\paragraph{Heuristic analysis with fixed SR matrix.}
Assume $S(\btheta)\equiv S\in\R^{N_p\times N_p}$ with $0<\rank(S)<N_p$, and let $P^{\calK}$ denote the associated kernel projector. Then \eqref{eq:kernel_project_change} reduces to
\begin{equation}
P^{\calK}\Delta\btheta_k = \mu P^{\calK}\Delta\btheta_{k-1}=\cdots = \mu^k P^{\calK}\Delta\btheta_0.
    \label{eq:kernel_fix_delta_theta}
\end{equation}
Consequently, the kernel component of the parameter iterate satisfies
\begin{equation}
P^{\calK}\btheta_{k+1} = P^{\calK}\btheta_0+\sum_{m=0}^{k}\eta_m P^{\calK}\Delta\btheta_m =  P^{\calK}\btheta_0+\lrbracket{\sum_{m=0}^{k}\eta_m\mu^m}P^{\calK}\Delta\btheta_0.
    \label{eq:kernel_norm_mu}
\end{equation}

\par When $0\le \mu<1$, the exponentially decay in \eqref{eq:kernel_fix_delta_theta} typically keeps the kernel contribution bounded for standard step-size schedules. In contrast, when $\mu=1$, \eqref{eq:kernel_norm_mu} becomes
\begin{equation*}
    P^{\calK}\btheta_{k+1} = P^{\calK}\btheta_0 + \lrbracket{\sum_{m=0}^k\eta_m}P^{\calK}\Delta\btheta_0.
\end{equation*}
If $P^{\calK}\Delta\btheta_0\neq 0$ and the step-size sequence is not summable, the kernel component grows without bound. Since commonly used step-size schedules in VMC satisfy $\eta_k=\mathcal{O}(1/k)$, this heuristic argument explains why $\mu=1$ can lead to divergence.

\begin{remark}
    In standard implementations, one often initializes $\Delta\btheta_{-1}=0$, so that the first update reduces to an SR step and satisfies
    \begin{equation*}
        \Delta\btheta_0=\dfrac{\btheta_1 - \btheta_0}{\eta_0}= -\dfrac{1}{2}\lrbracket{\lambda I +S}^{-1}g(\btheta_0) \Longrightarrow~P^{\calK}\Delta\btheta_0=0.
    \end{equation*}
    This does not contradict the above mechanism: in practice $S(\btheta_k)$ varies with $k$, and finite-precision effects can introduce nonzero kernel components after the first iterations. For the heuristic argument, we therefore view $\btheta_1$ as given and analyze the subsequent dynamics of $\{\btheta_k\}_{k\ge 2}$.
\end{remark}

More generally, if there exists a subspace 
\begin{equation*}
    \calK_0 \subset \bigcap_{k}\calK(S(\btheta_k;\calB_k)),
\end{equation*}
then for \eqref{eq:p-spring}, we can similarly have
\begin{equation*}
    P^{\calK_0}\Delta\btheta_k=\mu^k P^{\calK_0}\Delta\btheta_0,\quad P^{\calK_0}\btheta_{k+1}=P^{\calK_0}\btheta_0 + \lrbracket{\sum_{m=0}^k \eta_m\mu^m}P^{\calK_0}\Delta\btheta_0,
\end{equation*}
where $P^{\calK_0}$ is the orthogonal projector onto $\calK_0$.

\paragraph{Explicit counterexamples.}
The heuristic mechanism above can be realized by explicit wavefunction constructions for which the SR matrix is independent of $\btheta$. Highlight that, under the stochastic setting, the divergence will still happen in some situations.

\medskip
\noindent\textbf{Continuous example.}
Consider the Gaussian-type wavefunction on $\R^{3N}$,
\begin{equation}
\label{eq:gaussian_wave}
\psi_{\btheta}(x)=e^{-\frac14 (x-A\btheta)^\top \Sigma^{-1}(x-A\btheta)},\qquad x\in\R^{3N},
\end{equation}
where $\Sigma\in\R^{3N\times 3N}$ is symmetric positive definite and $A\in\R^{3N\times N_p}$ satisfies $0<\rank(A)<N_p$. Then $\pi_{\btheta}$ is the Gaussian distribution $\calN(A\btheta,\Sigma)$. A direct computation gives
\begin{equation*}
    \nabla_{\btheta}\log\psi_{\btheta}(\bfx)=\frac12 A^\top \Sigma^{-1}(\bfx-A\btheta),\quad \Expect_{X\sim\pi_{\btheta}}[\nabla_{\btheta}\log\psi_{\btheta}(\bfx)]=0.
\end{equation*}
Hence the SR matrix is 
\begin{equation*}
    S(\btheta) = \frac14\Expect_{X\sim\pi_{\btheta}}[A^\top \Sigma^{-1}(X-A\btheta)(X-A\btheta)^\top \Sigma^{-1}A] =\frac14 A^\top \Sigma^{-1} A,
\end{equation*}
which is independent of $\btheta$, and satisfies $0< \rank(S)\le \rank(A)<N_p$. In the stochastic setting,
\begin{equation*}
    O(\btheta_k;X_{k,i},\calB_k)=-\dfrac{1}{2}A^\top \Sigma^{-1}\lrbracket{X_{k,i}-A\btheta_k-\dfrac{1}{N_s}\sum_{j=1}^{N_s}\lrbracket{X_{k,j}-A\btheta_k}},
\end{equation*}
so, since 
\begin{equation*}
S(\btheta_k;\calB_k)=\frac{1}{N_s-1}\sum_{i=1}^{N_s}O(\btheta_k;X_{k,i},\calB_k)O(\btheta_k;X_{k,i},\calB_k)^\top,    
\end{equation*}
we have $\calK(A)\subset \calK(S(\btheta_k;\calB_k))$ for any $k\ge 0$.

\medskip
\noindent\textbf{Discrete example.}
Consider the complex-valued wavefunction
\begin{equation}
\label{eq:discrete_complex_wave}
\psi_{\btheta}(\bfx)=2^{-N/2}e^{\i\,\bfx^\top A\btheta},\qquad \bfx\in\{-1,1\}^N,
\end{equation}
where $\i$ denotes the imaginary unit and $A\in\R^{N\times N_p}$ satisfies $0<\rank(A)<N_p$. In this case, $\pi_{\btheta}$ is the uniform distribution on $\{-1,1\}^N$. Consequently, we have $\Expect_{X\sim\pi_{\btheta}}[X]=0$ and $\Expect_{X\sim\pi_{\btheta}}[XX^\top]=I_N$. For complex-valued wavefunctions, the logarithmic derivative is defined by
\begin{equation*}
    \nabla_{\btheta}\log\psi_{\btheta}(\bfx):=\dfrac{\nabla_{\btheta}\psi_{\btheta}(\bfx)}{\psi_{\btheta}(\bfx)} = \i A^\top \bfx,
\end{equation*}
which satisfies $\Expect_{X\sim\pi_{\btheta}}[\nabla_{\btheta}\log\psi_{\btheta}(\bfx)]=0$. The SR matrix is typically defined via the Hermitian covariance of logarithmic derivatives, or equivalently by taking its real part (see \cite{becca2017quantum} for details). In particular,
\begin{align*}
S(\btheta)
&:=\Re\,\Expect_{X\sim\pi_{\btheta}}\left[O(\btheta;X) O(\btheta;X)^{*}\right]\\
&=\Re\,\Expect_{X\sim\pi_{\btheta}}\left[(\i A^\top X)\,(-\i X^\top A)\right] \\
&=\Expect_{X\sim\pi_{\btheta}}\left[A^\top (XX^\top)A\right]\\
&=A^\top A,
\end{align*}
where $(\cdot)^{*}$ denotes the conjugate transpose. Thus, $S(\btheta)$ is independent of $\btheta$ and satisfies $0<\rank(S)\le \rank(A)<N_p$. In the stochastic setting,
\begin{equation*}
    O(\btheta_k;X_{k,i},\calB_k)=\i A^\top \lrbracket{X_{k,i}-\dfrac{1}{N_s}\sum_{j=1}^{N_s}X_{k,j}},
\end{equation*}
and therefore $\calK(A)\subset \calK(S(\btheta_k;\calB_k))$.

\medskip
\par Combining these observations yields the following result.
\begin{The}
Let $\{\btheta_k\}_{k\ge 2}$ be generated by \eqref{eq:full-spring} or \eqref{eq:p-spring} with $\mu=1$. Then there exist a wavefunction and an initialization such that
\[
    \norm{\btheta_k}
    =
    \Omega\!\left(\sum_{m=1}^k \eta_m\right).
\]
\label{the:mu_1_divergence}
\end{The}

\par Theorem~\ref{the:mu_1_divergence} describes a worst-case scenario, revealing that the kernel-related directions are the source of instability. In contrast, for the two explicit counterexamples above, we can also establish convergence for $0\le \mu<1$: in the continuous case with the electronic Hamiltonian in Eq.~\eqref{eq:elec_ham}, and in the discrete case with a spin-lattice Hamiltonian.

\begin{The}
    Let $\psi_{\btheta}(\bfx)$ be given by Eq.~\eqref{eq:gaussian_wave} or Eq.~\eqref{eq:discrete_complex_wave}, and let the Hamiltonian $\scrH$ be the electronic Hamiltonian in Eq.~\eqref{eq:elec_ham} or a spin-lattice Hamiltonian. Let $\{\btheta_k\}$ be generated by \eqref{eq:full-spring}. Under \cref{assume:1} (3), for any $0\le \mu<1$, there exist $C>0$, such that, for any $K\ge 1$,
    $$\sum_{k=1}^K \eta_k \norm{g(\btheta_k)}^2 \le C.$$
    \label{the:counter_less_1_converge}
\end{The}
\begin{proof}
    See Appendix \ref{sec:proof_counter}.
\end{proof}

    \section{PRIME-SR: A Momentum-Adaptive SR Method}
\label{sec:adaptive}

\par In this section, we introduce Principal Range Informed MomEntum SR (PRIME-SR), a tuning-free momentum-adaptive variant of SPRING. The design of PRIME-SR is guided by the theoretical insight from Section~\ref{sec:divergence_mu_1}: under aggressive momentum reuse at $\mu=1$, instability arises from kernel-related directions. This suggests that momentum control should link to the sampled spectral information revealed by the SR matrix at the current iteration.

\par Our basic viewpoint is that momentum control should depend on two complementary aspects of the SR information: first, how the sampled spectrum is distributed; second, how reliable the associated sampled directions are. The first aspect is encoded by an effective spectral dimension indicator, while the second is characterized by a subspace-overlap indicator. PRIME-SR combines these two signals to determine the momentum parameter adaptively: stronger momentum reuse is encouraged only when both the spectral and range space information are favorable. Throughout this section, unless stated otherwise, we use the shorthand notations
\[
\bar{E}_k:=\bar{E}(\btheta_k;\calB_k),\qquad
O_k:=O(\btheta_k;\calB_k),\qquad
S_k:=S(\btheta_k;\calB_k),\qquad
g_k:=g(\btheta_k;\calB_k).
\]

\subsection{Design Principle and Effective Spectral Dimension}
\label{subsec:adaptive_ideal}

\par We begin with the spectral aspect, and let the SVD of $O_k\in\R^{N_p\times N_s}$ be
\begin{equation}
    O_k =
    \begin{bmatrix}
        U_k^{\mathcal R} & U_k^{\mathcal K}
    \end{bmatrix}
    \begin{bmatrix}
        \Sigma_k & 0\\
        0 & 0
    \end{bmatrix}
    \begin{bmatrix}
        (V_k^{\mathcal R})^\top\\
        (V_k^{\mathcal K})^\top
    \end{bmatrix},
    \qquad
    \Sigma_k=\mathrm{diag}(\varsigma_{k,1},\dots,\varsigma_{k,r_k}),
    \label{eq:O_svd_decomp}
\end{equation}
where $\varsigma_{k,1},\dots,\varsigma_{k,r_k}>0$ are the nonzero singular values of $O_k$. Therefore, the SR matrix $S_k = O_k O_k^\top$ has nonzero eigenvalues $\varsigma_{k,1}^2,\dots,\varsigma_{k,r_k}^2$, and $r_k=\text{rank}(S_k)$.

\par To quantify how the sampled spectrum is distributed, we introduce the following effective spectral dimension:
\begin{equation}
    \alpha_k
    :=
    \frac{\left(\sum_{i=1}^{r_k}\varsigma_{k,i}^2\right)^2}
         {\sum_{i=1}^{r_k}\varsigma_{k,i}^4}
    \le r_k ,
    \label{eq:beta_k}
\end{equation}
where the upper bound $r_k$ is from the Cauchy-Schwarz inequality. This quantity approximately measures the number of principal spectral directions. Indeed, if $\varsigma_{k,1}^2\gg \varsigma_{k,i}^2$ for all $i\neq 1$, then $\alpha_k\approx 1$; if all nonzero singular values are equal, then $\alpha_k=r_k$. Hence $\alpha_k$ may be viewed as an effective dimension of the sampled spectrum.

\par The relevance of $\alpha_k$ to momentum control is the following. A smaller $\alpha_k$ means that the sampled spectral information is concentrated on fewer dominant directions, which in turn suggests a potentially larger kernel-related subspace. In this regime, strong momentum reuse carries a higher risk of amplifying unstable components, and the momentum parameter should therefore be reduced. Conversely, a larger $\alpha_k$ indicates that the sampled spectral information is spread over more effective directions, which is more favorable for aggressive momentum reuse.

\par Figure~\ref{fig:beta_indicator} illustrates the behavior of $\alpha_k$ on the 1D-TFI model with $N=10$, including the full-batch reference and sampled runs with different sample sizes. All other experimental settings are the same as those in Section~\ref{sec:convergence_mu_0_1}. An important empirical observation is that $\alpha_k$ is relatively insensitive to the sample size and remains close to its full-batch counterpart. This indicates that $\alpha_k$ stably captures the spectral structure at the current iteration and is therefore suitable as an online stability indicator.

\begin{figure}[htbp]
    \centering
    \includegraphics[width=0.45\linewidth]{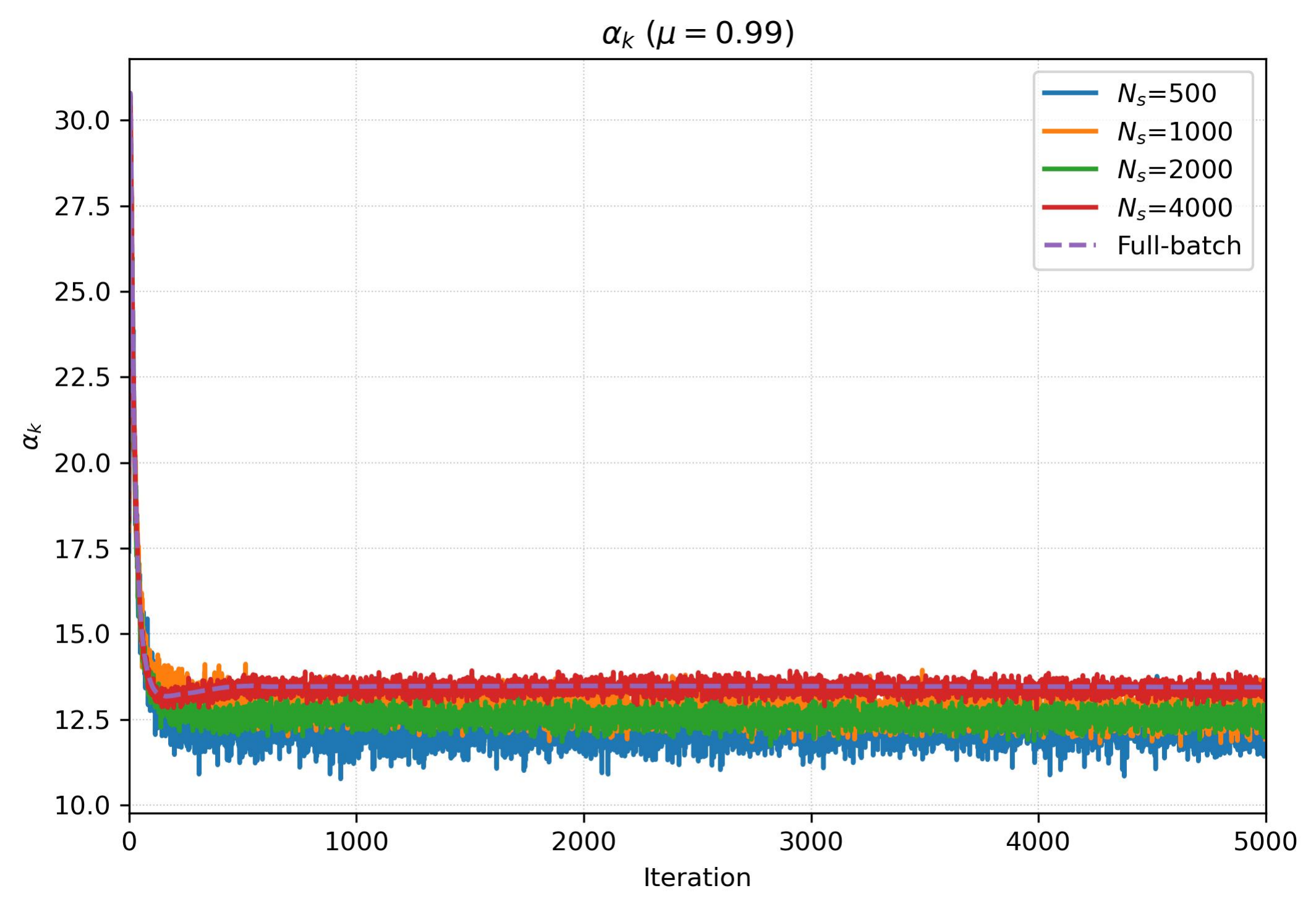}
    \caption{Effective spectral dimension $\alpha_k$ on the 1D-TFI model with $N=10$ under SPRING with $\mu=0.99$. The sampled values remain close to the full-batch reference across different sample sizes.}
    \label{fig:beta_indicator}
\end{figure}

\subsection{Subspace Overlap as a Reliability Indicator}

\par The indicator $\alpha_k$ describes only the sampled spectrum. It does not characterize whether the associated sampled directions are trustworthy. To address this second aspect, we now examine the range space information carried by the singular subspaces of $O_k$. From the decomposition in Eq.~\eqref{eq:O_svd_decomp}, the nonzero spectrum is shared by
\[
S_k = O_k O_k^\top \in \R^{N_p\times N_p}
\qquad\text{and}\qquad
T_k = O_k^\top O_k \in \R^{N_s\times N_s}.
\]
Meanwhile, the associated range space information is carried by the nonzero left and right singular subspaces:
\[
\mathrm{span}(U_k^{\mathcal R})=\mathcal R(S_k),
\qquad
\mathrm{span}(V_k^{\mathcal R})=\mathcal R(T_k).
\]
Since $\alpha_k$ already provides an effective dimension, we define the principal range spaces using the leading $\lceil\alpha_k\rceil$ directions. Specifically, let $U_{k,\alpha}^{\mathcal R}$ and $V_{k,\alpha}^{\mathcal R}$ denote the submatrices of $U_k^{\mathcal R}$ and $V_k^{\mathcal R}$ formed by the leading $\lceil\alpha_k\rceil$ columns.

\par The role of subspace overlap is to quantify the reliability of the sampled range space information. Ideally, such reliability should be measured by comparing the effective sampled subspace with the corresponding effective subspace of the true SR matrix $S(\btheta_k)$. In other words, the ideal quantity is the overlap between the principal range spaces of $S_k$ and $S(\btheta_k)$. Since $S(\btheta_k)$ is unavailable in practice, we introduce two practical surrogates:
\begin{itemize}
    \item the left-subspace surrogate, given by the overlap between the principal range spaces of $S_k$ and $S_{k-1}$;
    \item the right-subspace surrogate, given by the overlap between the principal range spaces of $T_k$ and $T_{k-1}$.
\end{itemize}
These are defined as
\begin{equation}
    \beta_k^{(U)}
    :=
    \left\|
    \left(U_{k,\alpha}^{\mathcal R}\right)^\top
    U_{k-1,\alpha}^{\mathcal R}
    \right\|_F,
    \qquad
    \beta_k^{(V)}
    :=
    \left\|
    \left(V_{k,\alpha}^{\mathcal R}\right)^\top
    V_{k-1,\alpha}^{\mathcal R}
    \right\|_F.
    \label{eq:left_right_overlap}
\end{equation}
These quantities measure the overall alignment between the principal sampled directions at two successive iterations. A small overlap indicates that the principal directions extracted from the two sampled matrices differ significantly, suggesting that sampling does not yet stably resolve the underlying effective subspace. By contrast, a large overlap indicates that the dominant sampled directions are consistently identified across iterations, suggesting that sampling is sufficiently informative to capture the underlying effective subspace more reliably. Therefore, larger values of $\beta_k^{(U)}$ or $\beta_k^{(V)}$ indicate more reliable sampled directions for momentum reuse.

Figure~\ref{fig:left_right_overlap} reports both indicators on the 1D-TFI model with $N=10$. We observe that both the left- and right-subspace overlaps clearly separate different sample sizes: larger sample sizes lead to larger overlap values. This confirms that both indicators reflect the reliability of sampled range space information.

\begin{figure}[htbp]
    \centering
    \begin{minipage}{0.45\linewidth}
        \centering
        \includegraphics[width=\linewidth]{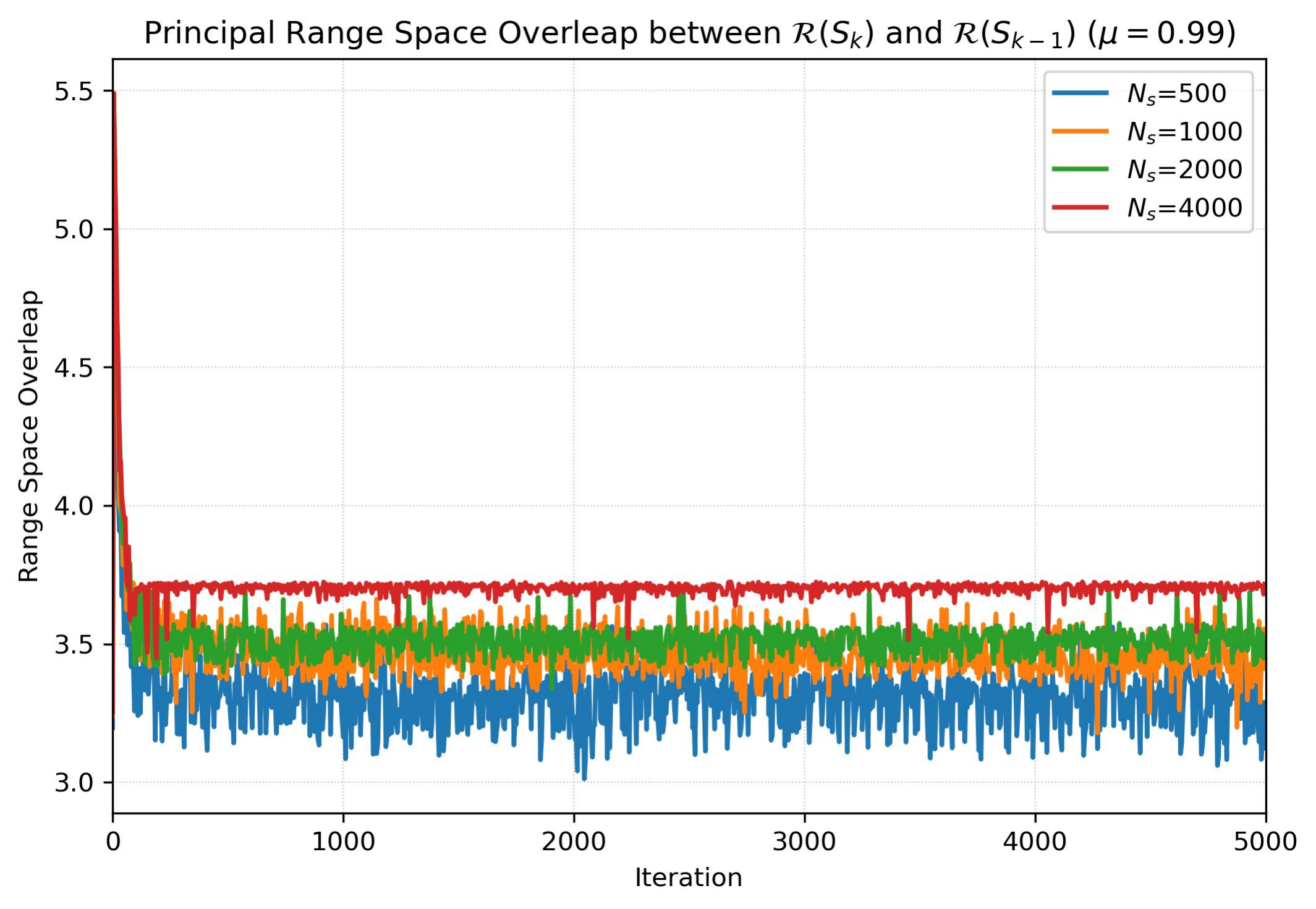}

        \vspace{0.2em}
        \small Left-subspace overlap: $\beta_k^{(U)}$
    \end{minipage}
    \hfill
    \begin{minipage}{0.45\linewidth}
        \centering
        \includegraphics[width=\linewidth]{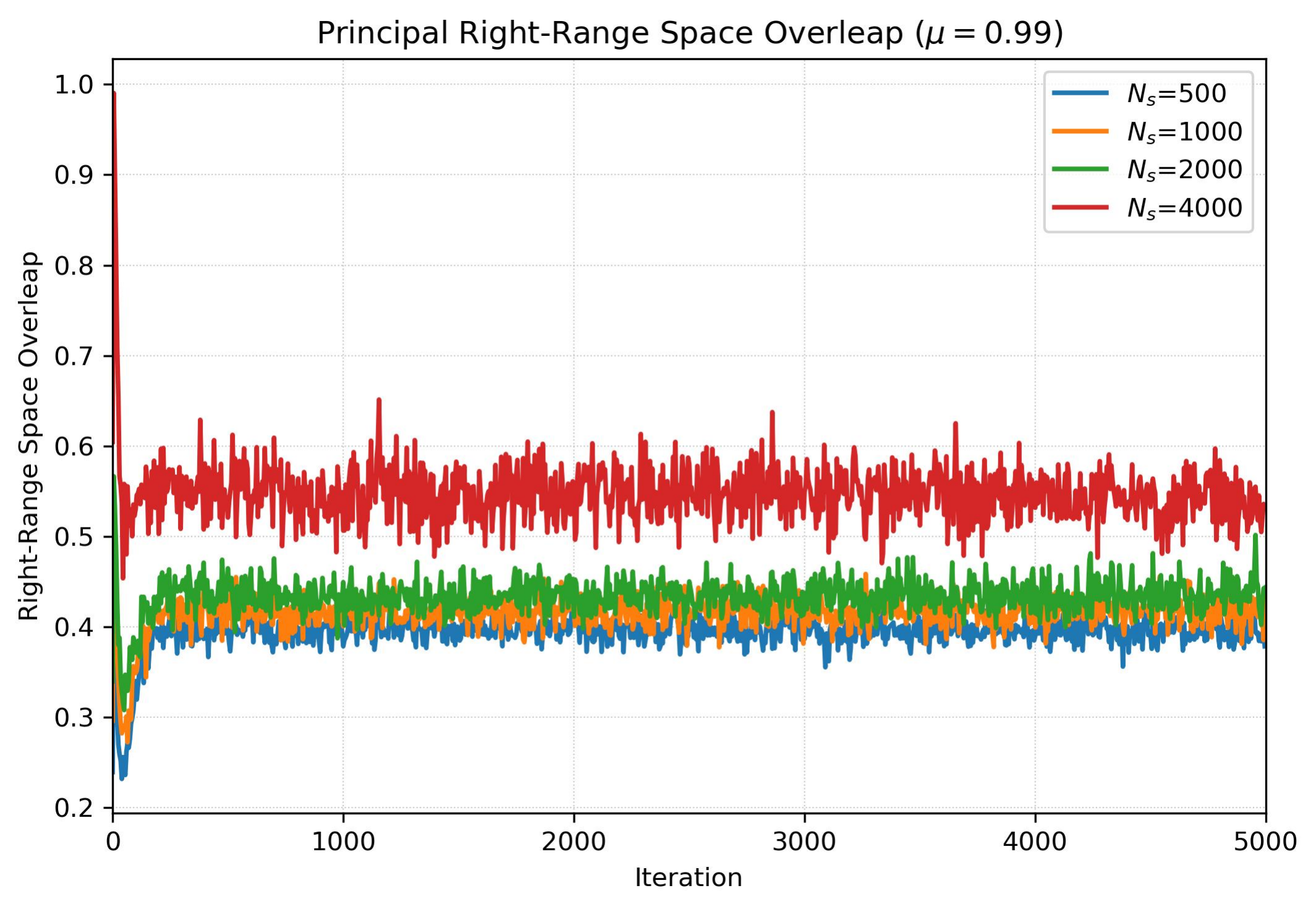}

        \vspace{0.2em}
        \small Right-subspace overlap: $\beta_k^{(V)}$
    \end{minipage}
    \caption{Left and right principal range space overlap indicators on the 1D-TFI model with $N=10$. Both indicators separate different sample sizes and thus reflect the reliability of sampled range space information.}
    \label{fig:left_right_overlap}
\end{figure}

\subsection{Practical Rule and the PRIME-SR Algorithm}
\label{subsec:adaptive_practical}

\par Both indicators in Eq.~\eqref{eq:left_right_overlap} are informative, but they differ in computational convenience. The left-subspace overlap is more directly connected to $\mathcal R(S_k)$. However, computing it requires access to the left singular vectors of $O_k$, which is expensive when $N_p$ is large. We therefore adopt the right-subspace viewpoint and work with the smaller matrix $T_k := O_k^\top O_k \in \R^{N_s\times N_s}$. Since $N_s\ll N_p$ in the applications considered here, this choice is substantially more efficient. Accordingly, PRIME-SR uses the right principal-range-space overlap as its practical reliability indicator:
\begin{equation}
    \tilde{\beta}_k
    :=
    \left\|
    \left(V_{k,\alpha}^{\mathcal R}\right)^\top
    V_{k-1,\alpha}^{\mathcal R}
    \right\|_F
    \in
    \left[0,\sqrt{\min\{\lceil\alpha_k\rceil,\lceil\alpha_{k-1}\rceil\}}\right].
    \label{eq:overleap_practice}
\end{equation}

\par We now combine the two signals into the momentum-control rule. The design principle is simple:
\begin{itemize}
    \item a smaller $\alpha_k$ indicates a potentially larger kernel-related subspace and therefore calls for smaller momentum;
    \item a larger $\tilde{\beta}_k$ indicates more reliable sampled directions and therefore supports larger momentum.
\end{itemize}
Hence, the adaptive momentum parameter $\mu_k$ in PRIME-SR should be increasing in both $\alpha_k$ and $\tilde{\beta}_k$. We use the following practical rule:
\begin{equation}
    \mu_k
    :=
    1-
    \left(
    1-\sqrt{
    \frac{\tilde{\beta}_k}{
    \sqrt{\min\{\lceil\alpha_k\rceil,\lceil\alpha_{k-1}\rceil\}}
    }}
    \right)
    \left(
    1-\left(\frac{\alpha_k}{r_k}\right)^{1/4}
    \right).
    \label{eq:adaptive_mu_practice}
\end{equation}
The resulting PRIME-SR update is
\begin{equation}
    \Delta\btheta_k
    =
    -O_k\left(\lambda I + O_k^\top O_k\right)^{-1}
    \left(\mu_k O_k^\top \Delta\btheta_{k-1} + \bar{E}_k\right)
    + \mu_k \Delta\btheta_{k-1}.
    \label{eq:prime_sr_update}
\end{equation}

\par For practical computation, we directly use the spectral decomposition $T_k = V_k \tilde{\Sigma}_k^2 V_k^\top$. Then $\alpha_k$ admits the equivalent form
\begin{equation}
    \alpha_k
    =
    \frac{\left(\sum_{i=1}^{r_k}\varsigma_{k,i}^2\right)^2}
         {\sum_{i=1}^{r_k}\varsigma_{k,i}^4}
    =
    \frac{\mathrm{tr}(\tilde{\Sigma}_k^2)^2}
         {\mathrm{tr}(\tilde{\Sigma}_k^4)},
    \label{eq:beta_k_practical}
\end{equation}
so it can be computed directly from the spectrum of $T_k$. The matrix $V_{k,\alpha}^{\mathcal R}$ is formed by the $\lceil\alpha_k\rceil$ leading eigenvectors in $V_k$, and $r_k$ is the numerical rank obtained by truncating the spectrum of $\tilde{\Sigma}_k^2$ using a tolerance $\varepsilon_r$. As in our implementation, we use the MATLAB-default choice in \texttt{rank} command: $\varepsilon_r = N_s \varepsilon_m$, where $\varepsilon_m$ denotes machine precision.

\par The additional computational and storage overhead of PRIME-SR relative to SPRING is modest. The extra cost is dominated by the computation of $\tilde{\beta}_k$, which requires $\mathcal O\left(N_s\lceil\alpha_k\rceil\lceil\alpha_{k-1}\rceil\right)$ operations per iteration, while the additional storage cost is $\mathcal O\left(N_s\lceil\alpha_k\rceil\right)$ for storing $V_{k,\alpha}^{\mathcal R}$. In our experiments, $N_s=1000$ and typically $\alpha_k\approx 20\sim 30$, so the overhead is relatively small. Algorithm~\ref{alg:adaptive_spring} summarizes PRIME-SR.

\begin{algorithm}[htbp]
    \caption{PRIME-SR}
    \label{alg:adaptive_spring}
    \KwIn{Initial parameter $\btheta_0\in\R^{N_p}$, sample size $N_s\in\N$, step-size schedule $\eta_k$, regularization parameter $\lambda>0$, norm-constraint parameter $C>0$ and rank tolerance $\varepsilon_r$.}

    $\Delta\btheta_{-1}:=\boldsymbol{0},\ \tilde{\beta}_0=1$;

    \While{the stopping criterion is not met}{
        Sample $\calB_k:=\{X_{k,1},\dots,X_{k,N_s}\}$ from $\pi_{\btheta_k}$;

        Construct $O_k\in\R^{N_p\times N_s}$ and $\bar{E}_k\in\R^{N_s}$ from $\calB_k$;

        Compute the eigendecomposition $T_k=O_k^\top O_k=V_k\tilde{\Sigma}_k^2V_k^\top$;

        Compute $\alpha_k$ via Eq.~\eqref{eq:beta_k_practical}, and compute $r_k$ using tolerance $\varepsilon_r$ on $\tilde{\Sigma}_k^2$;

        Form $V_{k,\alpha}^{\mathcal R}$ using the $\lceil\alpha_k\rceil$ leading eigenvectors of $T_k$;

        \If{$k\ge 1$}{
            Compute $\tilde{\beta}_k$ by Eq.~\eqref{eq:overleap_practice};
        }

        Compute $\mu_k$ by Eq.~\eqref{eq:adaptive_mu_practice};

        Update $\btheta_{k+1}$ by Algorithm \ref{alg:spring} with $\mu_k$;

        Record $\alpha_k$ and $V_{k,\alpha}^{\mathcal R}$, set $k:=k+1$, and update $\eta_k$.
    }

    \KwOut{Optimized parameters and energy.}
\end{algorithm}

\begin{remark}
At present, the convergence analysis developed in Section~\ref{sec:theory} does not directly extend to PRIME-SR, since the momentum parameter in PRIME-SR is updated adaptively, whereas our current theory is established for SPRING with a fixed momentum parameter. Nevertheless, the numerical results in the following section show that PRIME-SR provides a clear practical advantage: it substantially improves stability and robustness over fixed-$\mu$ SPRING, avoids manual tuning, and achieves performance comparable to, and in some cases better than, the optimally tuned SPRING.
\end{remark}
    \section{Numerical Experiments}
\label{sec:experiment}

\par In this section, we evaluate the proposed PRIME-SR method and compare it with SPRING using fixed, hand-tuned momentum parameters $\mu$. The goal of these experiments is not to claim uniform superiority over the best fixed choice of $\mu$, but rather to assess whether PRIME-SR can (i) achieve performance comparable to near-optimal fixed-$\mu$ SPRING runs, and (ii) provide improved robustness against unstable momentum choices in practical VMC settings.

\par All experiments are conducted on an Ubuntu 20.04.1 system equipped with an NVIDIA GeForce RTX 4090 GPU and CUDA 12.0. Spin-lattice experiments are implemented using NetKet~\cite{netket3:2022}, a VMC framework based on JAX \cite{jax2018github}, while electronic-structure experiments are performed using VMCNet \cite{vmcnet2024github}.

\par For spin-lattice models, we consider the TFI model and the Heisenberg model, both with periodic boundary conditions (see Appendix~\ref{sec:tfi_heisenberg_define} for details):
\begin{itemize}
    \item TFI model
    \begin{equation*}
        \scrH = -\sum_{<i,j>}\sigma_i^z\sigma_j^z - h\sum_{j}\sigma_j^x.
    \end{equation*}

    \item Heisenberg model
    \begin{equation*}
        \scrH = \sum_{<i,j>}\sigma_i^x\sigma_j^x+\sigma_i^y\sigma_j^y+\sigma_i^z\sigma_j^z.
    \end{equation*}
\end{itemize}
We test the 1D-Heisenberg model with $N=100$ sites, and the 2D-Heisenberg and 2D-TFI models with $N=10\times 10$ sites; for the 2D-TFI model we use transverse field strengths $h=2,3,4$. These settings follow \cite{armegioiu2025functional}. For all lattice experiments, we use RBM wavefunction with $D=5N$ hidden units, initialized with random seed \texttt{jax.random.PRNGKey(0)} and using 64-bit precision. Each run uses $K=10{,}000$ optimization iterations with $N_s=1000$ MCMC samples per iteration. We employ a burn-in of $3000$ steps, single-spin-flip proposals, and $10$ MCMC transitions between optimization iterations. We apply mean-centered local-energy clipping at $5$ standard deviations, following common VMC practice~\cite{pfau2020ab,goldshlager2024kaczmarz}. Reported energy trajectories are smoothed using a sliding window average over 100 iterations, as in \cite{goldshlager2024kaczmarz}.

\par For electronic systems, we consider $\mathrm{C}$, $\mathrm{N}$, and $\mathrm{O}$ atoms, as well as $\mathrm{LiH}$, $\mathrm{N}_2$ and $\mathrm{CO}$ molecules. For $\mathrm{N}_2$, we use the equilibrium bond length $2.016$ Bohr, and for $\mathrm{CO}$ we use $2.173$ Bohr, following \cite{goldshlager2024kaczmarz}. Benchmark ground-state energies for $\mathrm{C}$, $\mathrm{N}$, $\mathrm{O}$, $\mathrm{N}_2$, and $\mathrm{CO}$ are taken from \cite{goldshlager2024kaczmarz}, while the benchmark for $\mathrm{LiH}$ is taken from \cite{pfau2020ab}. A standard FermiNet architecture with 16 dense determinants and 32-bit precision is used, consistent with~\cite{goldshlager2024kaczmarz}. Each electronic experiment runs for $K=100{,}000$ iterations with $N_s=1000$ samples per iteration. Gaussian all-electron proposals are used, with adaptive standard deviation to maintain an acceptance ratio of approximately 50\% (See Appendix \ref{sec:ferminet_detail} for details). Reported energy trajectories are smoothed using a sliding window average over 3000 iterations. 

\par Across all experiments, we use the inverse-time step-size schedule $\eta_k=\dfrac{\eta_0}{1+ck}$ with $c=10^{-4}$. We set $\eta_0=0.02$ for all spin-lattice systems and for $\mathrm{C}$, $\mathrm{N}$, $\mathrm{O}$, and $\mathrm{LiH}$, and $\eta_0=0.002$ for $\mathrm{N}2$ and $\mathrm{CO}$, matching \cite{goldshlager2024kaczmarz}. We use regularization $\lambda=10^{-3}$ and a norm constraint $C=10^{-3}$ throughout. For fixed-$\mu$ 
SPRING, we tune $\mu=0,0.2,0.4,0.8,0.9,0.95,0.995$ for spin-lattice models, and we exclude $\mu=0.995$ for electronic systems due to stability issues. For PRIME-SR, we set the rank tolerance as $\varepsilon_r=N_s\varepsilon_m$, where $\varepsilon_m$ denotes machine precision.

\subsection{Experiments on Spin-Lattice Models}

\par Figures~\ref{fig:heisenberg_energy_mu} and \ref{fig:2D_Ising_adaptive} report results for the 2D-TFI model with $h=2,3,4$ and for the 1D/2D-Heisenberg models. The corresponding $\mu_k$ trajectories are plotted every 100 iterations. Across all tested lattice systems, PRIME-SR consistently reaches energy levels comparable to those obtained by the optimal or near-optimal fixed-$\mu$ runs. At the same time, the fixed-$\mu$ baselines exhibit a clear spread in performance across different choices, indicating that careful tuning is important for obtaining strong results.

\begin{figure}[H]
\centering
\begin{subfigure}{\linewidth}
\centering
\includegraphics[width=0.48\linewidth]{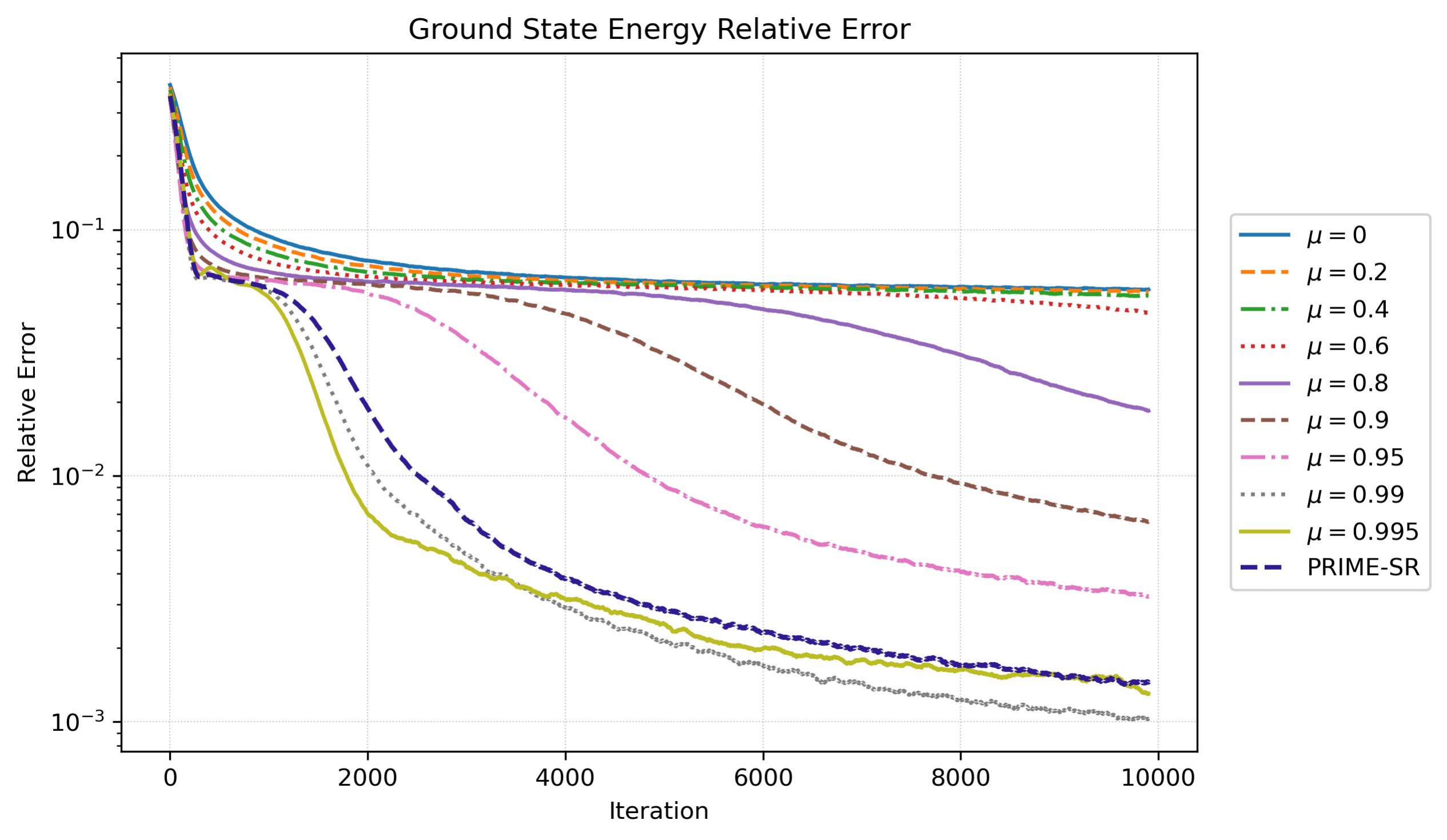}
\hfill
\includegraphics[width=0.42\linewidth]{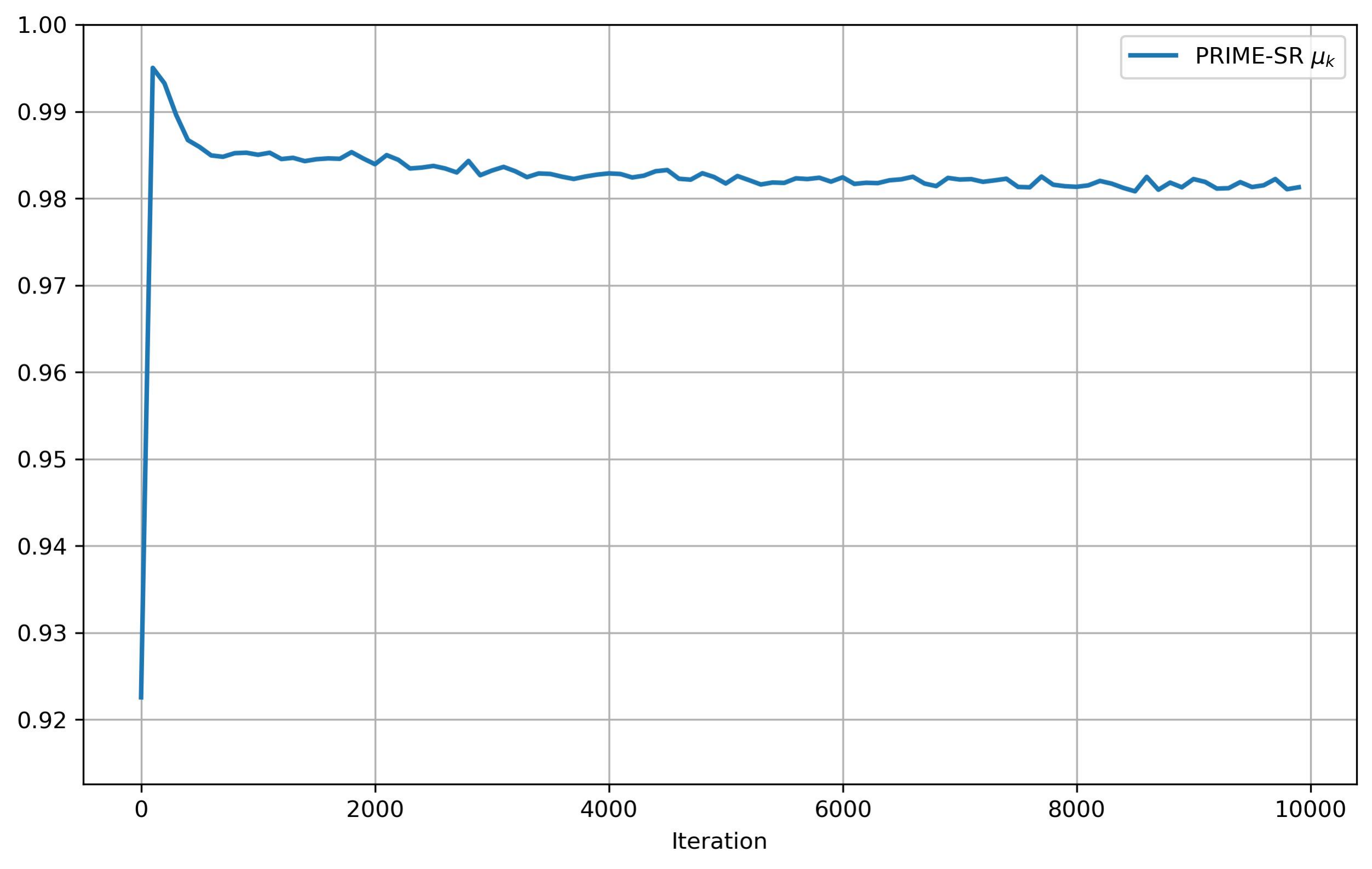}
\caption{1D-Heisenberg model with $N=100$ sites. Left: relative energy error. Right: $\mu_k$.}
\end{subfigure}

\vspace{0.5em}

\begin{subfigure}{\linewidth}
\centering
\includegraphics[width=0.48\linewidth]{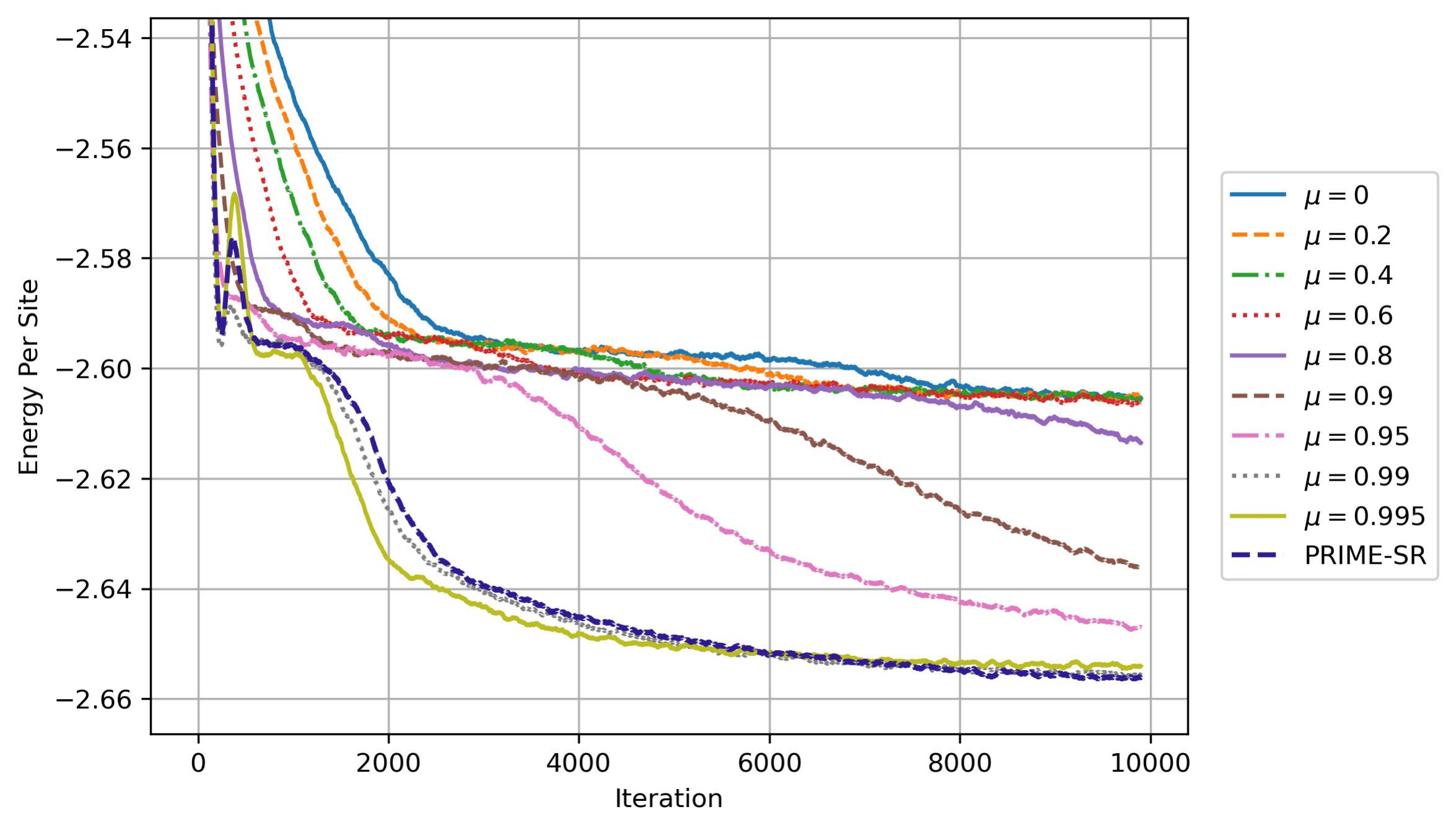}
\hfill
\includegraphics[width=0.42\linewidth]{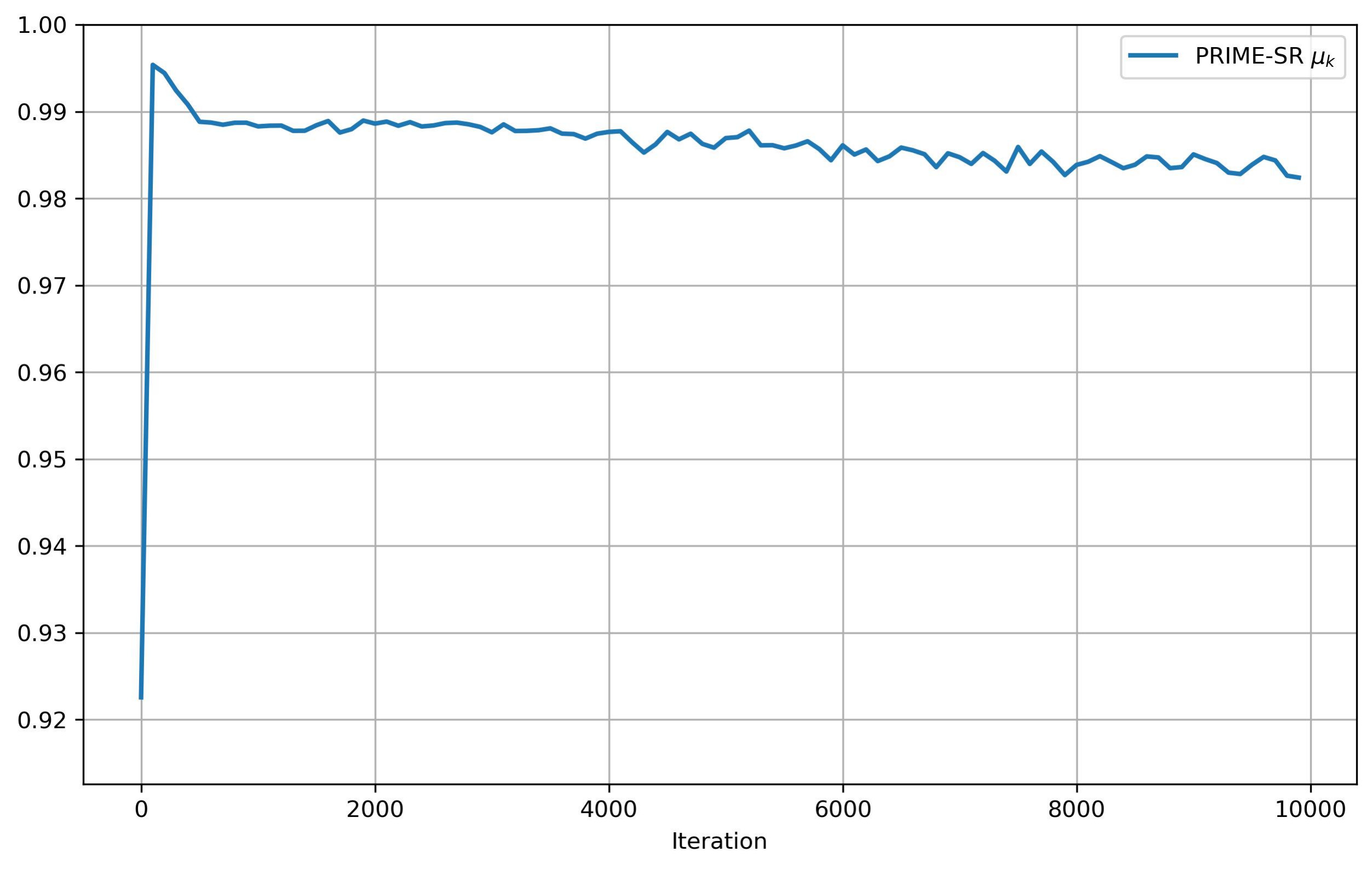}
\caption{2D-Heisenberg model with $N=10\times10$ sites. Left: relative energy error. Right: $\mu_k$.}
\end{subfigure}

\caption{PRIME-SR on the Heisenberg model. Left column: relative energy error. Right column: adaptive momentum parameter $\mu_k$.}

\label{fig:heisenberg_energy_mu}

\end{figure}

\par Values of $\mu$ that perform well on the 2D-TFI problems are not always equally effective on the Heisenberg models, and vice versa. This model dependence makes it difficult to choose a single fixed momentum that works reliably across different systems. In contrast, PRIME-SR remains competitive throughout these tests without requiring problem-specific tuning, which is precisely the practical advantage of the proposed adaptive strategy.

\subsection{Sensitivity of Initialization on Electronic Systems}
\label{sec:initial_sensitive}

\par For electronic systems, we observe that fixed-$\mu$ SPRING is also sensitive to initialization, i.e., to the initial parameter $\btheta_0$, especially for $\mu\ge 0.9$. We test five initializations corresponding to random seeds \texttt{jax.random.PRNGKey($i$)} for $i=0,1,2,3,4$. Fig.~\ref{fig:spring_senstive_initial} reports results for $\mathrm{N}_2$ and $\mathrm{CO}$ molecules with $\mu=0.9,0.95$. The dependence on initialization is pronounced: for some seeds, the optimization remains stable and reaches low energy error, whereas for others, the same hyperparameter setting leads to visibly worse behavior and can even lead to divergence. This indicates that, in the electronic setting, the performance of fixed-$\mu$ SPRING depends not only on the choice of $\mu$, but also strongly on the initialization.

\begin{figure}[H]
\centering

\begin{subfigure}{\linewidth}
\centering
\includegraphics[width=0.48\linewidth]{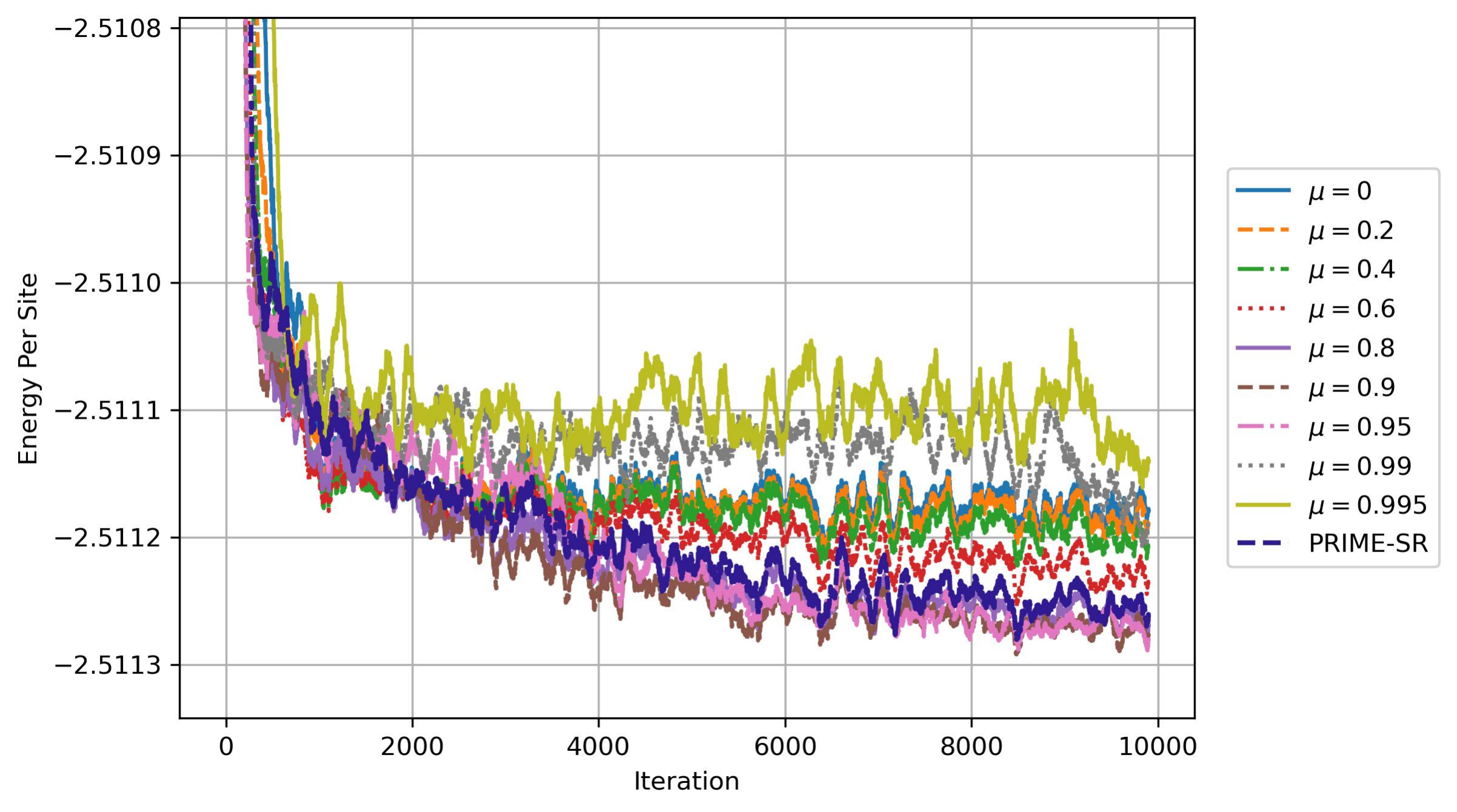}
\hfill
\includegraphics[width=0.42\linewidth]{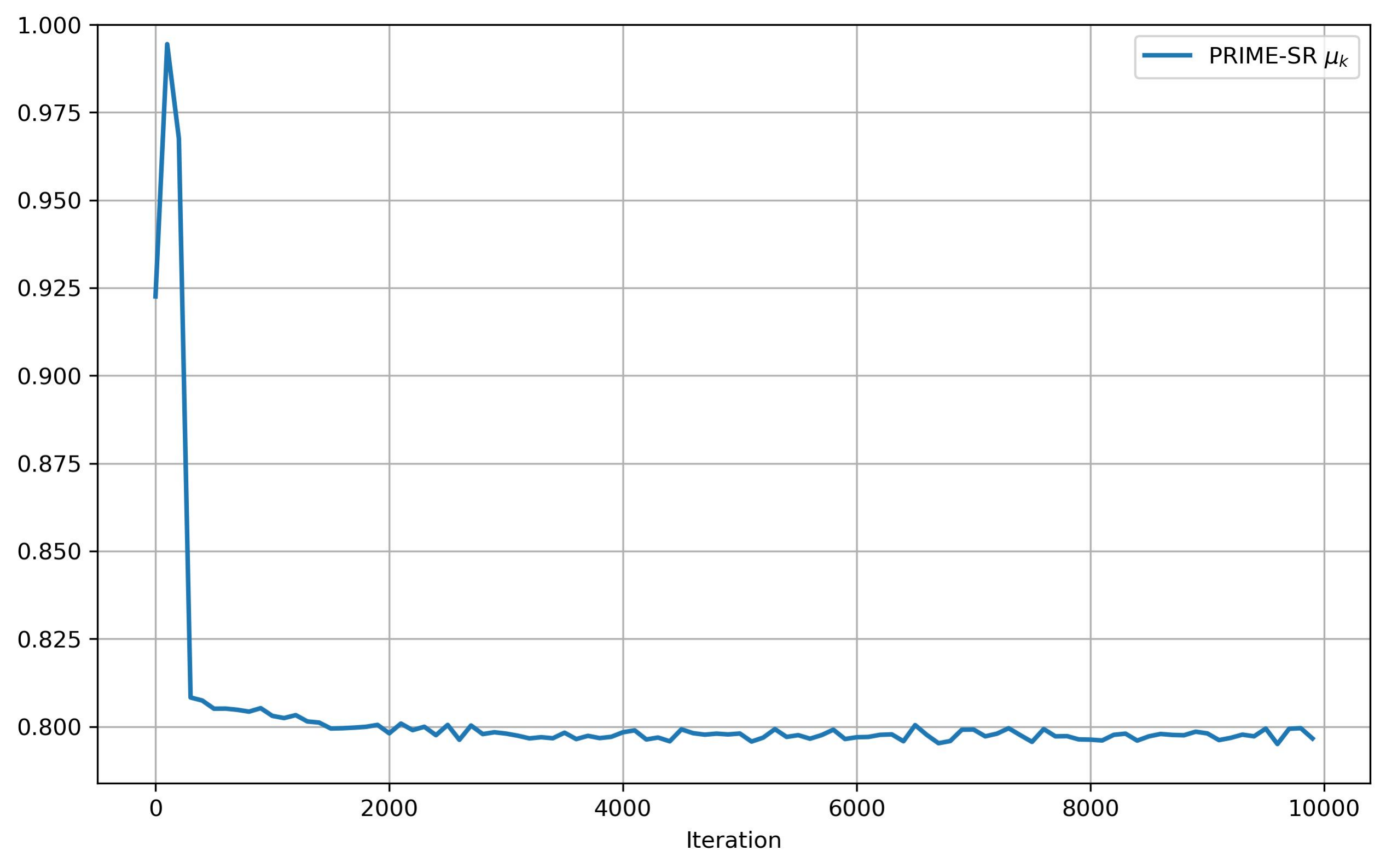}
\caption{$h=2$. Left: relative energy error. Right: $\mu_k$.}
\end{subfigure}

\vspace{0.5em}

\begin{subfigure}{\linewidth}
\centering
\includegraphics[width=0.48\linewidth]{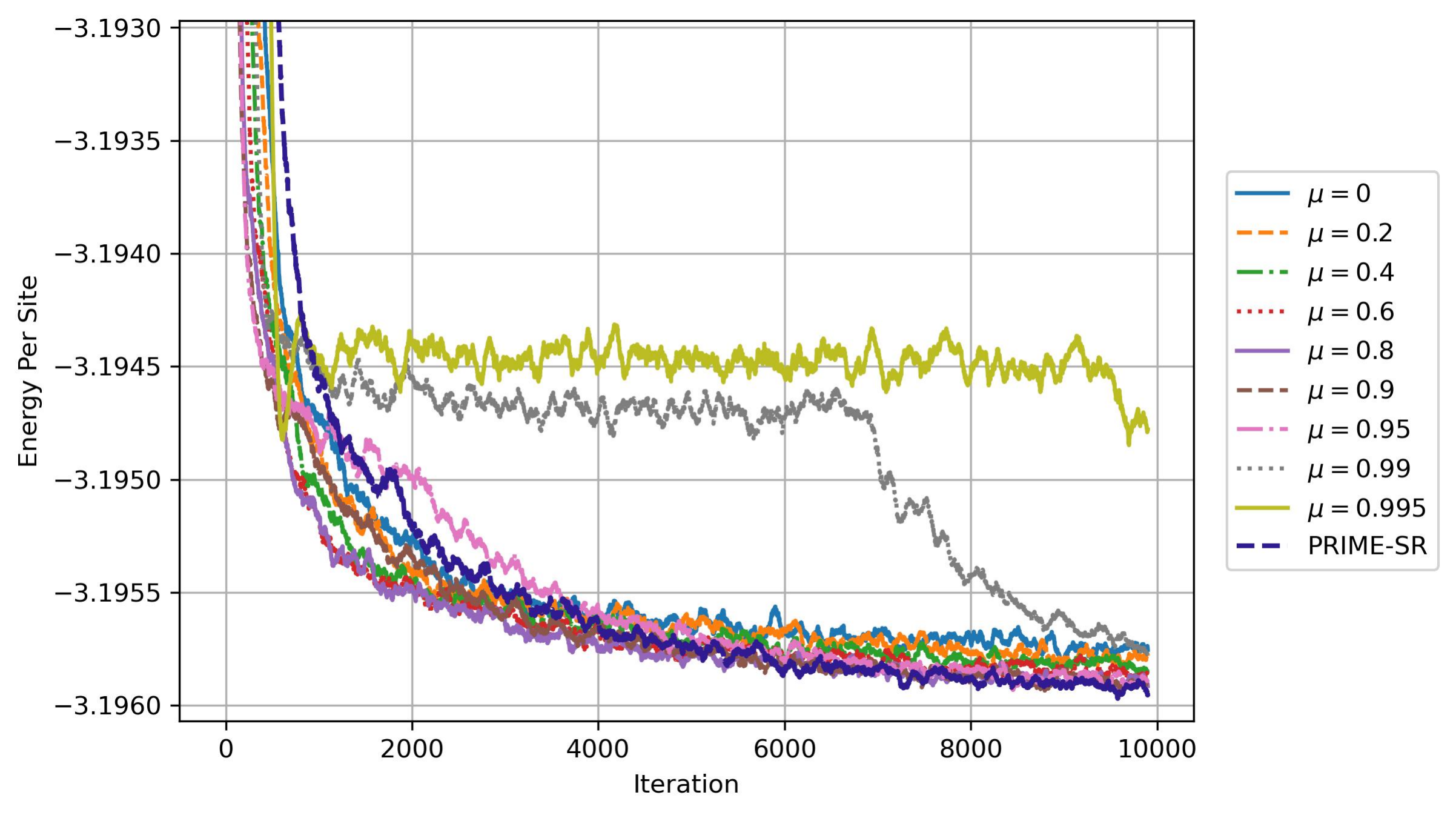}
\hfill
\includegraphics[width=0.42\linewidth]{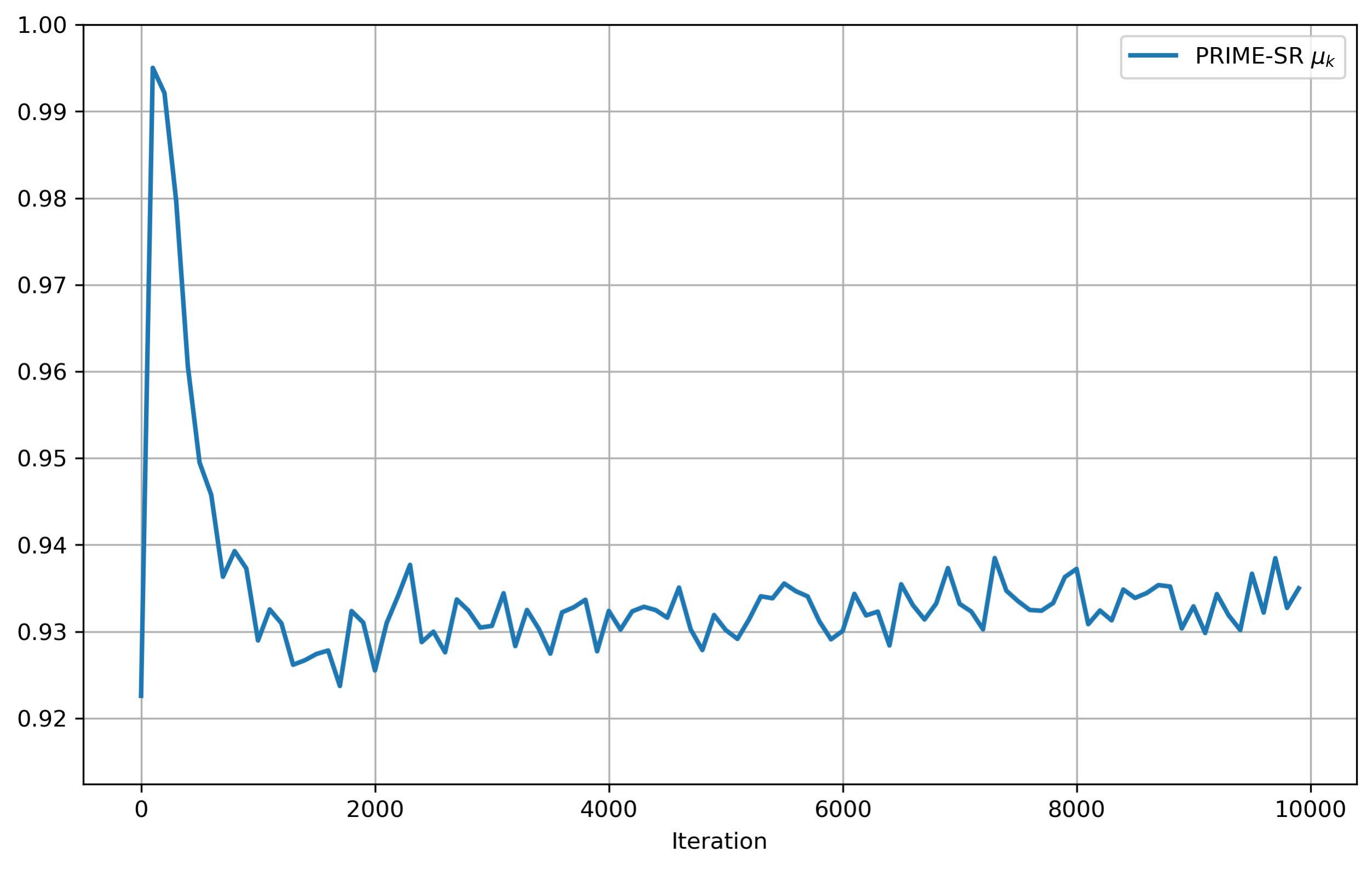}
\caption{$h=3$. Left: relative energy error. Right: $\mu_k$.}
\end{subfigure}

\vspace{0.5em}

\begin{subfigure}{\linewidth}
\centering
\includegraphics[width=0.48\linewidth]{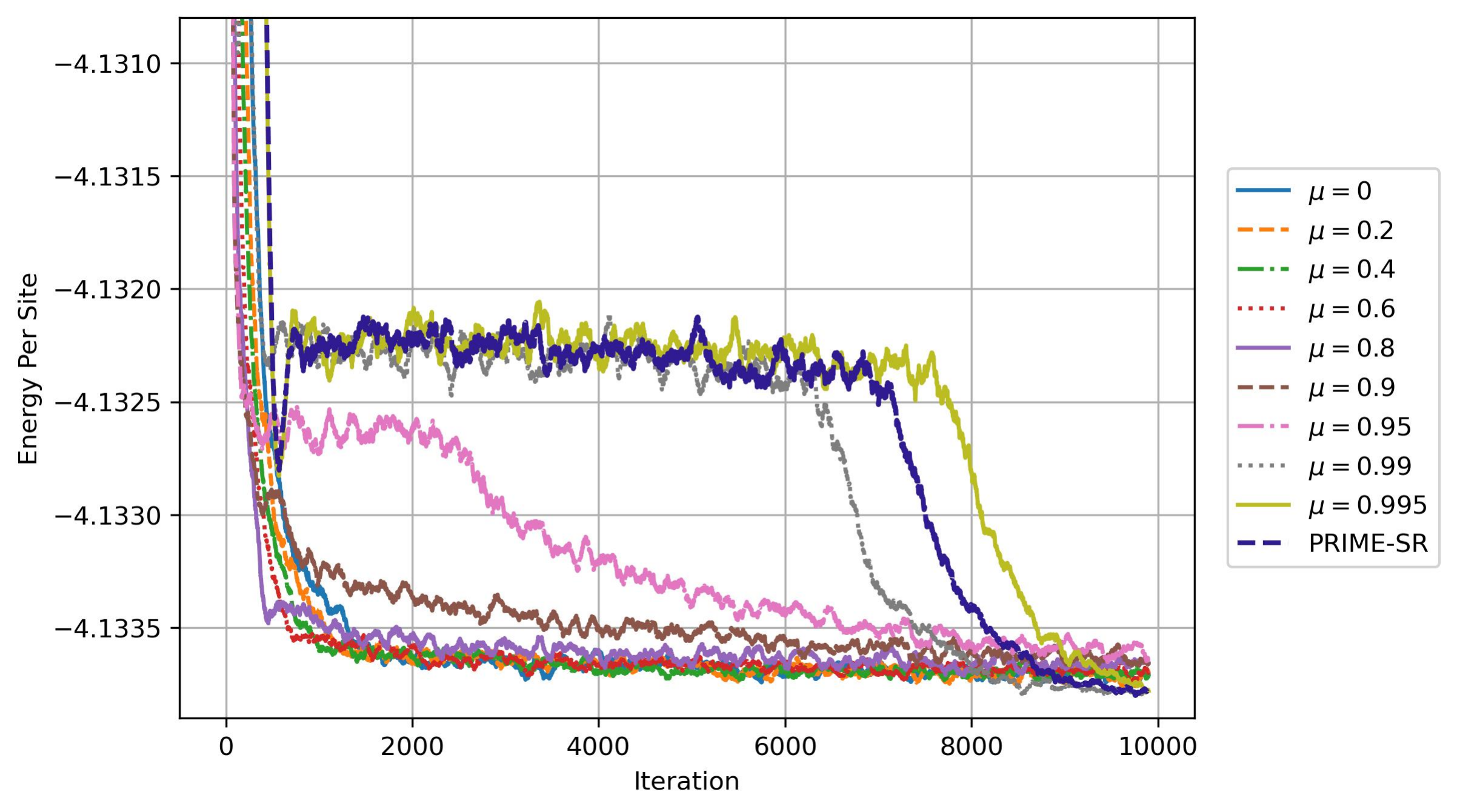}
\hfill
\includegraphics[width=0.42\linewidth]{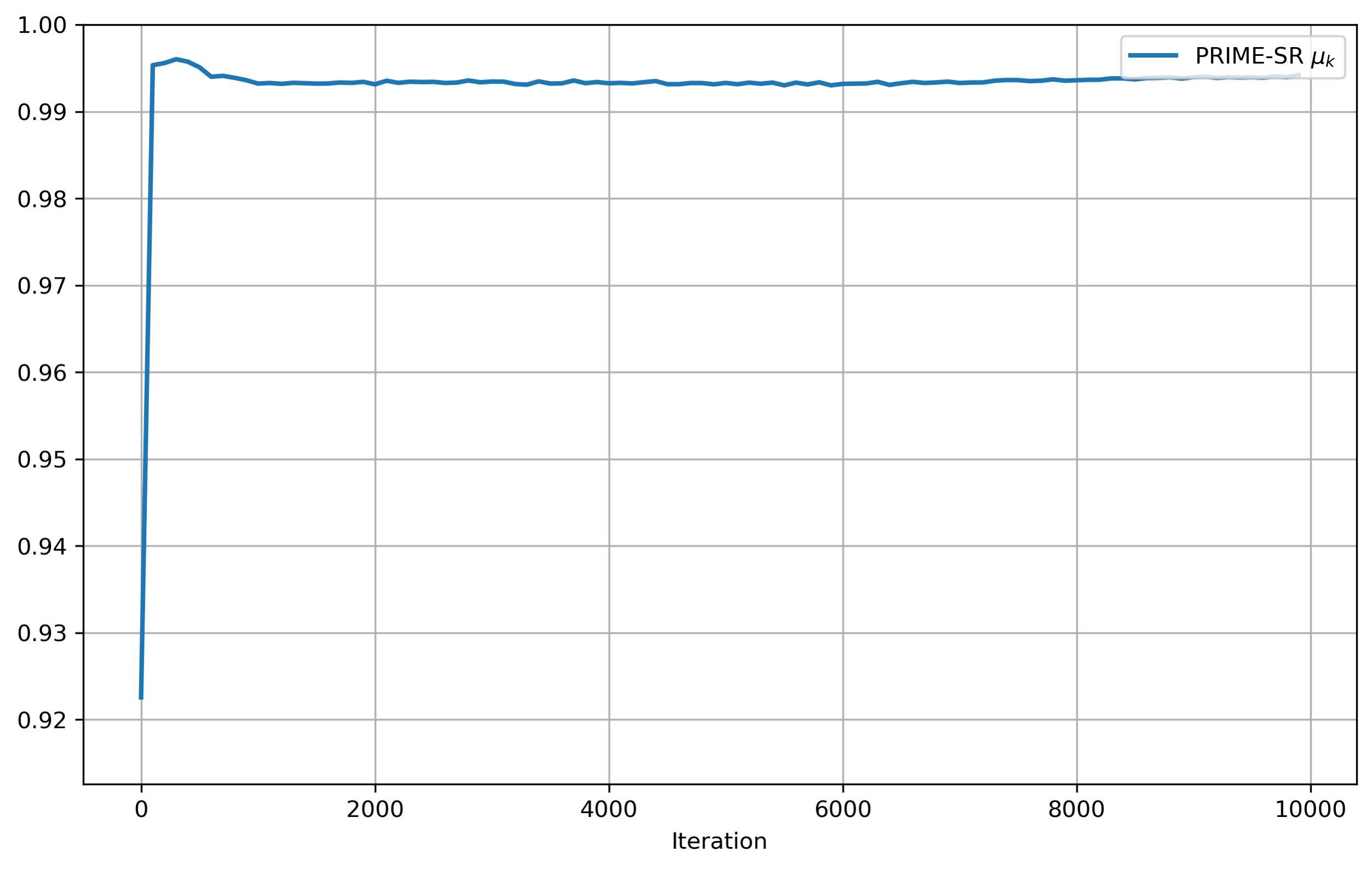}
\caption{$h=4$. Left: relative energy error. Right: $\mu_k$.}
\end{subfigure}

\caption{PRIME-SR on the 2D-TFI model with $N=10\times10$ sites. Left column: relative energy error. Right column: adaptive momentum parameter $\mu_k$.}

\label{fig:2D_Ising_adaptive}

\end{figure}

% ==========

\par As a direct comparison, Fig.~\ref{fig:adaptive_senstive_initial} shows the corresponding results of PRIME-SR on the $\mathrm{N}_2$ and $\mathrm{CO}$ molecules across the same set of random seeds. In contrast to fixed-$\mu$ SPRING, all runs remain stable and attain similar final accuracies. The reduction in seed sensitivity is one of the main practical advantages of PRIME-SR and highlights its substantially improved robustness in the electronic setting.

\subsection{Experiments on Atomic Systems}

\par We next report experiments on the atoms $\mathrm{C}$, $\mathrm{N}$, and $\mathrm{O}$, comparing fixed-$\mu$ SPRING and PRIME-SR. Due to instability at large $\mu$, we exclude some $\mu$ values from the comparison. Specifically, for $\mathrm{C}$ atom, we compare against $\mu=0,0.2,0.4,0.6,0.8,0.9,0.95,0.99$, while for $\mathrm{N}$ and $\mathrm{O}$ atoms we use $\mu=0,0.2,0.4,0.6,0.8,0.9,0.95$. Unstable runs for $\mathrm{N}$ and $\mathrm{O}$ atoms with $\mu=0.99$ are reported in Appendix~\ref{sec:spring_unstable_N_O}. Figure~\ref{fig:compare_spring_atom_seed_0} shows results for random seed 0, and the results of additional seeds are provided in Appendix~\ref{sec:atom_random_seed}. The trajectories of $\mu_k$ are plotted every 1000 iterations.

\begin{figure}[H]
    \centering
    \begin{subfigure}[t]{0.48\linewidth}
        \centering
        \includegraphics[width=\linewidth]{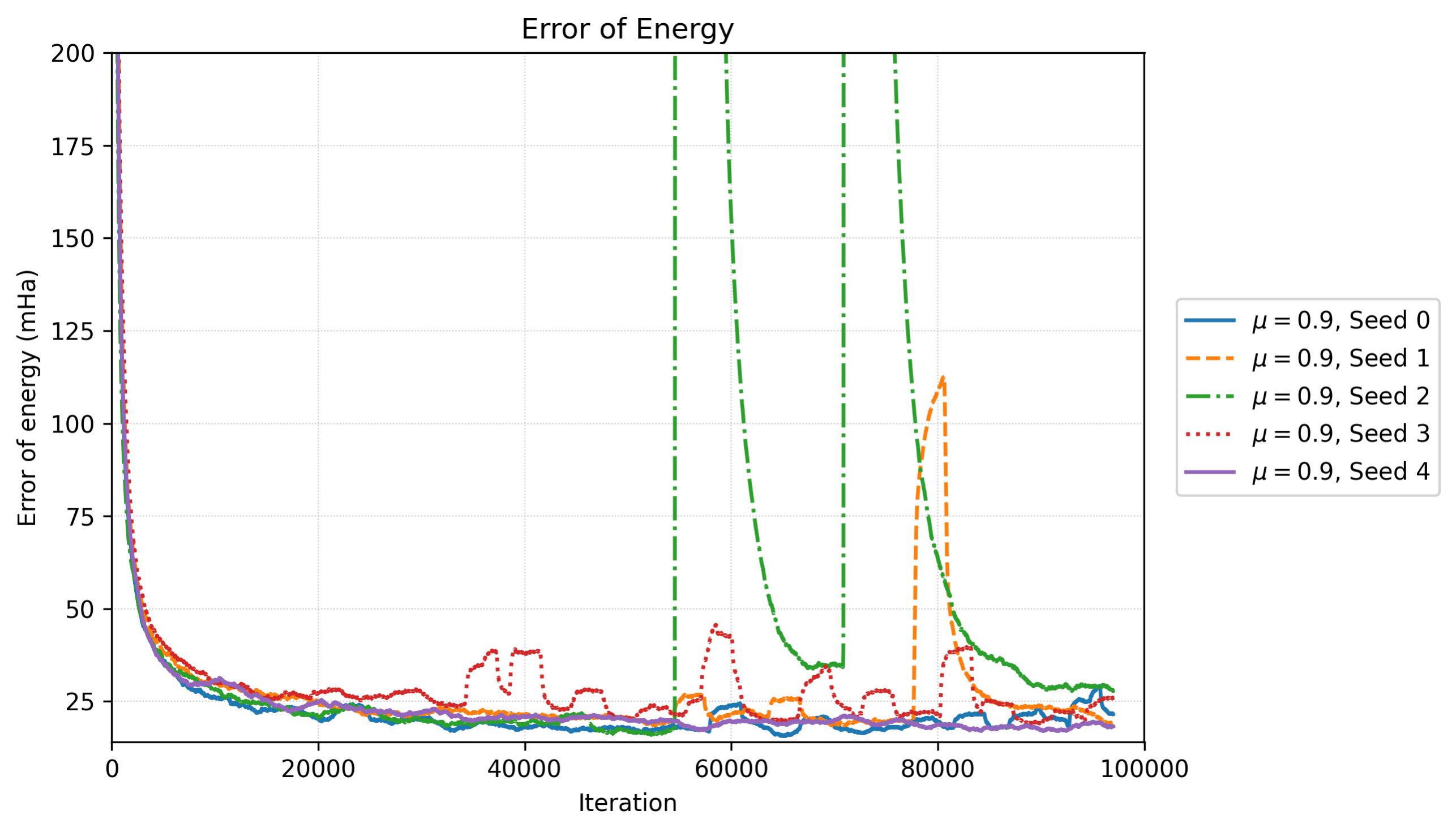}
        \caption{$\mathrm{N}_2$, $\mu=0.9$. Relative energy error.}
    \end{subfigure}
    \hfill
    \begin{subfigure}[t]{0.48\linewidth}
        \centering
        \includegraphics[width=\linewidth]{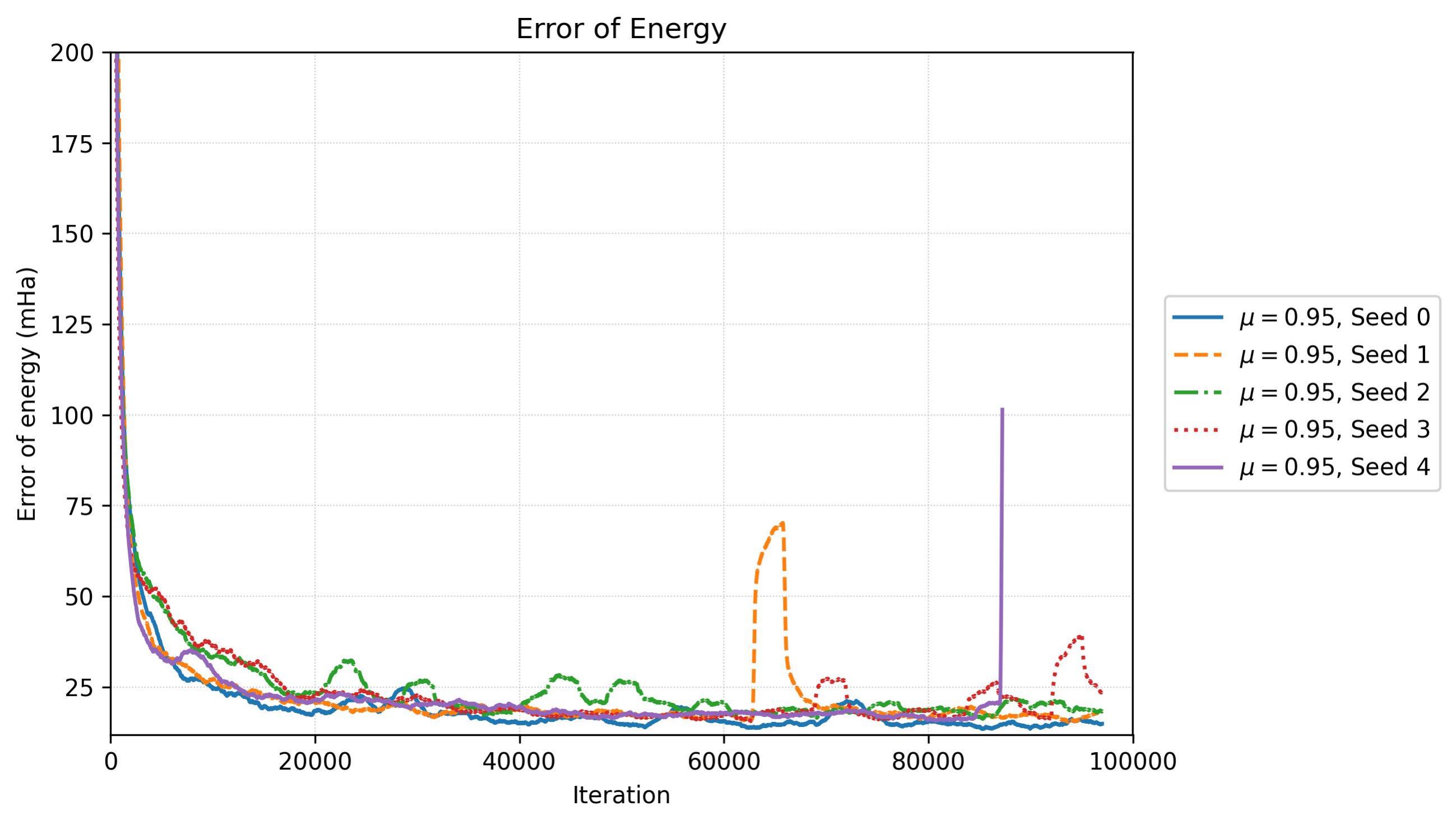}
        \caption{$\mathrm{N}_2$, $\mu=0.95$. Relative energy error.}
    \end{subfigure}

    \vspace{0.5em}

    \begin{subfigure}[t]{0.48\linewidth}
        \centering
        \includegraphics[width=\linewidth]{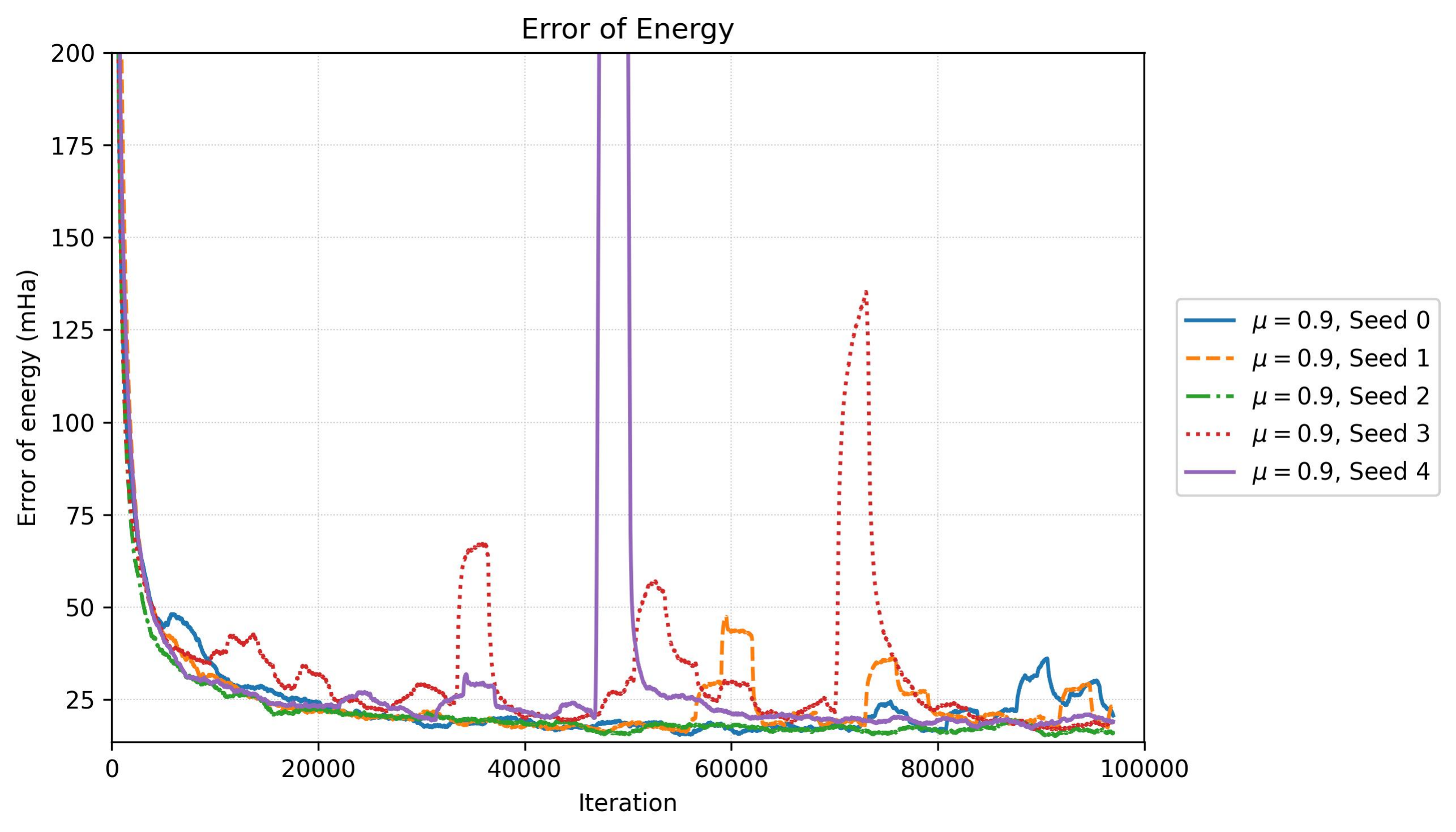}
        \caption{$\mathrm{CO}$, $\mu=0.9$. Relative energy error.}
    \end{subfigure}
    \hfill
    \begin{subfigure}[t]{0.48\linewidth}
        \centering
        \includegraphics[width=\linewidth]{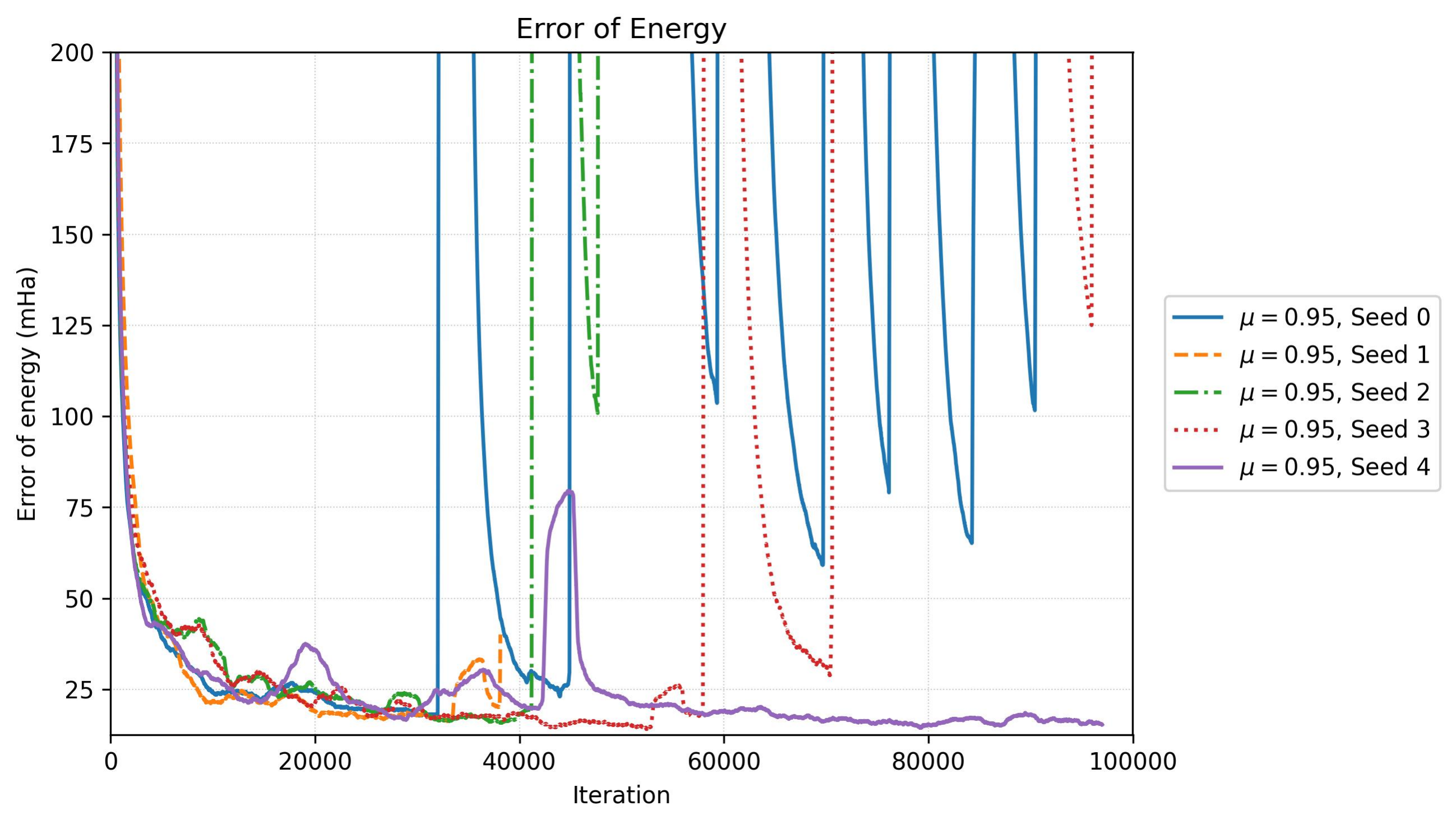}
        \caption{$\mathrm{CO}$, $\mu=0.95$. Relative energy error.}
    \end{subfigure}

    \caption{Sensitivity of fixed-$\mu$ SPRING to initialization on $\mathrm{N}_2$ (top row) and $\mathrm{CO}$ (bottom row) for $\mu=0.9,0.95$.}

    \label{fig:spring_senstive_initial}
\end{figure}

\begin{figure}[H]
    \centering
    \begin{subfigure}[t]{0.48\textwidth}
        \centering
        \includegraphics[width=\linewidth]{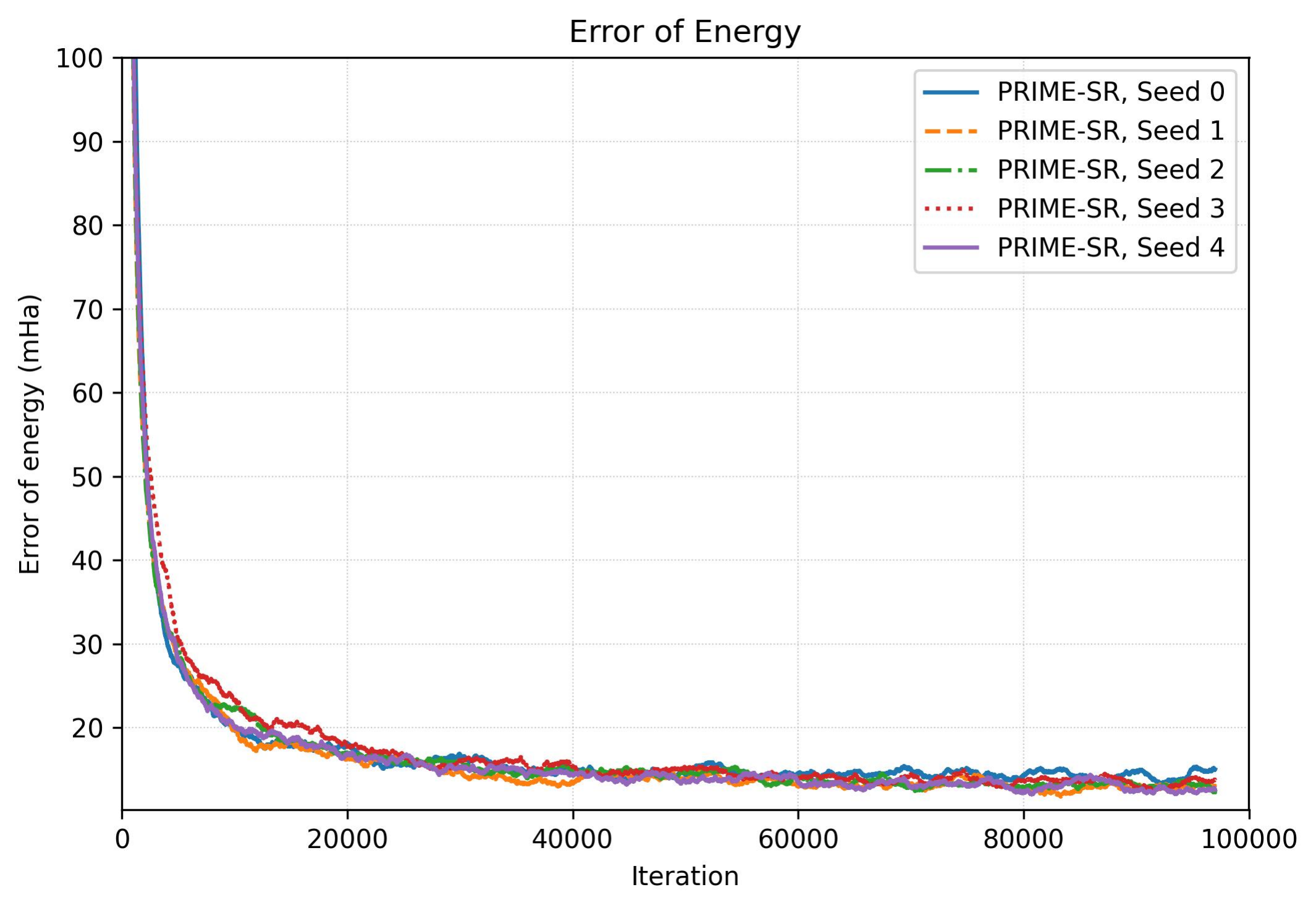}
        \caption{$\mathrm{N}_2$. Relative energy error.}
    \end{subfigure}
    \hfill
    \begin{subfigure}[t]{0.48\textwidth}
        \centering
        \includegraphics[width=\linewidth]{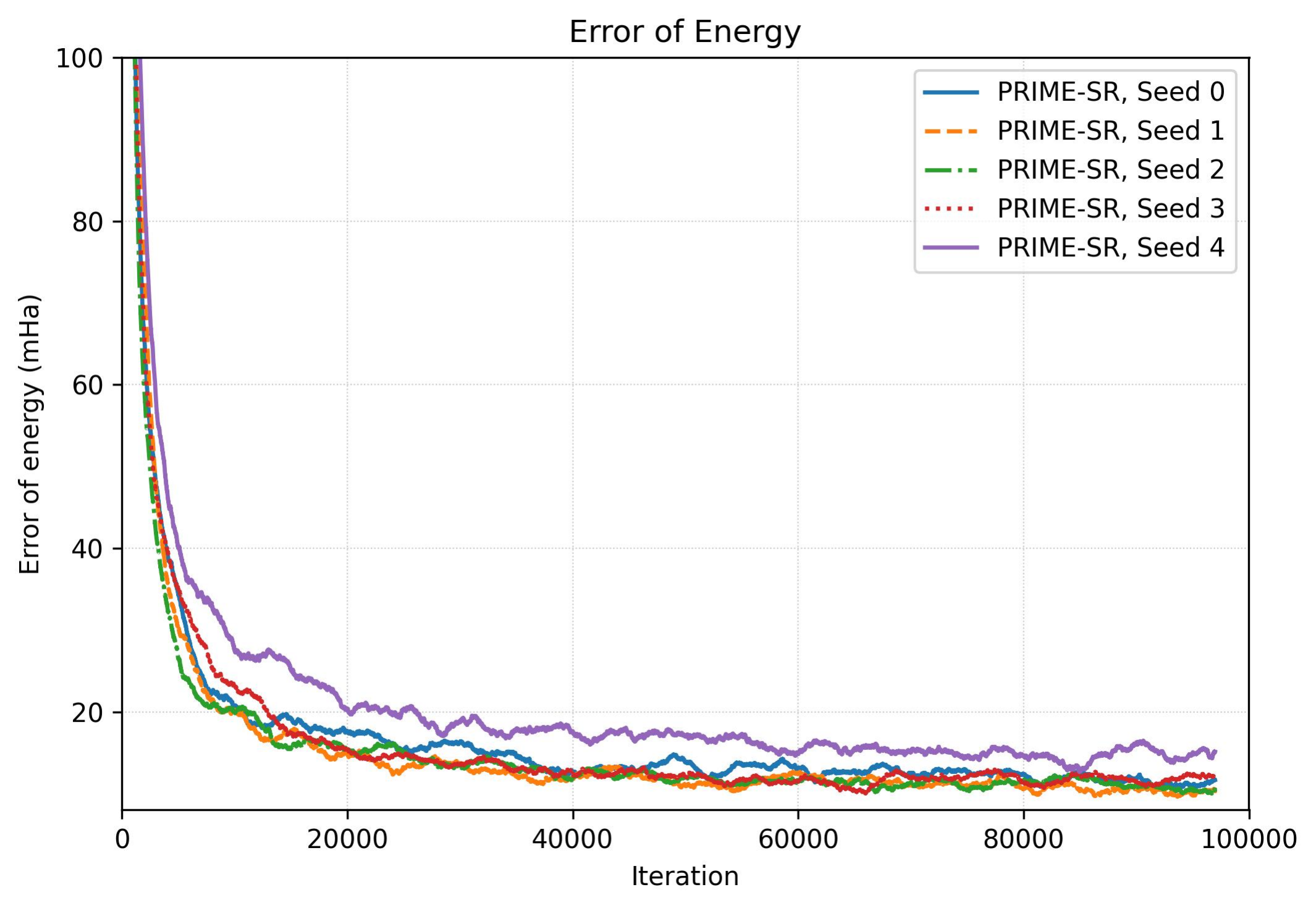}
        \caption{$\mathrm{CO}$. Relative energy error.}
    \end{subfigure}
    \caption{Sensitivity of PRIME-SR to initialization on $\mathrm{N}_2$ (left) and $\mathrm{CO}$ (right).}
    \label{fig:adaptive_senstive_initial}
\end{figure}

\par Across all three atomic systems, PRIME-SR remains stable throughout training and reaches final accuracies that are comparable to the optimal or near-optimal fixed-$\mu$ baselines. In particular, it consistently avoids the unstable large $\mu$ behavior while retaining the optimization benefits of momentum reuse. Together with the multi-seed results reported in Appendix~\ref{sec:compare_spring_random_seed}, these experiments indicate that PRIME-SR improves robustness to initialization without introducing a noticeable loss in accuracy.

\begin{figure}[H]
    \centering
    
    \begin{subfigure}{\linewidth}
    \centering
    \includegraphics[width=0.48\linewidth]{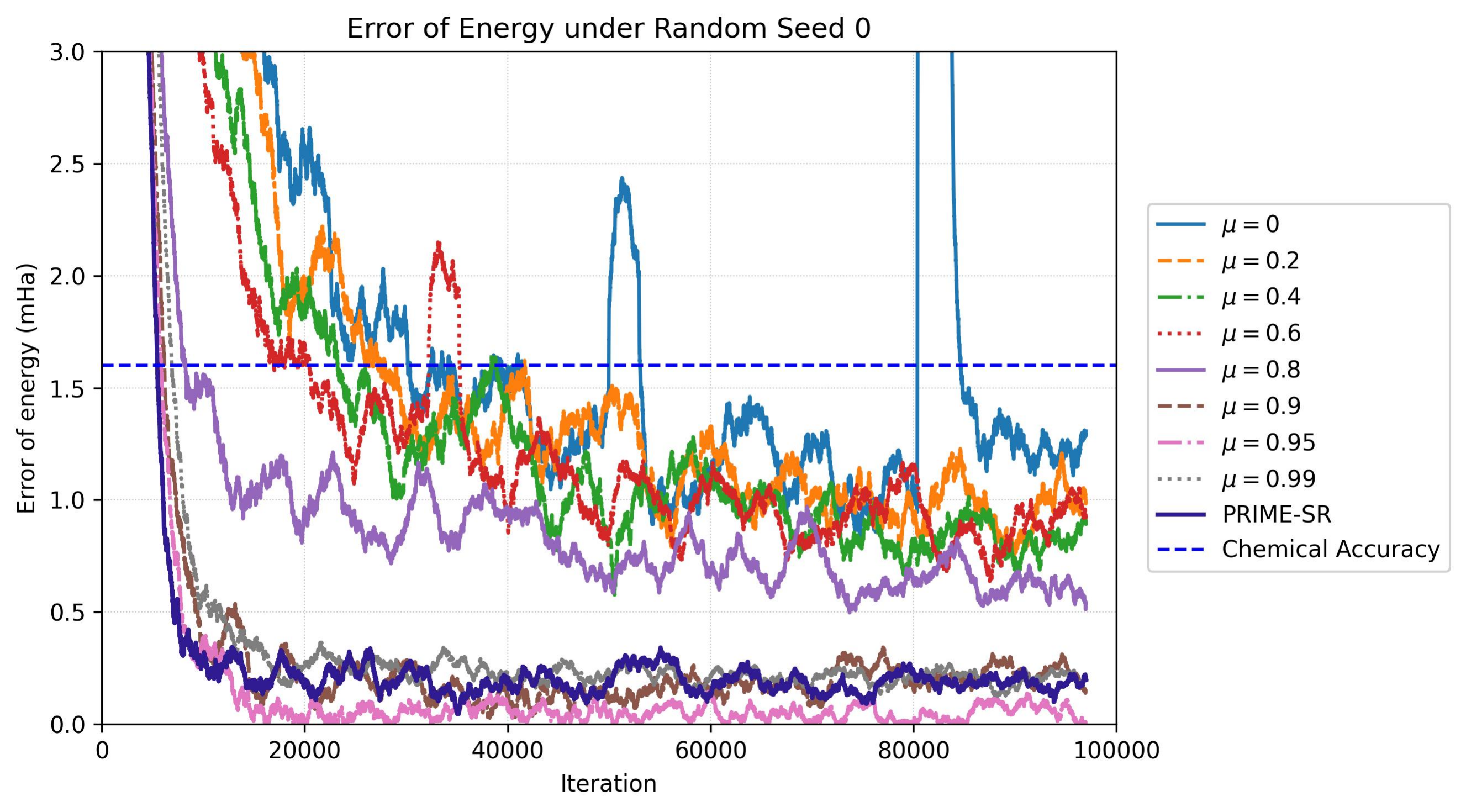}
    \hfill
    \includegraphics[width=0.42\linewidth]{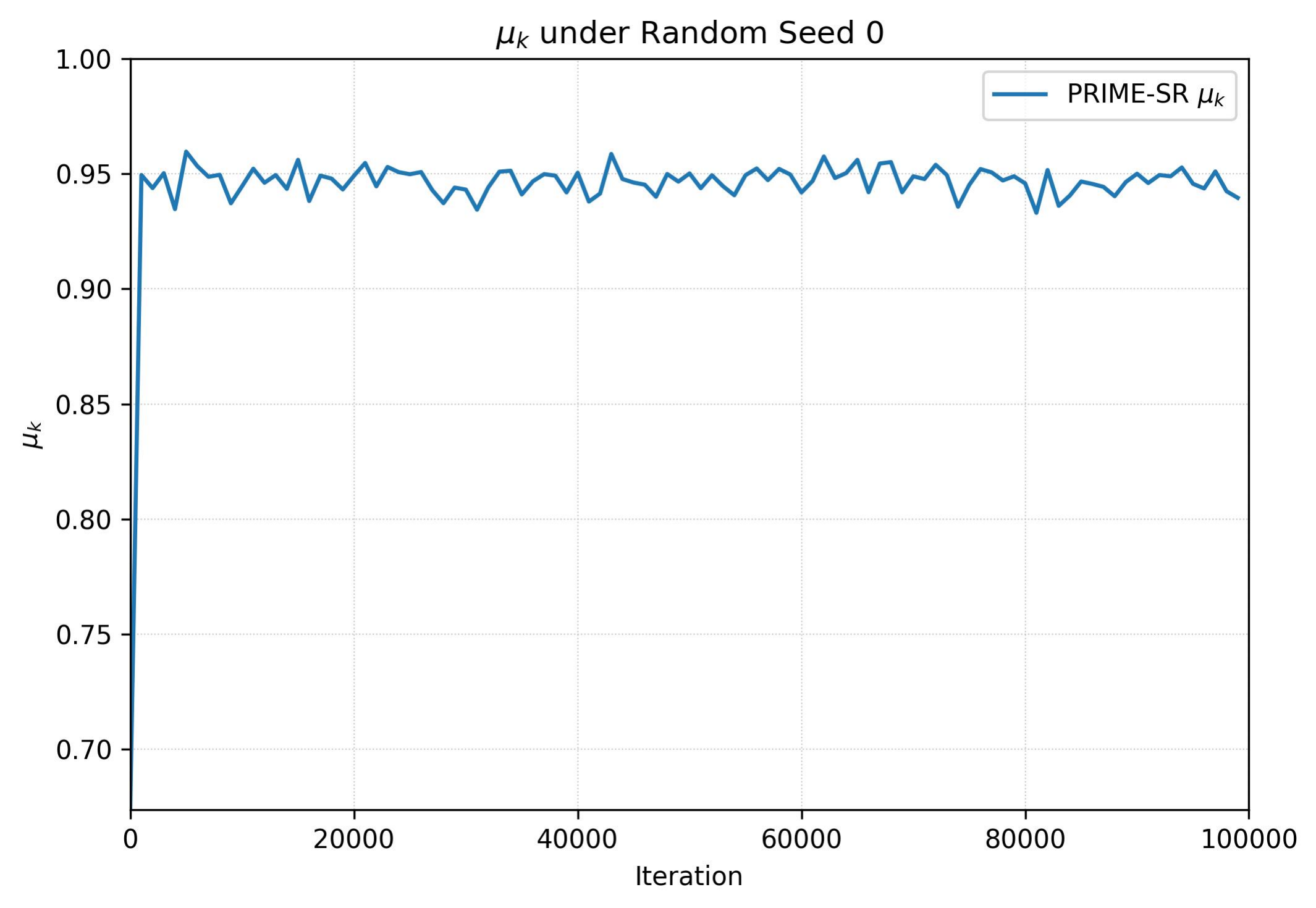}
    \caption{$\mathrm{C}$ atom. Left: relative energy error. Right: $\mu_k$.}
    \end{subfigure}
    
    \vspace{0.5em}
    
    \begin{subfigure}{\linewidth}
    \centering
    \includegraphics[width=0.48\linewidth]{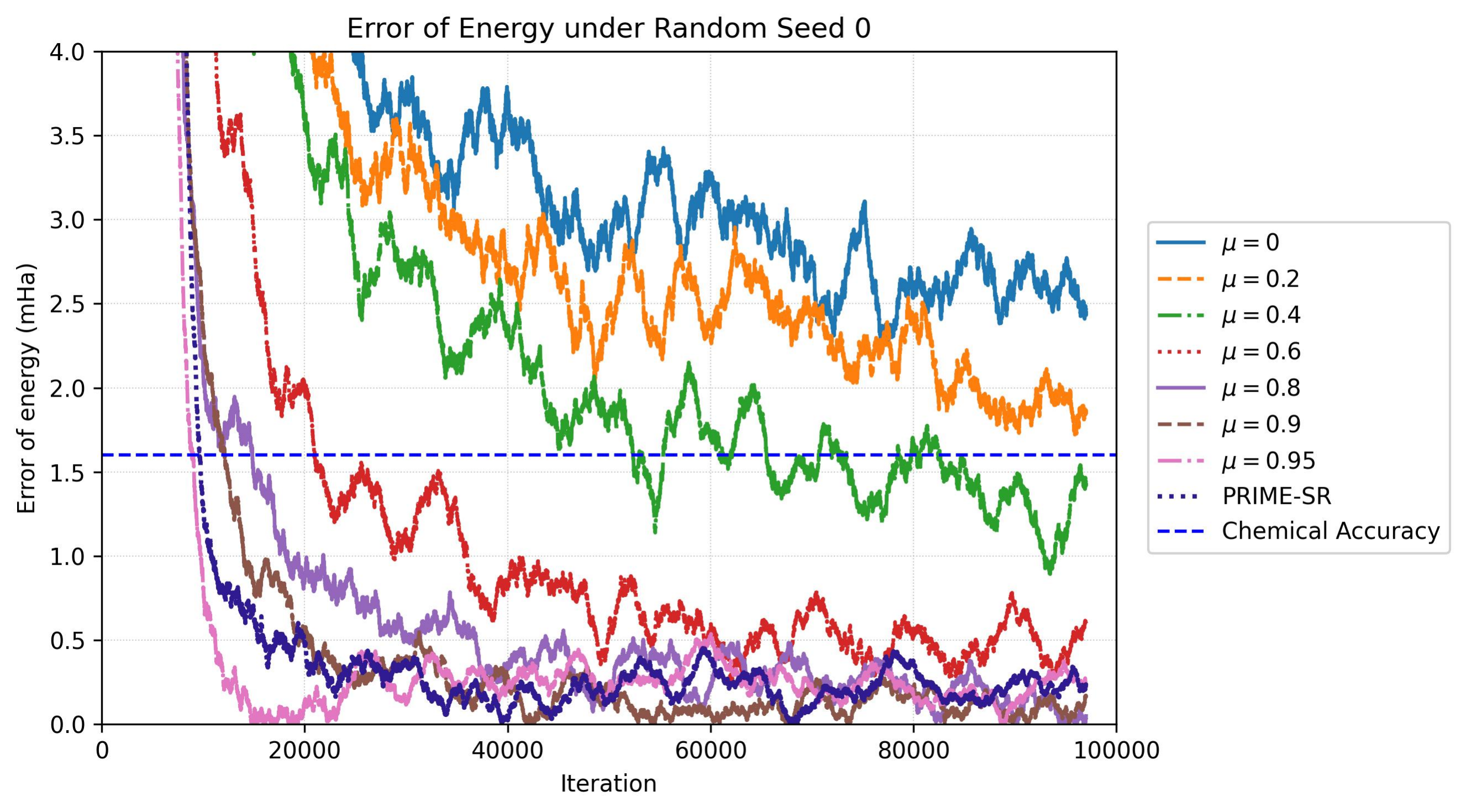}
    \hfill
    \includegraphics[width=0.42\linewidth]{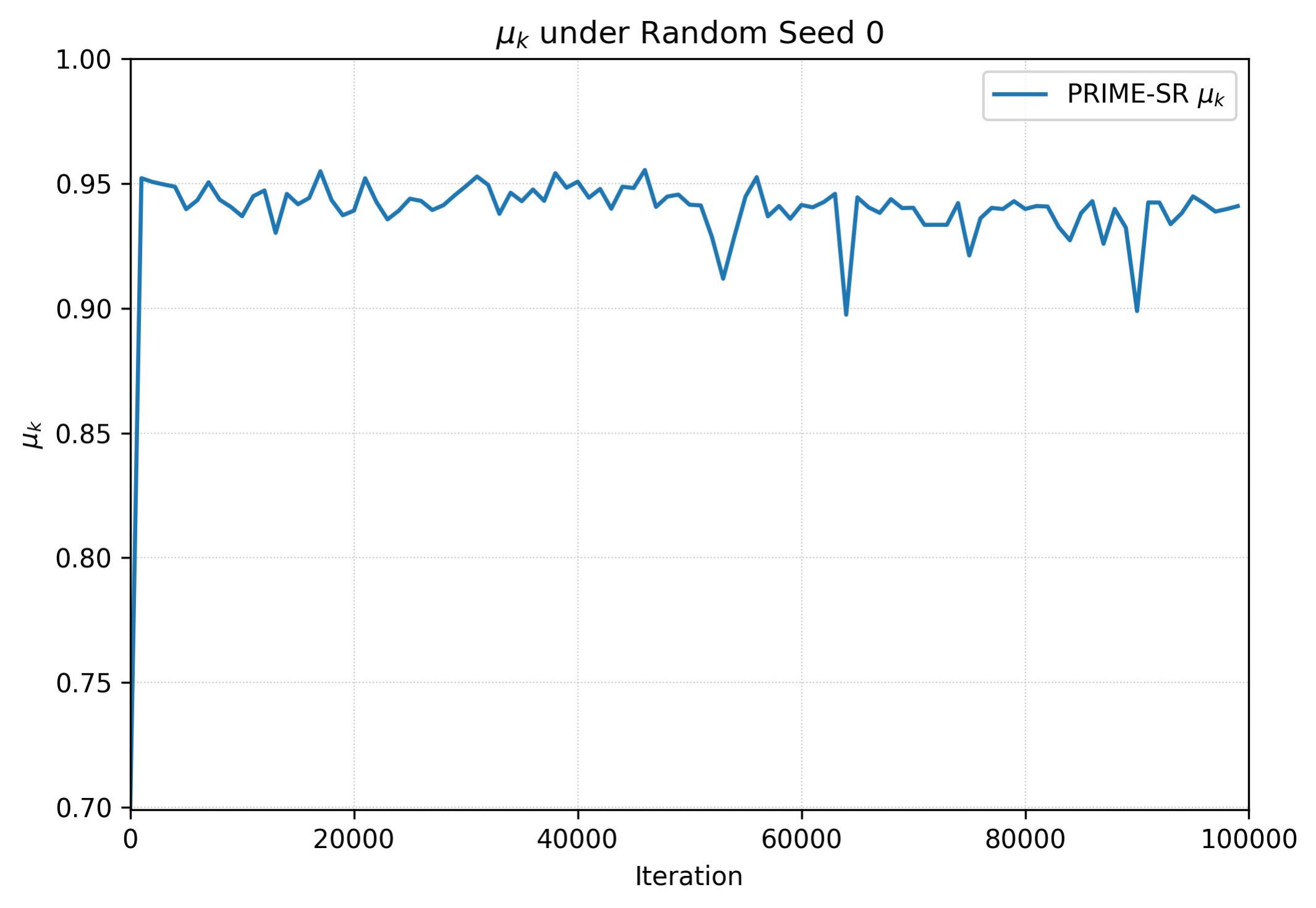}
    \caption{$\mathrm{N}$ atom. Left: relative energy error. Right: $\mu_k$.}
    \end{subfigure}
    
    \vspace{0.5em}
    
    \begin{subfigure}{\linewidth}
    \centering
    \includegraphics[width=0.48\linewidth]{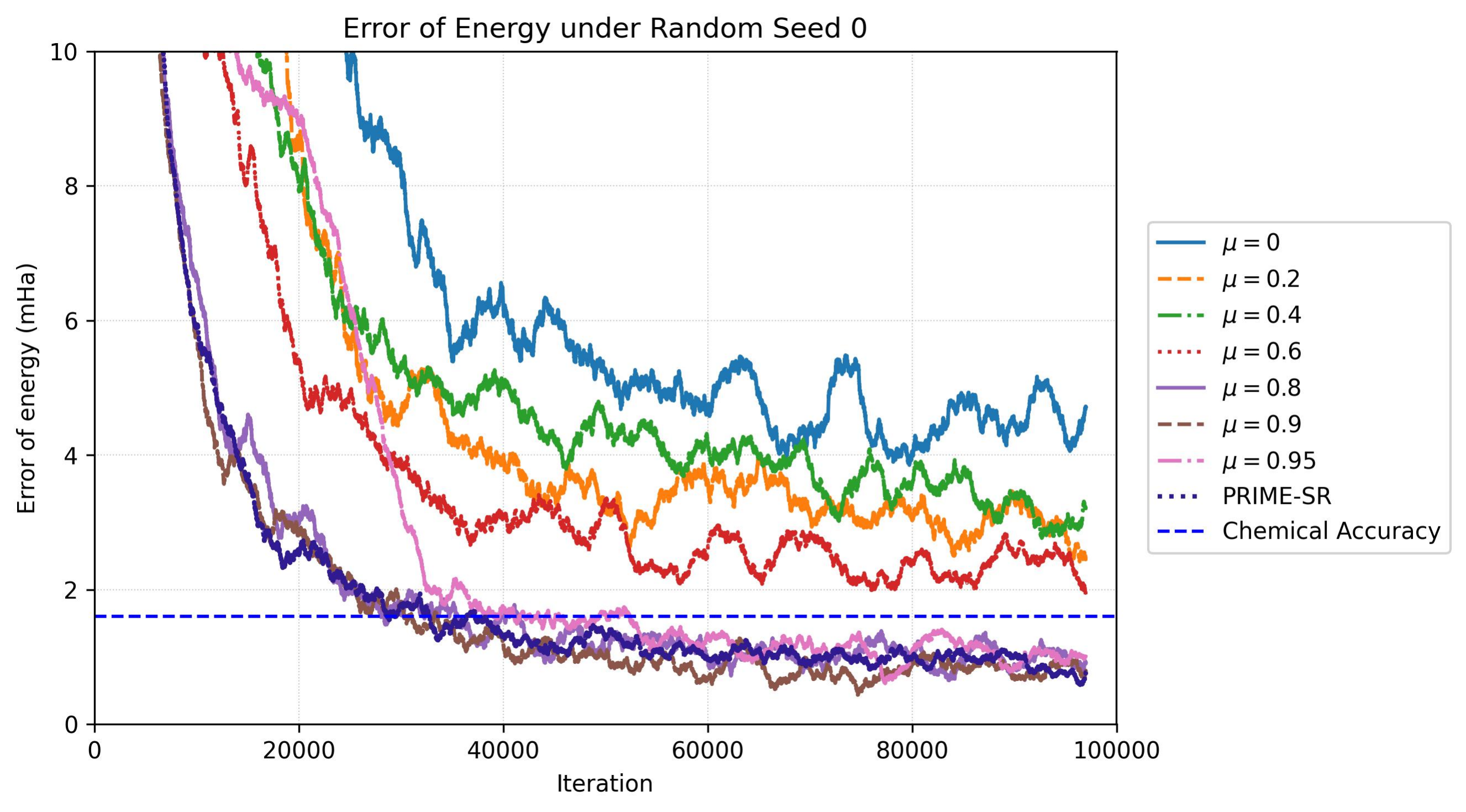}
    \hfill
    \includegraphics[width=0.42\linewidth]{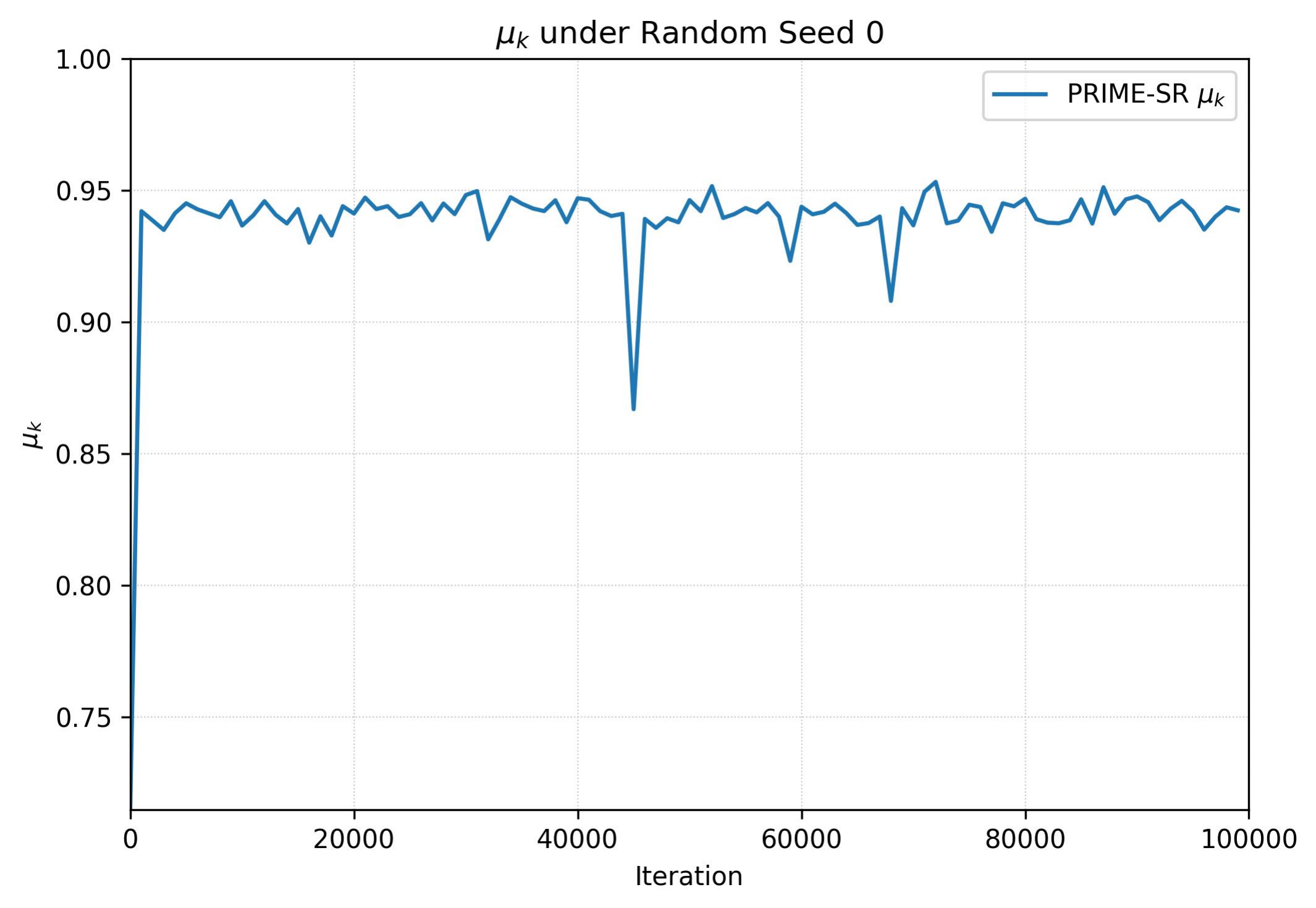}
    \caption{$\mathrm{O}$ atom. Left: relative energy error. Right: $\mu_k$.}
    \end{subfigure}
    
    \caption{Comparison of fixed-$\mu$ SPRING and PRIME-SR on $\mathrm{C}$, $\mathrm{N}$, $\mathrm{O}$ atoms for random seed 0.}
    \label{fig:compare_spring_atom_seed_0}

\end{figure}

\subsection{Experiments on Molecular Systems}

\par We finally consider the $\mathrm{LiH}$, $\mathrm{N}_2$, and $\mathrm{CO}$ molecules. We compare against fixed-$\mu$ SPRING with $\mu=0$, $0.2$, $0.4$, $0.6$, $0.8$, $0.9$, $0.95$, $0.99$ for $\mathrm{LiH}$ molecular, and $\mu=0$, $0.2$, $0.4$, $0.6$, $0.8$, $0.9$, $0.95$ for $\mathrm{N}_2$ and $\mathrm{CO}$ molecules. Unstable runs for $\mathrm{N}_2$ and $\mathrm{CO}$ molecules with $\mu=0.99$ are reported in Appendix~\ref{sec:spring_unstable_N_O}. Figure~\ref{fig:compare_spring_mol_seed_0} shows results for random seed 0, and results of additional seeds are reported in Appendix~\ref{sec:mol_random_seed}.

\par The molecular systems show the same overall trend. PRIME-SR remains stable over the full optimization horizon and reaches accuracies comparable to optimal fixed-$\mu$ SPRING. For $\mathrm{CO}$, PRIME-SR can even outperform the best stable fixed-$\mu$ run in our experiments. The additional multi-seed results in Appendix~\ref{sec:compare_spring_random_seed} further support the conclusion that PRIME-SR provides a more robust alternative to fixed-$\mu$ SPRING on molecular problems.

\begin{figure}[H]
    \centering
    
    \begin{subfigure}{\linewidth}
    \centering
    \includegraphics[width=0.48\linewidth]{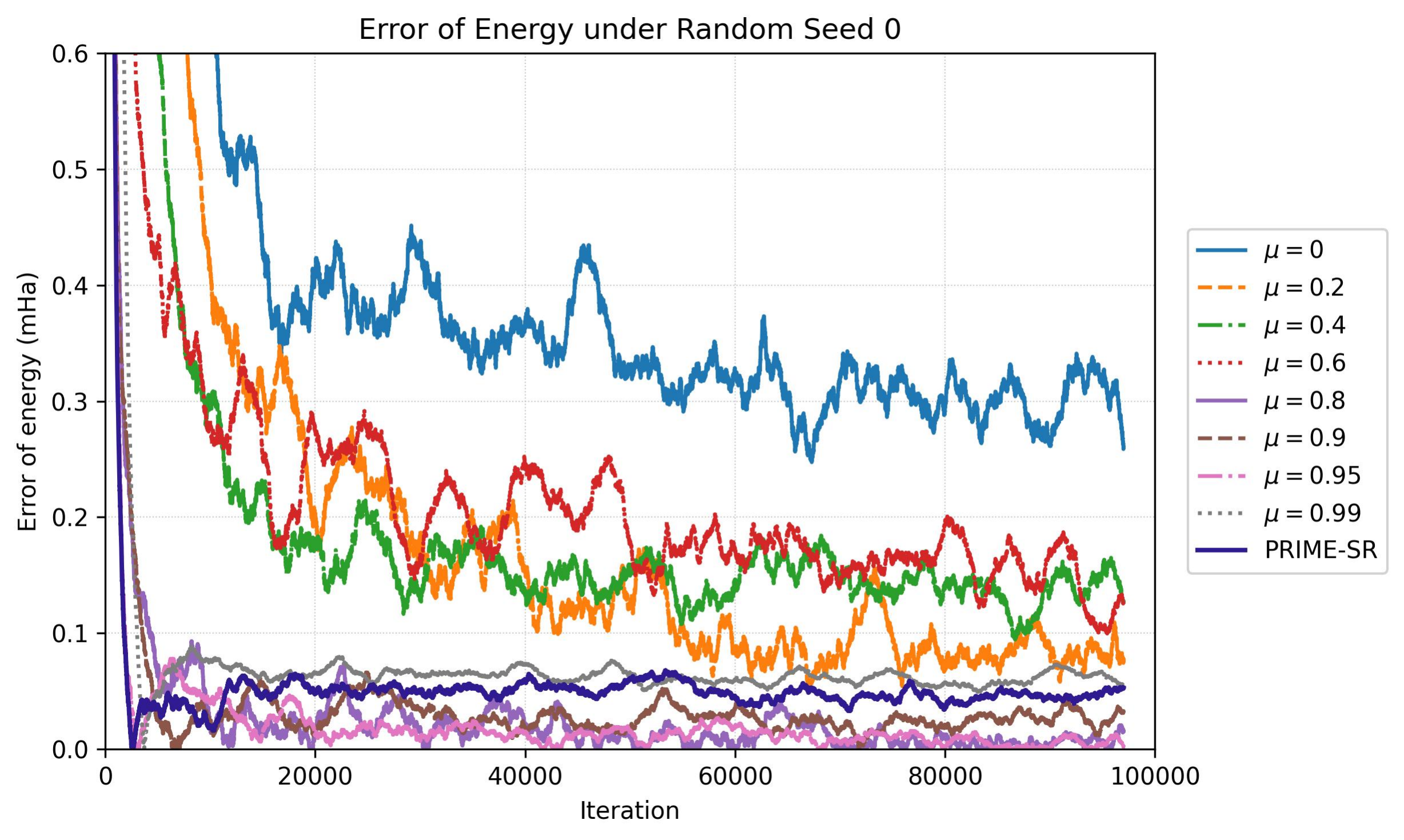}
    \hfill
    \includegraphics[width=0.42\linewidth]{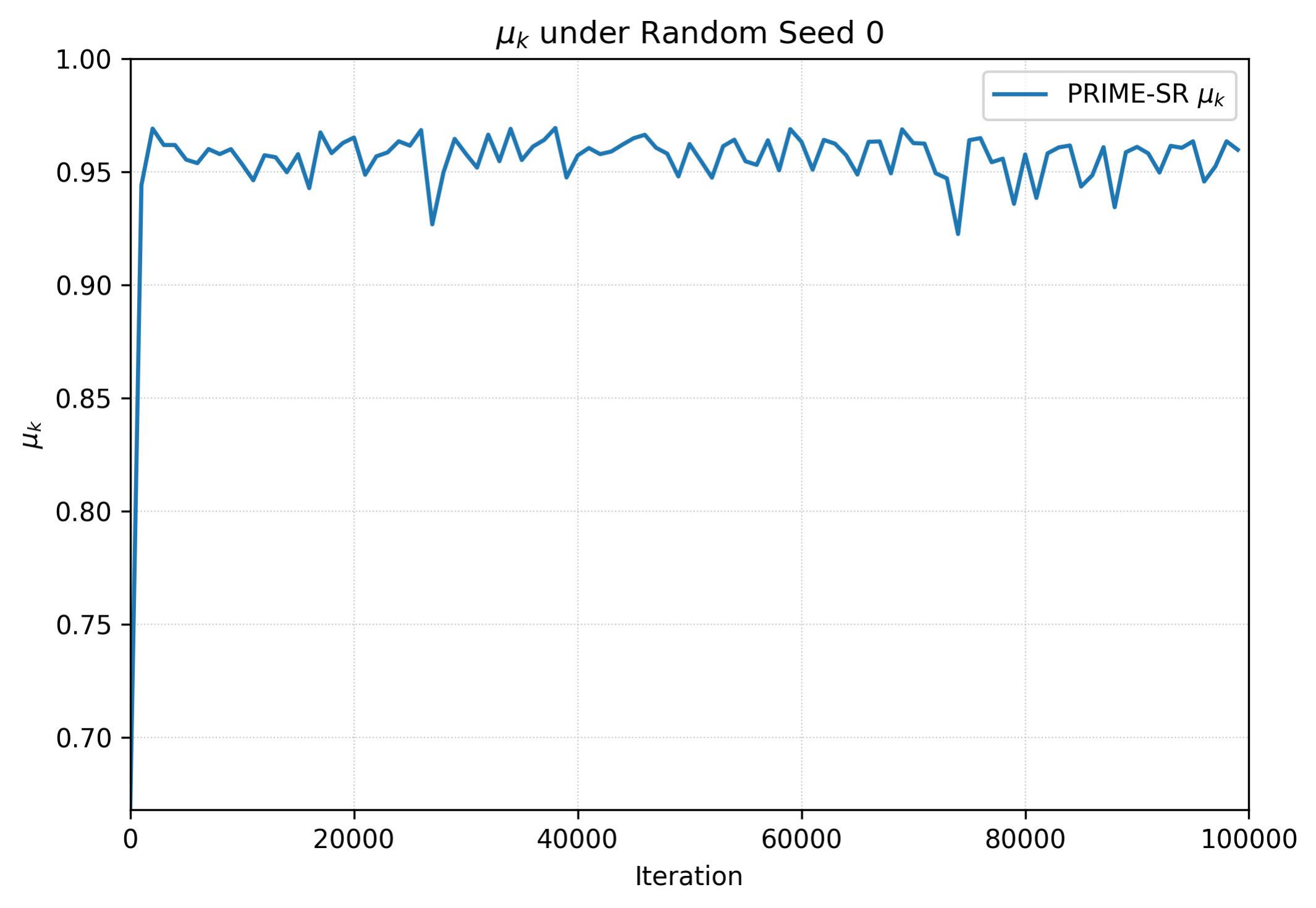}
    \caption{$\mathrm{LiH}$ molecule. Left: relative energy error. Right: $\mu_k$.}
    \end{subfigure}
    
    \vspace{0.5em}
    
    \begin{subfigure}{\linewidth}
    \centering
    \includegraphics[width=0.48\linewidth]{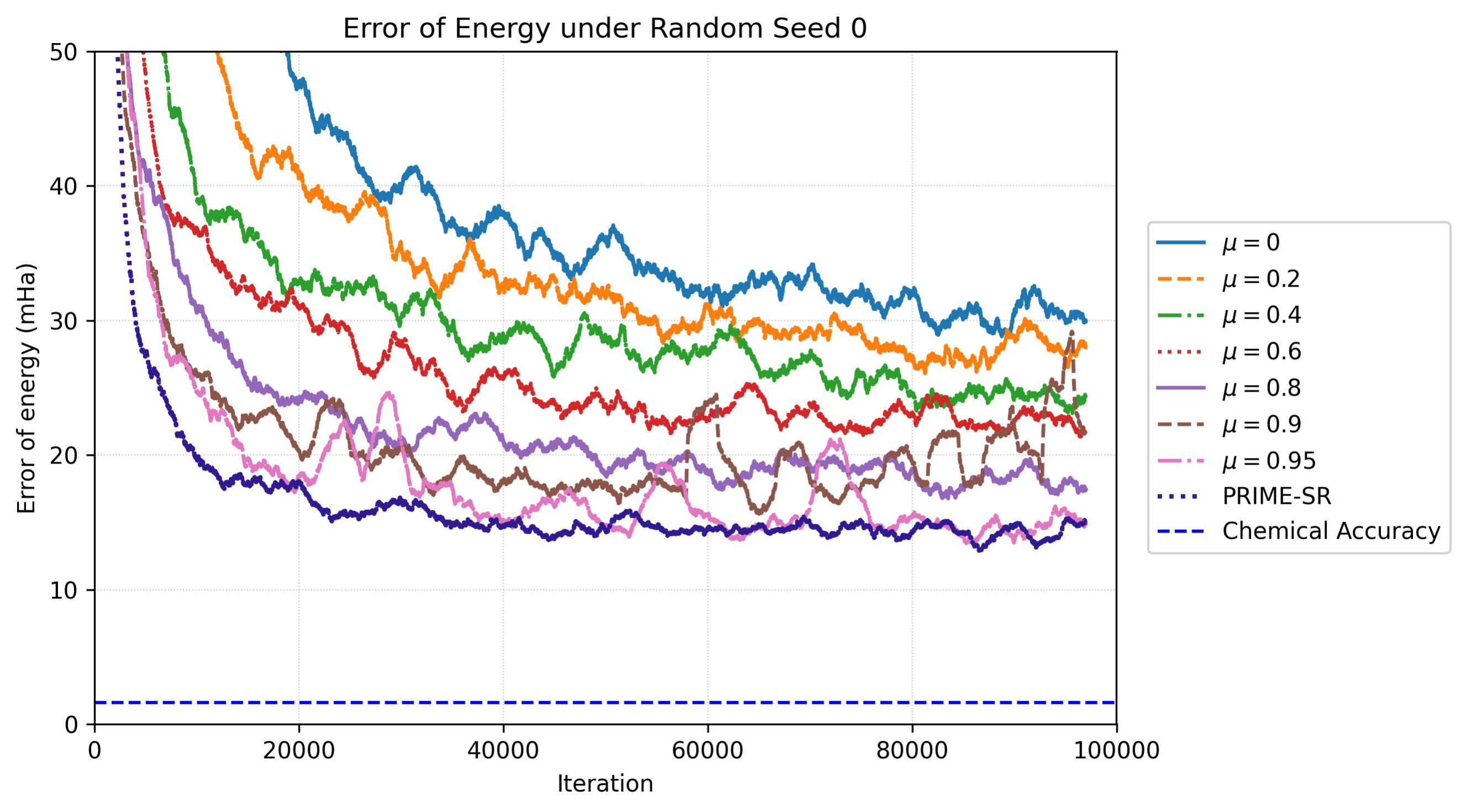}
    \hfill
    \includegraphics[width=0.42\linewidth]{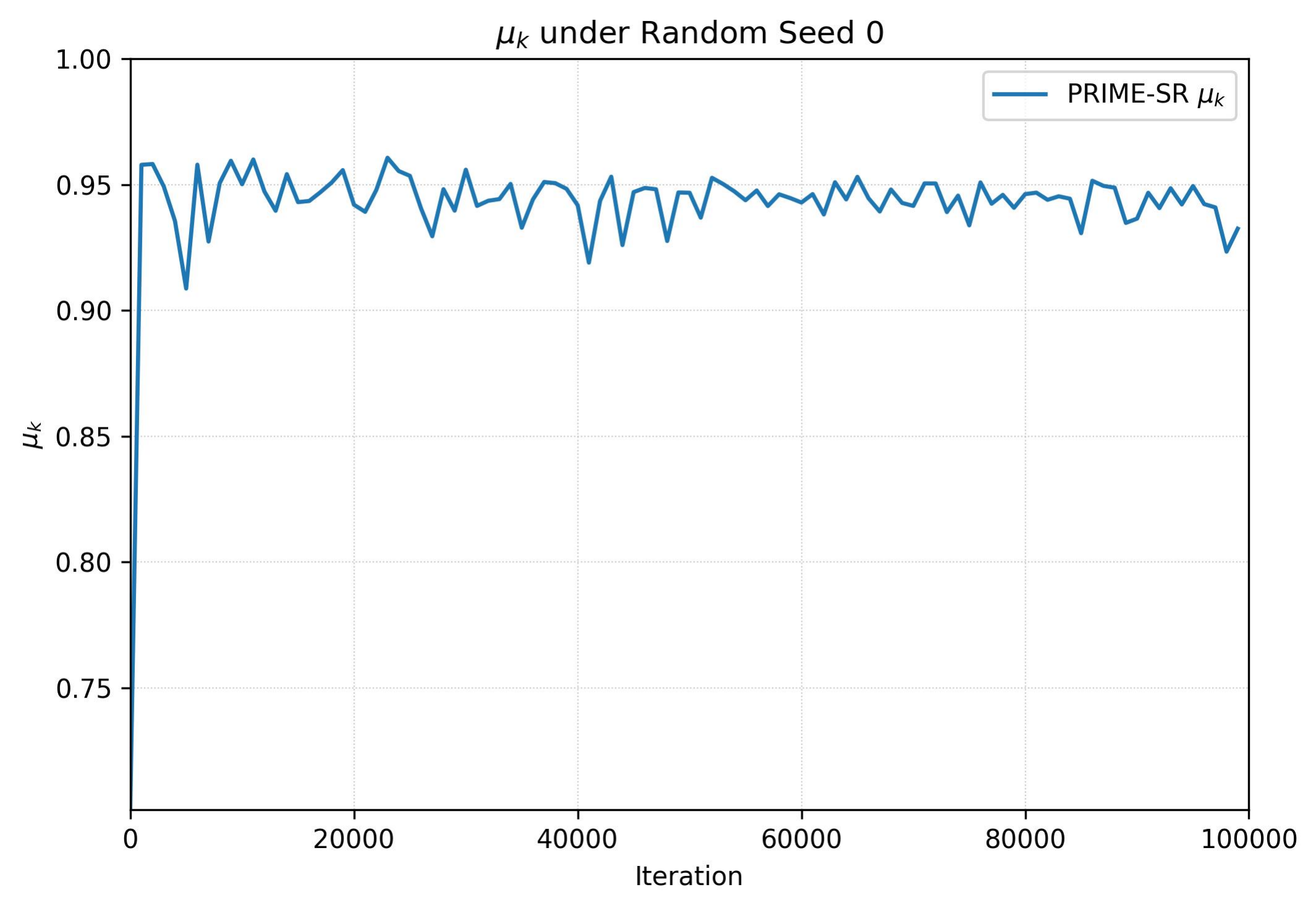}
    \caption{$\mathrm{N}_2$ molecular. Left: relative energy error. Right: $\mu_k$.}
    \end{subfigure}
    
    \vspace{0.5em}
    
    \begin{subfigure}{\linewidth}
    \centering
    \includegraphics[width=0.48\linewidth]{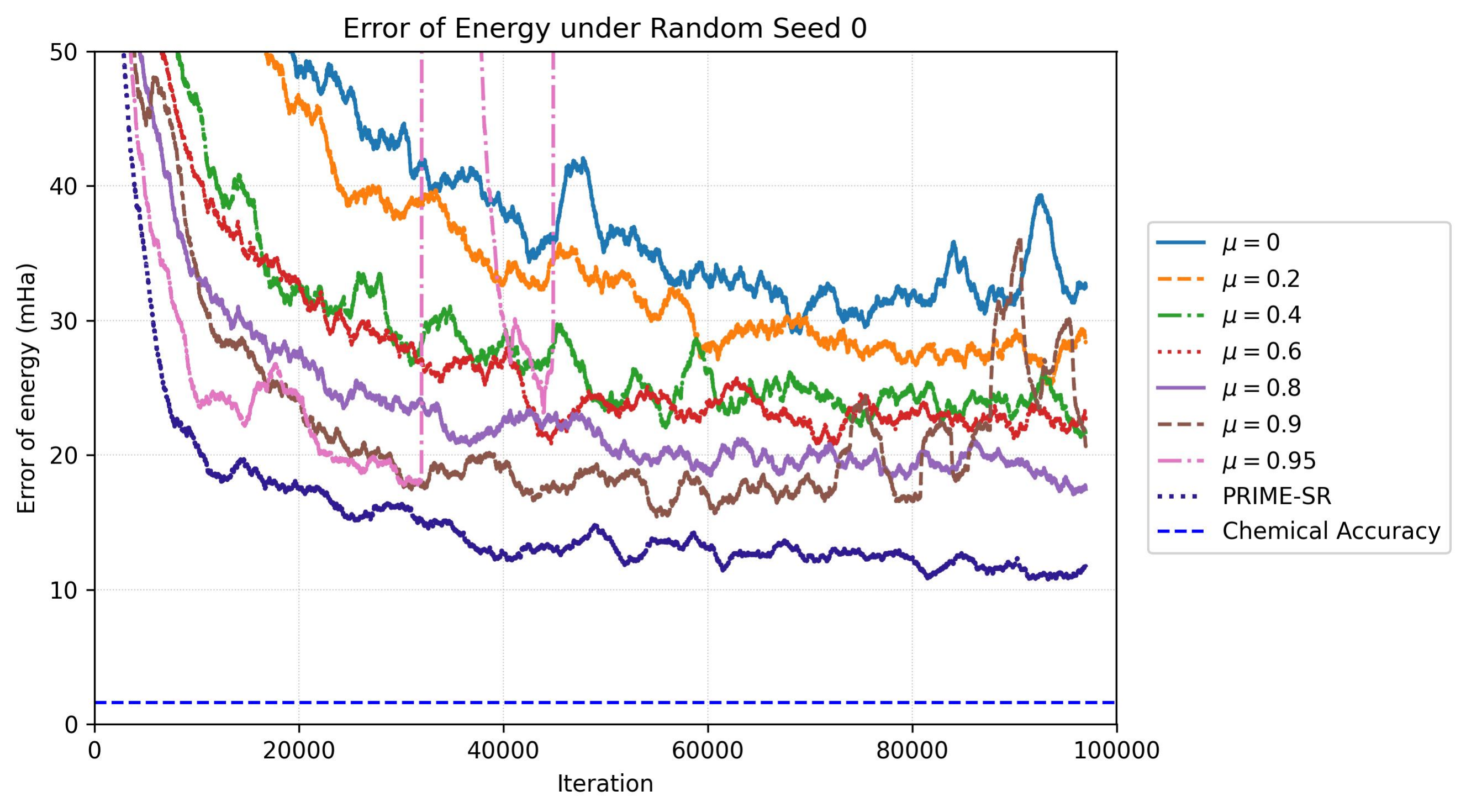}
    \hfill
    \includegraphics[width=0.42\linewidth]{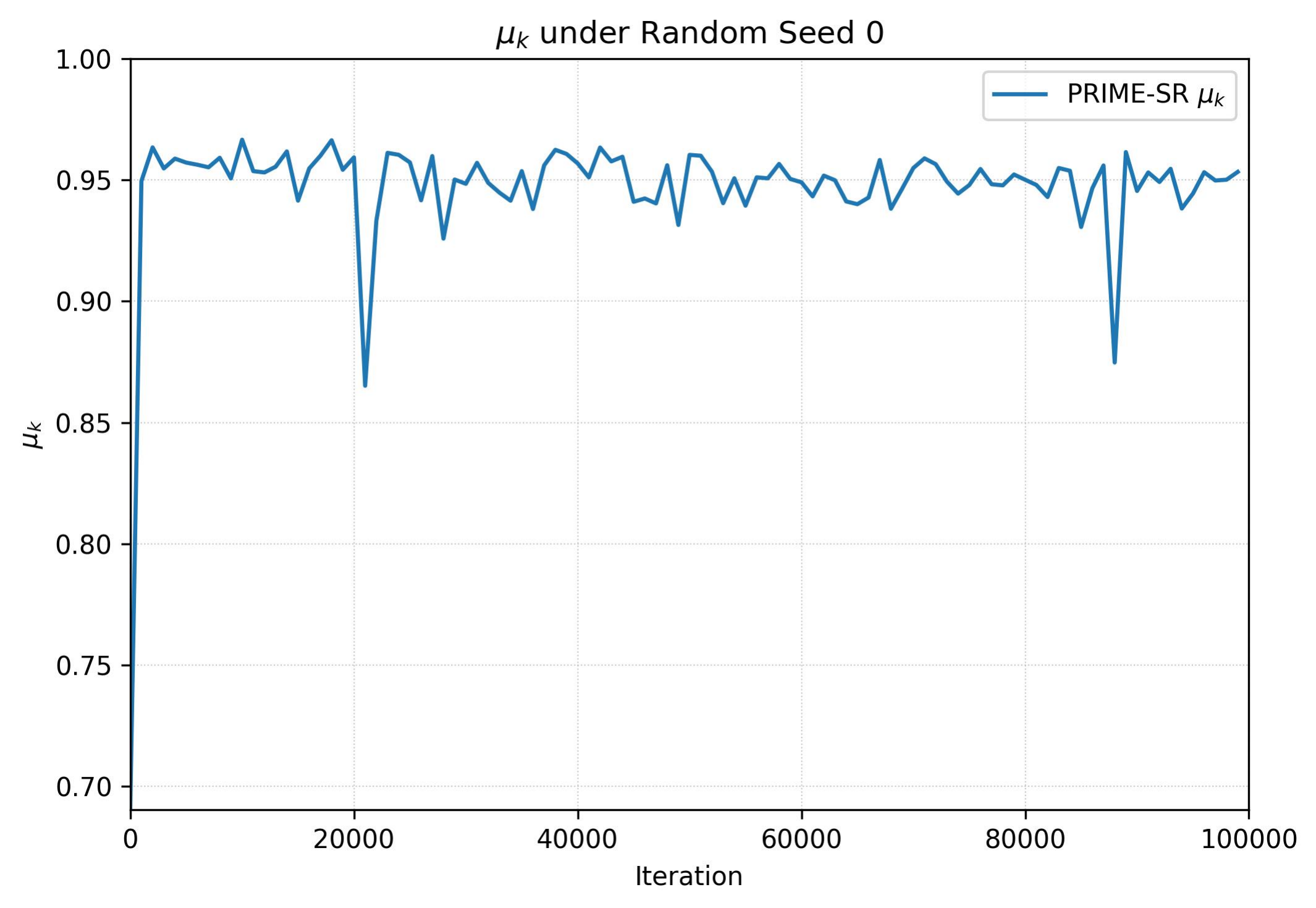}
    \caption{$\mathrm{CO}$ molecular. Left: relative energy error. Right: $\mu_k$.}
    \end{subfigure}
    
    \caption{Comparison of fixed-$\mu$ SPRING and PRIME-SR on $\mathrm{LiH}$, $\mathrm{N}_2$, and $\mathrm{CO}$ molecules for random seed 0.}
    
    \label{fig:compare_spring_mol_seed_0}

\end{figure}
    \section{Conclusion and Discussion}
\label{sec:conclusion}

\par VMC has become a central tool for studying strongly correlated quantum systems, where the expressiveness of modern neural network wavefunctions must be matched by stable and efficient optimization algorithms. Among existing approaches, SPRING, the momentum-accelerated variant of SR, has demonstrated state-of-the-art performance in VMC optimization. However, this strong practical performance relies on carefully tuned momentum parameters, and the stability and convergence behavior of SPRING have remained insufficiently understood.

\par In this work, we investigated the stability and convergence properties of SPRING for wavefunction optimization in VMC. We provided a theoretical characterization of the role of the momentum parameter $\mu$, showing that the regimes $0 \le \mu < 1$ and $\mu = 1$ are governed by fundamentally different mechanisms. In particular, we established convergence guarantees for $\mu < 1$ under mild assumptions, and constructed explicit counterexamples demonstrating that $\mu = 1$ can induce divergence through uncontrolled growth along kernel-related directions of the SR matrix. Building on these insights and numerical observations, we proposed PRIME-SR, a tuning-free momentum-adaptive SR method that controls momentum reuse through spectral flatness and subspace-overlap indicators derived from the SR matrix. Extensive experiments on spin-lattice systems as well as atomic and molecular electronic systems show that PRIME-SR achieves performance comparable to optimally or near-optimally tuned fixed-$\mu$ SPRING while substantially improving robustness.

\par Several directions for future work remain. From a theoretical perspective, it would be desirable to extend the current analysis beyond idealized sampling assumptions and fixed-$\mu$ settings, and to better understand the mechanisms behind the instability occasionally observed in electronic-structure runs even when $\mu<1$. Such behavior likely reflects a subtle interplay between sampling noise and SR ill-conditioning, and may also depend on the specific Hamiltonian and wavefunction parameterization. On the algorithmic side, exploring even simpler or theoretically justified indicators of sampling quality or curvature structure could further improve the adaptive mechanism. Finally, the kernel-range viewpoint underlying our design is not specific to SPRING, and similar adaptive principles may be applicable to other SR-based or geometry-aware optimization methods.

    \paragraph{Acknowledgements.} The work of Xin Liu was supported in part by the National Natural Science Foundation of China (12125108, 12288201), and RGC grant JLFS/P-501/24 for the CAS AMSS–PolyU Joint Laboratory in Applied Mathematics.

    \normalem
    \bibliographystyle{plain}
    \bibliography{ref}

    \appendix
    \section{Proof of \cref{the:full_spring_convergence}}
\label{sec:proof_full_convergence}

\par To prove \cref{the:full_spring_convergence}, we first establish the following lemmas.
\begin{lem}
    Under \cref{assume:1}, for any $\btheta\in \R^{N_p}$,
    \begin{equation*}
        \norm{g(\btheta)}\le 2\sqrt{C_m},\quad \norm{S(\btheta)}_2 \le \sqrt{C_m}.
    \end{equation*}
    \label{lem:g_S_bounded}
\end{lem}
\begin{proof}
    By the formulation of $g(\btheta)$ in Eq.~\eqref{eq:vmc_grad},
    \begin{equation*}
        \begin{split}
            \norm{g(\btheta)} &\le 2\Expect_{X\sim\pi_{\btheta}} \lrsquare{\abs{\bar{E}(\btheta;X)}\cdot\norm{O(\btheta;X)}}\\
            &\le 2\sqrt{\Expect_{X\sim\pi_{\btheta}} \lrsquare{\abs{\bar{E}(\btheta;X)}^2} \cdot \Expect_{X\sim\pi_{\btheta}} \lrsquare{\norm{O(\btheta;X)}^2}}\\
            &=2\sqrt{ \Expect_{X\sim\pi_{\btheta}}\lrsquare{ \abs{\Eloc(\btheta;X)-L(\btheta)}^2 } \cdot \Expect_{X\sim\pi_{\btheta}}\lrsquare{ \norm{\nabla_{\btheta}\log\abs{\psi_{\btheta}(X)} - \Expect_{X\sim \pi_{\btheta}} \lrsquare{ \nabla_{\btheta}\log\abs{\psi_{\btheta}(X)} }}^2 } }\\
            &\le 2\sqrt{ \Expect_{X\sim\pi_{\btheta}}\lrsquare{ \abs{\Eloc(\btheta;X)}^2 } \cdot \Expect_{X\sim\pi_{\btheta}}\lrsquare{ \norm{\nabla_{\btheta}\log\abs{\psi_{\btheta}(X)} }^2 } }\\
            &\le 2 \lrbracket{ \Expect_{X\sim\pi_{\btheta}}\lrsquare{ \abs{\Eloc(\btheta;X)}^4 } \cdot \Expect_{X\sim\pi_{\btheta}}\lrsquare{ \norm{\nabla_{\btheta}\log\abs{\psi_{\btheta}(X)} }^4 } }^{1/4}\le 2\sqrt{C_m},
        \end{split}
    \end{equation*}
    where the first inequality follows from Jensen inequality since the norm $\norm{\cdot}$ is convex, the second and fourth inequalities follow from the Cauchy-Schwarz inequality, the third follows from the property of random variables that $\Expect[\abs{Y-\Expect[Y]}^2]\le \Expect[Y^2]$, and the last inequality follows from \cref{assume:1}.

    Similarly, for $S(\btheta)$ we have,
    \begin{equation*}
        \begin{split}
            \norm{S(\btheta)}_F & =\norm{\Expect_{X\sim\pi_{\btheta}}\lrsquare{O(\btheta;X)O(\btheta;X)^\top}}_F\\
            &\le \Expect_{X\sim\pi_{\btheta}}\lrsquare{\norm{O(\btheta;X)O(\btheta;X)^\top}_F}\\
            &=\Expect_{X\sim\pi_{\btheta}}\lrsquare{\norm{O(\btheta;X)}^2}\\
            &=\Expect_{X\sim\pi_{\btheta}}\lrsquare{\norm{\nabla_{\btheta}\log\abs{\psi_{\btheta}(X)}}^2}\\
            &\le \sqrt{\Expect_{X\sim\pi_{\btheta}}\lrsquare{\norm{\nabla_{\btheta}\log\abs{\psi_{\btheta}(X)}}^4}}\le \sqrt{C}.
        \end{split}
    \end{equation*}
    Since the matrix norm satisfies $\norm{\cdot}_2\le \norm{\cdot}_F$, we obtain $\norm{S(\btheta)}_2\le \norm{S(\btheta)}_F\le \sqrt{C}$.
\end{proof}

\begin{lem}
    Under \cref{assume:1}, let $\{\btheta_k\}$ be generated by \eqref{eq:full-spring}, then for any $0\le \mu<1$,
    \begin{equation*}
        \sum_{k=0}^{\infty}\eta_k \norm{\Delta\btheta_k}^2 <\infty.
    \end{equation*}
\end{lem}
\begin{proof}
    Under \cref{assume:1}, $g(\btheta)$ is $C_g$-Lipschitz continuous, by \cref{lem:2_upper_bound}, we have
    \begin{equation}
        L(\btheta_{k+1})-L(\btheta_k)\le g(\btheta_k)^\top\lrbracket{\btheta_{k+1}-\btheta_k} +\dfrac{C_g}{2}\norm{\btheta_{k+1}-\btheta_k}^2.
        \label{eq:2_upper_proof}
    \end{equation}
    From the update in \eqref{eq:full-spring}:
    \begin{equation*}
        g(\btheta_k) = 2\lambda\mu\Delta\btheta_{k-1} - 2(\lambda I +S(\btheta_k))\Delta\btheta_k,\quad \btheta_{k+1}-\btheta_k = \eta_k\Delta\btheta_k.
    \end{equation*}
    And further by $2\Delta\btheta_k^\top\Delta\btheta_{k-1}=\norm{\Delta\btheta_k}^2+\norm{\Delta\btheta_{k-1}}^2-\norm{\Delta\btheta_k-\Delta\btheta_{k-1}}^2$, we have
    \begin{equation}
    \begin{split}
        L(\btheta_{k+1})-L(\btheta_k) &\le \lrbracket{\lambda\mu + \dfrac{C_g}{2}\eta_k}\eta_k\norm{\Delta\btheta_k}^2 + \lambda\mu\eta_k \lrbracket{\norm{\Delta\btheta_{k-1}}^2 - \norm{\Delta\btheta_k-\Delta\btheta_{k-1}}^2} \\
        &\quad - 2\eta_k\Delta\btheta_k^\top \lrbracket{\lambda I +S(\btheta_k)}\Delta\btheta_k\\
        &\le \lrbracket{\lambda\mu + \dfrac{C_g}{2}\eta_k-2\lambda}\eta_k\norm{\Delta\btheta_k}^2 + \lambda\mu\eta_k\norm{\Delta\btheta_{k-1}}^2 - \lambda\mu\eta_k\norm{\Delta\btheta_k-\Delta\btheta_{k-1}}^2.
    \end{split}
    \label{eq:proof_1}
    \end{equation}
    Consider the discrete energy $F_k:=L(\btheta_k)+\lambda\mu\eta_k\norm{\Delta\btheta_{k-1}}^2 \ge E_{\gs}$, which is bounded from below. Thus, by Eq.~\eqref{eq:proof_1}, we have
    \begin{equation*}
        \begin{split}
            F_{k+1}-F_k&=L(\btheta_{k+1}) - L(\btheta_k) + \lambda\mu\eta_{k+1}\norm{\Delta\btheta_k}^2 - \lambda\mu\eta_k\norm{\Delta\btheta_{k-1}}^2\\
            &\le \lrbracket{\lambda\mu+\dfrac{C_g}{2}\eta_k-2\lambda}\eta_k\norm{\Delta\btheta_k}^2 + \lambda\mu\eta_{k+1}\norm{\Delta\btheta_k}^2 - \lambda\mu\eta_k\norm{\Delta\btheta_k-\Delta\btheta_{k-1}}^2 \\
            &(\text{by }\eta_{k+1}\le \eta_k \text{ and }\eta_k \le \eta_0)\\
            &\le \lrbracket{-2\lambda(1-\mu)+\dfrac{C_g}{2}\eta_0}\eta_k\norm{\Delta\btheta_k}^2 - \lambda\mu\eta_k\norm{\Delta\btheta_k-\Delta\btheta_{k-1}}^2 <0,
        \end{split}
    \end{equation*}
    where the last $<0$ is from the \cref{assume:1} on $\eta_0$. Therefore, $F_k$ decreases monotonically. Assume that $\lim_{k\to\infty}F_k = F^*$, then
    \begin{equation*}
        F^*-F_0=\sum_{k=0}^{\infty}F_{k+1}-F_k\le \lrbracket{-2\lambda(1-\mu)+\dfrac{C_g}{2}\eta_0}\sum_{k=0}^{\infty}\eta_k\norm{\Delta\btheta_k}^2 - \lambda\mu\sum_{k=0}^{\infty}\eta_k\norm{\Delta\btheta_k-\Delta\btheta_{k-1}}^2<0.
    \end{equation*}
    Hence, we obtain $\sum_{k=0}^{\infty}\eta_k\norm{\Delta\btheta_k}^2 <\infty$.
\end{proof}

\par Now, we are ready to prove Theorem \ref{the:full_spring_convergence}.
\begin{proof}[Proof of Theorem \ref{the:full_spring_convergence}]
    We first estimate $g(\btheta_k)^\top \Delta\btheta_k$, from the update in \eqref{eq:full-spring}, we have
    \begin{equation*}
        \begin{split}
            g(\btheta_k)^\top \Delta\btheta_k&=-\dfrac{1}{2}g(\btheta_k)^\top \lrbracket{\lambda I +S(\btheta_k)}^{-1}g(\btheta_k) + \lambda\mu g(\btheta_k)^\top \lrbracket{\lambda I +S(\btheta_k)}^{-1} \Delta\btheta_{k-1}\\
            & (\text{by } S(\btheta_k)\succeq 0 \text{ and from \cref{lem:g_S_bounded} } \norm{S(\btheta_k)}_2\le \sqrt{C_m})\\
            &\le -\dfrac{1}{2(\lambda+\sqrt{C_m})}\norm{g(\btheta_k)}^2 + \mu\norm{g(\btheta_k)}\norm{\Delta\btheta_{k-1}}\\
            &(\text{by Young inequality}) \\
            &\le -\dfrac{1}{2(\lambda+\sqrt{C_m})}\norm{g(\btheta_k)}^2 + \dfrac{\mu}{4(\lambda + \sqrt{C_m})}\norm{g(\btheta_k)}^2 + \mu(\lambda+\sqrt{C_m})\norm{\Delta\btheta_{k-1}}^2\\
            &=\dfrac{\mu-2}{4(\lambda+\sqrt{C_m})}\norm{g(\btheta_k)}^2 + \mu(\lambda+\sqrt{C_m})\norm{\Delta\btheta_{k-1}}^2.
        \end{split}
    \end{equation*}
    Substituting this inequality into Eq.~\eqref{eq:2_upper_proof}, we obtain
    \begin{equation*}
        \begin{split}
            L(\btheta_{k+1})-L(\btheta_k)&\le \dfrac{\mu-2}{4(\lambda+\sqrt{C_m})}\eta_k\norm{g(\btheta_k)}^2 + \mu(\lambda+\sqrt{C_m})\eta_k\norm{\Delta\btheta_{k-1}}^2 + \dfrac{C_g}{2}\eta_k^2\norm{\Delta\btheta_{k}}^2\\
            &(\text{by }\eta_k\le \eta_{k-1} \text{ and }\eta_k\le \eta_0)\\
            &\le \dfrac{\mu-2}{4(\lambda+\sqrt{C_m})}\eta_k\norm{g(\btheta_k)}^2 + \mu(\lambda+\sqrt{C_m})\eta_{k-1}\norm{\Delta\btheta_{k-1}}^2 + \dfrac{C_g\eta_0}{2}\eta_k\norm{\Delta\btheta_{k}}^2.
        \end{split}
    \end{equation*}
      Rearranging the terms, we obtain
      \begin{equation*}
          \dfrac{2-\mu}{4(\lambda+\sqrt{C_m})}\eta_k\norm{g(\btheta_k)}^2 \le \mu(\lambda+\sqrt{C_m})\eta_{k-1}\norm{\Delta\btheta_{k-1}}^2 + \dfrac{C_g\eta_0}{2}\eta_k\norm{\Delta\btheta_k}^2 + L(\btheta_k)-L(\btheta_{k-1}).
      \end{equation*}
        Summing from $k=1$ to $k=K$, we obtain
    \begin{equation*}
        \begin{split}
        \dfrac{2-\mu}{4(\lambda+\sqrt{C_m})}\sum_{k=1}^K\eta_k\norm{g(\btheta_k)}^2&\le \mu(\lambda+\sqrt{C_m})\sum_{k=1}^K\eta_{k-1}\norm{\Delta\btheta_{k-1}}^2 + \dfrac{C_g\eta_0}{2}\sum_{k=1}^K\eta_k\norm{\Delta\btheta_k}^2 + L(\btheta_1)-L(\btheta_{K+1})\\
        &\le \lrbracket{\lambda\mu+\sqrt{C_m}\mu+\dfrac{C_g\eta_0}{2}}\sum_{k=0}^{\infty}\eta_k\norm{\Delta\btheta_k}^2 + L(\btheta_1)-E_{\gs}<\infty.
        \end{split}
    \end{equation*}
    Denote $C:=\dfrac{4(\lambda+\sqrt{C_m})}{2-\mu}\lrbracket{\lrbracket{\lambda\mu+\sqrt{C_m}\mu+\dfrac{C_g\eta_0}{2}}\sum_{k=0}^{\infty}\eta_k\norm{\Delta\btheta_k}^2 + L(\btheta_1)-E_{\gs}}$, then we have:
    \begin{equation*}
        \sum_{k=1}^K \eta_k\norm{g(\btheta_k)}^2 \le C.
    \end{equation*}
\end{proof}

\section{Proof of \cref{the:p_spring_convergence}}
\label{sec:proof_p_spring_convergence}

\subsection{Proof of Monte Carlo Estimators}
\label{sec:proof_sr_unbiased}

\par We begin with the proof of \cref{lem:sr_matrix_unbiased}, which states that the Monte Carlo estimator of the SR matrix is unbiased.

\begin{proof}[Proof of \cref{lem:sr_matrix_unbiased}]    
    Since $S(\btheta_k;\calB_k)=O(\btheta_k;\calB_k)O(\btheta_k;\calB_k)^\top$, by the definition of $O(\btheta_k;\calB_k)$ in Eqs.~ \eqref{eq:O_E_single} and \eqref{eq:O_E_mc_form}, we have:
    {\small\begin{equation}
        \begin{split}
    S(\btheta_k;\calB_k) &= \dfrac{1}{N_s-1}\sum_{i=1}^{N_s}\lrbracket{\nabla_{\btheta}\log\abs{\psi_{\btheta_k}(X_{k,i})} - \dfrac{1}{N_s}\sum_{j=1}^{N_s}\nabla_{\btheta}\log\abs{\psi_{\btheta_k}(X_{k,j})}}\lrbracket{\nabla_{\btheta}\log\abs{\psi_{\btheta_k}(X_{k,i})} - \dfrac{1}{N_s}\sum_{j=1}^{N_s}\nabla_{\btheta}\log\abs{\psi_{\btheta_k}(X_{k,j})}}^\top\\
    &=\dfrac{1}{N_s-1}\lrbracket{ \sum_{i=1}^{N_s} \nabla_{\btheta}\log\abs{\psi_{\btheta_k}(X_{k,i})}\nabla_{\btheta}\log\abs{\psi_{\btheta_k}(X_{k,i})}^\top - \dfrac{1}{N_s}\sum_{i=1}^{N_s}\nabla_{\btheta}\log\abs{\psi_{\btheta_k}(X_{k,i})} \sum_{j=1}^{N_s}\nabla_{\btheta}\log\abs{\psi_{\btheta_k}(X_{k,j})}^\top}\\
    &=\dfrac{1}{N_s}\sum_{i=1}^{N_s}\nabla_{\btheta}\log\abs{\psi_{\btheta_k}(X_{k,i})}\nabla_{\btheta}\log\abs{\psi_{\btheta_k}(X_{k,i})}^\top - \dfrac{1}{N_s(N_s-1)}\sum_{1\le i\neq j\le N_s}\nabla_{\btheta}\log\abs{\psi_{\btheta_k}(X_{k,i})}\nabla_{\btheta}\log\abs{\psi_{\btheta_k}(X_{k,j})}^\top.
        \end{split}
        \label{eq:sr_matrix_decomp}
    \end{equation}}
    By \cref{assume:sample}, for any $1\le i\le N_s$ and $i\neq j$,
    \begin{equation*}
        \begin{split}
            \Expect_{\calB_k}\lrsquare{ \nabla_{\btheta}\log\abs{\psi_{\btheta_k}(X_{k,i})}\nabla_{\btheta}\log\abs{\psi_{\btheta_k}(X_{k,i})}^\top } &=  \Expect_{X\sim\pi_{\btheta_k}}\lrsquare{ \nabla_{\btheta}\log\abs{\psi_{\btheta_k}(X)}\nabla_{\btheta}\log\abs{\psi_{\btheta_k}(X)}^\top }\\
        \Expect_{\calB_k}\lrsquare{\nabla_{\btheta}\log\abs{\psi_{\btheta_k}(X_{k,i})}\nabla_{\btheta}\log\abs{\psi_{\btheta_k}(X_{k,j})}^\top}&=\Expect_{X\sim\pi_{\btheta_k}}\lrsquare{\nabla_{\btheta}\log\abs{\psi_{\btheta_k}(X)}}\Expect_{X\sim\pi_{\btheta_k}}\lrsquare{\nabla_{\btheta}\log\abs{\psi_{\btheta_k}(X)}}^\top.
        \end{split}
    \end{equation*}
    Therefore,
    \begin{equation*}
        \begin{split}
            \Expect_{\calB_k}\lrsquare{S(\btheta_k;\calB_k)}&=\Expect_{X\sim\pi_{\btheta_k}}\lrsquare{ \nabla_{\btheta}\log\abs{\psi_{\btheta_k}(X)}\nabla_{\btheta}\log\abs{\psi_{\btheta_k}(X)}^\top } - \Expect_{X\sim\pi_{\btheta_k}}\lrsquare{\nabla_{\btheta}\log\abs{\psi_{\btheta_k}(X)}}\Expect_{X\sim\pi_{\btheta_k}}\lrsquare{\nabla_{\btheta}\log\abs{\psi_{\btheta_k}(X)}}^\top\\
            &=S(\btheta_k).
        \end{split}
    \end{equation*}
\end{proof}

\par Next, we estimate the error induced by Monte Carlo sampling. Before doing so, we prove an auxiliary counting lemma that will be used in the subsequent bounds.

\begin{lem}
        Let $\calI:=\{(i_1,i_2,j_1,j_2):1\le i_1\neq j_1\le N,~1\le i_2\neq j_2\le N\}$. Then
        \begin{equation*}
            \sharp\{(i_1,i_2,j_1,j_2)\in\calI: i_1=i_2\text{ or }j_1=j_2\} = 2N(N-1)^2-N(N-1),
        \end{equation*}
    where $\sharp$ denotes the cardinality of a set.
    \label{lem:counter_number}
\end{lem}
\begin{proof}
        We count the tuples in $\calI$ satisfying $i_1=i_2$ or $j_1=j_2$ as following.
\begin{enumerate}
    \item[(1)]  Count the case $i_1=i_2$.
    
        Choose $i:=i_1=i_2$ in $N$ ways. For each fixed $i$, we must choose $j_1\neq i$ and $j_2\neq i$, which can be done in $(N-1)^2$ ways. Hence there are $N(N-1)^2$ tuples with $i_1=i_2$.
        
    \item[(2)] Count the case $j_1=j_2$.

    By the same argument, there are $N(N-1)^2$ tuples with $j_1=j_2$.

    \item[(3)] Subtract the overlap.

    The overlap consists of tuples with both $i_1=i_2$ and $j_1=j_2$. Choose $i:=i_1=i_2$ in $N$ ways and then choose $j:=j_1=j_2$ with $j\neq i$, in $(N-1)$ ways. Thus the overlap has size $N(N-1)$.
    
\end{enumerate}
    Combining the three steps gives
    \begin{equation*}
        N(N-1)^2 + N(N-1)^2 - N(N-1) = 2N(N-1)^2 - N(N-1),
    \end{equation*}
    which completes the proof.
\end{proof}

\begin{lem}
    Under \cref{assume:1,assume:sample}, for any $k\ge 0$, 
    \begin{equation*}
        \Expect_{\calB_k}\lrsquare{\norm{g(\btheta_k)-g(\btheta_k;\calB_k)}^2}\le \dfrac{24C_m}{N_s},\quad \Expect_{\calB_k}\lrsquare{\norm{S(\btheta_k)-S(\btheta_k;\calB_k)}_F^2}\le \dfrac{6C_m}{N_s}.
    \end{equation*}
    \label{lem:grad_S_mc_error}
\end{lem}
\begin{proof}  
     We first bound the error of $g(\btheta_k;\calB_k)$. By the definition of $g(\btheta_k;\calB_k)$ in Eqs.~\eqref{eq:grad_mc} and \eqref{eq:O_and_bar_E}, and similarly to Eq.~\eqref{eq:sr_matrix_decomp}, we have
     \begin{equation*}
         g(\btheta_k;\calB_k) = \dfrac{2}{N_s}\sum_{i=1}^{N_s}\Eloc(\btheta_k;X_{k,i})\nabla_{\btheta}\log\abs{\psi_{\btheta_k}(X_{k,i})} - \dfrac{2}{N_s(N_s-1)}\sum_{1\le i\neq j\le N_s} \Eloc(\btheta_k;X_{k,i})\nabla_{\btheta}\log\abs{\psi_{\btheta_k}(X_{k,j})}.
     \end{equation*}
    Hence
    {\small\begin{equation*}
        \begin{split}
            g(\btheta_k;\calB_k)& - g(\btheta_k)= \dfrac{2}{N_s}\sum_{i=1}^{N_s}\lrbracket{ \Eloc(\btheta_k;X_{k,i})\nabla_{\btheta}\log\abs{\psi_{\btheta_k}(X_{k,i})} - \Expect_{X\sim\pi_{\btheta_k}}\lrsquare{ \Eloc(\btheta_k;X)\nabla_{\btheta}\log\abs{\psi_{\btheta_k}(X)} } }\\
            &-\dfrac{2}{N_s(N_s-1)}\sum_{1\le i\neq j\le N_s}\lrbracket{ \Eloc(\btheta_k;X_{k,i})\nabla_{\btheta}\log\abs{\psi_{\btheta_k}(X_{k,j})} - \Expect_{X\sim\pi_{\btheta_k}}\lrsquare{ \Eloc(\btheta_k;X) }\Expect_{X\sim\pi_{\btheta_k}}\lrsquare{\nabla_{\btheta}\log\abs{\psi_{\btheta_k}(X)} }}.
        \end{split}
    \end{equation*}}
    Using $\norm{a+b}^2\le 2\norm{a}^2 + 2\norm{b}^2$, we obtain
    {\small\begin{equation*}
        \begin{split}
            &\norm{g(\btheta_k;\calB_k) - g(\btheta_k)}^2  \le \dfrac{8}{N_s^2}\norm{\sum_{i=1}^{N_s}\lrbracket{ \Eloc(\btheta_k;X_{k,i})\nabla_{\btheta}\log\abs{\psi_{\btheta_k}(X_{k,i})} - \Expect_{X\sim\pi_{\btheta_k}}\lrsquare{ \Eloc(\btheta_k;X)\nabla_{\btheta}\log\abs{\psi_{\btheta_k}(X)} } }}^2\\
            &+\dfrac{8}{N_s^2(N_s-1)^2}\norm{\sum_{1\le i\neq j\le N_s}\lrbracket{ \Eloc(\btheta_k;X_{k,i})\nabla_{\btheta}\log\abs{\psi_{\btheta_k}(X_{k,j})} - \Expect_{X\sim\pi_{\btheta_k}}\lrsquare{ \Eloc(\btheta_k;X) }\Expect_{X\sim\pi_{\btheta_k}}\lrsquare{\nabla_{\btheta}\log\abs{\psi_{\btheta_k}(X)} }}}^2.
        \end{split}
    \end{equation*}}
    Taking expectation over $\calB_k$ and using \cref{assume:sample}, the cross terms involving different samples vanish. Therefore,
    {\small\begin{equation*}
        \begin{split}
       &\Expect_{\calB_k}\lrsquare{\norm{g(\btheta_k)-g(\btheta_k;\calB_k)}^2}\le \dfrac{8}{N_s}\Expect_{X\sim\pi_{\btheta_k}}\lrsquare{ \norm{\Eloc(\btheta_k;X)\nabla_{\btheta}\log\abs{\psi_{\btheta_k}(X)} - \Expect_{X\sim\pi_{\btheta_k}}\lrsquare{\Eloc(\btheta_k;X)\nabla_{\btheta}\log\abs{\psi_{\btheta_k}(X)}}}^2 }\\
    &+\dfrac{8}{N_s^2(N_s-1)^2}\sum_{\substack{i_1\neq j_1\text{ and }i_2\neq j_2,\\i_1=i_2 \text{ or }j_1=j_2}}\Expect_{(X,X')\sim \pi_{\btheta_k}^{\otimes^2}}\lrsquare{ \norm{ \Eloc(\btheta_k;X)\nabla_{\btheta}\log\abs{\psi_{\btheta_k}(X')} - \Expect_{X\sim\pi_{\btheta_k}}\lrsquare{ \Eloc(\btheta_k;X) }\Expect_{X\sim\pi_{\btheta_k}}\lrsquare{ \nabla_{\btheta}\log\abs{\psi_{\btheta_k}(X)} }}^2 }.
        \end{split}
    \end{equation*}}
       Then, by \cref{lem:counter_number}, the inequality $\Expect[\abs{Y-\Expect[Y]}^2]\le \Expect[\abs{Y}^2]$, and the Cauchy-Schwarz inequality, we have
   {\small \begin{equation*}
        \begin{split}
       \Expect_{\calB_k}\lrsquare{\norm{g(\btheta_k)-g(\btheta_k;\calB_k)}^2}&\le \dfrac{8}{N_s}\Expect_{X\sim\pi_{\btheta_k}}\lrsquare{ \norm{\Eloc(\btheta_k;X)\nabla_{\btheta}\log\abs{\psi_{\btheta_k}(X)}}^2 } + \dfrac{16}{N_s}\Expect_{(X,X')\sim \pi_{\btheta_k}^{\otimes^2}}\lrsquare{ \norm{\Eloc(\btheta_k;X)\nabla_{\btheta}\log\abs{\psi_{\btheta_k}(X')}}^2 } \\
       &\le \dfrac{24}{N_s}\sqrt{ \Expect_{X\sim\pi_{\btheta_k}}\lrsquare{ \norm{\Eloc(\btheta_k;X)}^4 }\Expect_{X\sim\pi_{\btheta_k}}\lrsquare{ \norm{\nabla_{\btheta}\log\abs{\psi_{\btheta_k}(X)}}^4 } }\le \dfrac{24C_m}{N_s}.
        \end{split}
    \end{equation*}}

   The bound for $S(\btheta_k;\calB_k)$ is similar:
    {\small\begin{equation*}
        \begin{split}
        &\Expect_{\calB_k}\lrsquare{\norm{S(\btheta_k)-S(\btheta_k;\calB_k)}_F^2}\\
        &\le \dfrac{2}{N_s}\Expect_{X\sim\pi_{\btheta_k}} \lrsquare{ \norm{ \nabla_{\btheta}\log\abs{\psi_{\btheta_k}(X)}\nabla_{\btheta}\log\abs{\psi_{\btheta_k}(X)}^\top -\Expect_{X\sim\pi_{\btheta_k}}\lrsquare{ \nabla_{\btheta}\log\abs{\psi_{\btheta_k}(X)}\nabla_{\btheta}\log\abs{\psi_{\btheta_k}(X)}^\top } }_F^2 }\\
        &+\dfrac{4}{N_s}\Expect_{(X,X')\sim\pi_{\btheta_k}^{\otimes^2}} \lrsquare{ \norm{ \nabla_{\btheta}\log\abs{\psi_{\btheta_k}(X)}\nabla_{\btheta}\log\abs{\psi_{\btheta_k}(X')}^\top -\Expect_{X\sim\pi_{\btheta_k}}\lrsquare{ \nabla_{\btheta}\log\abs{\psi_{\btheta_k}(X)}}\Expect_{X\sim\pi_{\btheta_k}}\lrsquare{\nabla_{\btheta}\log\abs{\psi_{\btheta_k}(X)}^\top } }_F^2 }\\
        &\le \dfrac{2}{N_s}\Expect_{X\sim\pi_{\btheta_k}} \lrsquare{ \norm{ \nabla_{\btheta}\log\abs{\psi_{\btheta_k}(X)}\nabla_{\btheta}\log\abs{\psi_{\btheta_k}(X)}^\top }_F^2 } +\dfrac{4}{N_s}\Expect_{(X,X')\sim\pi_{\btheta_k}^{\otimes^2}} \lrsquare{ \norm{ \nabla_{\btheta}\log\abs{\psi_{\btheta_k}(X)}\nabla_{\btheta}\log\abs{\psi_{\btheta_k}(X')}^\top }_F^2 }\\
        &=\dfrac{2}{N_s}\Expect_{X\sim\pi_{\btheta_k}} \lrsquare{ \norm{ \nabla_{\btheta}\log\abs{\psi_{\btheta_k}(X)} }_F^4 } + \dfrac{4}{N_s}\Expect_{(X,X')\sim\pi_{\btheta_k}^{\otimes^2}} \lrsquare{ \norm{ \nabla_{\btheta}\log\abs{\psi_{\btheta_k}(X)}}^2 \norm{ \nabla_{\btheta}\log\abs{\psi_{\btheta_k}(X')}}^2 }\\
        &\le \dfrac{2C_m}{N_s} + \dfrac{4}{N_s}\sqrt{\lrbracket{\Expect_{X\sim\pi_{\btheta_k}} \lrsquare{ \norm{ \nabla_{\btheta}\log\abs{\psi_{\btheta_k}(X)} }_F^4 } }^2}\le \dfrac{6C_m}{N_s}.
        \end{split}
    \end{equation*}}
\end{proof}

\begin{lem}
    Under \cref{assume:1,assume:sample}, let $\{\btheta_k,\Delta\btheta_{k}\}$ be generated by \eqref{eq:p-spring}, then for any $0\le \mu<1$, $k\ge 0$,
    \begin{equation*}
        \begin{split}
            \norm{ \Expect\lrsquare{ \Delta\btheta_k-\Delta\btheta_k^{\star}|\scrF_{k-1} } }^2 &\le \dfrac{6C_m\Expect\lrsquare{\norm{\Delta\btheta_k}^2|\scrF_{k-1}}}{\lambda^2 N_s},\\
            \Expect\lrsquare{ \norm{\Delta\btheta_k-\Delta\btheta_k^{\star}}^2|\scrF_{k-1} } &\le \lrbracket{1+2\mu^2 \norm{\Delta\btheta_{k-1}}^2+\dfrac{2C_m}{\lambda^2}} \dfrac{12C_m}{\lambda^2 N_s},
        \end{split}
    \end{equation*}
    where $\Delta\btheta_k^{\star}:=\lrbracket{\lambda I +S(\btheta_k)}^{-1}\lrbracket{\lambda\mu \Delta\btheta_{k-1}-\dfrac{1}{2}g(\btheta_k)}$, and $\scrF_{k-1}:=\sigma\lrbracket{\{X_{i,j}\}_{0\le i\le k-1,1\le j\le N_s}}$ is the $\sigma$-algebra generated by previous samples.
    \label{lem:biased_direction_error}
\end{lem}
\begin{proof}
    We first prove the first inequality.
   \begin{align*}
            \Delta\btheta_k&-\Delta\btheta_k^{\star} = \lrbracket{\lambda I +S(\btheta_k;\calB_k)}^{-1}\lrbracket{\lambda\mu\Delta\btheta_{k-1}-\dfrac{1}{2}g(\btheta_k;\calB_k)} -  \lrbracket{\lambda I +S(\btheta_k)}^{-1}\lrbracket{\lambda\mu\Delta\btheta_{k-1}-\dfrac{1}{2}g(\btheta_k)}\\
            &=\lrsquare{ \lrbracket{\lambda I + S(\btheta_k;\calB_k)}^{-1}-\lrbracket{\lambda I + S(\btheta_k)}^{-1} }\lrbracket{ \lambda\mu\Delta\btheta_{k-1}-\dfrac{1}{2}g(\btheta_k;\calB_k) } + \dfrac{1}{2}\lrbracket{\lambda I +S(\btheta_k)}^{-1}\lrbracket{g(\btheta_k)-g(\btheta_k;\calB_k)}\\
            &=\lrbracket{\lambda I +S(\btheta_k)}^{-1} \lrbracket{S(\btheta_k)-S(\btheta_k;\calB_k)} \lrbracket{\lambda I +S(\btheta_k;\calB_k)}^{-1}\lrbracket{ \lambda\mu\Delta\btheta_{k-1}-\dfrac{1}{2}g(\btheta_k;\calB_k) }\\
            &\quad\quad+ \dfrac{1}{2}\lrbracket{\lambda I +S(\btheta_k)}^{-1}\lrbracket{g(\btheta_k)-g(\btheta_k;\calB_k)}\\
            &=\lrbracket{\lambda I +S(\btheta_k)}^{-1} \lrbracket{S(\btheta_k)-S(\btheta_k;\calB_k)}\Delta\btheta_k+ \dfrac{1}{2}\lrbracket{\lambda I +S(\btheta_k)}^{-1}\lrbracket{g(\btheta_k)-g(\btheta_k;\calB_k)}.
    \end{align*}
    By \cref{assume:sample}, taking the conditional expectation given $\scrF_{k-1}$ yields
    \begin{equation*}
        \Expect\lrsquare{ \Delta\btheta_k-\Delta\btheta_k^{\star}|\scrF_{k-1} } = \lrbracket{\lambda I +S(\btheta_k)}^{-1} \Expect\lrsquare{\lrbracket{S(\btheta_k)-S(\btheta_k;\calB_k)}\Delta\btheta_k|\scrF_{k-1}}.
    \end{equation*}
    Then by the Cauchy-Schwarz inequality and \cref{lem:grad_S_mc_error}, we have
    \begin{equation*}
        \begin{split}
            \norm{ \Expect\lrsquare{ \Delta\btheta_k-\Delta\btheta_k^{\star}|\scrF_{k-1} } }&\le \dfrac{1}{\lambda}\Expect\lrsquare{ \norm{S(\btheta_k)-S(\btheta_k;\calB_k)}_2\cdot\norm{\Delta\btheta_k}|\scrF_{k-1} }\\
            &\le \dfrac{1}{\lambda}\sqrt{ \Expect\lrsquare{ \norm{S(\btheta_k)-S(\btheta_k;\calB_k)}_2^2|\scrF_{k-1} } \Expect\lrsquare{\norm{\Delta\btheta_k}^2|\scrF_{k-1} }}\\
            &\le \dfrac{1}{\lambda}\sqrt{ \Expect\lrsquare{ \norm{S(\btheta_k)-S(\btheta_k;\calB_k)}_F^2|\scrF_{k-1} } \Expect\lrsquare{\norm{\Delta\btheta_k}^2|\scrF_{k-1} }}\\
            &\le \dfrac{1}{\lambda}\sqrt{\dfrac{6C_m}{N_s}}\sqrt{ \Expect\lrsquare{\norm{\Delta\btheta_k}^2|\scrF_{k-1} } }.
        \end{split}
    \end{equation*}
      Next, we prove the second inequality. On the other hand, we can rewrite $\Delta\btheta_k-\Delta\btheta_k^{\star}$ as 
    \begin{equation*}
        \begin{split}
            \Delta\btheta_k-\Delta\btheta_k^{\star}   &=\lrbracket{\lambda I +S(\btheta_k)}^{-1} \lrbracket{S(\btheta_k)-S(\btheta_k;\calB_k)} \lrbracket{\lambda I +S(\btheta_k;\calB_k)}^{-1}\lrbracket{ \lambda\mu\Delta\btheta_{k-1}-\dfrac{1}{2}g(\btheta_k) }\\
            &\quad\quad+ \dfrac{1}{2}\lrbracket{\lambda I +S(\btheta_k;\calB_k)}^{-1}\lrbracket{g(\btheta_k)-g(\btheta_k;\calB_k)}.
        \end{split}
    \end{equation*}
    Since $S(\btheta_k;\calB_k)\succeq 0$ a.s., we have $\norm{(\lambda I +S(\btheta_k;\calB_k))^{-1}}_2\le \frac{1}{\lambda}$ a.s. Therefore, almost surely,
    \begin{equation*}
        \begin{split}
            \norm{\Delta\btheta_k-\Delta\btheta_k^{\star}}&\le \dfrac{1}{\lambda^2}\norm{S(\btheta_k)-S(\btheta_k;\calB_k)}_F\cdot\norm{\lambda\mu \Delta\btheta_{k-1}-\dfrac{1}{2}g(\btheta_k)} + \dfrac{1}{2\lambda}\norm{g(\btheta_k)-g(\btheta_k;\calB_k)}\\
            \Longrightarrow\norm{\Delta\btheta_k-\Delta\btheta_k^{\star}}^2&\le \dfrac{2}{\lambda^4}\norm{S(\btheta_k)-S(\btheta_k;\calB_k)}_F^2\cdot\norm{\lambda\mu \Delta\btheta_{k-1}-\dfrac{1}{2}g(\btheta_k)}^2 + \dfrac{1}{2\lambda^2}\norm{g(\btheta_k)-g(\btheta_k;\calB_k)}^2.
        \end{split}
    \end{equation*}
      Taking the conditional expectation given $\scrF_{k-1}$ and using \cref{lem:g_S_bounded,lem:grad_S_mc_error} yields
    \begin{equation*}
        \begin{split}
            \Expect\lrsquare{ \norm{\Delta\btheta_k-\Delta\btheta_k^{\star}}^2|\scrF_{k-1} } &\le \dfrac{2\norm{\lambda\mu \Delta\btheta_{k-1}-\dfrac{1}{2}g(\btheta_k)}^2}{\lambda^4}\Expect\lrsquare{ \norm{S(\btheta_k)-S(\btheta_k;\calB_k)}_F^2|\scrF_{k-1} } \\
            &\qquad + \dfrac{1}{2\lambda^2}\Expect\lrsquare{ \norm{g(\btheta_k)-g(\btheta_k;\calB_k)}^2|\scrF_{k-1} }\\
            &\le \dfrac{12C_m}{\lambda^4N_s}\lrbracket{ 2\lambda^2\mu^2\norm{\Delta\btheta_{k-1}}^2+\dfrac{1}{2}\norm{g(\btheta_k)}^2 } + \dfrac{12C_m}{\lambda^2N_s}\\
            &\le \dfrac{12C_m}{\lambda^4N_s}\lrbracket{ 2\lambda^2\mu^2\norm{\Delta\btheta_{k-1}}^2+2C_m } + \dfrac{12C_m}{\lambda^2N_s}\\
            & = \lrbracket{1+2\mu^2 \norm{\Delta\btheta_{k-1}}^2+\dfrac{2C_m}{\lambda^2}} \dfrac{12C_m}{\lambda^2 N_s}.
        \end{split}
    \end{equation*}
\end{proof}

\subsection{Estimation of $\Delta\theta_k$}

\begin{lem}
    Under \cref{assume:1,assume:sample}, let $\{\btheta_k,\Delta\btheta_k\}$ be generated from \eqref{eq:p-spring}, for any $0\le \mu<1$, $k\ge 0$, we have:
    \begin{equation*}
        \Expect[\norm{\Delta\btheta_k}^2] \le \Expect[\norm{\Delta\btheta_0}^2] + \dfrac{4(1+\mu^2)C_m}{\lambda^2(1-\mu^2)^2} +\dfrac{24(1+\mu^2)C_m}{\lambda^2(1-\mu^2)^2}\cdot\dfrac{1}{N_s},
    \end{equation*}
    where $\Expect[\cdot]$ is the total expectation.
    \label{lem:expect_delta_theta_bound}
\end{lem}
\begin{proof}
    By $\Delta\btheta_k=(\lambda I +S(\btheta_k;\calB_k))^{-1}(\lambda\mu \Delta\btheta_{k-1}-\dfrac{1}{2}g(\btheta_k;\calB_k))$ and Young inequality, almost surely,
    \begin{equation*}
    \begin{split}
        \norm{\Delta\btheta_k}^2 &\le \lrbracket{\mu\norm{\Delta\btheta_{k-1}} + \dfrac{1}{2\lambda}\norm{g(\btheta_k;\calB_k)}}^2\\
        &\le \dfrac{\mu^2+1}{2}\norm{\Delta\btheta_{k-1}}^2 + \dfrac{1+\mu^2}{4\lambda^2 (1-\mu^2)}\norm{g(\btheta_k;\calB_k)}^2\\
            &\le \dfrac{\mu^2+1}{2}\norm{\Delta\btheta_{k-1}}^2 + \dfrac{1+\mu^2}{2\lambda^2(1-\mu^2)} \norm{g(\btheta_k;\calB_k)-g(\btheta_k)}^2 + \dfrac{1+\mu^2}{2\lambda^2(1-\mu^2)}\norm{g(\btheta_k)}^2.
    \end{split}
    \end{equation*}
Taking the conditional expectation given $\scrF_{k-1}$ and using \cref{lem:g_S_bounded,lem:grad_S_mc_error} yields
\begin{equation*}
    \Expect[\norm{\Delta\btheta_k}^2|\scrF_{k-1}]\le \dfrac{\mu^2+1}{2}\norm{\Delta\btheta_{k-1}}^2+ \dfrac{12(1+\mu^2)C_m}{\lambda^2(1-\mu^2)}\cdot\dfrac{1}{N_s} + \dfrac{2(1+\mu^2)C_m}{\lambda^2(1-\mu^2)}.
\end{equation*}
Taking total expectation, we obtain
\begin{equation*}
    \begin{split}
        \Expect[\norm{\Delta\btheta_k}^2] &\le \dfrac{\mu^2+1}{2}\Expect[\norm{\Delta\btheta_{k-1}}^2]+ \dfrac{12(1+\mu^2)C_m}{\lambda^2(1-\mu^2)}\cdot\dfrac{1}{N_s} + \dfrac{2(1+\mu^2)C_m}{\lambda^2(1-\mu^2)}\\
        &=\Expect[\norm{\Delta\btheta_{0}}^2]+\lrbracket{\dfrac{12(1+\mu^2)C_m}{\lambda^2(1-\mu^2)}\cdot\dfrac{1}{N_s} + \dfrac{2(1+\mu^2)C_m}{\lambda^2(1-\mu^2)}}\lrbracket{1+\dfrac{1+\mu^2}{2}+\cdots +\lrbracket{\dfrac{1+\mu^2}{2}}^{k-1}}\\
        &\le \Expect[\norm{\Delta\btheta_{0}}^2]+\lrbracket{\dfrac{12(1+\mu^2)C_m}{\lambda^2(1-\mu^2)}\cdot\dfrac{1}{N_s} + \dfrac{2(1+\mu^2)C_m}{\lambda^2(1-\mu^2)}}\sum_{i=0}^{\infty}\lrbracket{\dfrac{1+\mu^2}{2}}^{i}\\
        &=\Expect[\norm{\Delta\btheta_0}^2] + \dfrac{4(1+\mu^2)C_m}{\lambda^2(1-\mu^2)^2} +\dfrac{24(1+\mu^2)C_m}{\lambda^2(1-\mu^2)^2}\cdot\dfrac{1}{N_s}.
    \end{split}
\end{equation*}
\end{proof}

\begin{lem}
    Under \cref{assume:1,assume:sample}, let $\{\btheta_k,\Delta\btheta_k\}$ be generated from \eqref{eq:p-spring}, for any $0\le \mu<1$, there exists $C_1,C_2>0$, such that, for any $K\ge 1$, we have:
    \begin{equation*}
        \sum_{k=1}^K \eta_k \Expect[\norm{\Delta\btheta_k}^2] \le \dfrac{\Expect[L(\btheta_1)]-E_{\gs}}{\lambda(1-\mu)} + \dfrac{\mu\eta_0\Expect[\norm{\Delta\btheta_0}^2]}{1-\mu} + \dfrac{C_1}{\lambda(1-\mu)}\dfrac{\sum_{k=1}^K\eta_k}{N_s} + \dfrac{C_2}{\lambda(1-\mu)}\dfrac{\sum_{k=1}^K \eta_k}{N_s^2}.
    \end{equation*}
    \label{lem:sum_delta_theta}
\end{lem}
\begin{proof}
    By \cref{lem:2_upper_bound}, and $g(\btheta_k)=2\lambda\mu\Delta\btheta_{k-1}-2(\lambda I +S(\btheta_k))\Delta\btheta_k^{\star}$, where $\Delta\btheta_k^{\star}$ is defined in \cref{lem:biased_direction_error}, we have
    \begin{align*}
    &L(\btheta_{k+1})-L(\btheta_k)\\
    &\le 2\lambda\mu \eta_k \Delta\btheta_k^\top \Delta\btheta_{k-1} - 2\eta_k \Delta\btheta_k(\lambda I +S(\btheta_k))\Delta\btheta_k^{\star} + \dfrac{C_g}{2}\eta_k^2\norm{\Delta\btheta_k}^2\\
            & = \lambda\mu\eta_k( \norm{\Delta\btheta_k}^2 + \norm{\Delta\btheta_{k-1}}^2 - \norm{\Delta\btheta_k-\Delta\btheta_{k-1}}^2 )-2\eta_k\Delta\btheta_k^\top (\lambda I +S(\btheta_k))\Delta\btheta_k \\
            &\qquad+ 2\eta_k\Delta\btheta_k^\top(\lambda I +S(\btheta_k))(\Delta\btheta_k-\Delta\btheta_k^{\star}) + \dfrac{C_g}{2}\eta_k^2\norm{\Delta\btheta_k}^2\\
            &(\text{by }\eta_k\le \eta_0)\\
            &\le\lambda\mu\eta_k\norm{\Delta\btheta_{k-1}}^2 + (\lambda(\mu-2)+\dfrac{C_g}{2}\eta_0)\eta_k\norm{\Delta\btheta_k}^2 + 2\eta_k\norm{\lambda I +S(\btheta_k)}_2 \norm{\Delta\btheta_k}\norm{\Delta\btheta_k-\Delta\btheta_k^{\star}}\\
            &(\text{by triangle inequality and Young inequality})\\
            &\le \lambda\mu\eta_k\norm{\Delta\btheta_{k-1}}^2 + (\lambda(\mu-2)+\dfrac{C_g}{2}\eta_0)\eta_k\norm{\Delta\btheta_k}^2 + 2\eta_k (\lambda + \norm{S(\btheta_k)}_F)\lrbracket{\varepsilon \norm{\Delta\btheta_k}^2 + \dfrac{1}{4\varepsilon}\norm{\Delta\btheta_k-\Delta\btheta_k^{\star}}^2}\\
            &(\text{by \cref{lem:g_S_bounded} and let }\varepsilon = \dfrac{2\lambda(1-\mu)-C_g\eta_0}{4(\lambda+\sqrt{C_m})})\\
            &\le \lambda\mu\eta_k\norm{\Delta\btheta_{k-1}}^2 + (\lambda(\mu-2)+\dfrac{C_g}{2}\eta_0)\eta_k\norm{\Delta\btheta_k}^2 + (\lambda(1-\mu)-\dfrac{C_g}{2}\eta_0)\eta_k\norm{\Delta\btheta_k}^2 \\
            &\qquad + \dfrac{2(\lambda+\sqrt{C_m})^2}{2\lambda(1-\mu)-C_g\eta_0}\eta_k\norm{\Delta\btheta_k-\Delta\btheta_k^{\star}}^2\\
            &=\lambda\mu\eta_k\norm{\Delta\btheta_{k-1}}^2 - \lambda\eta_k\norm{\Delta\btheta_k}^2 + \dfrac{2(\lambda+\sqrt{C_m})^2}{2\lambda(1-\mu)-C_g\eta_0} \eta_k\norm{\Delta\btheta_k-\Delta\btheta_k^{\star}}^2.
    \end{align*}
    Taking conditional expectation conditional on $\scrF_{k-1}$, and by \cref{lem:biased_direction_error}, we have
    \begin{equation*}
        \Expect[L(\btheta_{k+1})|\scrF_{k-1}]-L(\btheta_k)\le \lambda\mu\eta_k\norm{\Delta\btheta_{k-1}}^2 - \lambda\eta_k\Expect[\norm{\Delta\btheta_k}^2|\scrF_{k-1}] +\tilde{C}_1\dfrac{\eta_k}{N_s} + \tilde{C}_2\dfrac{\eta_k\norm{\Delta\btheta_{k-1}}^2}{N_s},
    \end{equation*}
    where $\tilde{C}_1:=\dfrac{24C_m(\lambda+\sqrt{C_m})^2}{\lambda^2(2\lambda(1-\mu)-C_g\eta_0)}\lrbracket{1+\dfrac{2C_m}{\lambda^2}}$, and $\tilde{C}_2:=\dfrac{48C_m\mu^2(\lambda+\sqrt{C_m})^2}{\lambda^2(2\lambda(1-\mu)-C_g\eta_0)}$.

    Taking total expectation,
    \begin{align*}
        \Expect[L(\btheta_{k+1})]-\Expect[L(\btheta_k)]&\le \lambda\mu\eta_k \Expect[\norm{\Delta\btheta_{k-1}}^2] - \lambda\eta_k\Expect[\norm{\Delta\btheta_k}^2] + \tilde{C}_2\dfrac{\eta_k\Expect[\norm{\Delta\btheta_{k-1}}^2]}{N_s} + \tilde{C}_1\dfrac{\eta_k}{N_s}\\
        & (\text{for the third term, by \cref{lem:expect_delta_theta_bound}})\\
        &\le \lambda\mu\eta_k \Expect[\norm{\Delta\btheta_{k-1}}^2] - \lambda\eta_k\Expect[\norm{\Delta\btheta_k}^2] + C_1\dfrac{\eta_k}{N_s} + C_2\dfrac{\eta_k}{N_s^2},
    \end{align*}
    where $C_1:=\tilde{C}_1+\lrbracket{\Expect[\norm{\Delta\btheta_0}^2] +\dfrac{4(1+\mu^2)C_m}{\lambda^2(1-\mu^2)^2}} \tilde{C}_2$, and $C_2:=\dfrac{24(1+\mu^2)C_m\tilde{C}_2}{\lambda^2(1-\mu^2)^2}$.

    Define the discrete energy $F_k:=\Expect[L(\btheta_k)]+\lambda\mu\eta_k\Expect[\norm{\Delta\btheta_{k-1}}^2]$, and then again by $\eta_{k+1}\le \eta_k$, we have
    \begin{align*}
        F_{k+1}-F_k&=\Expect[L(\btheta_{k+1})]-\Expect[L(\btheta_k)] + \lambda\mu\eta_{k+1}\Expect[\norm{\Delta\btheta_k}^2]-\lambda\mu\eta_k\Expect[\norm{\Delta\btheta_{k-1}}^2]\\
        &\le \lambda(\mu-1)\eta_k\Expect[\norm{\Delta\btheta_k}^2] + C_1\dfrac{\eta_k}{N_s} + C_2\dfrac{\eta_k}{N_s^2}.
    \end{align*}
    Rearranging terms and summing from $k=1$ to $k=K$,  we get 
    \begin{align*}
        \lambda(1-\mu)\sum_{k=1}^K\eta_k \Expect[\norm{\Delta \btheta_k}^2]&\le F_1-F_{K+1} +\dfrac{C_1}{N_s}\sum_{k=1}^K\eta_k + \dfrac{C_2}{N_s^2}\sum_{k=1}^K\eta_k\\
        &\le \Expect[L(\btheta_1)]-\Expect[L(\btheta_{K+1})] + \lambda\mu\eta_0\Expect[\norm{\Delta\btheta_0}^2]+\dfrac{C_1}{N_s}\sum_{k=1}^K\eta_k + \dfrac{C_2}{N_s^2}\sum_{k=1}^K\eta_k\\
        &\le \Expect[L(\btheta_1)]-E_{\gs} + \lambda\mu\eta_0\Expect[\norm{\Delta\btheta_0}^2]+\dfrac{C_1}{N_s}\sum_{k=1}^K\eta_k + \dfrac{C_2}{N_s^2}\sum_{k=1}^K\eta_k.
    \end{align*}
    Dividing both sides by $\lambda(1-\mu)$ yields the claim.
\end{proof}

\subsection{Proof of \cref{the:p_spring_convergence}}

\par Now, we are ready to prove Theorem \ref{the:p_spring_convergence}.
\begin{proof}[Proof of Theorem \ref{the:p_spring_convergence}]
    By \cref{lem:2_upper_bound}, we have:
    \begin{align*}
        L(\btheta_{k+1})-L(\btheta_k)&\le \eta_kg(\btheta_k)^\top \Delta\btheta_k + \dfrac{C_g}{2}\eta_k^2\norm{\Delta\btheta_k}^2\\
        &= \eta_kg(\btheta_k)^\top \Delta\btheta_k^{\star} + \eta_kg(\btheta_k)^\top (\Delta\btheta_k-\Delta\btheta_k^{\star}) + \dfrac{C_g}{2}\eta_k^2\norm{\Delta\btheta_k}^2.
    \end{align*}
    We start by estimating $g(\btheta_k)^\top \Delta\btheta_k^{\star}$. By \cref{lem:g_S_bounded} and Young inequality, we have
    \begin{align*}
        g(\btheta_k)^\top\Delta\btheta_k^{\star}& = -\dfrac{1}{2}g(\btheta_k)^\top (\lambda I+S(\btheta_k))^{-1}g(\btheta_k) + \lambda\mu g(\btheta_k)^\top (\lambda I +S(\btheta_k))^{-1}\Delta\btheta_{k-1} \\
        &\le -\dfrac{1}{2(\lambda+\sqrt{C_m})} \norm{g(\btheta_k)}^2 + \mu \norm{g(\btheta_k)} \norm{\Delta\btheta_{k-1}}\\
        &\le -\dfrac{1}{2(\lambda+\sqrt{C_m})} \norm{g(\btheta_k)}^2 + \dfrac{\mu}{8(\lambda+\sqrt{C_m})} \norm{g(\btheta_k)}^2 + 2\mu(\lambda +\sqrt{C_m}) \norm{\Delta\btheta_{k-1}}^2\\
        &=\dfrac{\mu-4}{8(\lambda+\sqrt{C_m})}\norm{g(\btheta_k)}^2 + 2\mu(\lambda +\sqrt{C_m})\norm{\Delta\btheta_{k-1}}^2.
    \end{align*}
And thus, we have
\begin{equation*}
    L(\btheta_{k+1})-L(\btheta_k)\le \dfrac{\mu-4}{8(\lambda+\sqrt{C_m})}\eta_k\norm{g(\btheta_k)}^2 + 2\mu(\lambda +\sqrt{C_m})\eta_k\norm{\Delta\btheta_{k-1}}^2+\eta_kg(\btheta_k)^\top (\Delta\btheta_k-\Delta\btheta_k^{\star}) + \dfrac{C_g}{2}\eta_k^2\norm{\Delta\btheta_k}^2.
\end{equation*}
Take the conditional expectation conditional on $\scrF_{k-1}$ and use Cauchy-Schwarz and Young inequalities, and by \cref{lem:biased_direction_error},
{\small\begin{align*}
    &\Expect[L(\btheta_{k+1})|\scrF_{k-1}] -L(\btheta_k)\\
    &\le \dfrac{\mu-4}{8(\lambda+\sqrt{C_m})}\eta_k\norm{g(\btheta_k)}^2 + 2\mu(\lambda +\sqrt{C_m})\eta_k\norm{\Delta\btheta_{k-1}}^2 + \eta_k \norm{g(\btheta_k)}\cdot\norm{\Expect[\Delta\btheta_k-\Delta\btheta_k^{\star}|\scrF_{k-1}]} + \dfrac{C_g}{2}\eta_k^2 \Expect[\norm{\Delta\btheta_k}^2|\scrF_{k-1}]\\
    &\le \dfrac{\mu-4}{8(\lambda+\sqrt{C_m})}\eta_k\norm{g(\btheta_k)}^2 + 2\mu(\lambda +\sqrt{C_m})\eta_k\norm{\Delta\btheta_{k-1}}^2 +\dfrac{\eta_k}{8(\lambda+\sqrt{C_m})}\norm{g(\btheta_k)}^2 + 2(\lambda+\sqrt{C_m})\eta_k\norm{\Expect[\Delta\btheta_k-\Delta\btheta_k^{\star}|\scrF_{k-1}]}^2\\
    &\qquad + \dfrac{C_g}{2}\eta_k^2 \Expect[\norm{\Delta\btheta_k}^2|\scrF_{k-1}]\\
    &<-\dfrac{1}{4(\lambda+\sqrt{C_m})}\eta_k\norm{g(\btheta_k)}^2+ 2\mu(\lambda +\sqrt{C_m})\eta_k\norm{\Delta\btheta_{k-1}}^2+ 2(\lambda+\sqrt{C_m})\eta_k\norm{\Expect[\Delta\btheta_k-\Delta\btheta_k^{\star}|\scrF_{k-1}]}^2+ \dfrac{C_g}{2}\eta_k^2 \Expect[\norm{\Delta\btheta_k}^2|\scrF_{k-1}]\\
    &\le -\dfrac{1}{4(\lambda+\sqrt{C_m})}\eta_k\norm{g(\btheta_k)}^2+ 2\mu(\lambda +\sqrt{C_m})\eta_k\norm{\Delta\btheta_{k-1}}^2+ \dfrac{12C_m(\lambda+\sqrt{C_m})}{\lambda^2N_s}\eta_k\Expect[\norm{\Delta\btheta_k}^2|\scrF_{k-1}]+ \dfrac{C_g}{2}\eta_k^2 \Expect[\norm{\Delta\btheta_k}^2|\scrF_{k-1}].
\end{align*}}
       Taking total expectation, and by $\eta_k\le \eta_{k-1}\le \eta_0$ and \cref{lem:expect_delta_theta_bound},
       {\small\begin{align*}
        &\Expect[L(\btheta_{k+1})]-\Expect[L(\btheta_k)]\\
        &\le  -\dfrac{1}{4(\lambda+\sqrt{C_m})}\eta_k\Expect[\norm{g(\btheta_k)}^2]+ 2\mu(\lambda +\sqrt{C_m})\eta_k\Expect[\norm{\Delta\btheta_{k-1}}^2]+ \dfrac{12C_m(\lambda+\sqrt{C_m})}{\lambda^2N_s}\eta_k\Expect[\norm{\Delta\btheta_k}^2]+ \dfrac{C_g}{2}\eta_k^2 \Expect[\norm{\Delta\btheta_k}^2]\\
        &\le -\dfrac{1}{4(\lambda+\sqrt{C_m})}\eta_k\Expect[\norm{g(\btheta_k)}^2]+ 2\mu(\lambda +\sqrt{C_m})\eta_{k-1}\Expect[\norm{\Delta\btheta_{k-1}}^2]+\tilde{C}_1\dfrac{\eta_k}{N_s}+\tilde{C}_2\dfrac{\eta_k}{N_s^2} + \dfrac{C_g\eta_0}{2}\eta_k \Expect[\norm{\Delta\btheta_k}^2],
       \end{align*}}
       where $\tilde{C}_1=\dfrac{12C_m(\lambda+\sqrt{C_m})}{\lambda^2}\lrbracket{\Expect[\norm{\Delta\btheta_0}^2] + \dfrac{4(1+\mu^2)C_m}{\lambda^2(1-\mu^2)^2}}$, and $\tilde{C}_2:=\dfrac{288C_m^2(1+\mu^2)(\lambda+\sqrt{C_m})}{\lambda^4(1-\mu^2)^2}$.

       Rearranging terms, we get:
       {\small\begin{align*}
           \eta_k\Expect[\norm{g(\btheta_k)}^2]\le &4(\lambda+\sqrt{C_m})(\Expect[L(\btheta_k)]-\Expect[L(\btheta_{k+1})]) + 8\mu(\lambda+\sqrt{C_m})^2\eta_{k-1}\Expect[\norm{\Delta\btheta_{k-1}}^2]\\
           &+4(\lambda+\sqrt{C_m})\tilde{C}_1\dfrac{\eta_k}{N_s} + 4(\lambda+\sqrt{C_m})\tilde{C}_2\dfrac{\eta_k}{N_s^2} +2C_g(\lambda+\sqrt{C_m})\eta_0\eta_k\Expect[\norm{\Delta\btheta_k}^2].
       \end{align*}}
       Summing from $k=1$ to $k=K$, and by \cref{lem:sum_delta_theta},
       {\small\begin{align*}
          \sum_{k=1}^K \eta_k\Expect[\norm{g(\btheta_k)}^2] &\le 4(\lambda+\sqrt{C_m})(\Expect[L(\btheta_1)]-\Expect[L(\btheta_{K+1})]) +4(\lambda+\sqrt{C_m})\tilde{C}_1\dfrac{\sum_{k=1}^K\eta_k}{N_s} +4(\lambda+\sqrt{C_m})\tilde{C}_2\dfrac{\sum_{k=1}^K\eta_k}{N_s^2}  \\
          &\qquad+ 2(\lambda+\sqrt{C_m})(4\mu(\lambda+\sqrt{C_m})+C_g\eta_0)\sum_{k=0}^{K}\eta_k\Expect[\norm{\Delta\btheta_{k}}^2]\\
          &\le 4(\lambda+\sqrt{C_m})(\Expect[L(\btheta_1)]-E_{\gs}) +4(\lambda+\sqrt{C_m})\tilde{C}_1\dfrac{\sum_{k=1}^K\eta_k}{N_s} +4(\lambda+\sqrt{C_m})\tilde{C}_2\dfrac{\sum_{k=1}^K\eta_k}{N_s^2}  \\
          &\qquad+ 2(\lambda+\sqrt{C_m})(4\mu(\lambda+\sqrt{C_m})+C_g\eta_0)\sum_{k=0}^{K}\eta_k\Expect[\norm{\Delta\btheta_{k}}^2]\\
          &\le M_1 +M_2 \dfrac{\sum_{k=1}^K\eta_k}{N_s} + M_3\dfrac{\sum_{k=1}^K\eta_k}{N_s^2},
       \end{align*}}
       where $M_1,M_2,M_3>0$ are constants defined as:
       {\small
       \begin{align*}
           M_1&:=2(\lambda+\sqrt{C_m})\lrbracket{2+\dfrac{4\mu(\lambda+\sqrt{C_m})+C_g\eta_0}{\lambda(1-\mu)}}\lrbracket{\Expect[L(\btheta_1)]-E_{\gs}} + \dfrac{2\eta_0(\lambda+\sqrt{C_m})}{1-\mu}\lrsquare{4\mu (\lambda+\sqrt{C_m})+C_g\eta_0}\Expect[\norm{\Delta\btheta_0}^2]\\
           M_2&:=2(\lambda+\sqrt{C_m})\lrbracket{2\tilde{C}_1 + \dfrac{4\mu(\lambda+\sqrt{C_m}C_1+C_gC_1\eta_0)}{\lambda(1-\mu)}}\\
           M_3&:=2(\lambda+\sqrt{C_m})\lrbracket{2\tilde{C}_2 + \dfrac{4\mu(\lambda+\sqrt{C_m}C_2+C_gC_2\eta_0)}{\lambda(1-\mu)}},
       \end{align*}
       }
       where $C_1,C_2>0$ are constants in \cref{lem:sum_delta_theta}. 
\end{proof}

\section{Proof of \cref{the:counter_less_1_converge}}
\label{sec:proof_counter}

\begin{proof}
According to the proof of \cref{the:full_spring_convergence} in Appendix~\ref{sec:proof_full_convergence}, it suffices to show that the gradient $g(\btheta)$ is uniformly bounded and Lipschitz continuous. It is therefore enough to prove that both $g(\btheta)$ and $\nabla^2L(\btheta)$ are uniformly bounded.

\medskip
\noindent\textbf{Continuous Gaussian example.}
Since a constant does not change the function value, we consider the normalized wavefunction
\begin{equation*}
    \psi_\theta(x)=C_\Sigma e^{-\frac14 (x-A\theta)^\top \Sigma^{-1}(x-A\theta)},
\qquad \bfx=(\bx_1,\dots,\bx_N),\quad \bx_i\in\R^3,
\end{equation*}
where\(A\in\mathbb R^{3N\times N_p}\),\(\Sigma\in\mathbb R^{3N\times 3N}\) is symmetric positive definite and \(C_\Sigma = (2\pi)^{-3N/4}\det (\Sigma)^{-1/4}\). Then \(\pi_{\btheta}(\bfx)=|\psi_{\btheta}(\bfx)|^2\) is the density of the Gaussian distribution $X\sim N(A\btheta,\Sigma)$.

The electronic Hamiltonian is
\begin{equation*}
    \scrH=-\frac12\sum_{i=1}^N \Delta_{\bx_i}
-\sum_{i=1}^N\sum_{I=1}^M \frac{Z_I}{\|\bx_i-\RR_I\|}
+\sum_{1\le i<j\le N}\frac1{\|\bx_i-\bx_j\|}
+C_{\rm nuc},\quad C_{\rm nuc}:=\sum_{1\le I<I'\le M}\frac{Z_I Z_{I'}}{\|\RR_I-\RR_{I'}\|}.
\end{equation*}
Hence $L(\btheta)=\braket{\psi_{\btheta}|\scrH|\psi_{\btheta}}$ can be decomposed into a kinetic part and a potential part. For the kinetic term, using integration by parts
\begin{equation*}
    -\dfrac{1}{2}\sum_{i=1}^N\braket{\psi_{\btheta}|\Delta_{\bx_i} |\psi_{\btheta}} = \dfrac{1}{2}\int_{\R^{3N}}\norm{\nabla_{\bfx}\psi_{\btheta}(\bfx)}^2\dd\bfx=\dfrac{1}{8}\Expect_{X\sim \pi_{\btheta}}\lrsquare{(X-A\btheta)^\top \Sigma^{-2}(X-A\btheta)}=\dfrac{1}{8}\text{tr}(\Sigma^{-1}),
\end{equation*}
where the last equality is from the the standard identity for centered Gaussian quadratic forms: if \(Y\sim N(0,\Sigma)\), then for any matrix \(M\), $\Expect[Y^\top MY]=\text{tr}(M\Sigma)$.

For the potential terms, we first rewrite $A$ and $\Sigma$ in block forms:
\begin{equation*}
    A=
\begin{pmatrix}
A_1\\ \vdots \\ A_N
\end{pmatrix},\quad \Sigma=\begin{pmatrix}
    \Sigma_{11} & \cdots & \Sigma_{1N}\\
    \vdots &\ddots & \vdots\\
    \Sigma_{N1} & \cdots & \Sigma_{NN}
\end{pmatrix},\quad A_i\in \R^{3\times N_p},~\Sigma_{ij}\in \R^{3\times 3}.
\end{equation*}
Then the random variable of each electronic coordination $X_i\sim N(A_i\btheta,\Sigma_{ii})$, and therefore
\[
\Expect_{X\sim \pi_{\btheta}}\!\left[\frac1{\|X_i-\RR_I\|}\right]
=
\frac{1}{(2\pi)^{3/2}(\det \Sigma_{ii})^{1/2}}
\int_{\mathbb R^3}\frac1{\|\by\|}
e^{
-\frac12 (\by-(A_i\btheta-\RR_I))^\top \Sigma_{ii}^{-1}(\by-(A_i\btheta-\RR_I))} \dd\by.
\]
For the electron-electron term, define $Y_{ij}:=X_i-X_j$, and \(Y_{ij}\) is again Gaussian with the mean and covariance
\begin{equation*}
    \mathbb E[Y_{ij}] = (A_i-A_j)\btheta,\quad \operatorname{Cov}(Y_{ij}) =\Sigma_{ii}+\Sigma_{jj}-\Sigma_{ij}-\Sigma_{ji}=:\Gamma_{ij}.
\end{equation*}
Therefore, we have
\[
\Expect_{X\sim \pi_{\btheta}}\!\left[\frac1{\|X_i-X_j\|}\right]
=
\Expect\!\left[\frac1{\|Y_{ij}\|}\right]
=
\frac{1}{(2\pi)^{3/2}(\det \Gamma_{ij})^{1/2}}
\int_{\mathbb R^3}\frac1{\|\by\|}
e^{
-\frac12 (\by-(A_i-A_j)\btheta)^\top \Gamma_{ij}^{-1}(\by-(A_i-A_j)\btheta)}
\dd\by.
\]

For any symmetric positive definite \(3\times 3\) matrix \(\Gamma\), define
\[
U_\Gamma(\bz)
:=
\frac{1}{(2\pi)^{3/2}(\det\Gamma)^{1/2}}
\int_{\mathbb R^3}\frac{1}{\|\by\|}
e^{-\frac12 (\by-\bz)^\top \Gamma^{-1}(\by-\bz)}\dd\by.
\]
Then the above two identities become
\[
\Expect_{X\sim \pi_{\btheta}}\!\left[\frac1{\|X_i-\RR_I\|}\right]=U_{\Sigma_{ii}}(A_i\btheta-\RR_I),
\qquad
\Expect_{X\sim \pi_{\btheta}}\!\left[\frac1{\|X_i-X_j\|}\right]=U_{\Gamma_{ij}}((A_i-A_j)\btheta).
\]
Hence
\[
L(\theta)
=
\frac18\,\operatorname{tr}(\Sigma^{-1})
-\sum_{i=1}^N\sum_{I=1}^M Z_I\,U_{\Sigma_{ii}}(A_i\btheta-\RR_I)
+\sum_{1\le i<j\le N}U_{\Gamma_{ij}}((A_i-A_j)\btheta)
+C_{\rm nuc}.
\]

\medskip
Let
\[
\gamma_\Gamma(\bw)
=
\frac{1}{(2\pi)^{3/2}(\det\Gamma)^{1/2}}
e^{-\frac12 \bw^\top \Gamma^{-1}\bw},
\]
then we have $U_\Gamma(\bz)=\displaystyle\int_{\R^3}\dfrac{1}{\norm{\by}}\gamma_{\Gamma}(\by-\bz)\dd\by$. Differentiating with respect to $\bz$ gives
\begin{equation*}
    \nabla_{\bz} U_\Gamma(\bz)
=
\int_{\mathbb R^3}\frac1{\|\by\|}\,\Gamma^{-1}(\by-\bz)\,\gamma_\Gamma(\by-\bz)\dd\by, \quad \nabla^2_{\bz} U_\Gamma(\bz)
=
\int_{\mathbb R^3}\frac1{\|\by\|}
\Big(\Gamma^{-1}(\by-\bz)(\by-\bz)^\top\Gamma^{-1}-\Gamma^{-1}\Big)\gamma_\Gamma(\by-\bz)\dd\by.
\end{equation*}
Since \(\Gamma^{-1}\) is fixed and positive definite, there exist constants \(C_1,C_2,C_3,c_1>0\) such that
\[
\|\Gamma^{-1}(\by-\bz)\gamma_\Gamma(\by-\bz)\|
\le C_1\|\by-\bz\|e^{-c_1\|\by-\bz\|^2},
\]
and
\begin{equation*}
\norm{\Big(\Gamma^{-1}(\by-\bz)(\by-\bz)^\top\Gamma^{-1}-\Gamma^{-1}\Big)\gamma_\Gamma(\by-\bz)}
\le (C_2+C_3\|\by-\bz\|^2)e^{-c_1\|\by-\bz\|^2}.    
\end{equation*}
Therefore
\begin{equation*}
    \norm{\nabla_{\bz} U_\Gamma(\bz)}\le C_1\int_{\R^3}\dfrac{\norm{\by-\bz}}{\norm{\by}}e^{-c_1\norm{\by-\bz}^2}\dd\by = C_1\int_{\R^3}\dfrac{\norm{\bu}}{\norm{\bu+\bz}}e^{-c_1\norm{\bu}^2}\dd\bu,
\end{equation*}
and
\begin{equation*}
    \|\nabla^2_{\bz} U_\Gamma(\bz)\|
\le
\int_{\mathbb R^3}\frac{C_2+C_3\|\by-\bz\|^2}{\norm{\by}}e^{-c_1\|\by-\bz\|^2}\dd\by = \int_{\mathbb R^3}\frac{C_2+C_3\|\bu\|^2}{\norm{\bu+\bz}}e^{-c_1\|\bu\|^2}\dd\bu,
\end{equation*}
where $\bu=\by-\bz$. We split the domain into $\{\|\bu+\bz\|\le 1\}\bigcup \{\|\bu+\bz\|>1\}$. On \(\{\|\bu+\bz\|\le 1\}\), using the local integrability of \(\|v\|^{-1}\) in \(\mathbb R^3\),
\[
\int_{\R^3}\dfrac{\norm{\bu}^k}{\norm{\bu+\bz}}e^{-c_1\norm{\bu}^2}\dd\bu
\le
\lrbracket{\sup_{\bu\in\mathbb R^3}\|\bu\|^ke^{-c_1\|\bu\|^2}}
\int_{\|\bv\|\le 1}\frac{1}{\|\bv\|}\dd\bv
<\infty,
\]
for \(k=0,1,2\). On \(\{\|\bu+\bz\|>1\}\), 
\[
\int_{\|\bu+\bz\|>1}\frac{(\|\bu\|^k)e^{-c_1\|\bu\|^2}}{\|\bu+\bz\|}\dd\bu
\le
\int_{\mathbb R^3}\|\bu\|^ke^{-c_1\|\bu\|^2}\dd\bu
<\infty.
\]
Hence
\[
\sup_{\bz\in\mathbb R^3}|\nabla_{\bz} U_\Gamma(\bz)|<\infty,
\qquad
\sup_{\bz\in\mathbb R^3}\|\nabla^2_{\bz} U_\Gamma(\bz)\|<\infty.
\]

\medskip
For $g(\btheta)$ and $\nabla^2L(\btheta)$, by the chain rule,
\[
g(\btheta)=
-\sum_{i=1}^N\sum_{I=1}^M Z_I\,A_i^\top \nabla U_{\Sigma_{ii}}(A_i\btheta-\RR_I)
+\sum_{1\le i<j\le N}(A_i-A_j)^\top \nabla U_{\Gamma_{ij}}((A_i-A_j)\btheta),
\]
and
\[
\nabla^2L(\theta)
=
-\sum_{i=1}^N\sum_{I=1}^M Z_I\,A_i^\top \nabla^2 U_{\Sigma_{ii}}(A_i\btheta-\RR_I)\,A_i
+\sum_{1\le i<j\le N}(A_i-A_j)^\top \nabla^2 U_{\Gamma_{ij}}((A_i-A_j)\btheta)(A_i-A_j).
\]
Using the uniform boundedness of \(\nabla_{\bz} U_\Gamma(\bz)\) and \(\nabla^2_{\bz} U_\Gamma(\bz)\), we obtain
\[
\sup_{\btheta\in\mathbb R^{N_p}}\|g(\btheta)\|<\infty,
\qquad
\sup_{\btheta\in\mathbb R^{N_p}}\|\nabla^2L(\btheta)\|<\infty.
\]
Therefore \(g(\theta)\) is uniformly bounded and globally Lipschitz continuous.

\medskip
\noindent\textbf{Discrete example.}
Consider
\[
\psi_\theta(\bfx)=2^{-N/2}e^{\i\bfx^\top A\btheta},
\qquad \bfx\in\{-1,1\}^N.
\]
Then
\[
L(\btheta)
= \braket{\psi_{\btheta}|\scrH|\psi_{\btheta}}
=
2^{-N}\sum_{\bfx,\mathbf{y}\in\{-1,1\}^N}\braket{\bfx|\scrH|\mathbf{y}} e^{\i(\mathbf{y}-\bfx)^\top A\btheta}.
\]
Therefore, we have 
\[
g(\btheta)
=2^{-N} \sum_{\bfx,\mathbf{y}\in \{-1,1\}^N} \i \braket{\bfx|\scrH|\mathbf{y}} A^\top (\mathbf{y}-\bfx)e^{\i (\mathbf{y}-\bfx)^\top A\btheta}
\]
and
\[
\nabla^2L(\theta)
=
-2^{-N}\sum_{\bfx,\mathbf{y}\in\{-1,1\}^N} \braket{\bfx|\scrH|\mathbf{y}}
A^\top (\mathbf{y}-\bfx)(\mathbf{y}-\bfx)^\top A e^{\i(\mathbf{y}-\bfx)^\top A\btheta}.
\]
Since \(x,y\in\{-1,1\}^N\), we have $\|\mathbf{y}-\bfx\|\le 2\sqrt{N}$. Hence
\[
\sup_{\btheta\in\mathbb R^{N_p}}\|g(\btheta)\|<\infty,
\qquad
\sup_{\btheta\in\mathbb R^{N_p}}\|\nabla^2L(\btheta)\|<\infty.
\]
Therefore \(g(\btheta)\) is uniformly bounded and globally Lipschitz continuous also in the discrete case.
\end{proof}

\section{Experimental Settings}
\label{sec:experiment_setting}

\subsection{Restricted Boltzmann Machine}
\label{sec:rbm_define}

\par The restricted Boltzmann machine (RBM) proposed in \cite{carleo2017solving}, is used as the wavefunction ans\"atz for spin-lattice models and takes the form
    \begin{equation*}
        \psi_{\btheta}(\bfx) = e^{\sum_{j=1}^N a_jx_j}\prod_{k=1}^D \cosh(b_k + \sum_{j=1}^N w_{kj}x_j),
    \end{equation*}
    where $\btheta = (a_j, b_k, w_{kj})_{j=1,\dots,N;\,k=1,\dots,D}$. 
For spin-lattice systems, the configuration satisfies $\bfx\in \{-1,1\}^N$.

\subsection{Electronic Hamiltonian}
In electron structure theory, the electronic Hamiltonian takes the form:
    \begin{equation}
        \scrH:=-\frac{1}{2}\sum_{i=1}^N\Delta_{\bx_i}-\sum_{i=1}^N\sum_{I=1}^M\frac{Z_I}{\norm{\bx_i-\RR_I}}+\sum_{1\le i<j \le N}\frac{1}{\norm{\bx_i-\bx_j}} + \sum_{1\le I < I'\le M} \dfrac{Z_I Z_{I'}}{\norm{\RR_I-\RR_{I'}}},
        \label{eq:elec_ham}
    \end{equation}
where $\{\RR_I\}_{I=1}^M\subseteq\R^3$ are nuclear positions and $\{Z_I\}_{I=1}^M\subseteq\N$ are nuclear charges.

\subsection{Transverse-Field Ising Model and Heisenberg Model}
\label{sec:tfi_heisenberg_define}

\par Both TFI and Heisenberg models consist of chains of Pauli matrices. The Pauli matrices are defined as
\begin{equation*}
    \sigma^x = \begin{pmatrix}
    0 & 1\\
    1 & 0
\end{pmatrix},\quad \sigma^y = \begin{pmatrix}
    0 & -i\\
    i & 0
\end{pmatrix}, \quad \sigma^z = \begin{pmatrix}
    1 & 0\\
    0 & -1
\end{pmatrix}.
\end{equation*}
The Pauli operator acting on site $i$ is defined as
\begin{equation*}
    \sigma_{i}^{\ell}:=I^{\otimes (i-1)} \otimes \sigma^{\ell} \otimes I^{\otimes (N-i)}\in \C^{2^N\times 2^N},\ell \in \{x,y,z\},~1\le i\le N,
\end{equation*}
where $\otimes$ denotes the Kronecker product. The TFI and Heisenberg Hamiltonians are defined as
\begin{itemize}
    \item TFI model: $ \scrH = -\sum_{<i,j>}\sigma_i^z\sigma_j^z - h\sum_{j}\sigma_j^x$, where $h>0$ is the transverse-field strength.
    
    \item Heisenberg model: $\scrH = \sum_{<i,j>}\sigma_i^x\sigma_j^x+\sigma_i^y\sigma_j^y+\sigma_i^z\sigma_j^z$.
\end{itemize}
Here $<i,j>$ denotes a nearest-neighbor pair.

\subsection{FermiNet Architecture and Hyperparameters}
\label{sec:ferminet_detail}

\begin{table}[H]
\centering
\caption{FermiNet architecture and MCMC hyperparameters used in all electronic-system experiments.}
\begin{tabular}{lc}
\toprule
Hyperparameter & Value \\
\midrule
One-electron stream width & 256 \\
Two-electron stream width & 16 \\
Number of equivariant layers & 4 \\
Backflow activation function & tanh \\
Number of determinants & 16 \\
Exponential envelope structure & Isotropic \\
Standard deviations for local energy clipping & 5 \\
Number of walkers & 1000 \\
MCMC burn-in steps & 5000 \\
MCMC steps between updates & 10 \\
\bottomrule
\end{tabular}

\label{tab:ferminet_hyperparameters}
\end{table}

\section{PRIME-SR vs. SPRING Across Random Seeds}
\label{sec:compare_spring_random_seed}

\subsection{Instability of SPRING with $\mu=0.99$ on $\mathrm{N}$, $\mathrm{O}$, $\mathrm{N}_2$ and $\mathrm{CO}$}
\label{sec:spring_unstable_N_O}

\begin{figure}[H] 
    \centering
    \begin{subfigure}[t]{0.45\linewidth}
        \centering
        \includegraphics[width=0.9\linewidth]{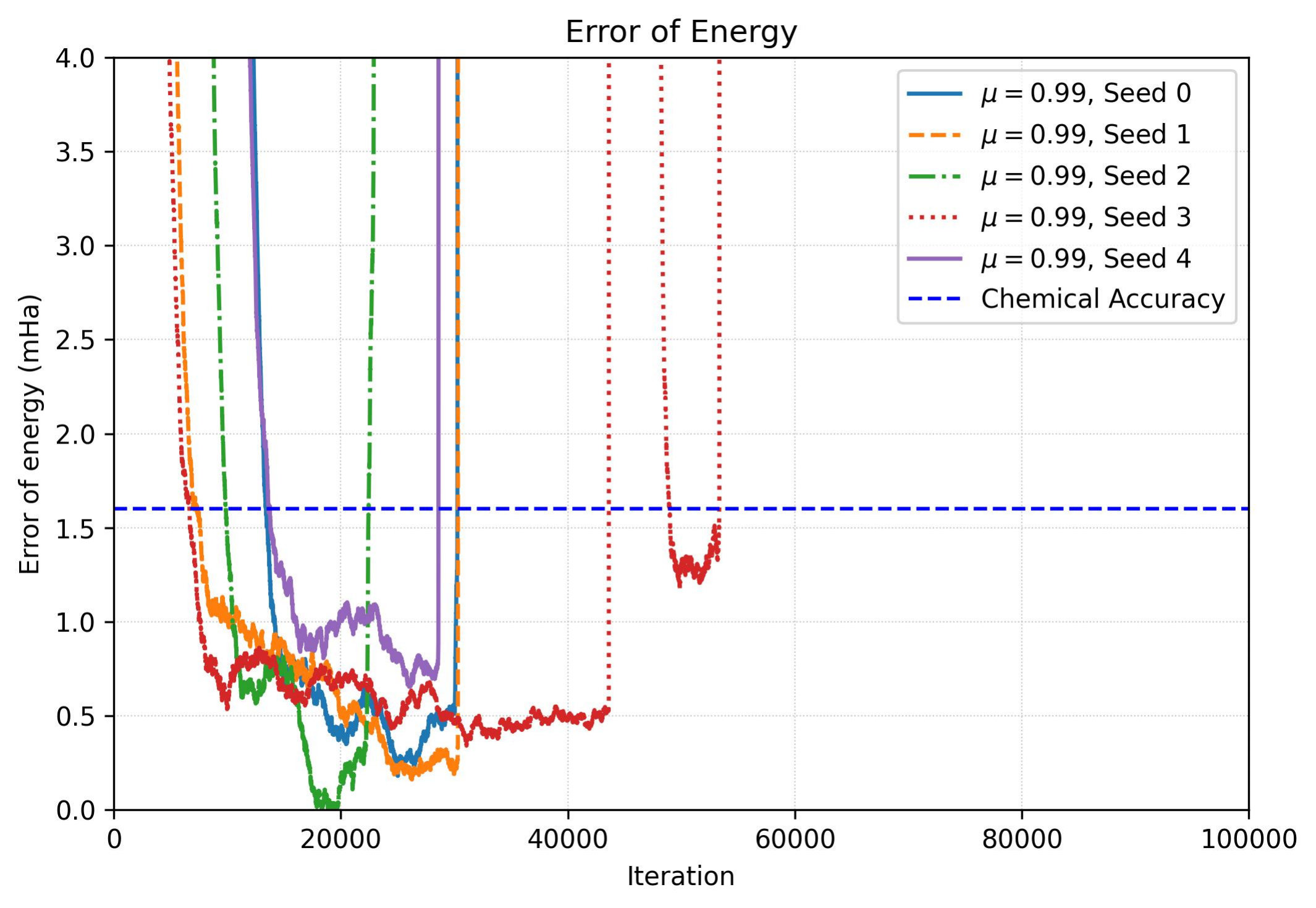}
        \caption{$\mathrm{N}$ (energy)}
    \end{subfigure}
    \hfill
    \begin{subfigure}[t]{0.45\linewidth}
        \centering
        \includegraphics[width=0.9\linewidth]{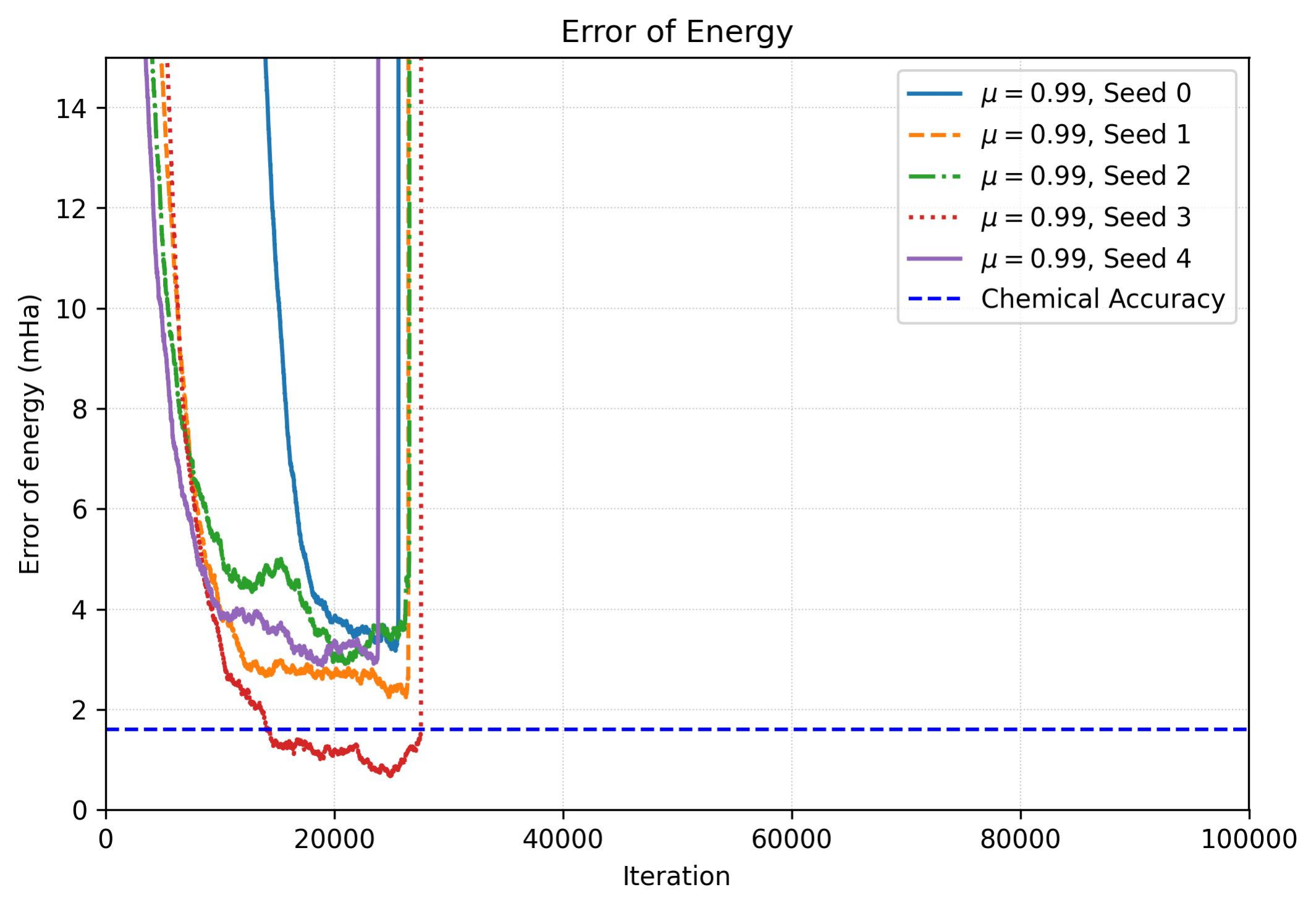}
        \caption{$\mathrm{O}$ (energy)}
    \end{subfigure}

    \vspace{0.4em}

    \begin{subfigure}[t]{0.49\linewidth}
        \centering
        \includegraphics[width=0.9\linewidth]{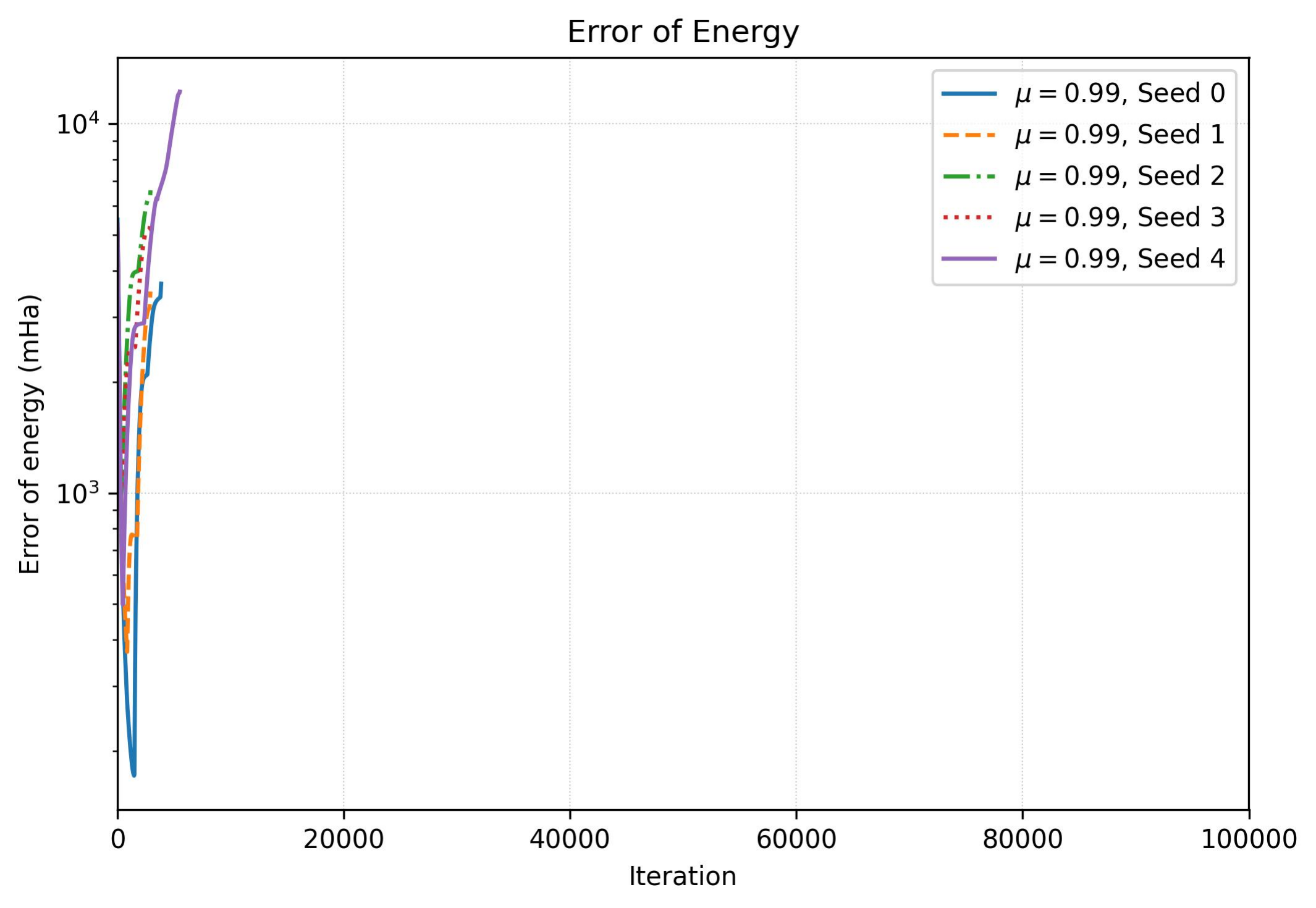}
        \caption{$\mathrm{N}_2$ (energy)}
    \end{subfigure}
    \hfill
    \begin{subfigure}[t]{0.49\linewidth}
        \centering
        \includegraphics[width=0.9\linewidth]{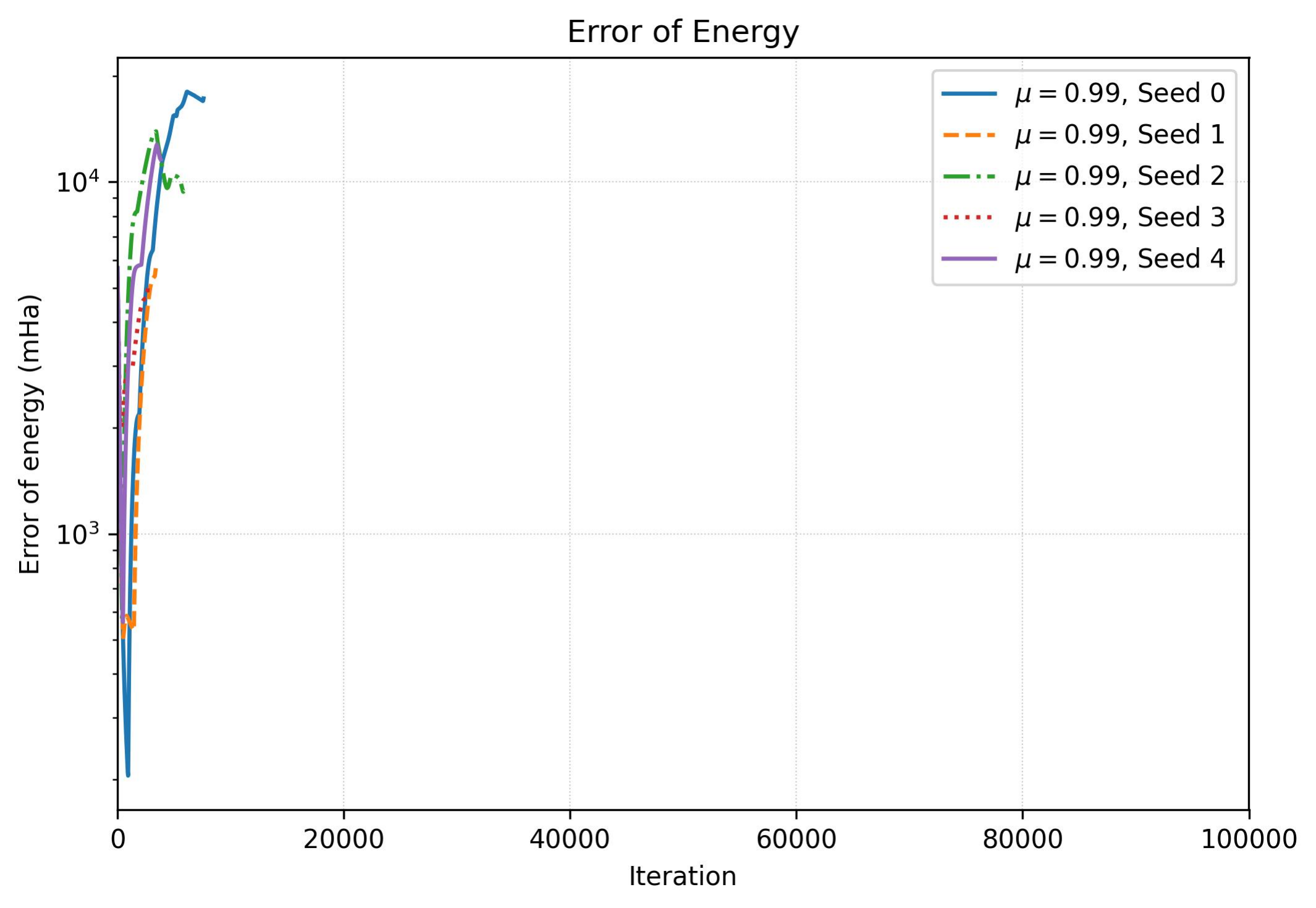}
        \caption{$\mathrm{CO}$ (energy)}
    \end{subfigure}

    \caption{Energy trajectories of SPRING with $\mu=0.99$ on $\mathrm{N}$, $\mathrm{O}$ $\mathrm{N}_2$ and $\mathrm{CO}$.}

    \label{fig:spring_unstable_N_O}
\end{figure}

\newpage

\subsection{Atomic Systems: Results Across Random Seeds}
\label{sec:atom_random_seed}

\textbf{Random seed 1.}

\begin{figure}[htbp]
    \centering
    
    \begin{subfigure}{\linewidth}
    \centering
    \includegraphics[width=0.48\linewidth]{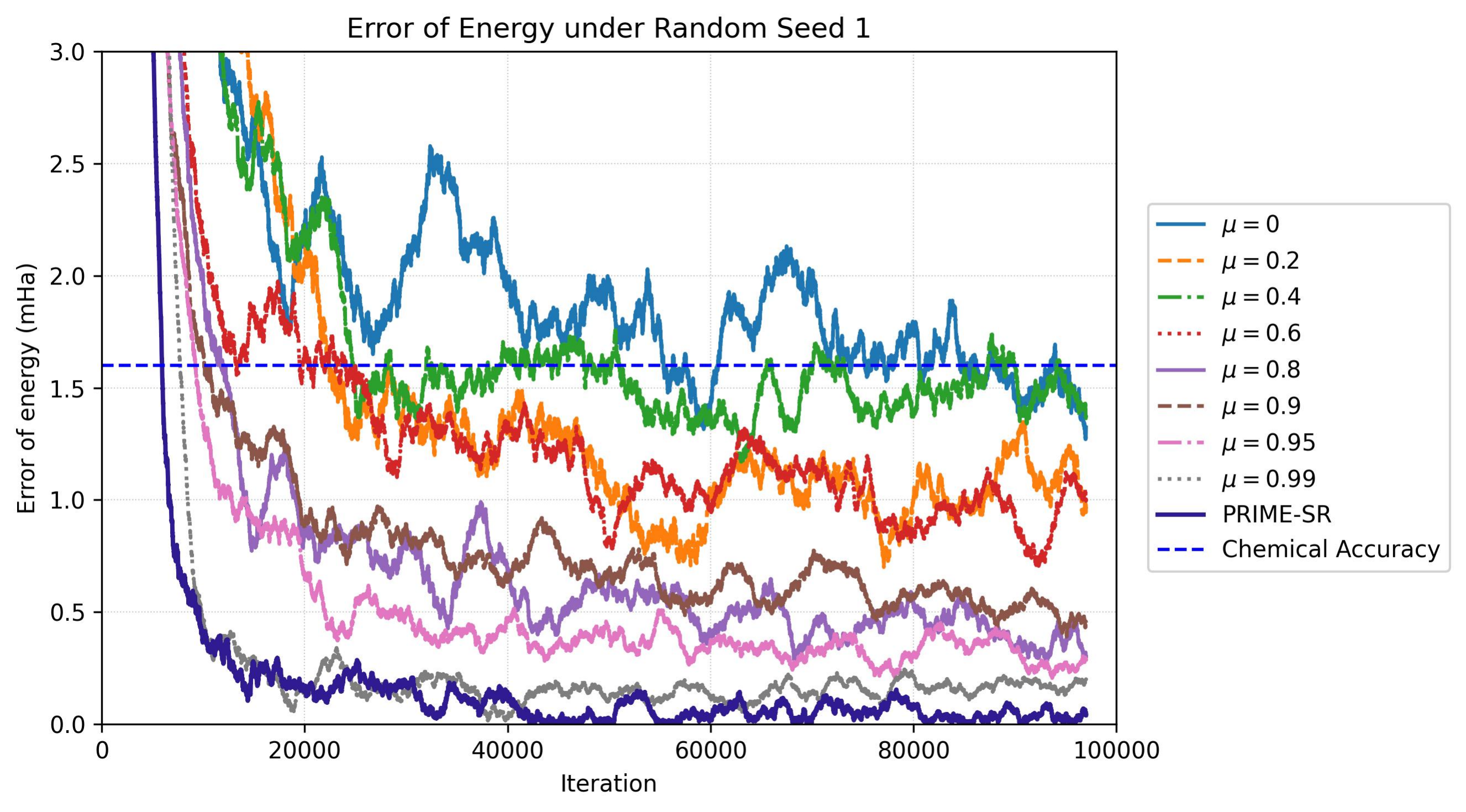}
    \hfill
    \includegraphics[width=0.42\linewidth]{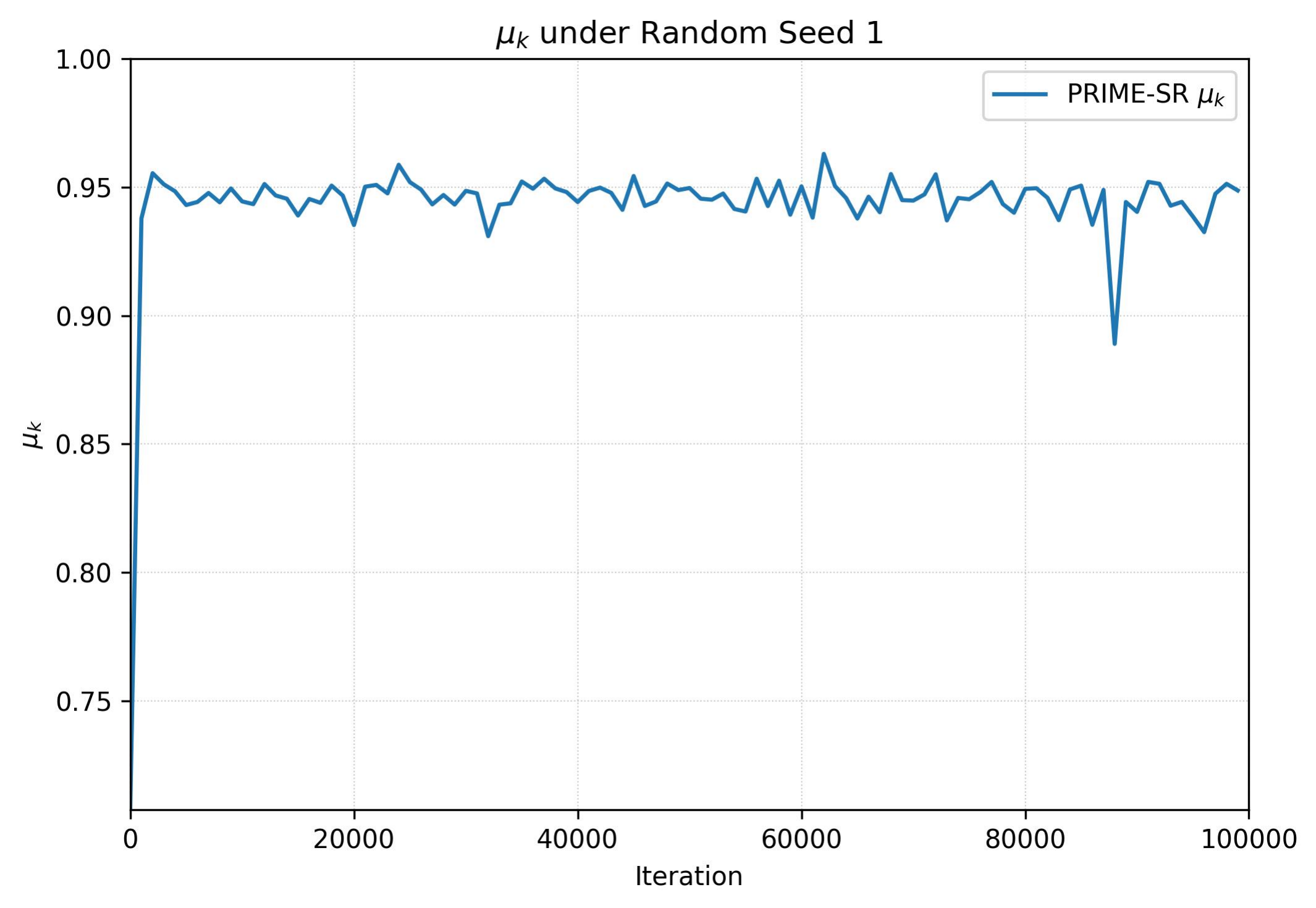}
    \caption{$\mathrm{C}$ atom. Left: relative energy error. Right: $\mu_k$.}
    \end{subfigure}
    
    \vspace{0.4em}
    
    \begin{subfigure}{\linewidth}
    \centering
    \includegraphics[width=0.48\linewidth]{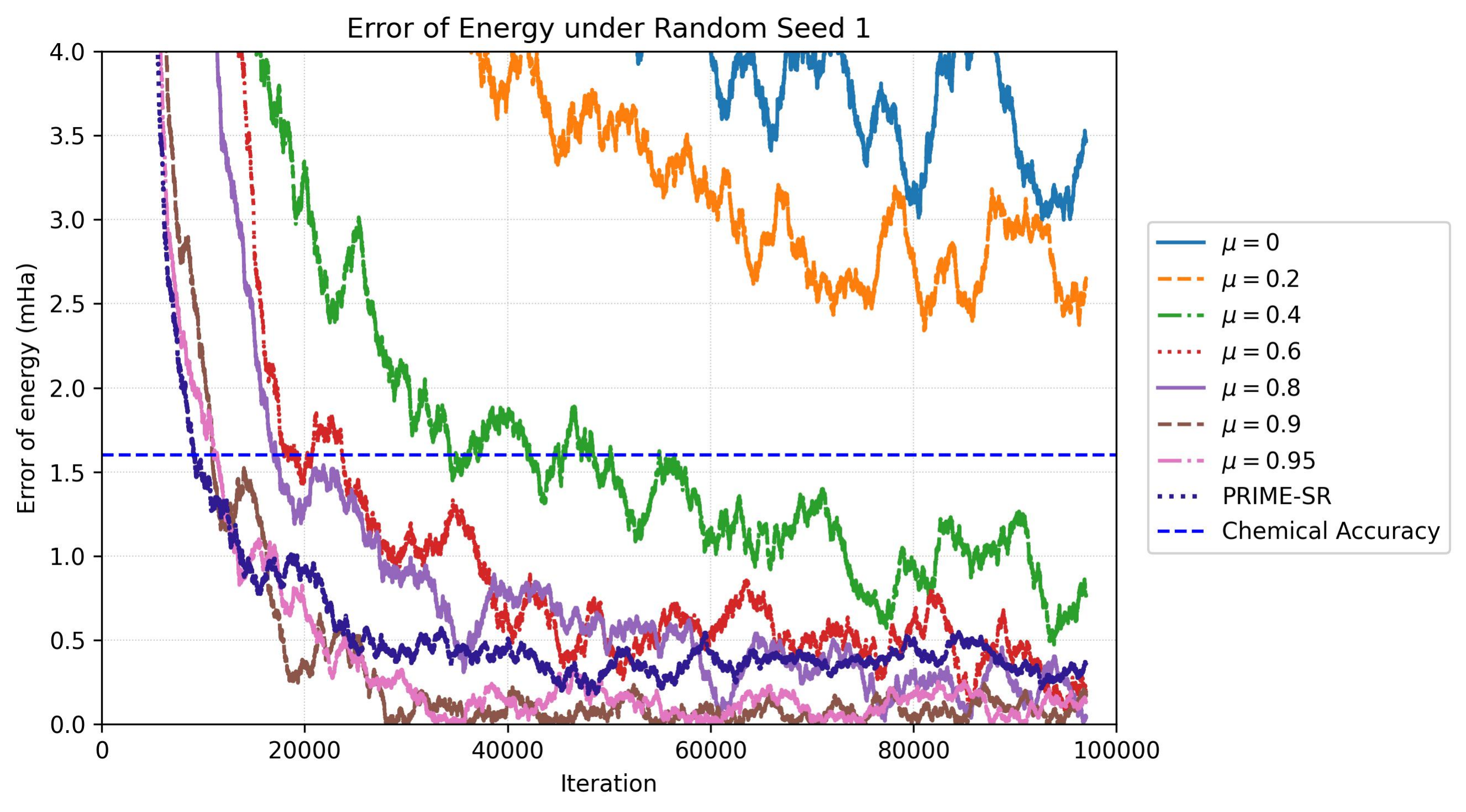}
    \hfill
    \includegraphics[width=0.42\linewidth]{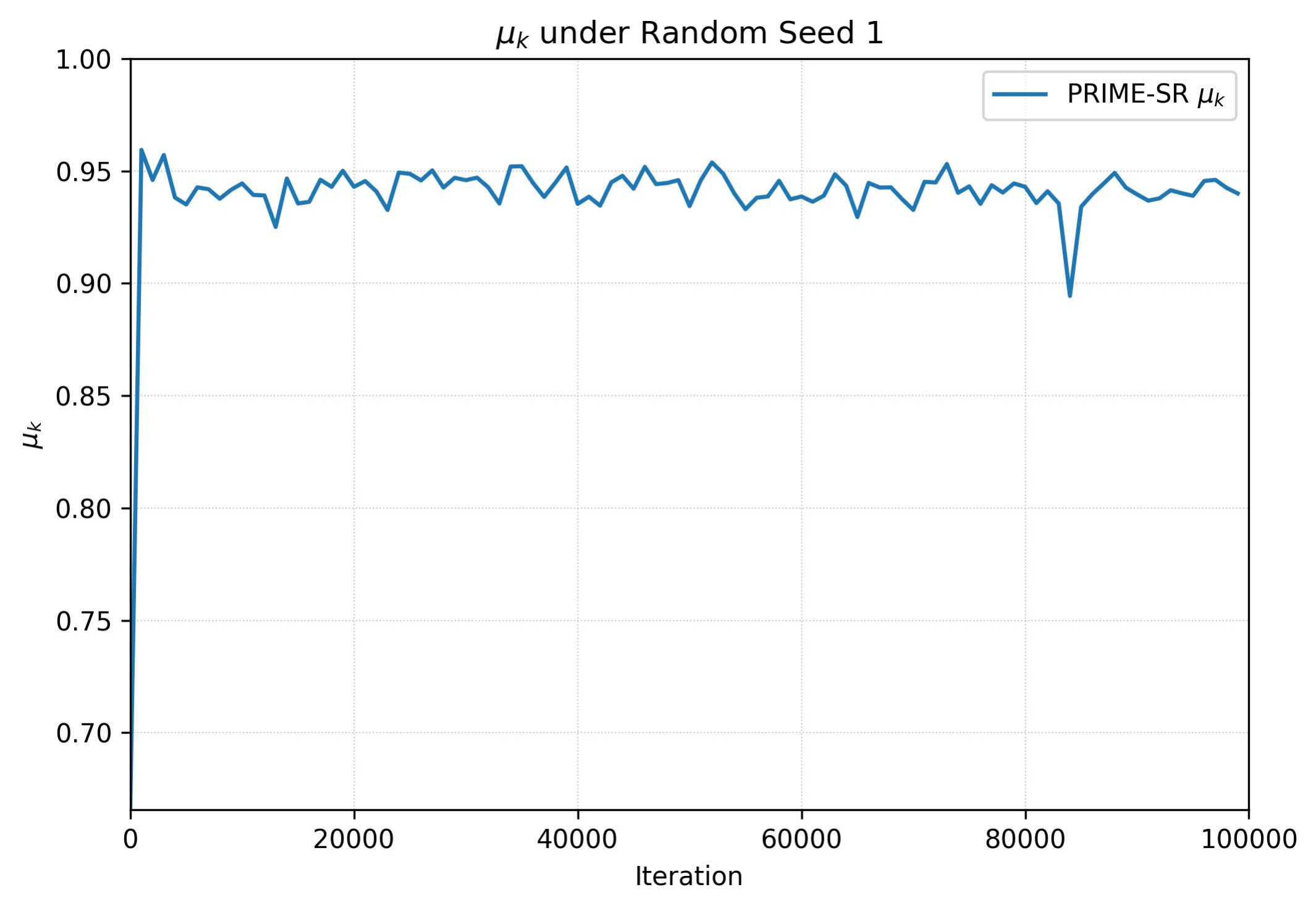}
    \caption{$\mathrm{N}$ atom. Left: relative energy error. Right: $\mu_k$.}
    \end{subfigure}
    
    \vspace{0.4em}
    
    \begin{subfigure}{\linewidth}
    \centering
    \includegraphics[width=0.48\linewidth]{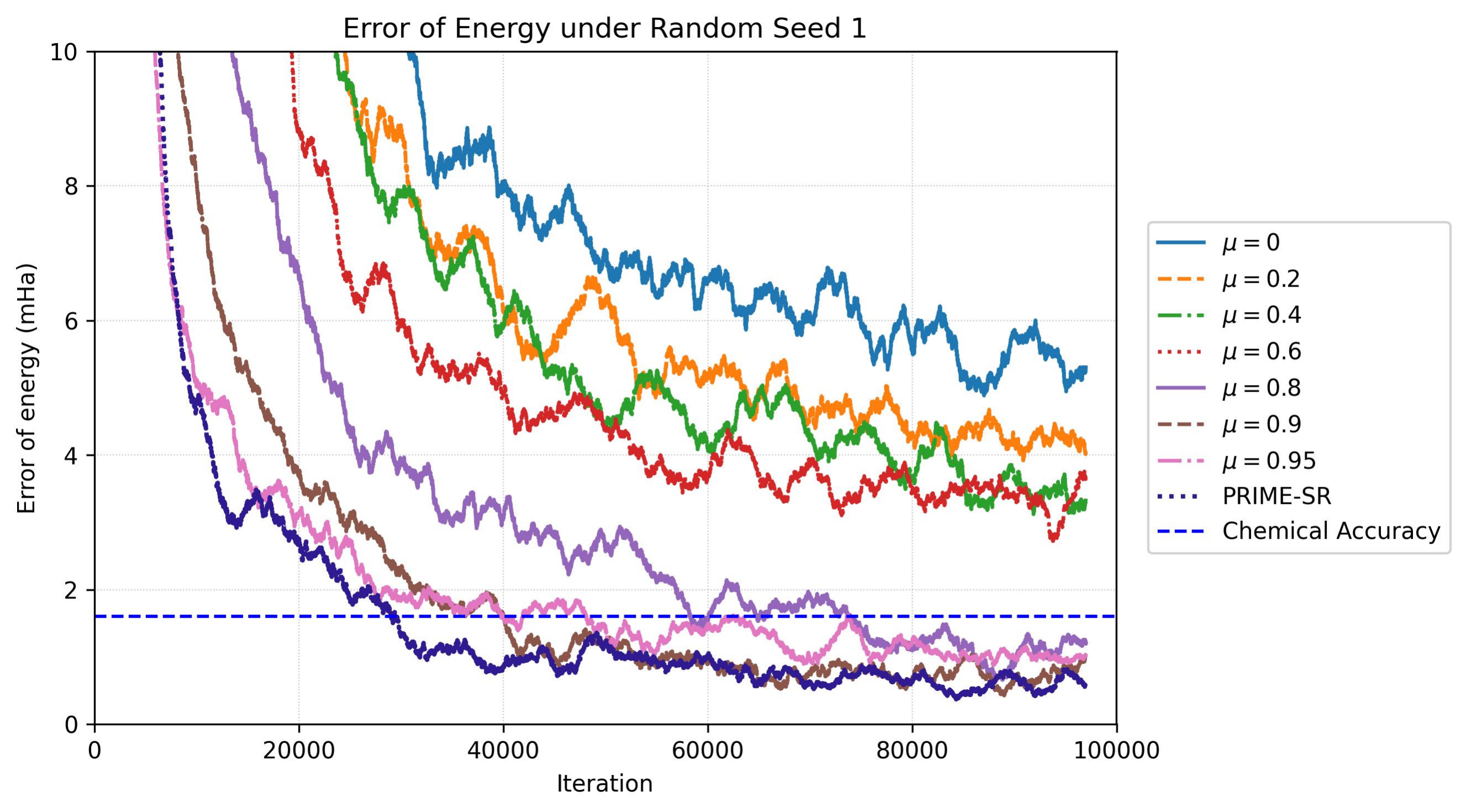}
    \hfill
    \includegraphics[width=0.42\linewidth]{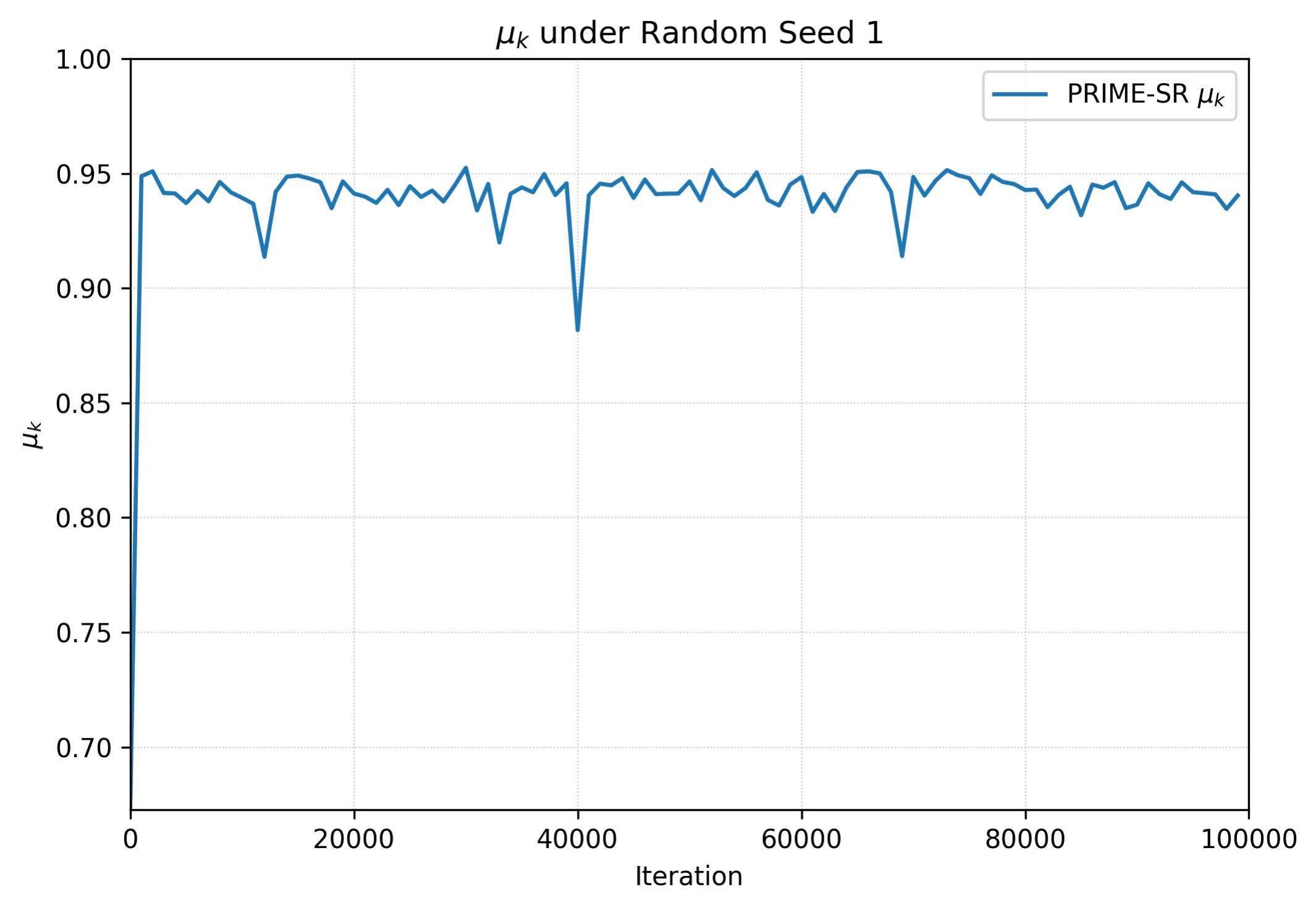}
    \caption{$\mathrm{O}$ atom. Left: relative energy error. Right: $\mu_k$.}
    \end{subfigure}
    
    \caption{Comparison of fixed-$\mu$ SPRING and PRIME-SR on $\mathrm{C}$, $\mathrm{N}$, $\mathrm{O}$ atoms for random seed 1.}
    \label{fig:compare_spring_atom_seed_1}

\end{figure}

 \newpage

\textbf{Random seed 2.}

\begin{figure}[htbp]
    \centering
    
    \begin{subfigure}{\linewidth}
    \centering
    \includegraphics[width=0.48\linewidth]{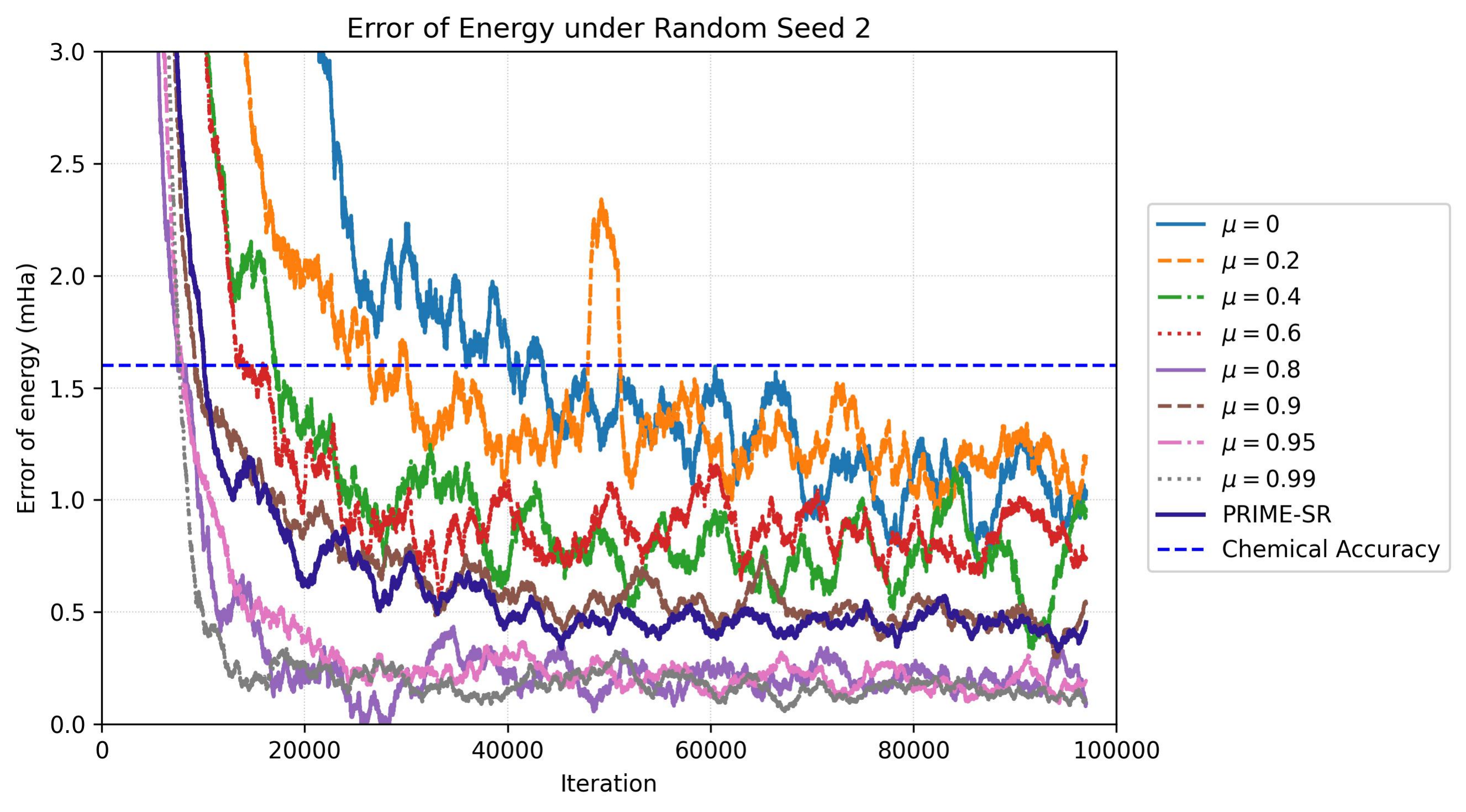}
    \hfill
    \includegraphics[width=0.42\linewidth]{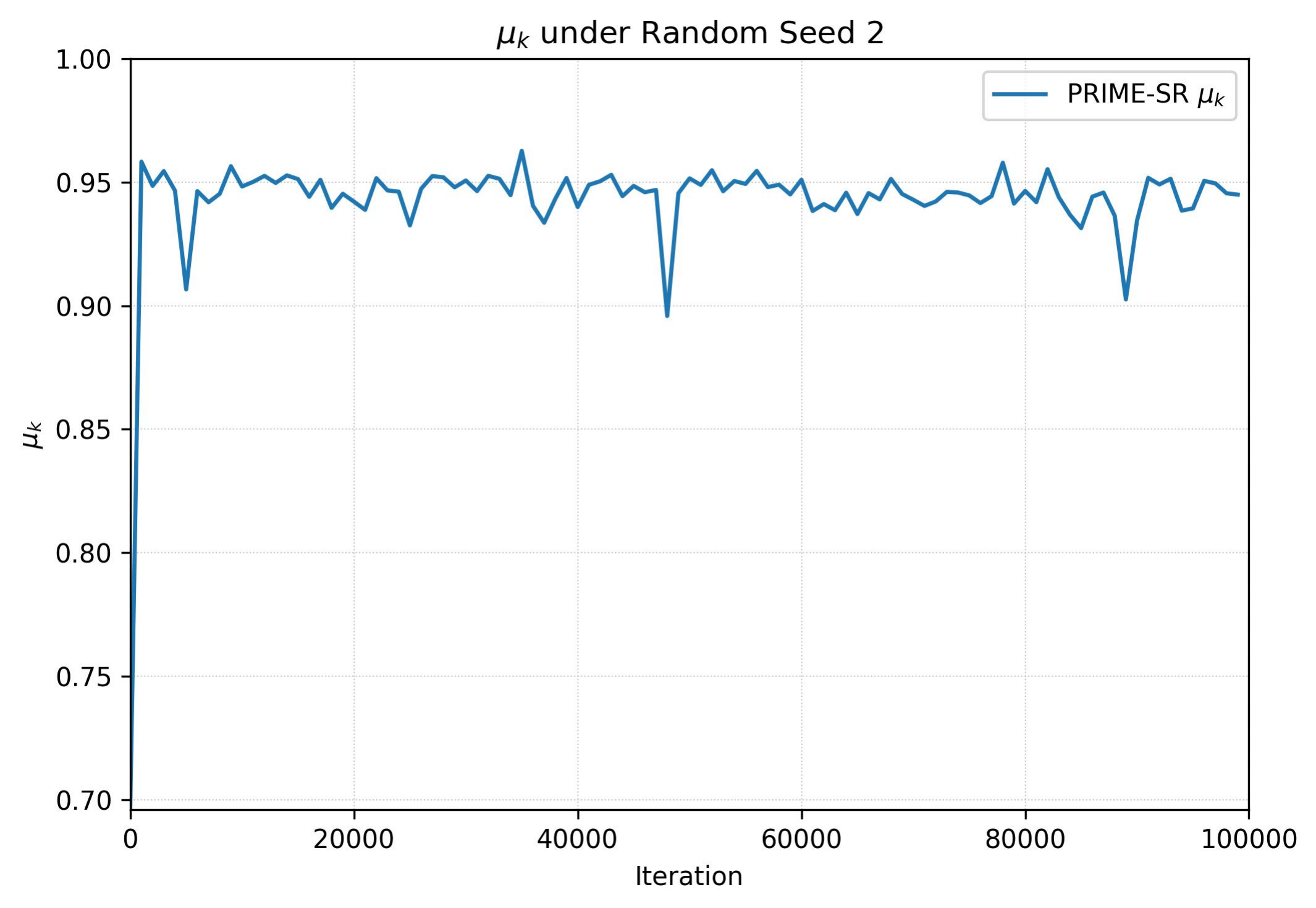}
    \caption{$\mathrm{C}$ atom. Left: relative energy error. Right: $\mu_k$.}
    \end{subfigure}
    
    \vspace{0.4em}
    
    \begin{subfigure}{\linewidth}
    \centering
    \includegraphics[width=0.48\linewidth]{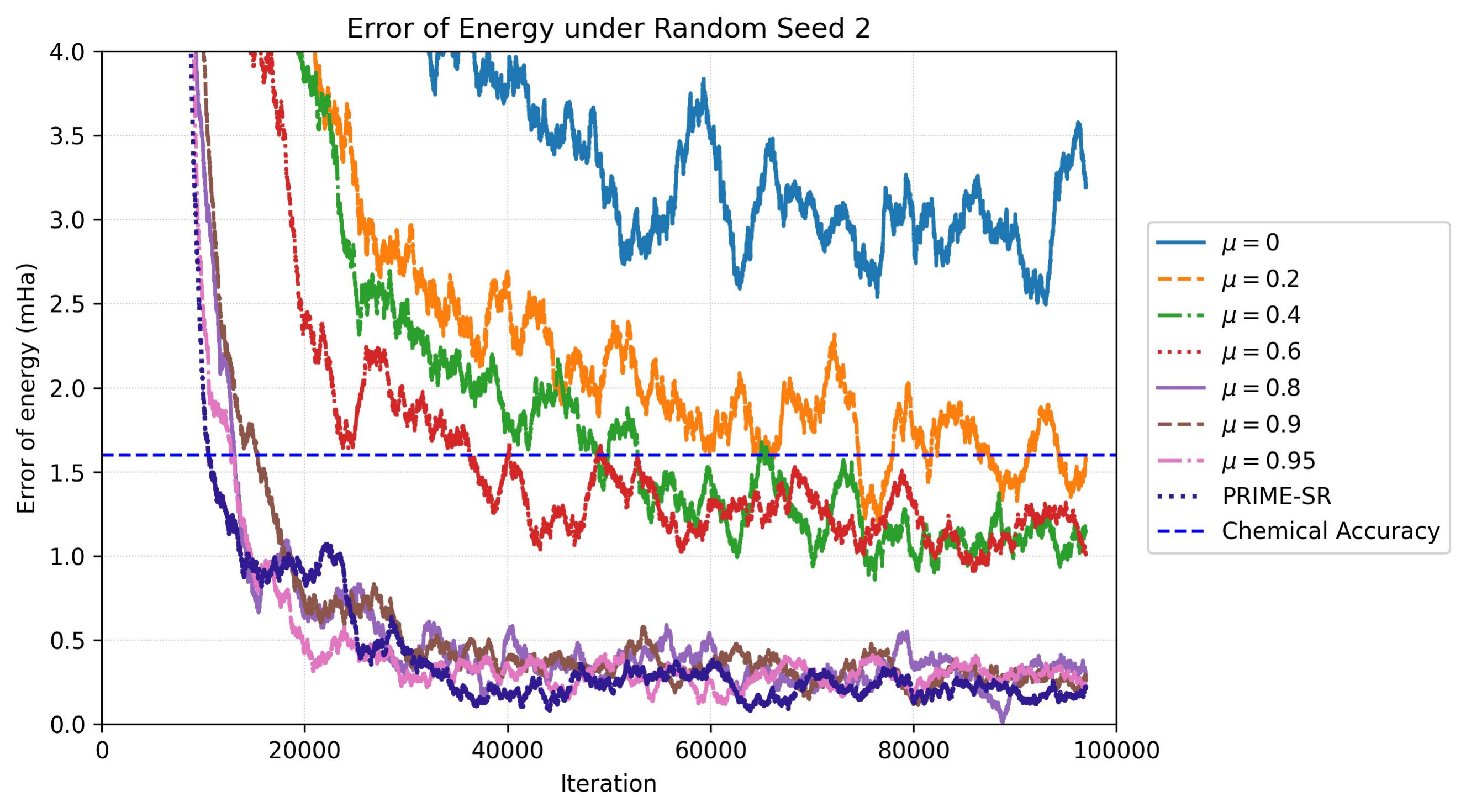}
    \hfill
    \includegraphics[width=0.42\linewidth]{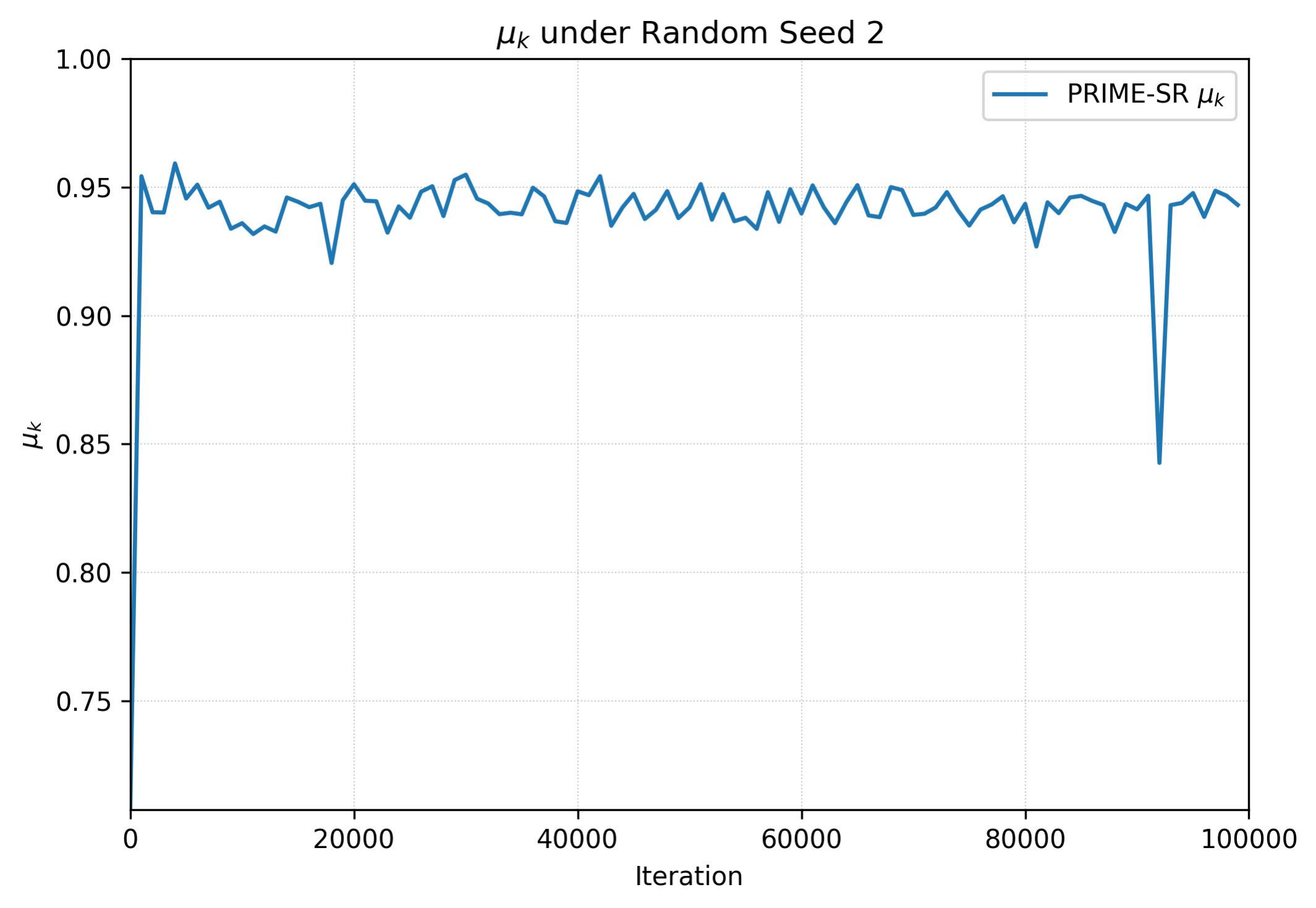}
    \caption{$\mathrm{N}$ atom. Left: relative energy error. Right: $\mu_k$.}
    \end{subfigure}
    
    \vspace{0.4em}
    
    \begin{subfigure}{\linewidth}
    \centering
    \includegraphics[width=0.48\linewidth]{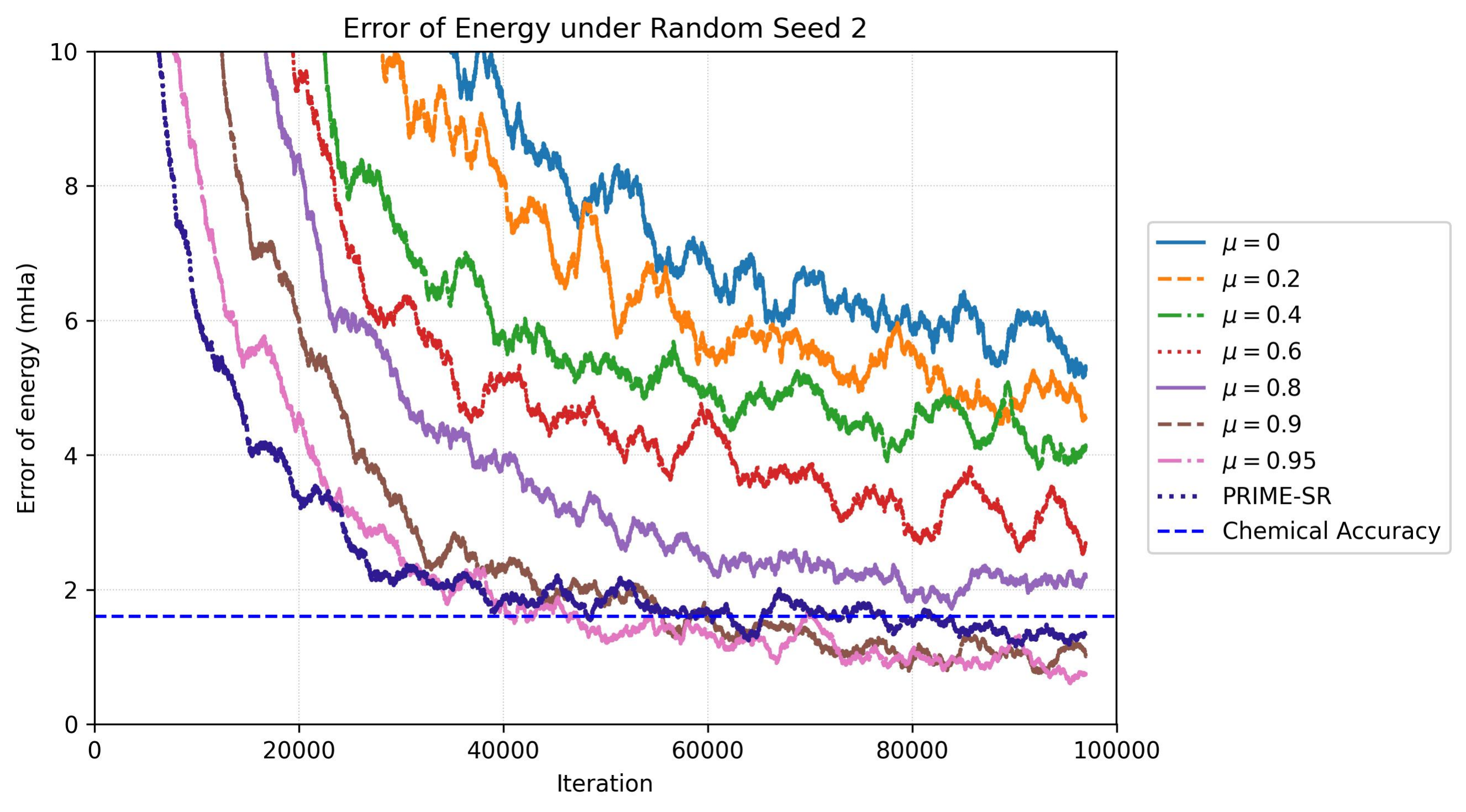}
    \hfill
    \includegraphics[width=0.42\linewidth]{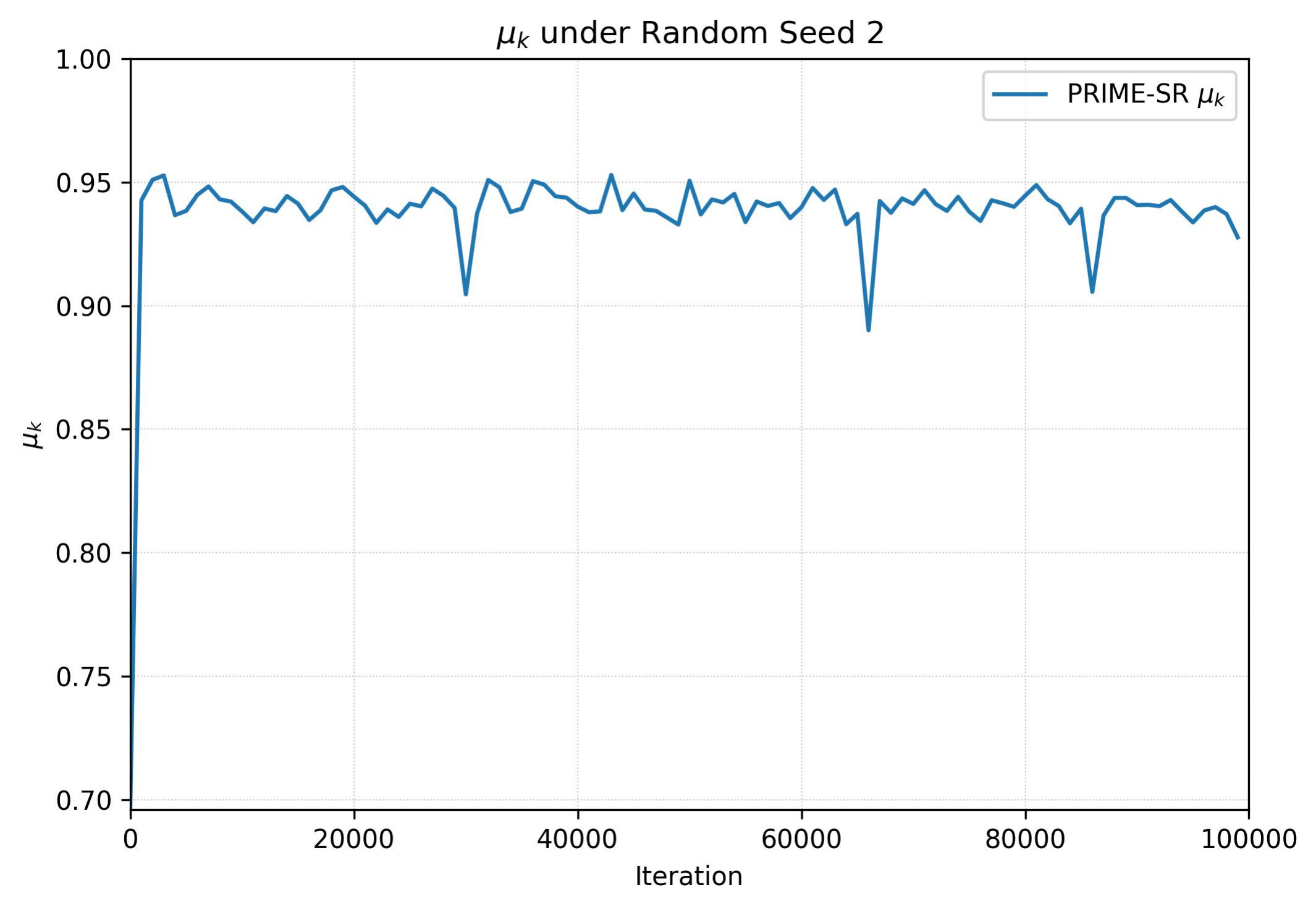}
    \caption{$\mathrm{O}$ atom. Left: relative energy error. Right: $\mu_k$.}
    \end{subfigure}
    
    \caption{Comparison of fixed-$\mu$ SPRING and PRIME-SR on $\mathrm{C}$, $\mathrm{N}$, $\mathrm{O}$ atoms for random seed 1.}
    \label{fig:compare_spring_atom_seed_2}

\end{figure}

 \newpage

\textbf{Random seed 3.}

\begin{figure}[htbp]
    \centering
    
    \begin{subfigure}{\linewidth}
    \centering
    \includegraphics[width=0.48\linewidth]{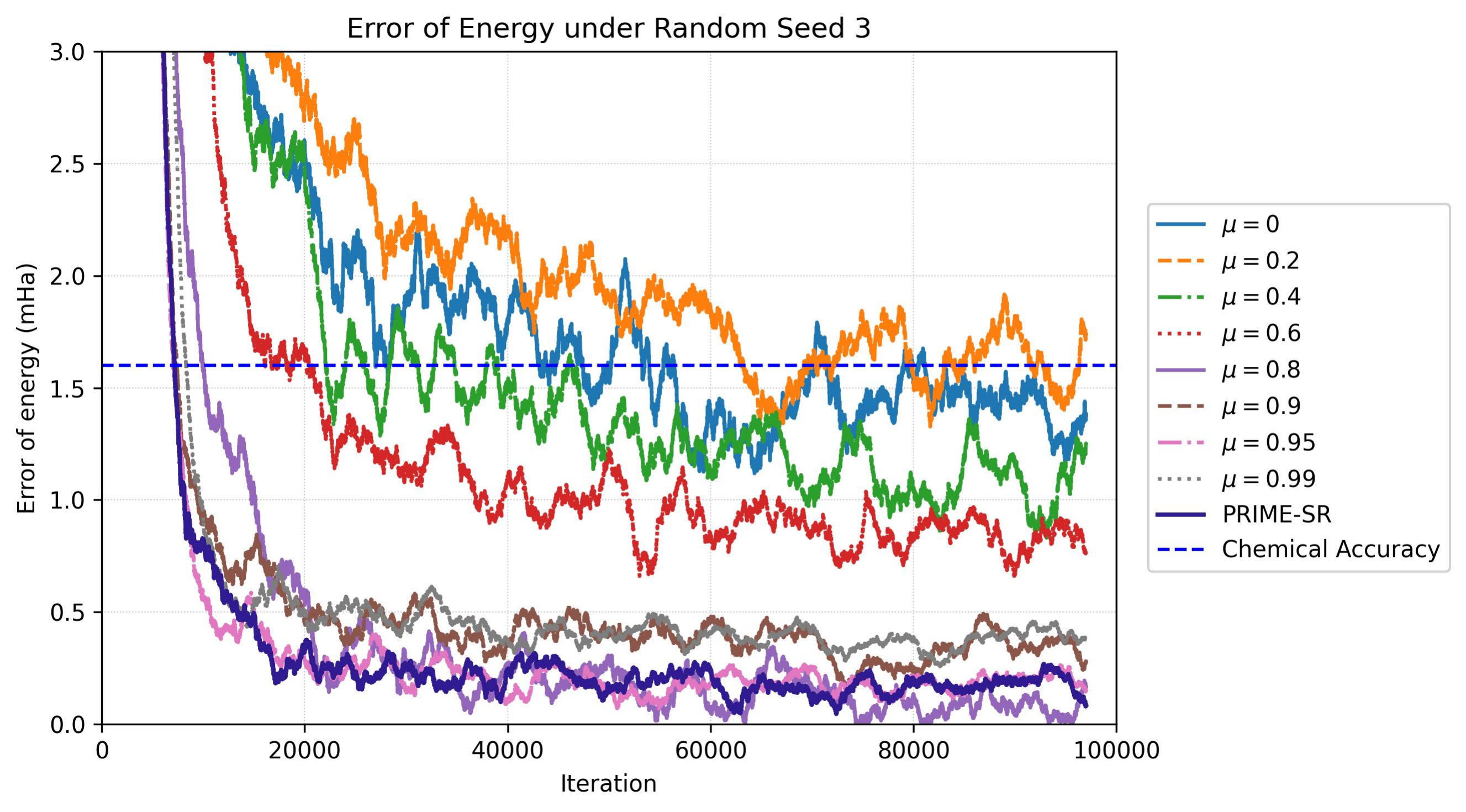}
    \hfill
    \includegraphics[width=0.42\linewidth]{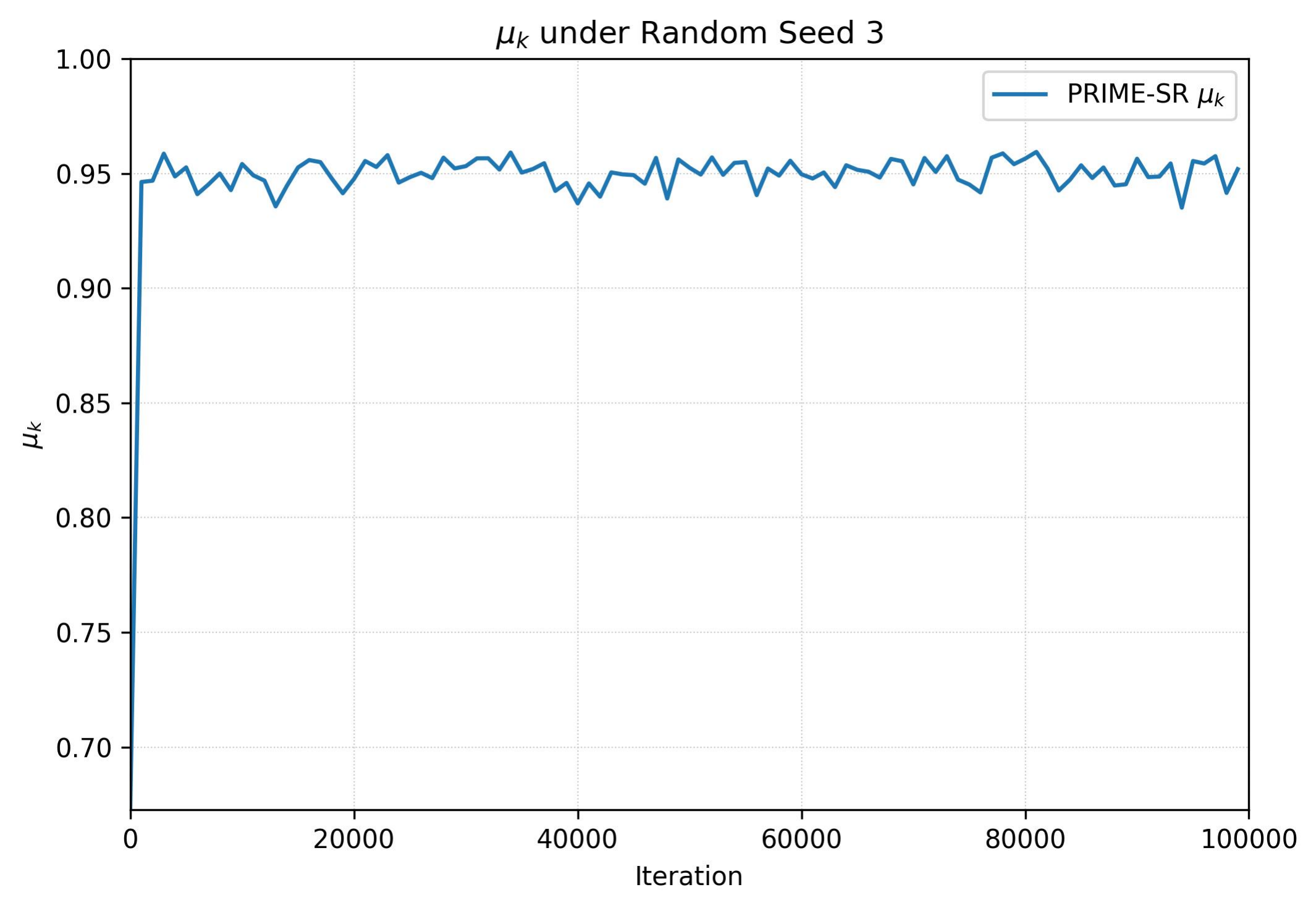}
    \caption{$\mathrm{C}$ atom. Left: relative energy error. Right: $\mu_k$.}
    \end{subfigure}
    
    \vspace{0.4em}
    
    \begin{subfigure}{\linewidth}
    \centering
    \includegraphics[width=0.48\linewidth]{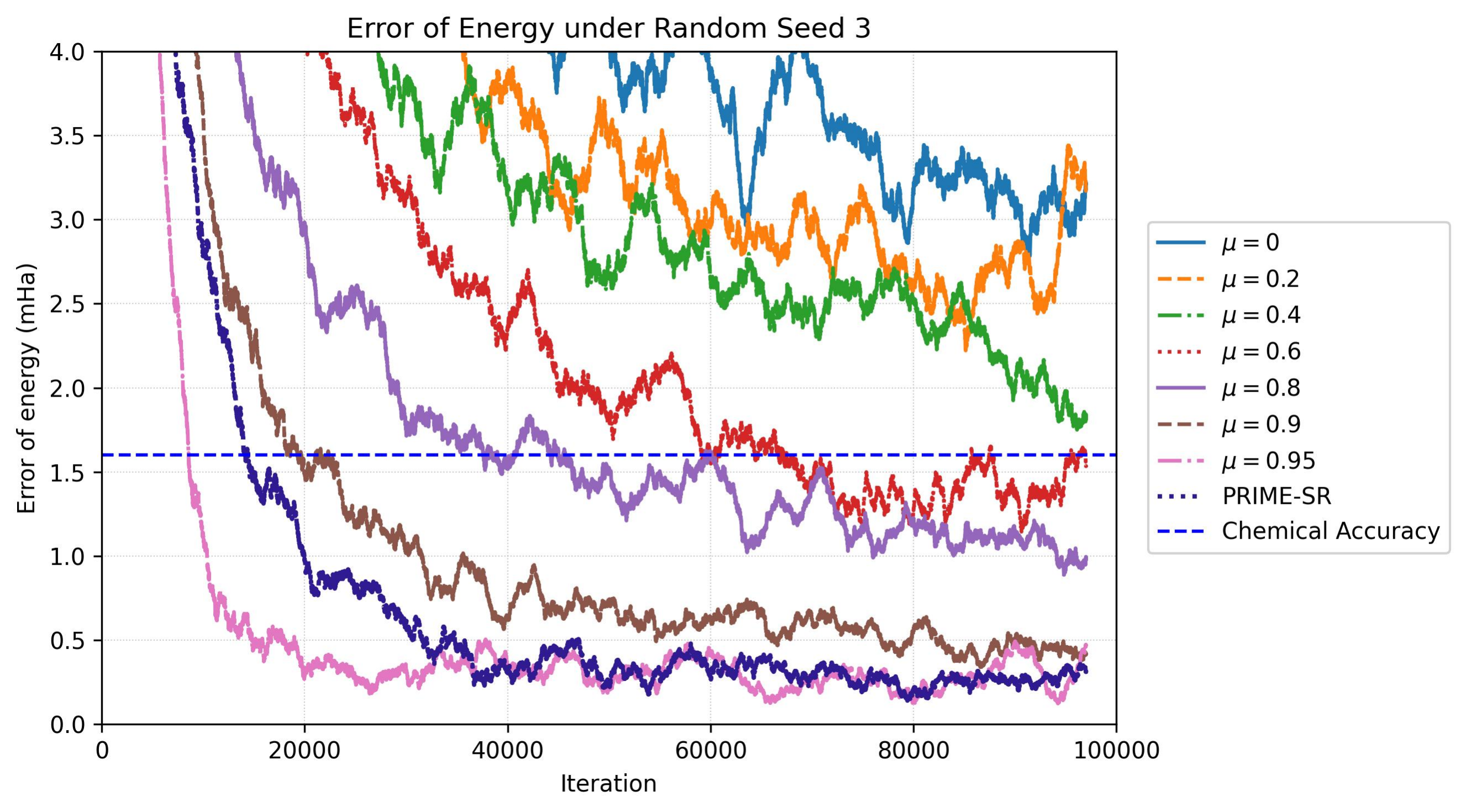}
    \hfill
    \includegraphics[width=0.42\linewidth]{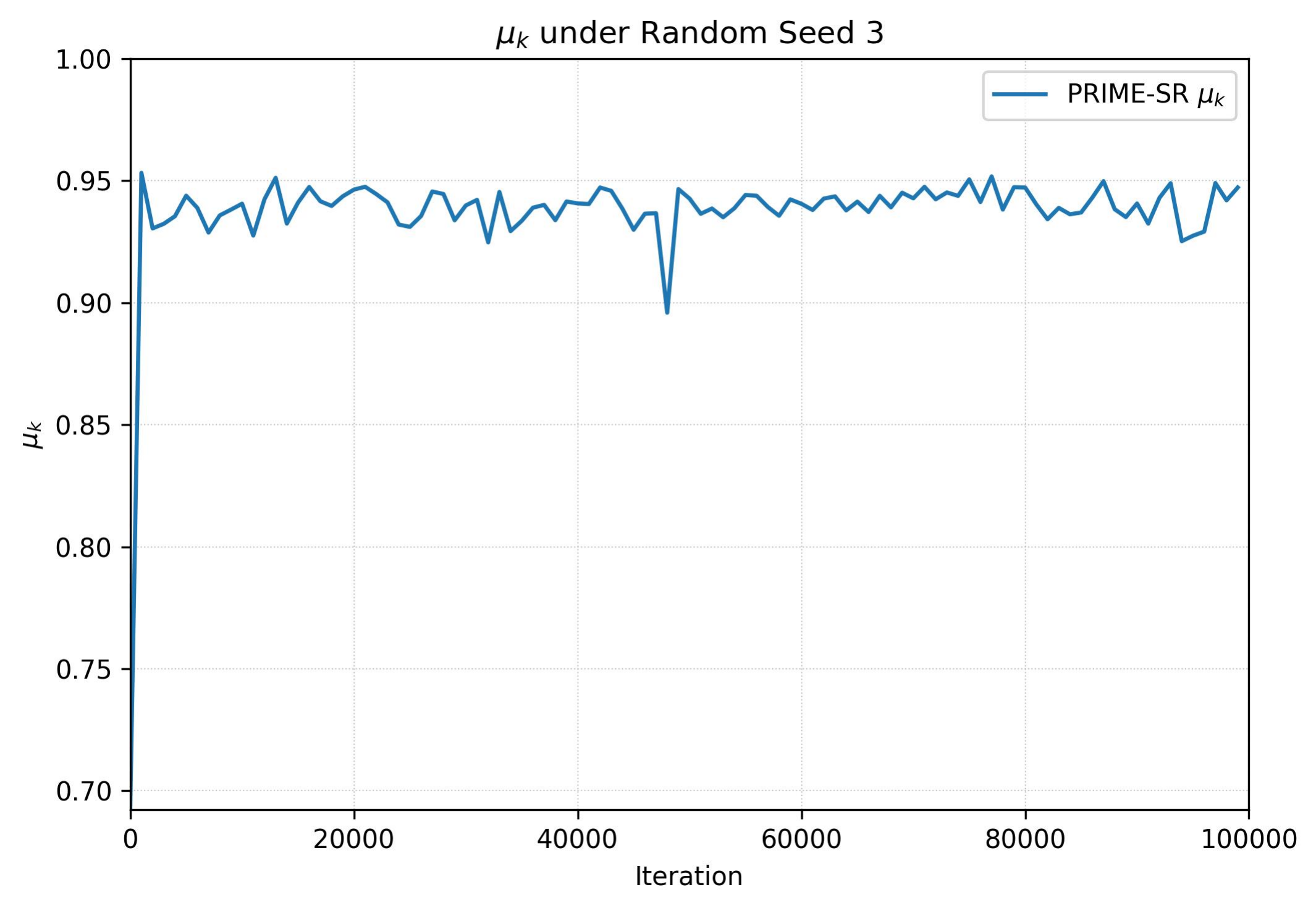}
    \caption{$\mathrm{N}$ atom. Left: relative energy error. Right: $\mu_k$.}
    \end{subfigure}
    
    \vspace{0.4em}
    
    \begin{subfigure}{\linewidth}
    \centering
    \includegraphics[width=0.48\linewidth]{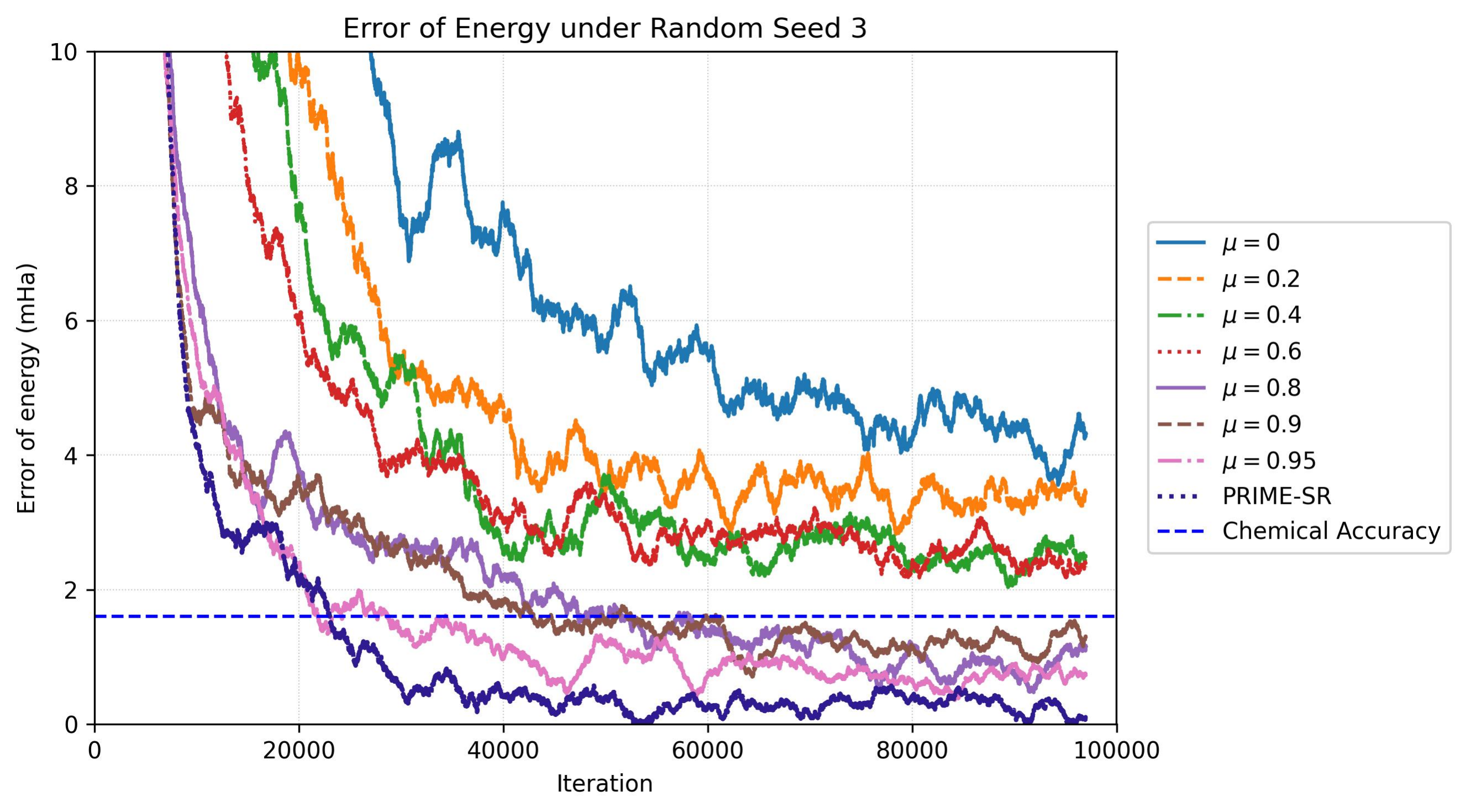}
    \hfill
    \includegraphics[width=0.42\linewidth]{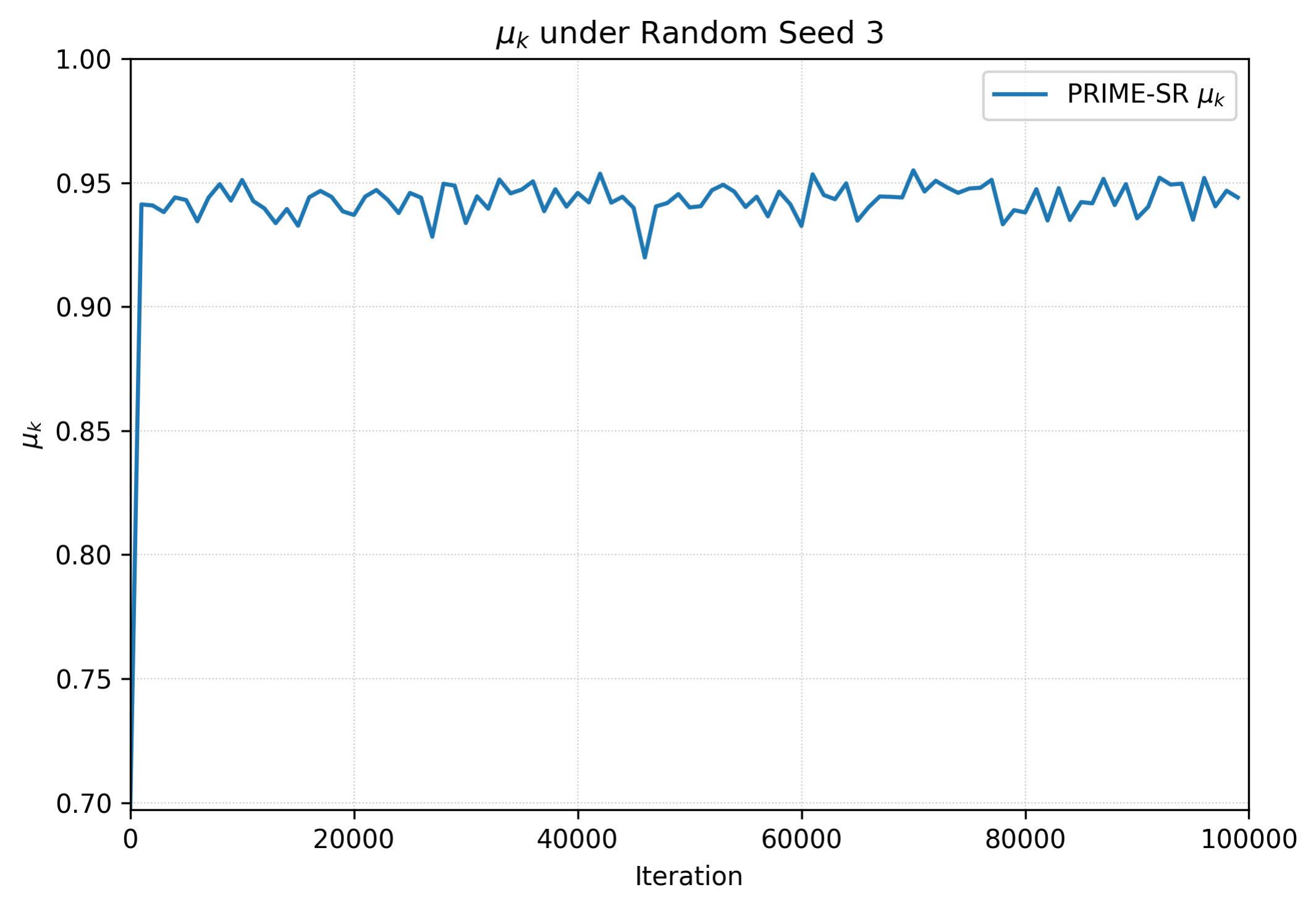}
    \caption{$\mathrm{O}$ atom. Left: relative energy error. Right: $\mu_k$.}
    \end{subfigure}
    
    \caption{Comparison of fixed-$\mu$ SPRING and PRIME-SR on $\mathrm{C}$, $\mathrm{N}$, $\mathrm{O}$ atoms for random seed 3.}
    \label{fig:compare_spring_atom_seed_3}

\end{figure}

 \newpage

\textbf{Random seed 4.}

\begin{figure}[htbp]
    \centering
    
    \begin{subfigure}{\linewidth}
    \centering
    \includegraphics[width=0.48\linewidth]{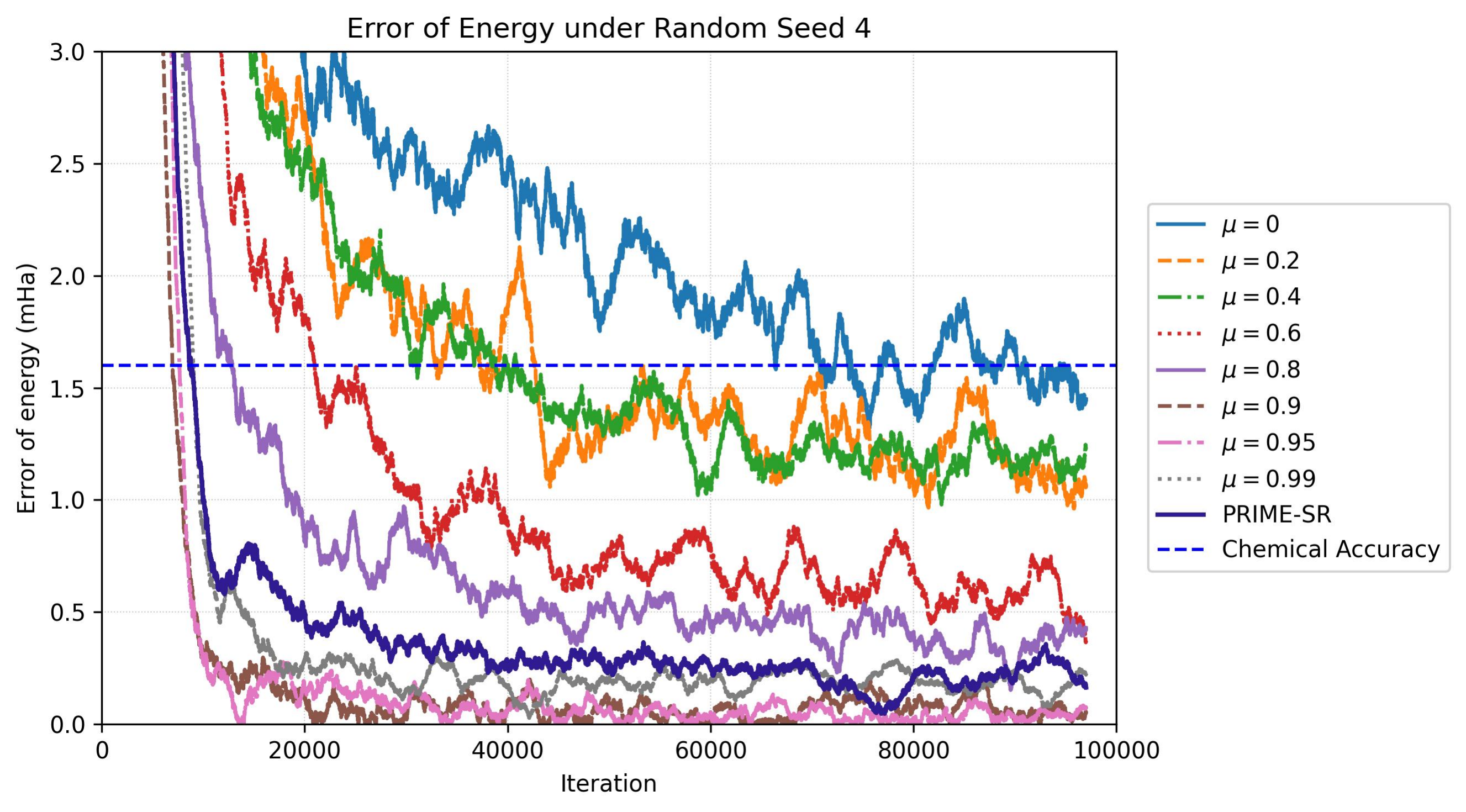}
    \hfill
    \includegraphics[width=0.42\linewidth]{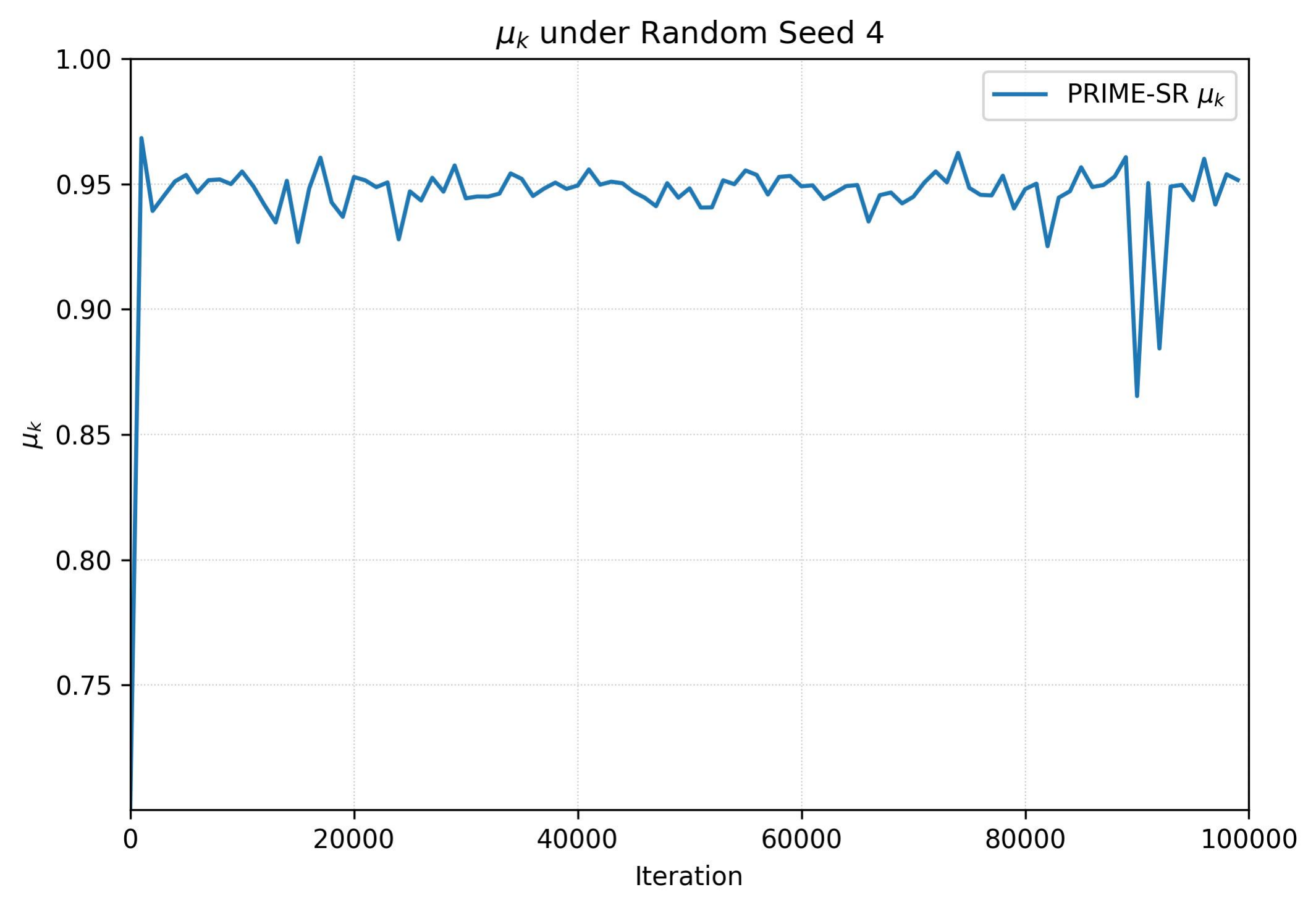}
    \caption{$\mathrm{C}$ atom. Left: relative energy error. Right: $\mu_k$.}
    \end{subfigure}
    
    \vspace{0.4em}
    
    \begin{subfigure}{\linewidth}
    \centering
    \includegraphics[width=0.48\linewidth]{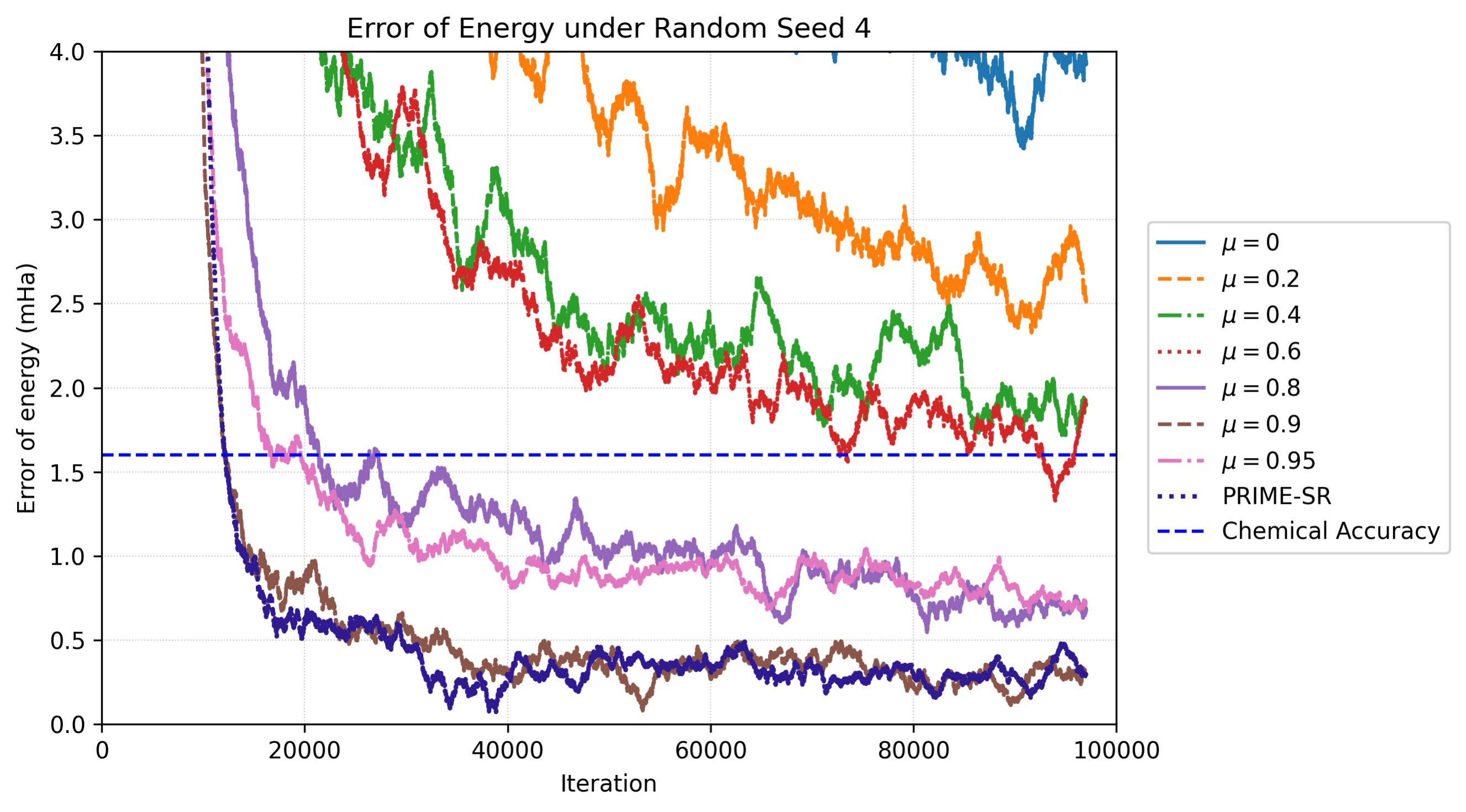}
    \hfill
    \includegraphics[width=0.42\linewidth]{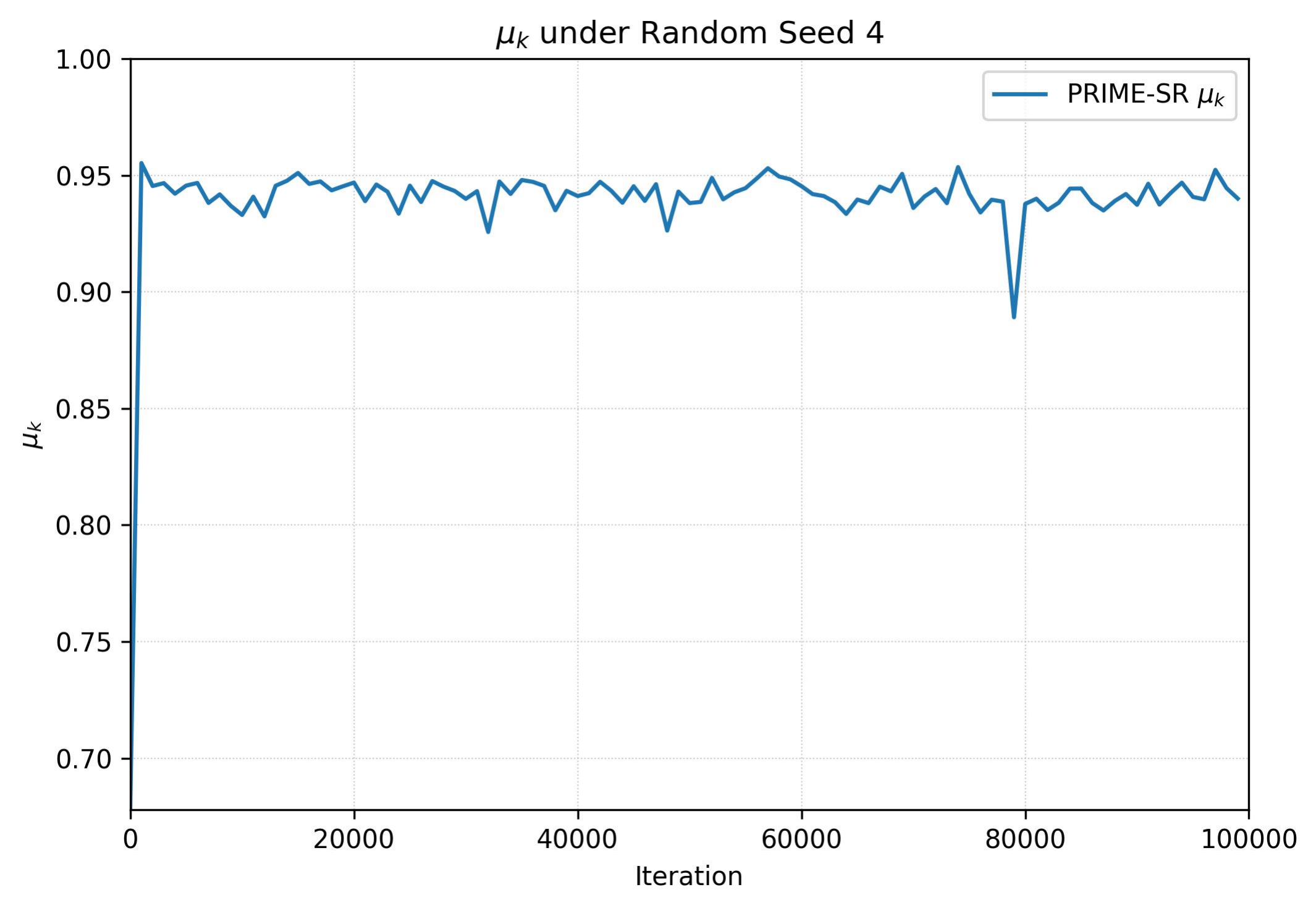}
    \caption{$\mathrm{N}$ atom. Left: relative energy error. Right: $\mu_k$.}
    \end{subfigure}
    
    \vspace{0.4em}
    
    \begin{subfigure}{\linewidth}
    \centering
    \includegraphics[width=0.48\linewidth]{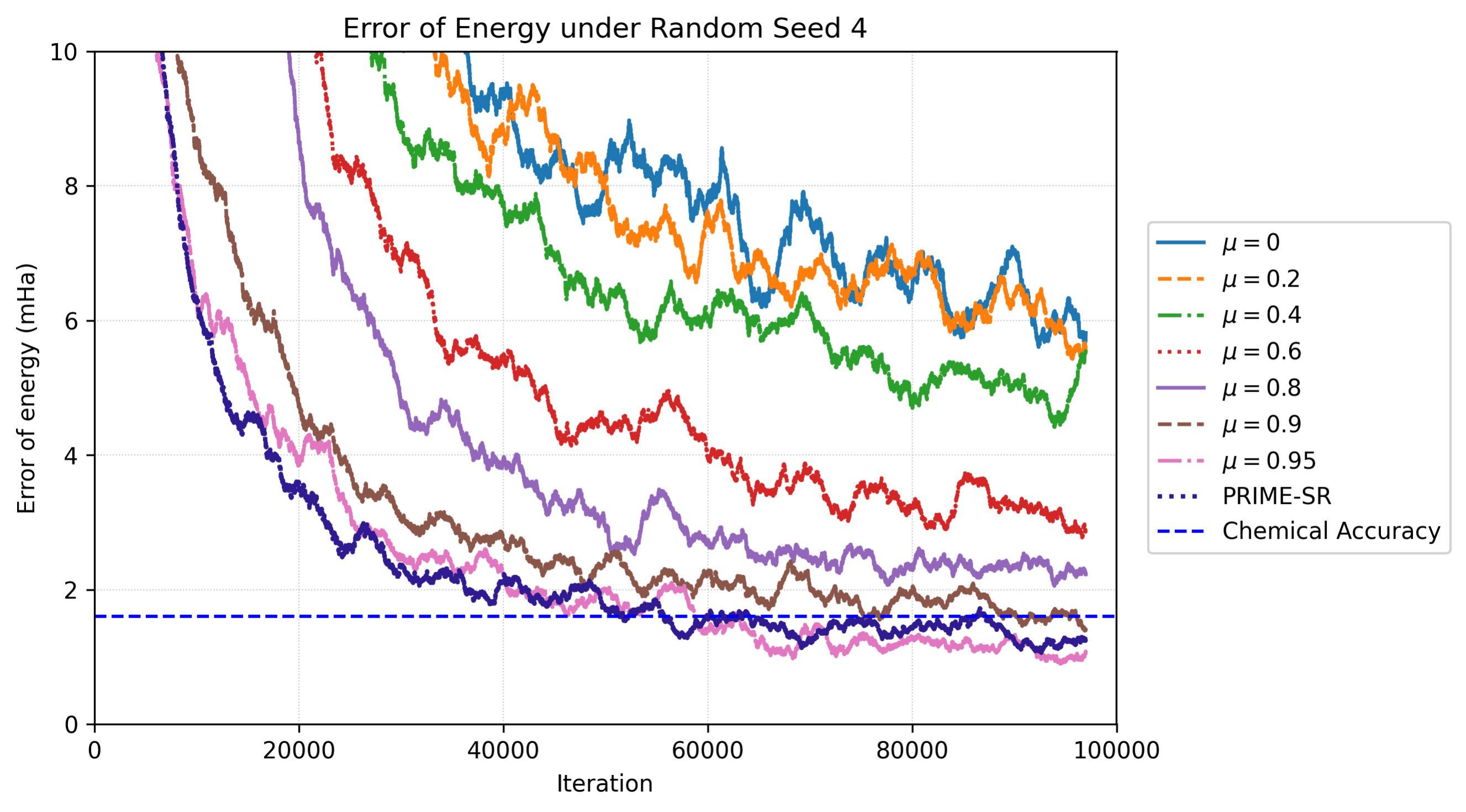}
    \hfill
    \includegraphics[width=0.42\linewidth]{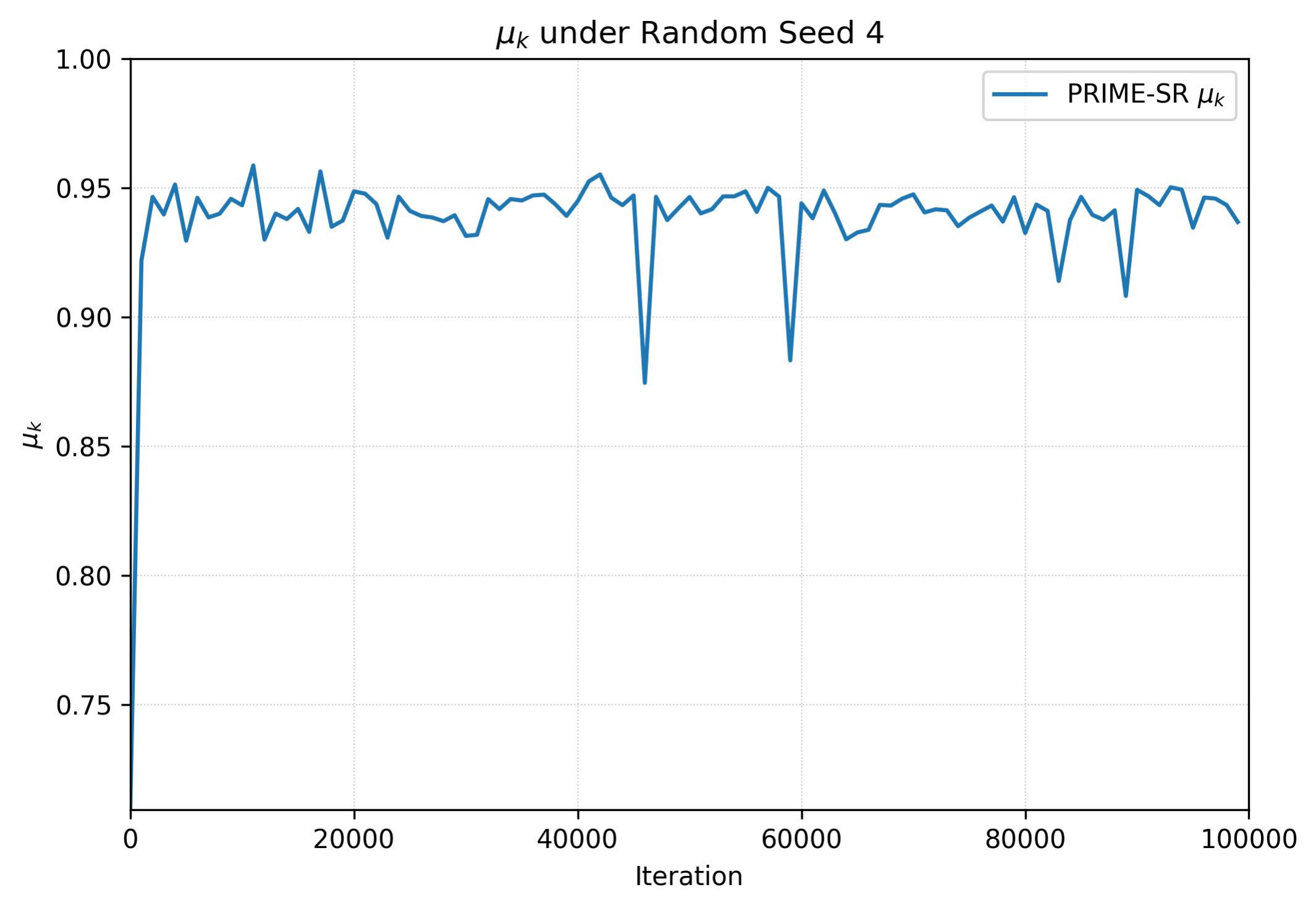}
    \caption{$\mathrm{O}$ atom. Left: relative energy error. Right: $\mu_k$.}
    \end{subfigure}
    
    \caption{Comparison of fixed-$\mu$ SPRING and PRIME-SR on $\mathrm{C}$, $\mathrm{N}$, $\mathrm{O}$ atoms for random seed 4.}
    \label{fig:compare_spring_atom_seed_4}

\end{figure}

 \newpage

\subsection{Molecular Systems: Results Across Random Seeds}
\label{sec:mol_random_seed}

\textbf{Random Seed 1.}

\begin{figure}[htbp]
    \centering
    
    \begin{subfigure}{\linewidth}
    \centering
    \includegraphics[width=0.48\linewidth]{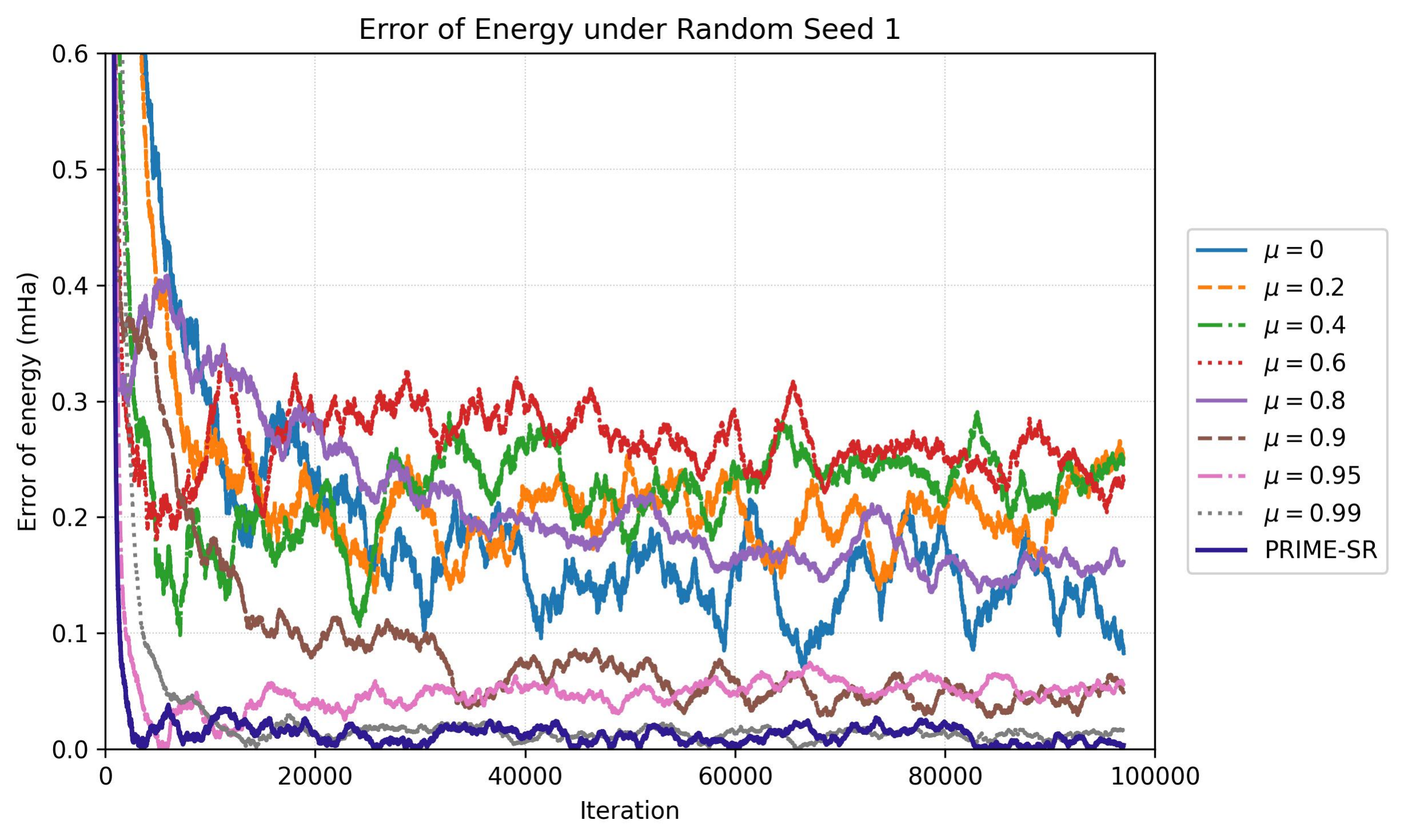}
    \hfill
    \includegraphics[width=0.42\linewidth]{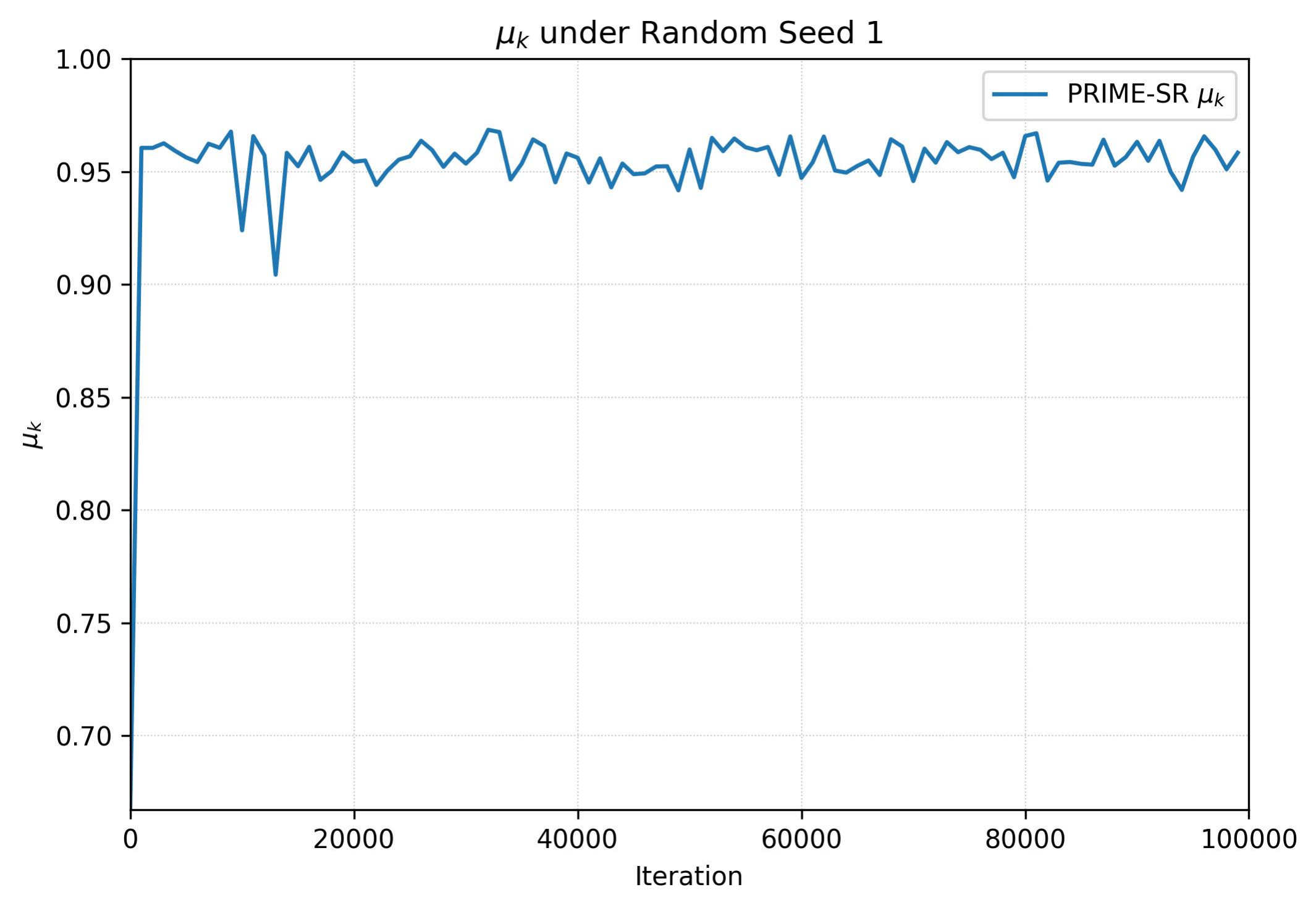}
    \caption{$\mathrm{LiH}$ molecular. Left: relative energy error. Right: $\mu_k$.}
    \end{subfigure}
    
    \vspace{0.4em}
    
    \begin{subfigure}{\linewidth}
    \centering
    \includegraphics[width=0.48\linewidth]{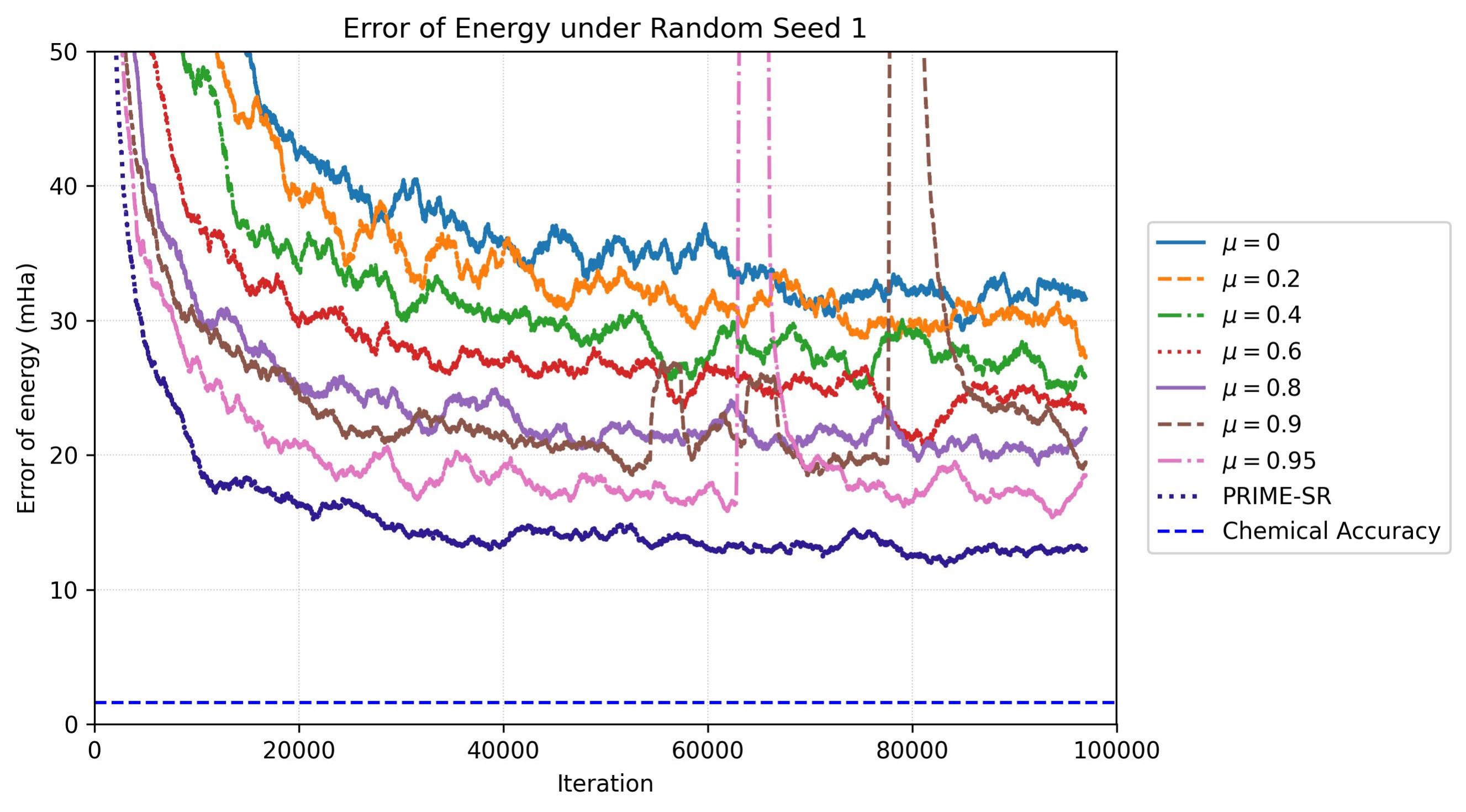}
    \hfill
    \includegraphics[width=0.42\linewidth]{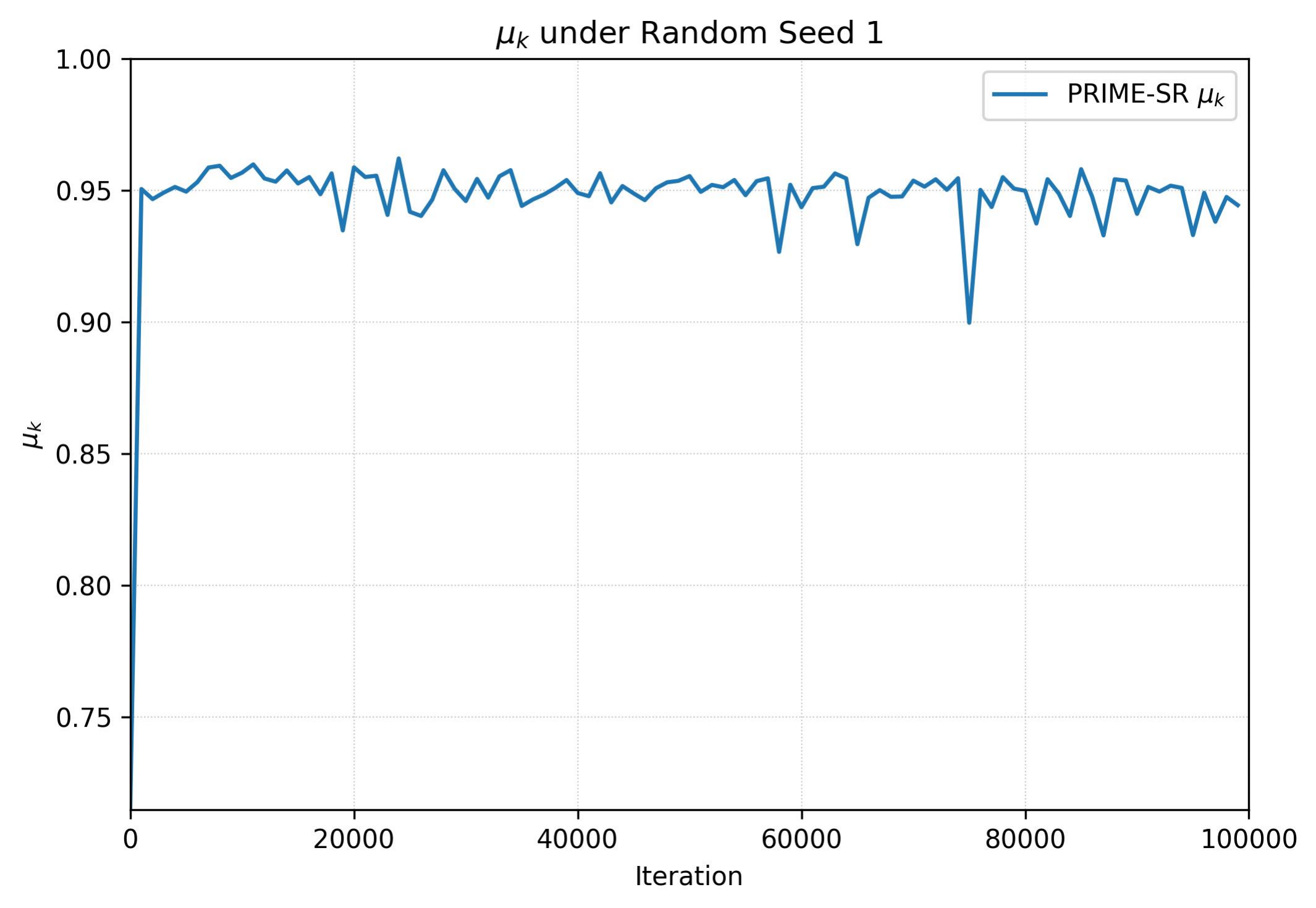}
    \caption{$\mathrm{N}_2$ molecular. Left: relative energy error. Right: $\mu_k$.}
    \end{subfigure}
    
    \vspace{0.4em}
    
    \begin{subfigure}{\linewidth}
    \centering
    \includegraphics[width=0.48\linewidth]{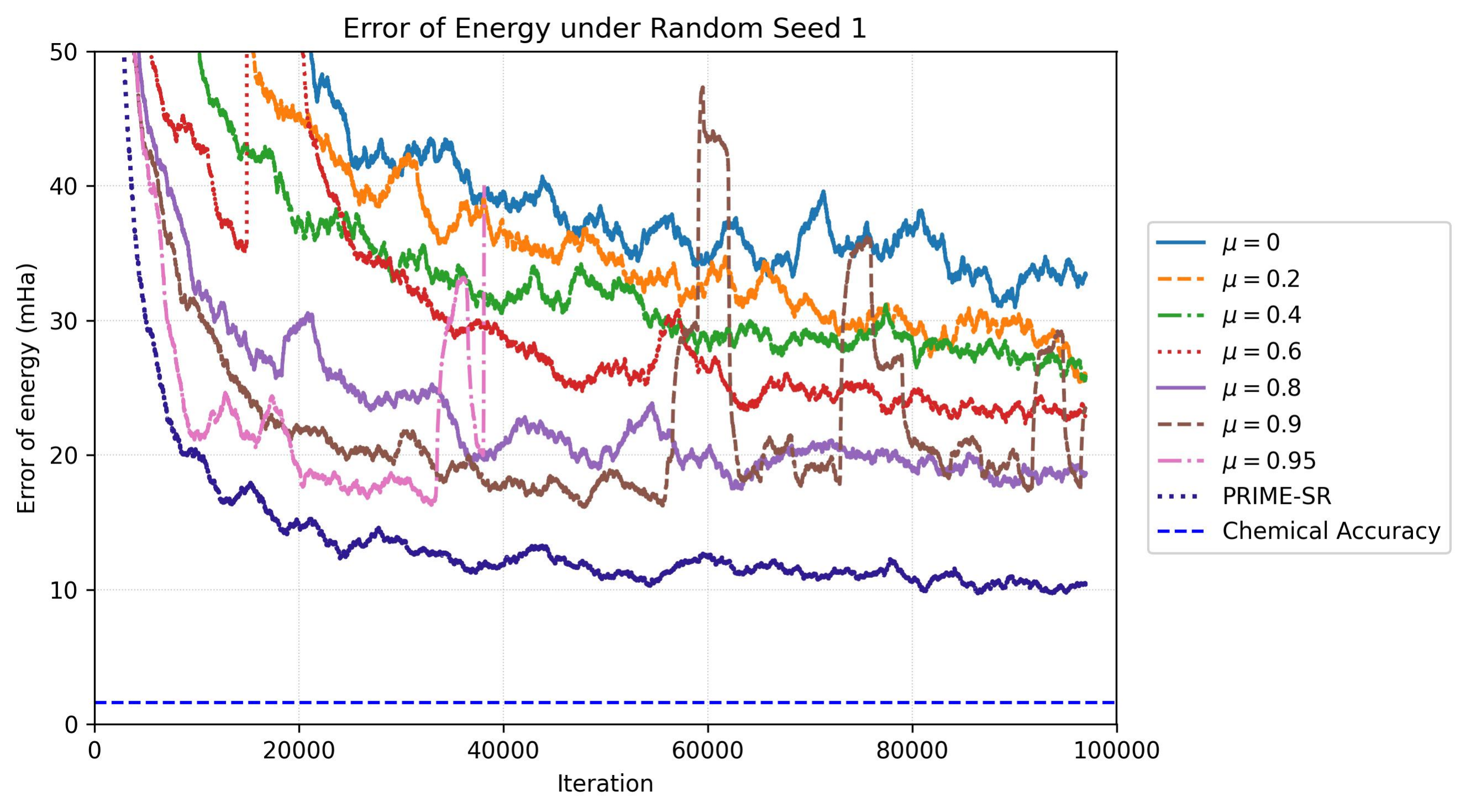}
    \hfill
    \includegraphics[width=0.42\linewidth]{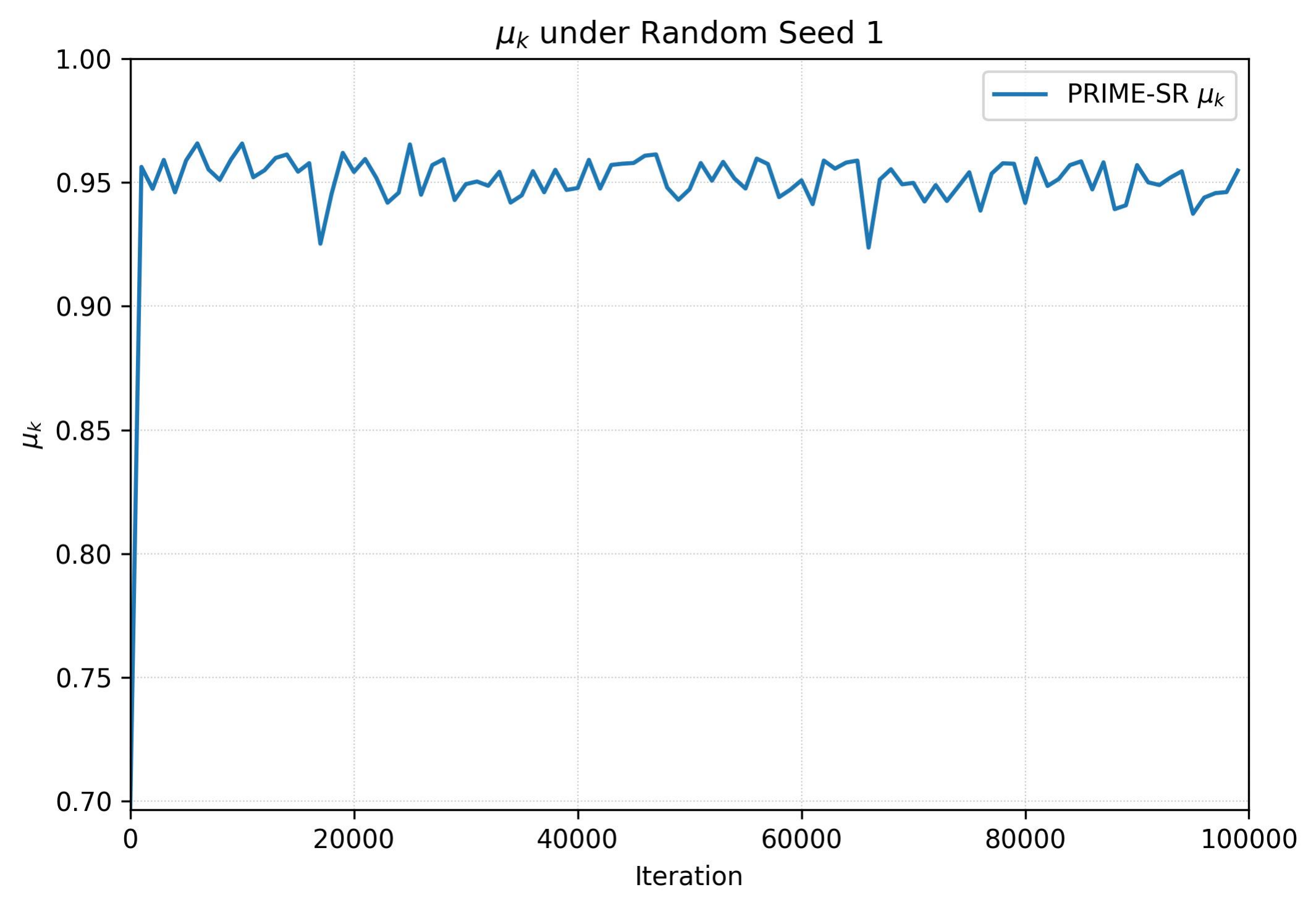}
    \caption{$\mathrm{CO}$ molecular. Left: relative energy error. Right: $\mu_k$.}
    \end{subfigure}
    
    \caption{Comparison of fixed-$\mu$ SPRING and PRIME-SR on $\mathrm{LiH}$, $\mathrm{N}_2$, and $\mathrm{CO}$ molecules for random seed 1.}
    
    \label{fig:compare_spring_mol_seed_1}

\end{figure}

 \newpage

\textbf{Random Seed 2.}

\begin{figure}[htbp]
    \centering
    
    \begin{subfigure}{\linewidth}
    \centering
    \includegraphics[width=0.48\linewidth]{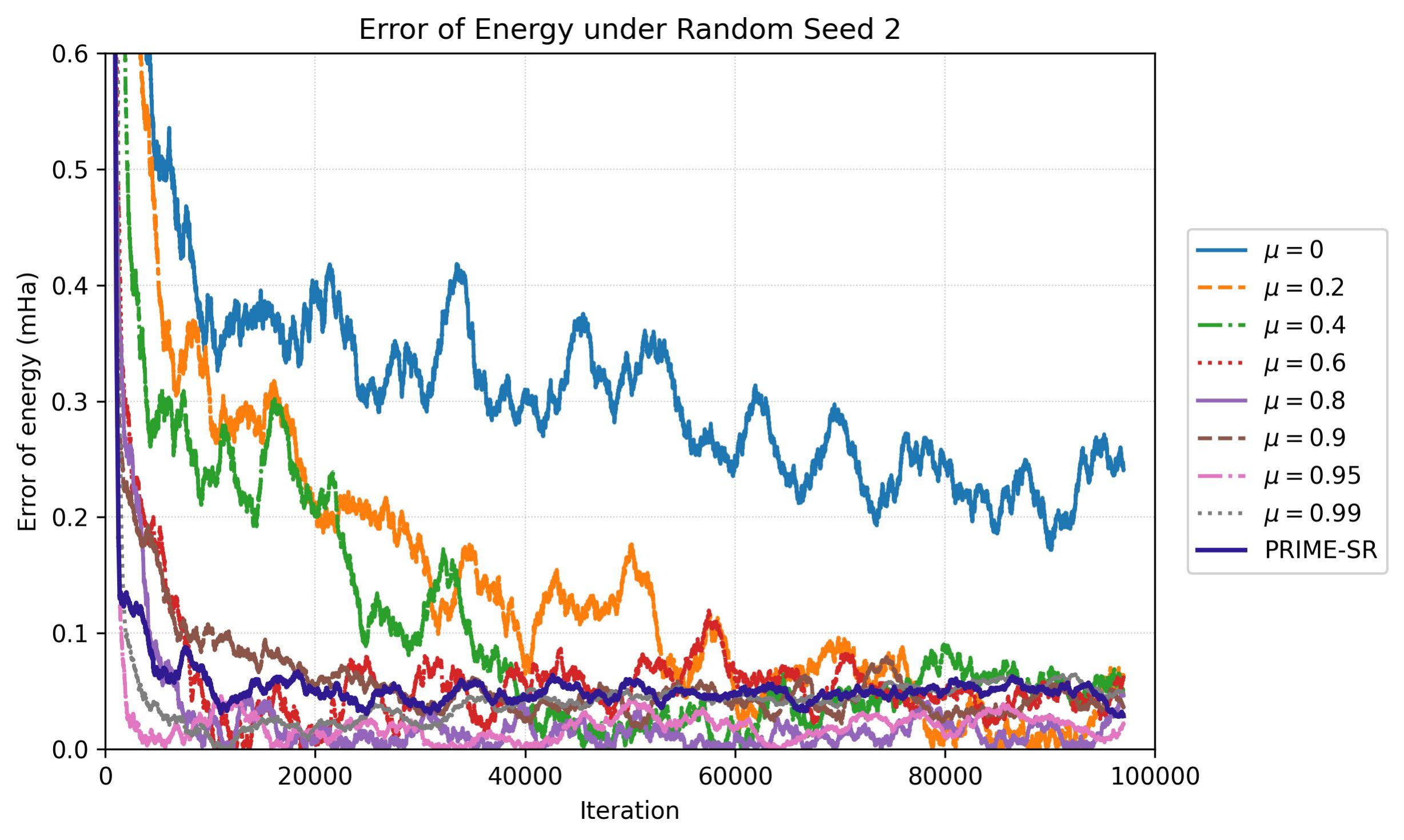}
    \hfill
    \includegraphics[width=0.42\linewidth]{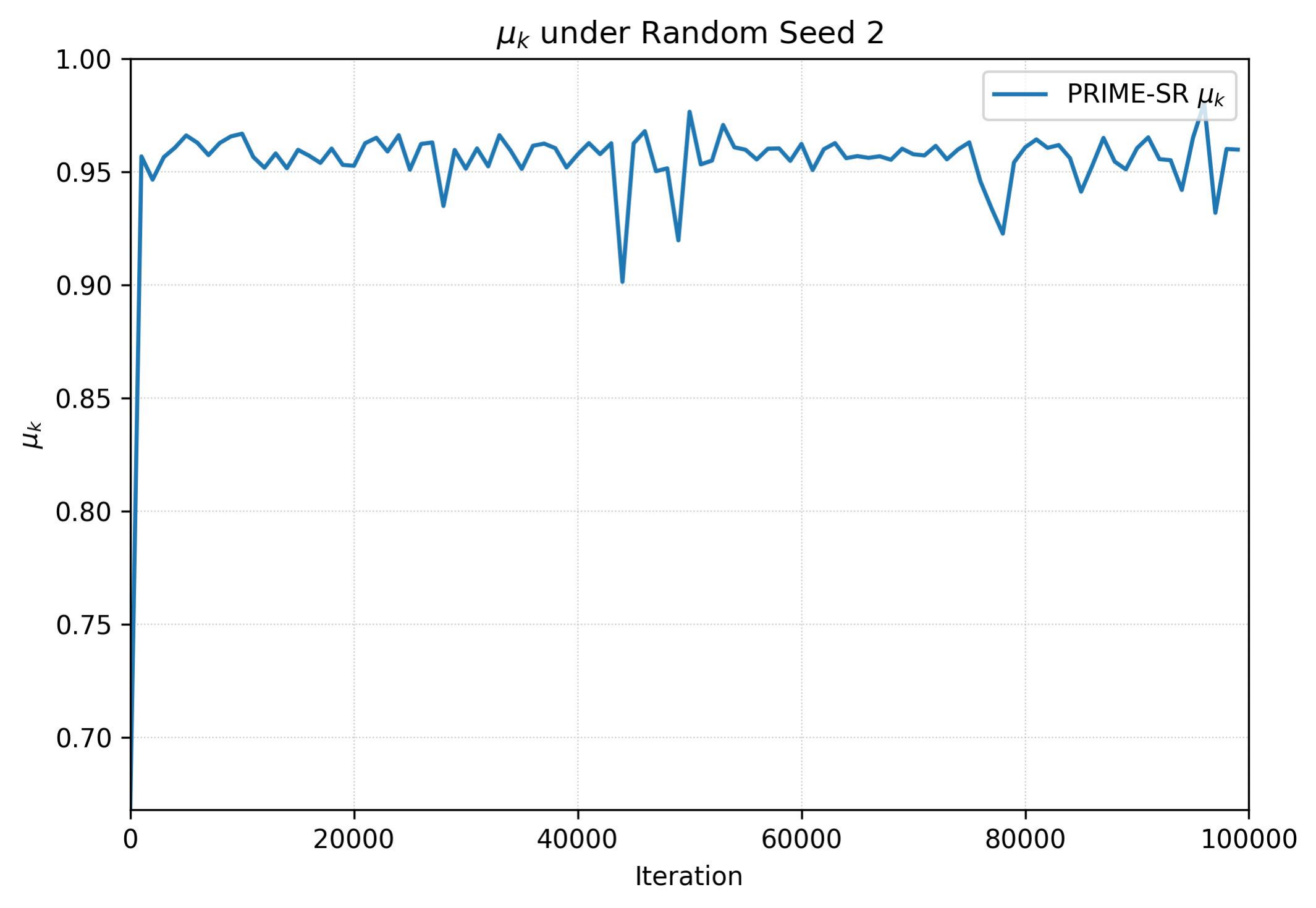}
    \caption{$\mathrm{LiH}$ molecular. Left: relative energy error. Right: $\mu_k$.}
    \end{subfigure}
    
    \vspace{0.4em}
    
    \begin{subfigure}{\linewidth}
    \centering
    \includegraphics[width=0.48\linewidth]{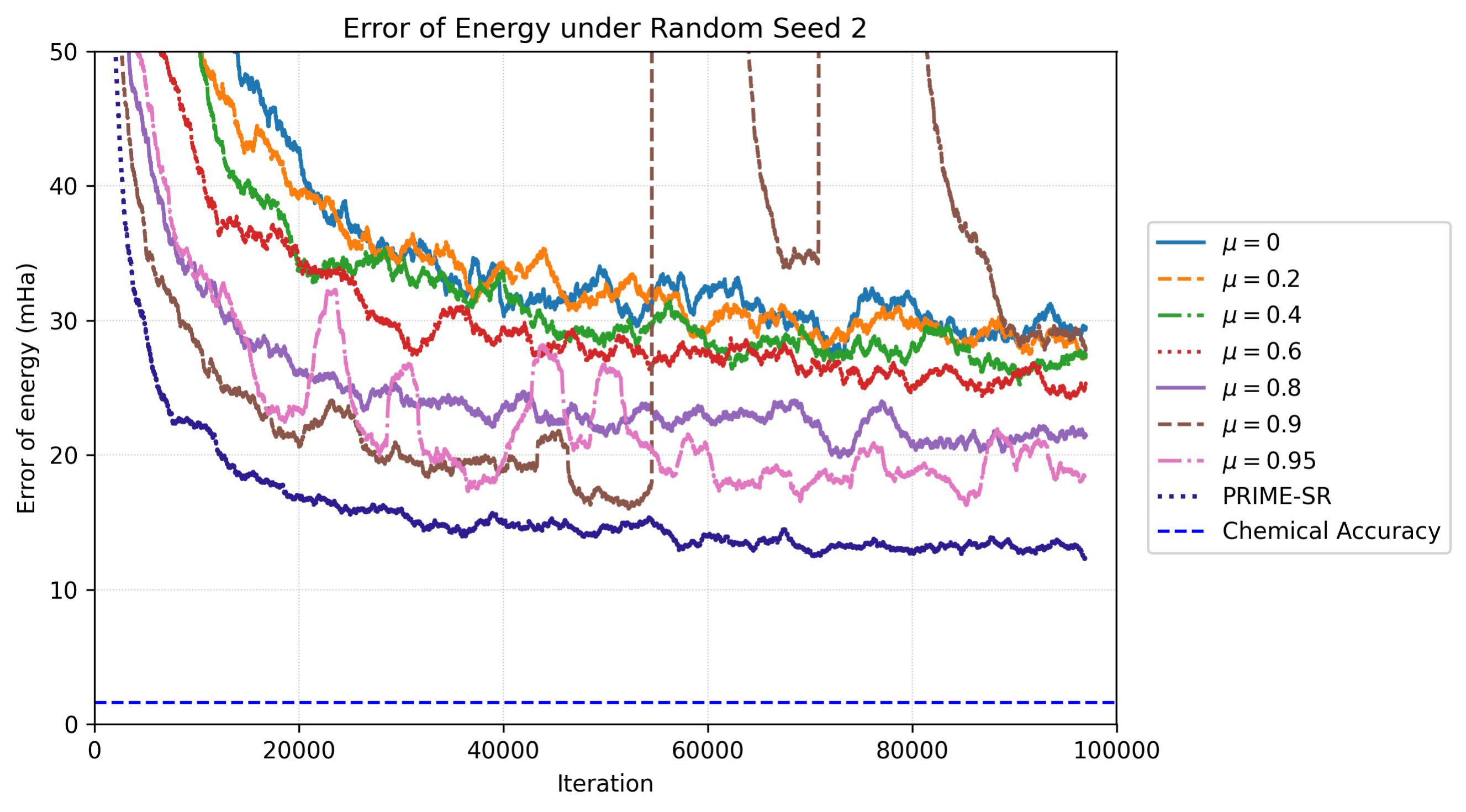}
    \hfill
    \includegraphics[width=0.42\linewidth]{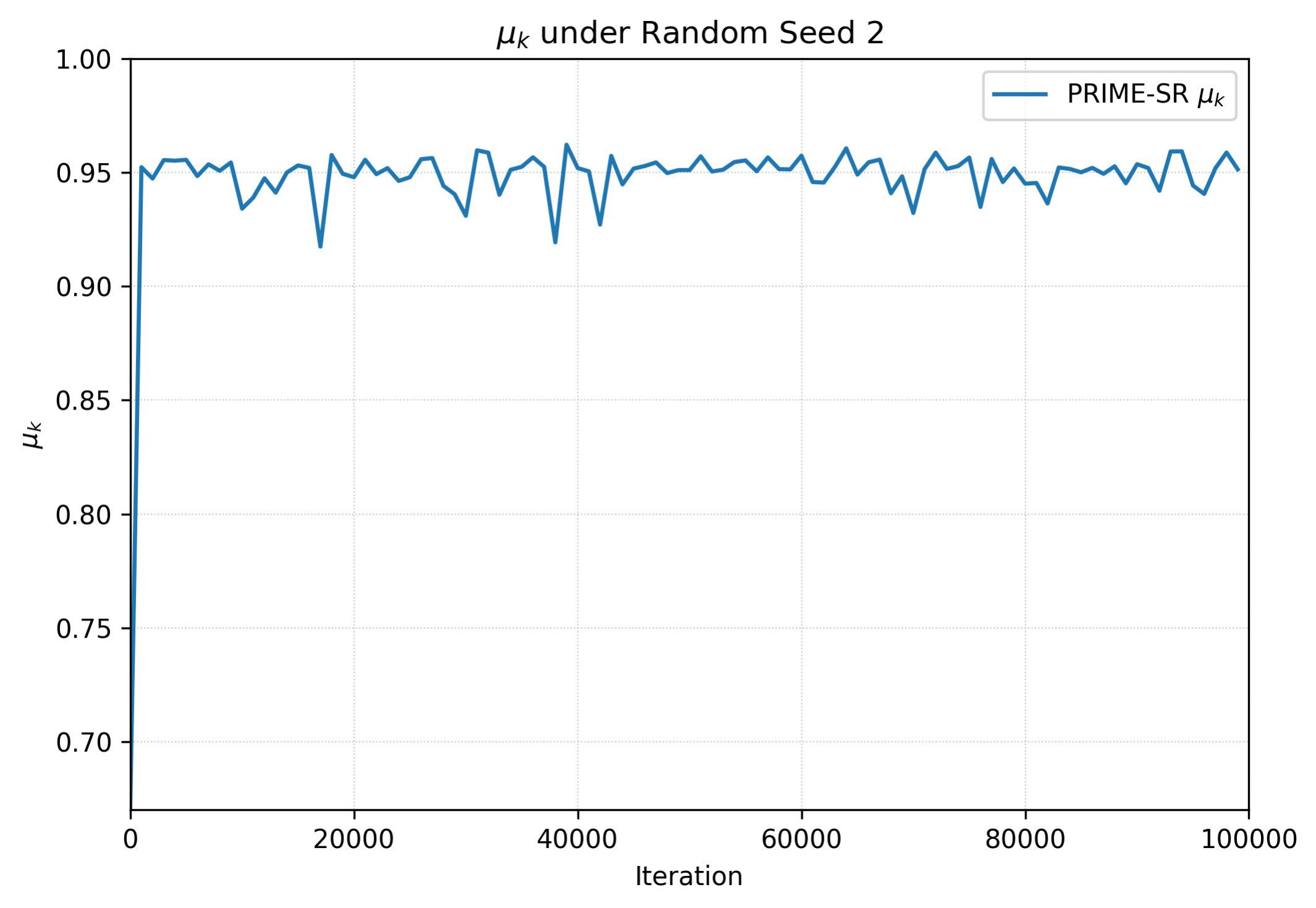}
    \caption{$\mathrm{N}_2$ molecular. Left: relative energy error. Right: $\mu_k$.}
    \end{subfigure}
    
    \vspace{0.4em}
    
    \begin{subfigure}{\linewidth}
    \centering
    \includegraphics[width=0.48\linewidth]{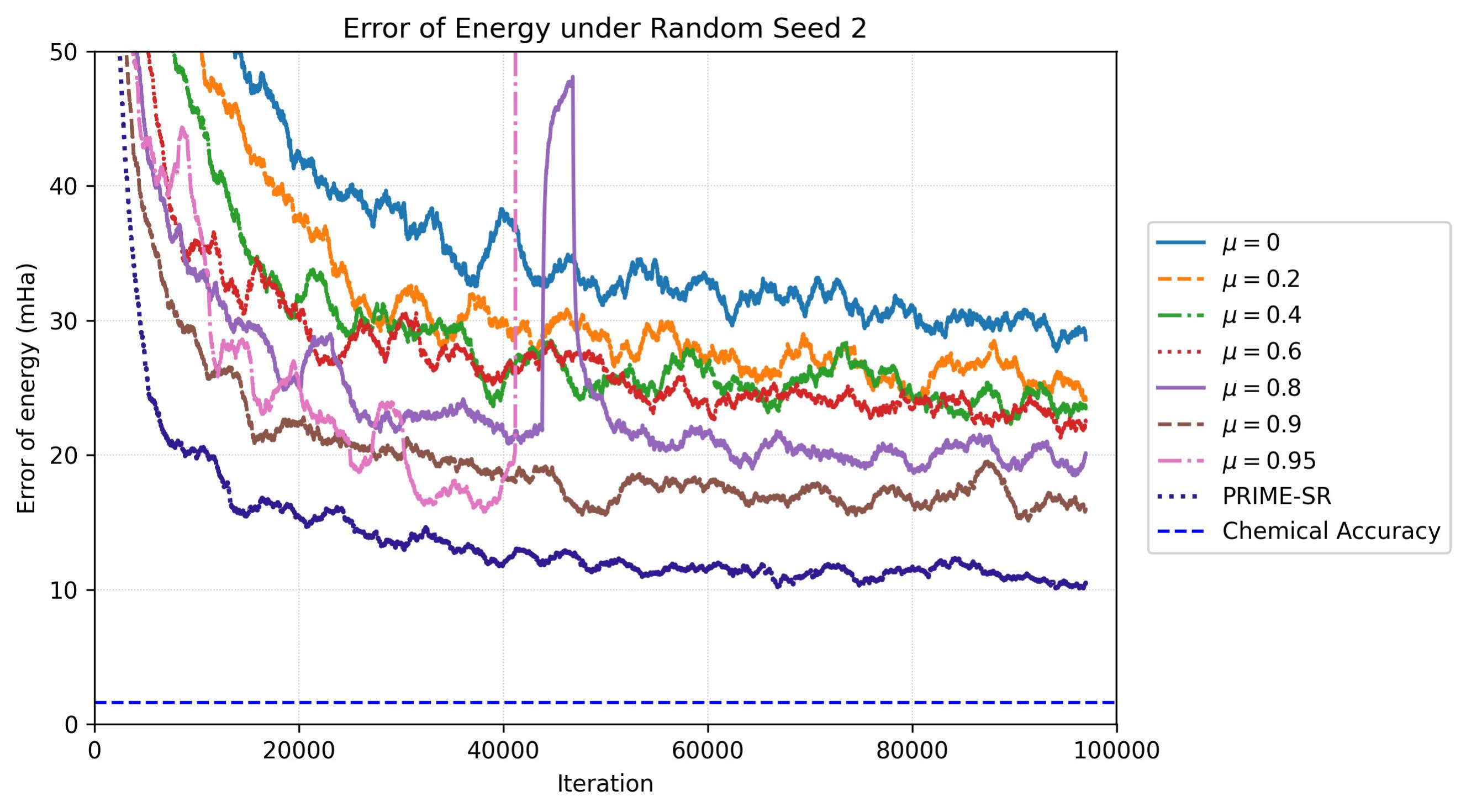}
    \hfill
    \includegraphics[width=0.42\linewidth]{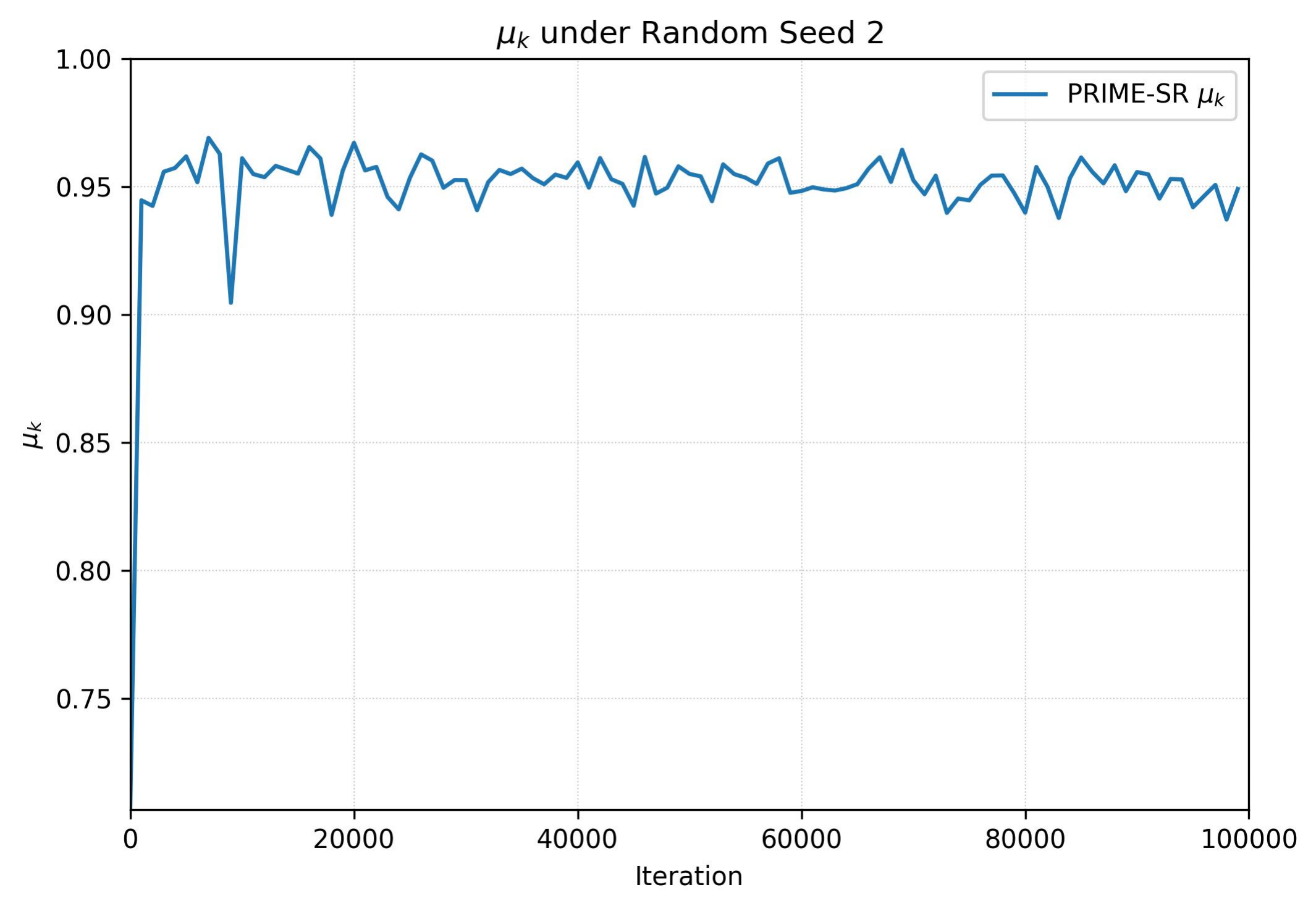}
    \caption{$\mathrm{CO}$ molecular. Left: relative energy error. Right: $\mu_k$.}
    \end{subfigure}
    
    \caption{Comparison of fixed-$\mu$ SPRING and PRIME-SR on $\mathrm{LiH}$, $\mathrm{N}_2$, and $\mathrm{CO}$ molecules for random seed 2.}
    
    \label{fig:compare_spring_mol_seed_2}

\end{figure}

 \newpage

\textbf{Random Seed 3.}

\begin{figure}[htbp]
    \centering
    
    \begin{subfigure}{\linewidth}
    \centering
    \includegraphics[width=0.48\linewidth]{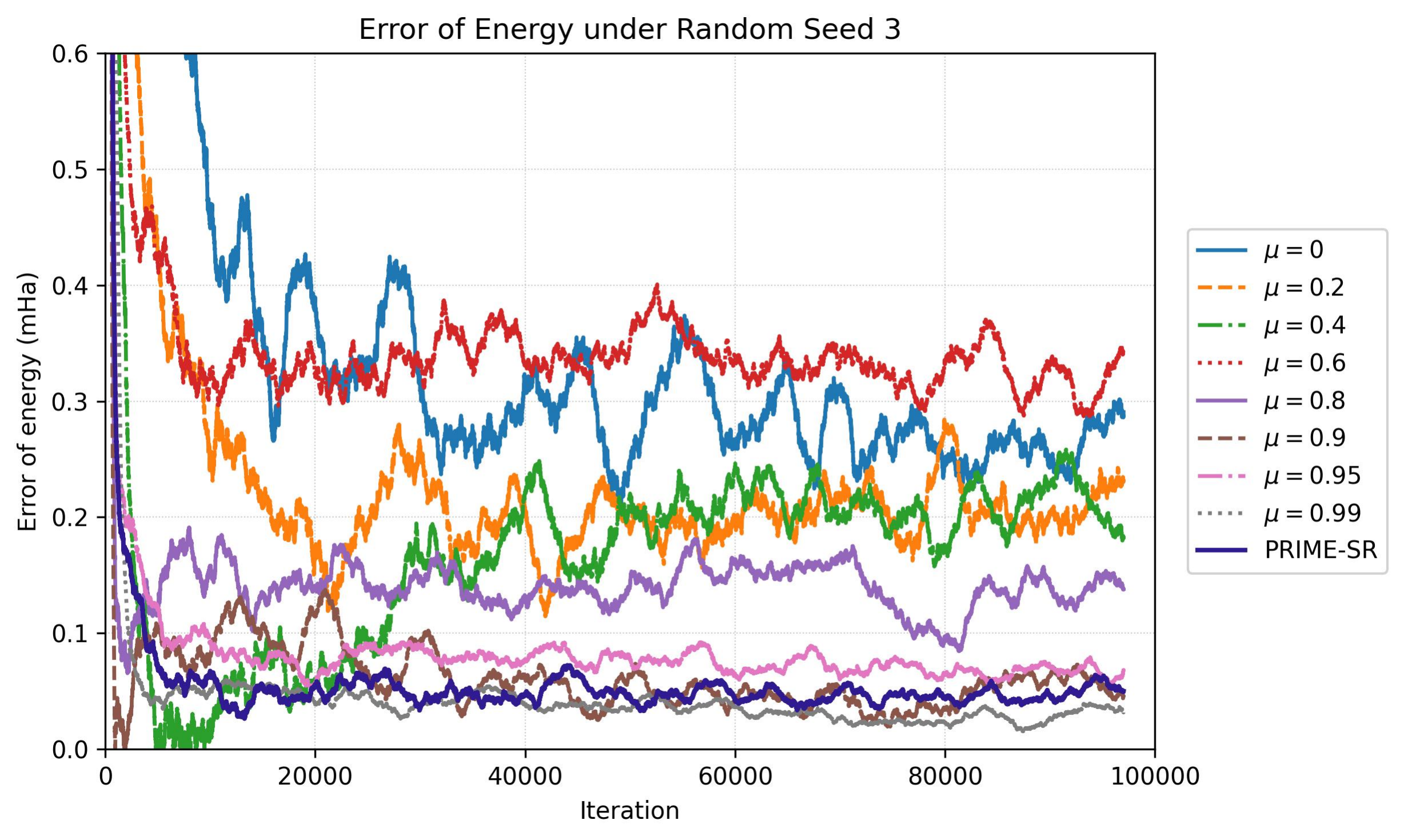}
    \hfill
    \includegraphics[width=0.42\linewidth]{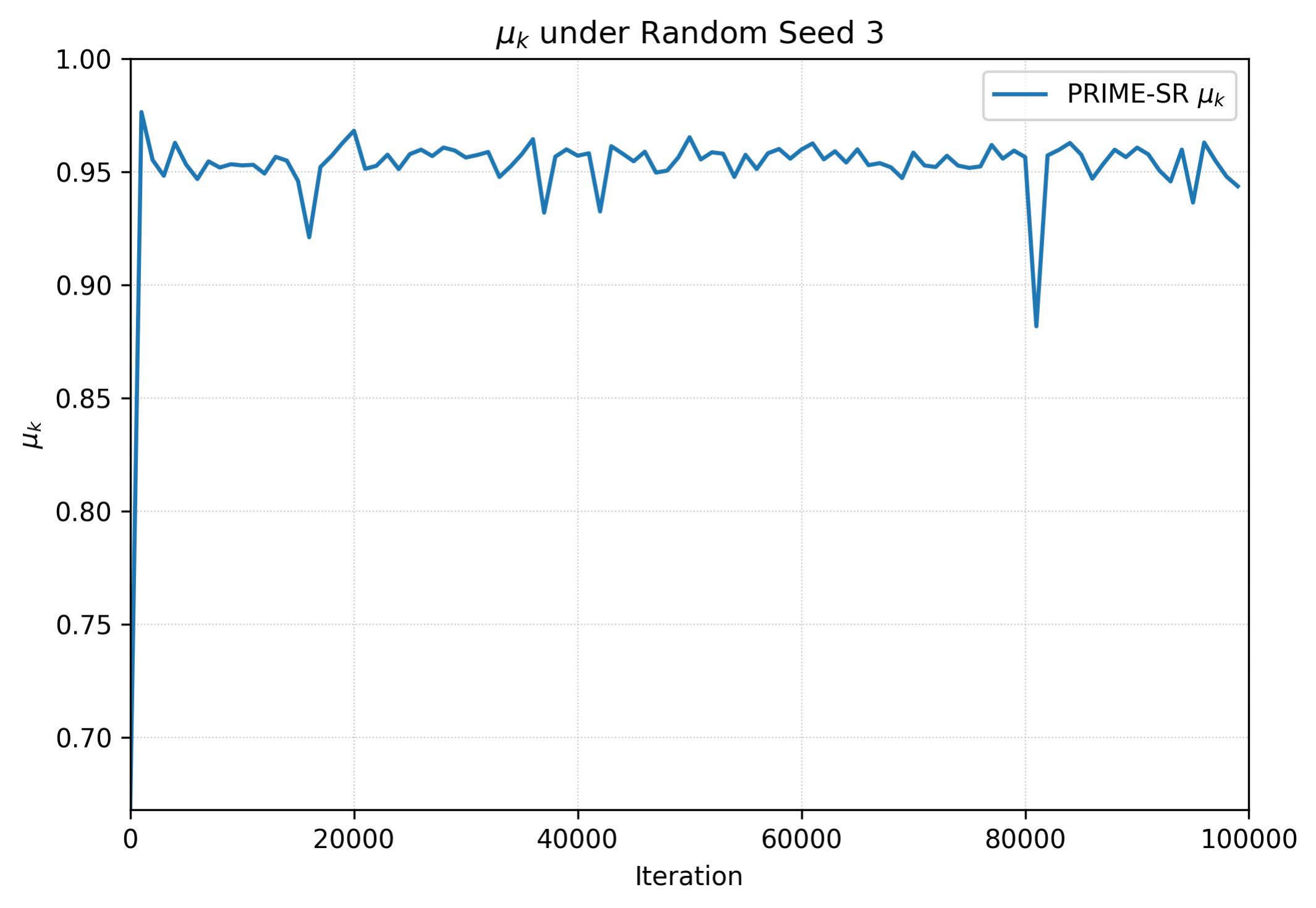}
    \caption{$\mathrm{LiH}$ molecular. Left: relative energy error. Right: $\mu_k$.}
    \end{subfigure}
    
    \vspace{0.4em}
    
    \begin{subfigure}{\linewidth}
    \centering
    \includegraphics[width=0.48\linewidth]{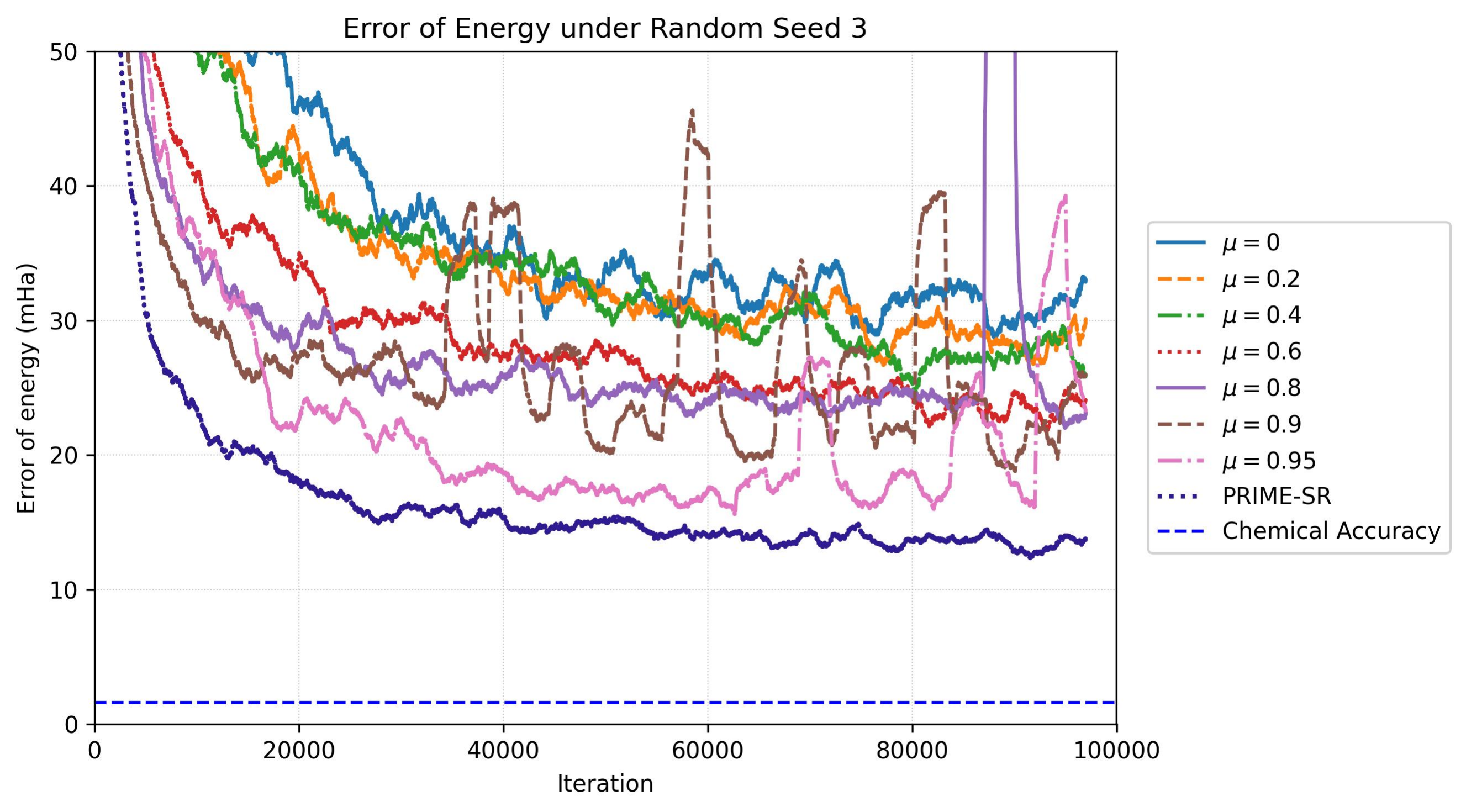}
    \hfill
    \includegraphics[width=0.42\linewidth]{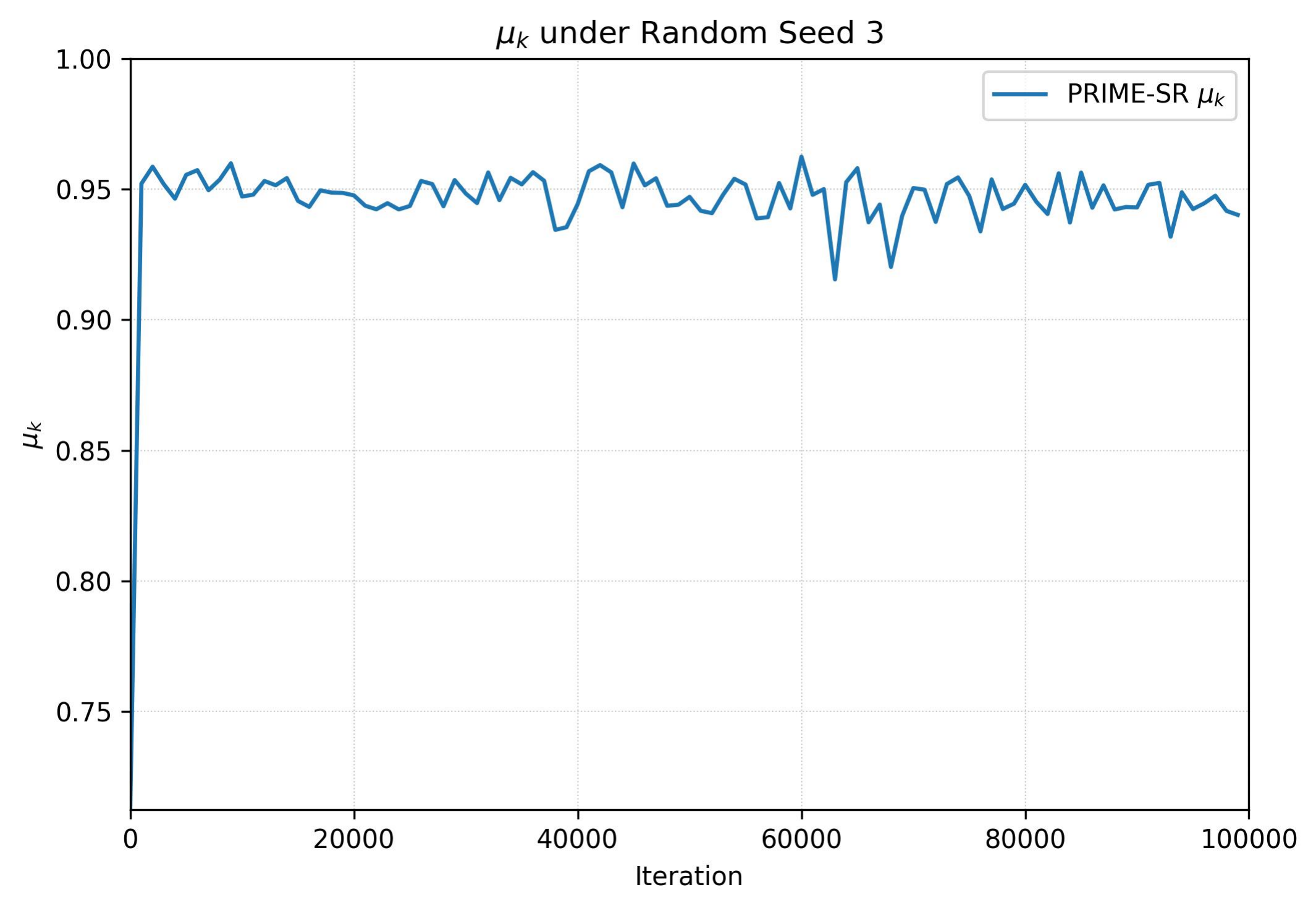}
    \caption{$\mathrm{N}_2$ molecular. Left: relative energy error. Right: $\mu_k$.}
    \end{subfigure}
    
    \vspace{0.4em}
    
    \begin{subfigure}{\linewidth}
    \centering
    \includegraphics[width=0.48\linewidth]{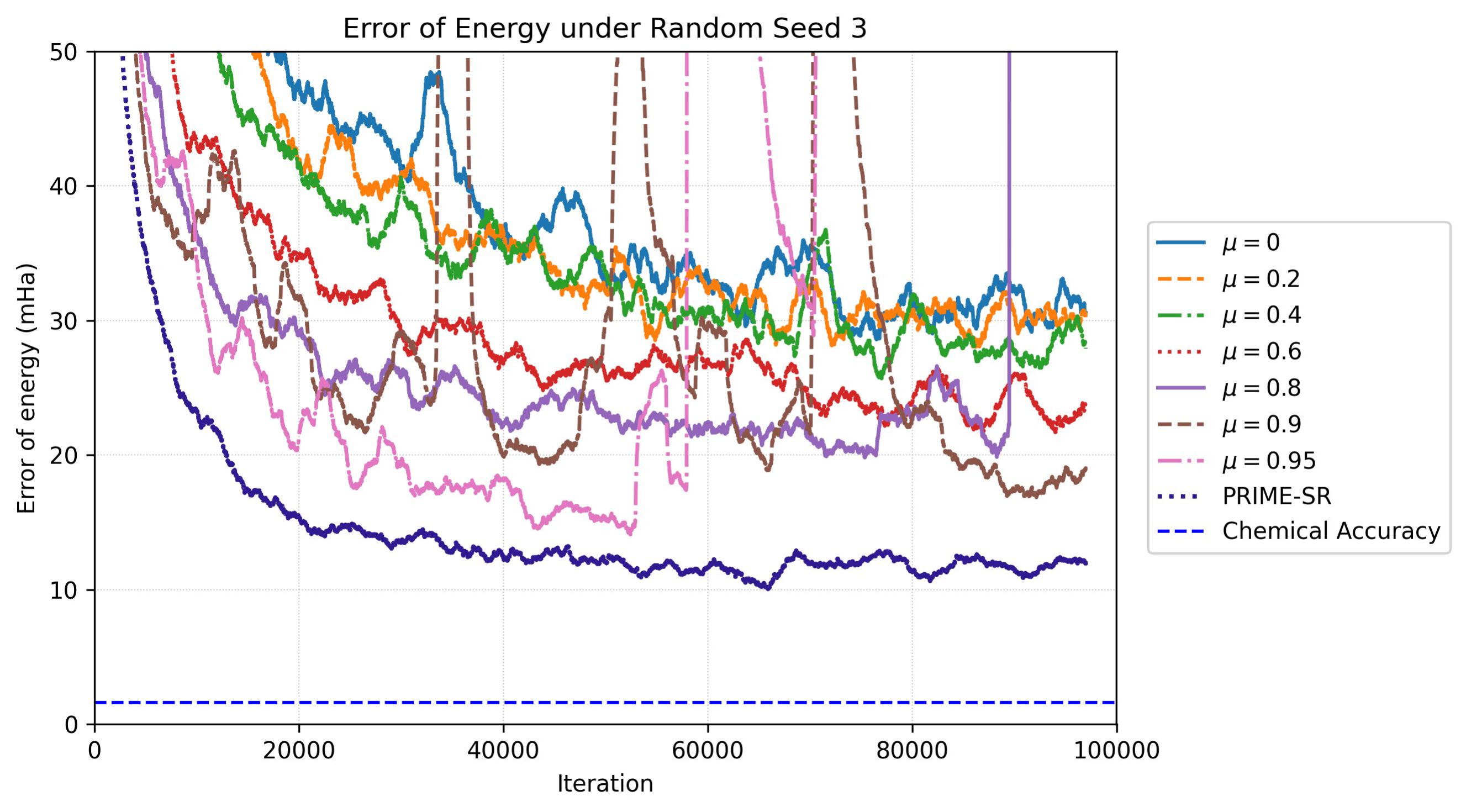}
    \hfill
    \includegraphics[width=0.42\linewidth]{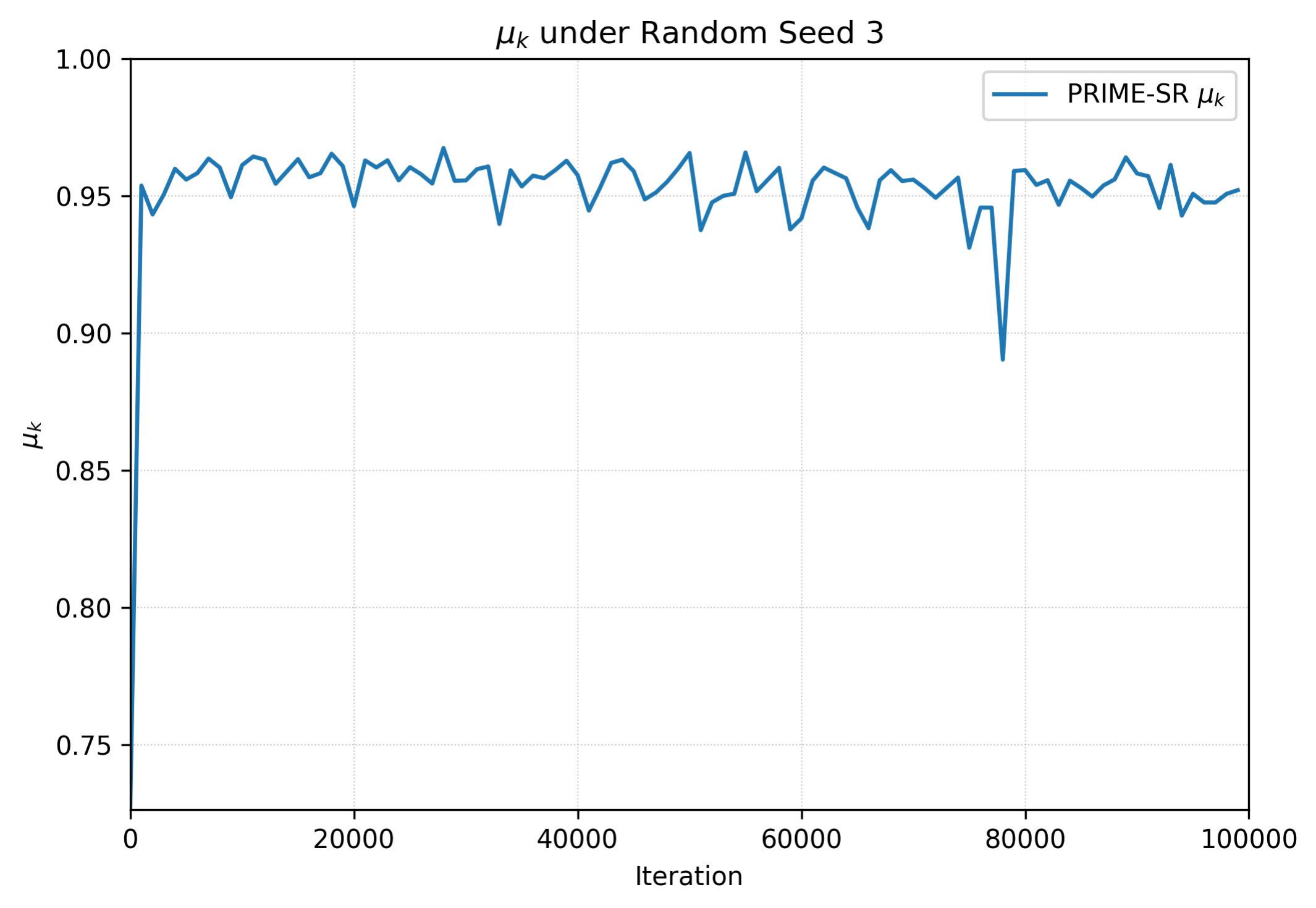}
    \caption{$\mathrm{CO}$ molecular. Left: relative energy error. Right: $\mu_k$.}
    \end{subfigure}
    
    \caption{Comparison of fixed-$\mu$ SPRING and PRIME-SR on $\mathrm{LiH}$, $\mathrm{N}_2$, and $\mathrm{CO}$ molecules for random seed 3.}
    
    \label{fig:compare_spring_mol_seed_3}

\end{figure}

 \newpage

\textbf{Random Seed 4.}

\begin{figure}[htbp]
    \centering
    
    \begin{subfigure}{\linewidth}
    \centering
    \includegraphics[width=0.48\linewidth]{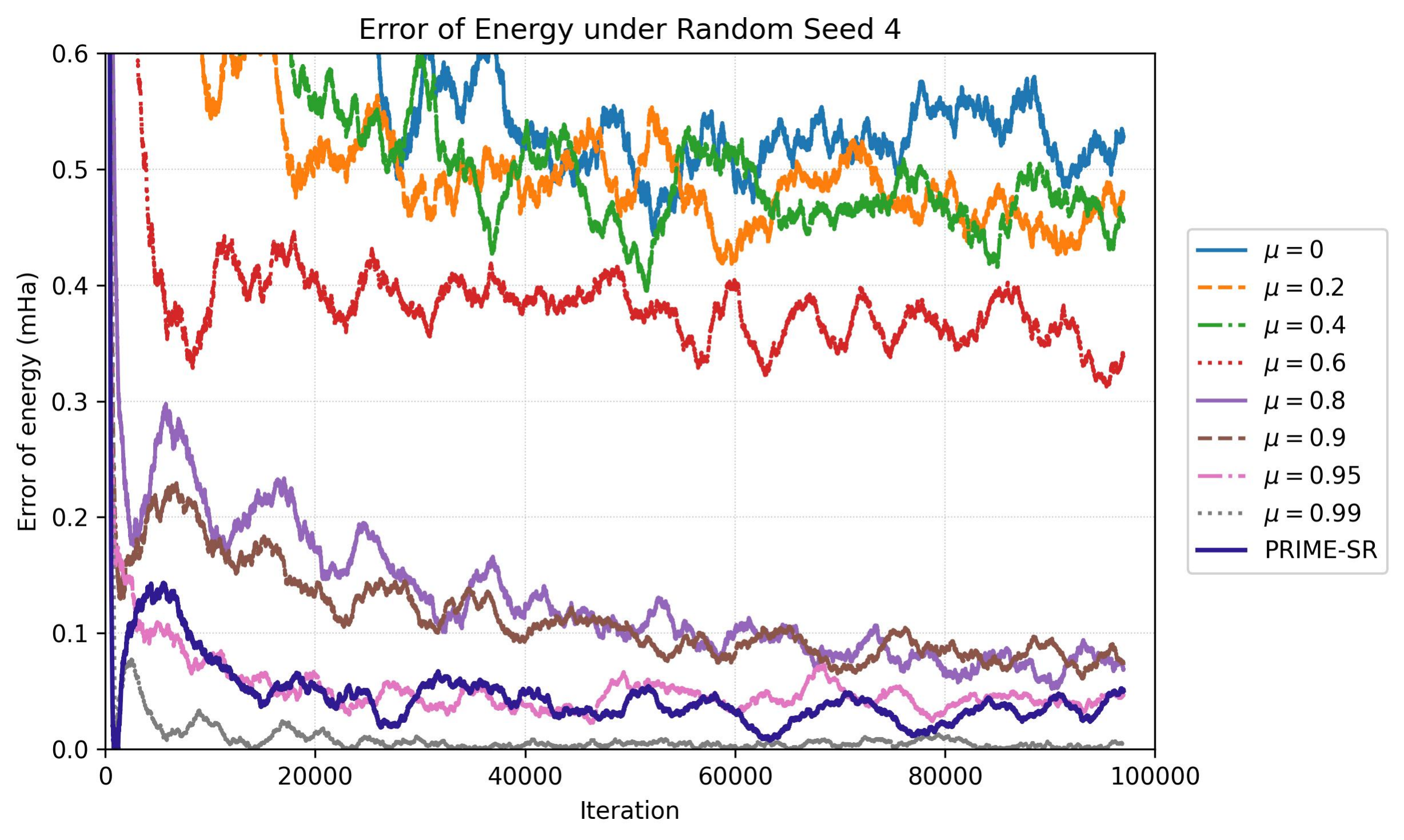}
    \hfill
    \includegraphics[width=0.42\linewidth]{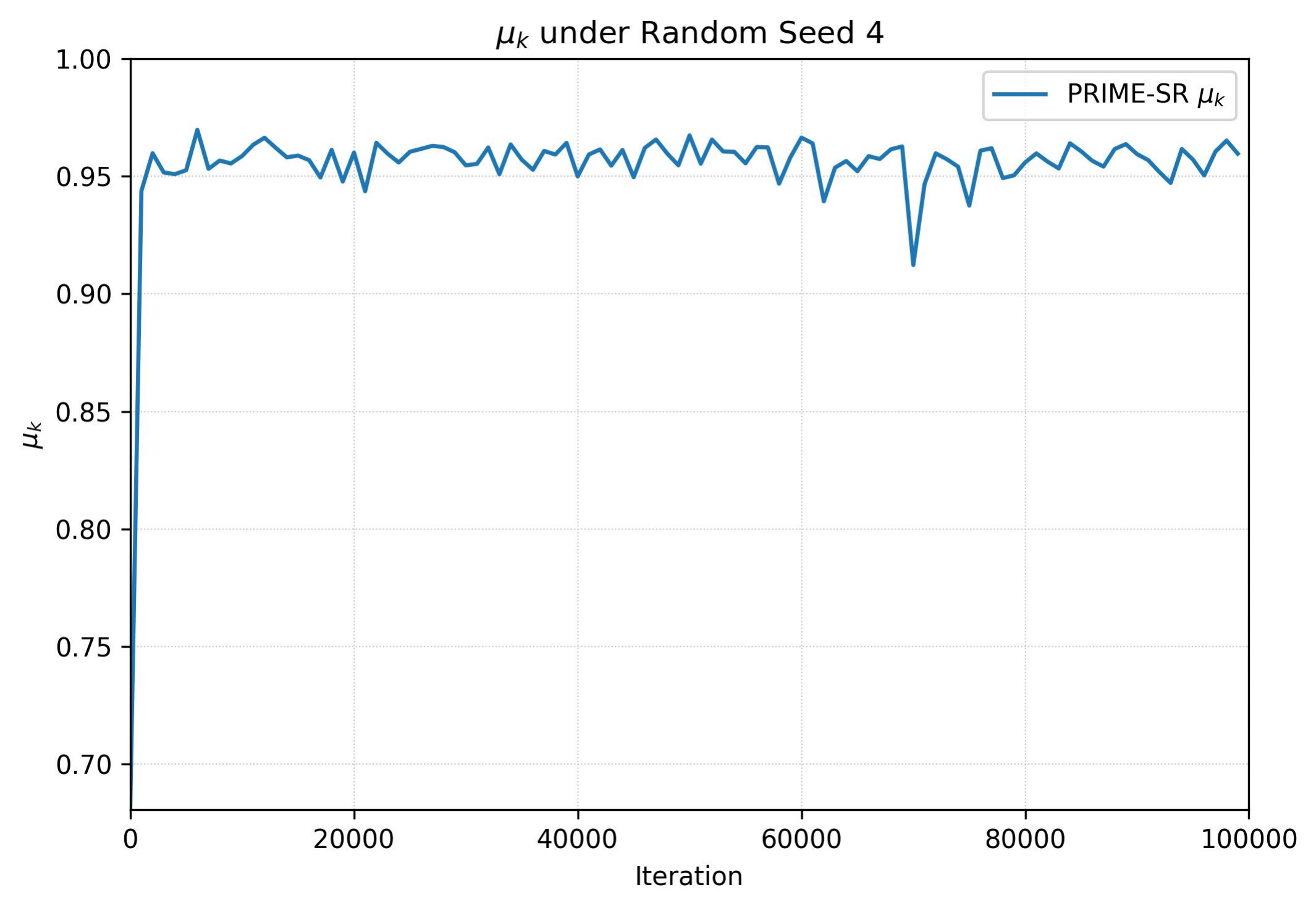}
    \caption{$\mathrm{LiH}$ molecular. Left: relative energy error. Right: $\mu_k$.}
    \end{subfigure}
    
    \vspace{0.4em}
    
    \begin{subfigure}{\linewidth}
    \centering
    \includegraphics[width=0.48\linewidth]{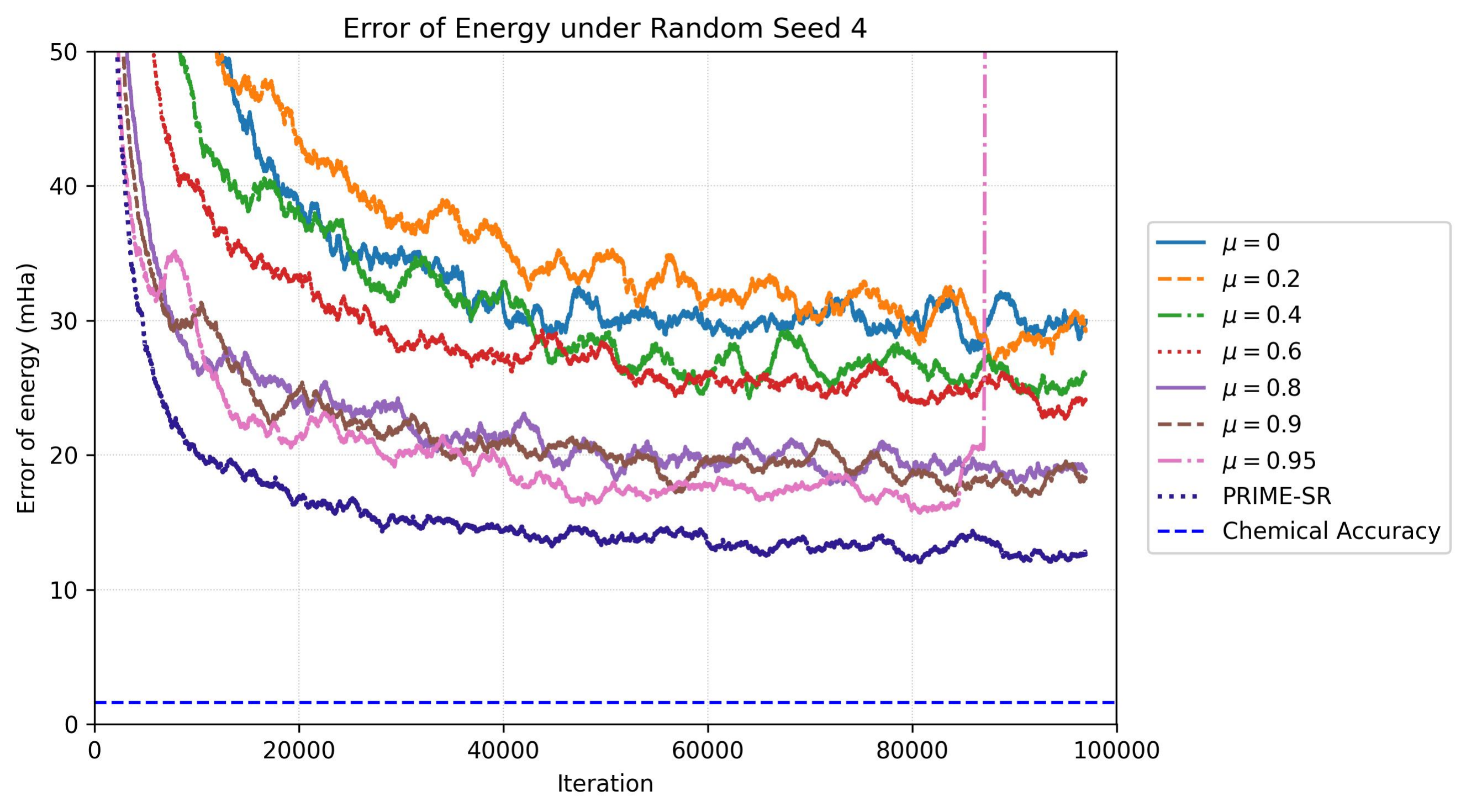}
    \hfill
    \includegraphics[width=0.42\linewidth]{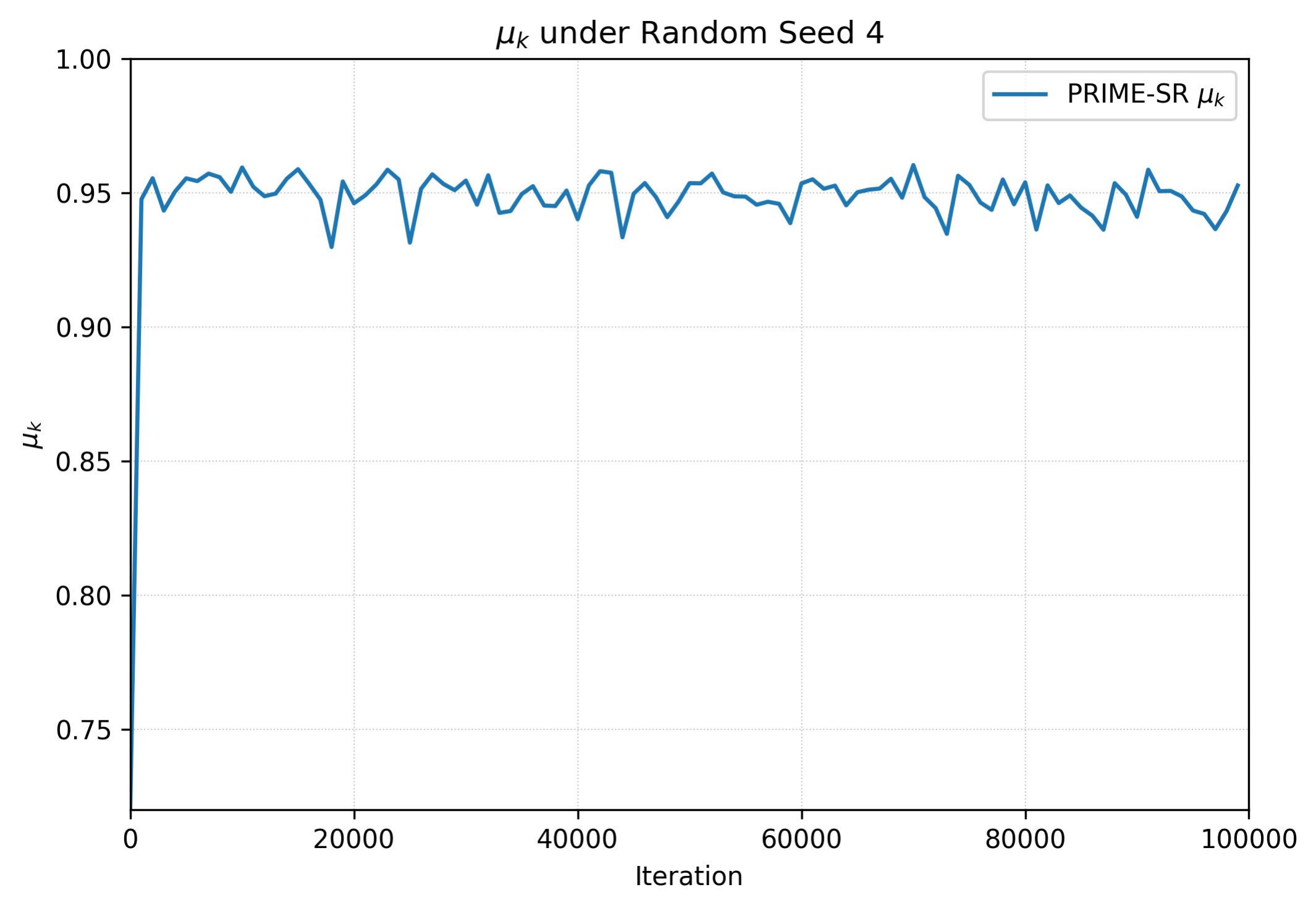}
    \caption{$\mathrm{N}_2$ molecular. Left: relative energy error. Right: $\mu_k$.}
    \end{subfigure}
    
    \vspace{0.4em}
    
    \begin{subfigure}{\linewidth}
    \centering
    \includegraphics[width=0.48\linewidth]{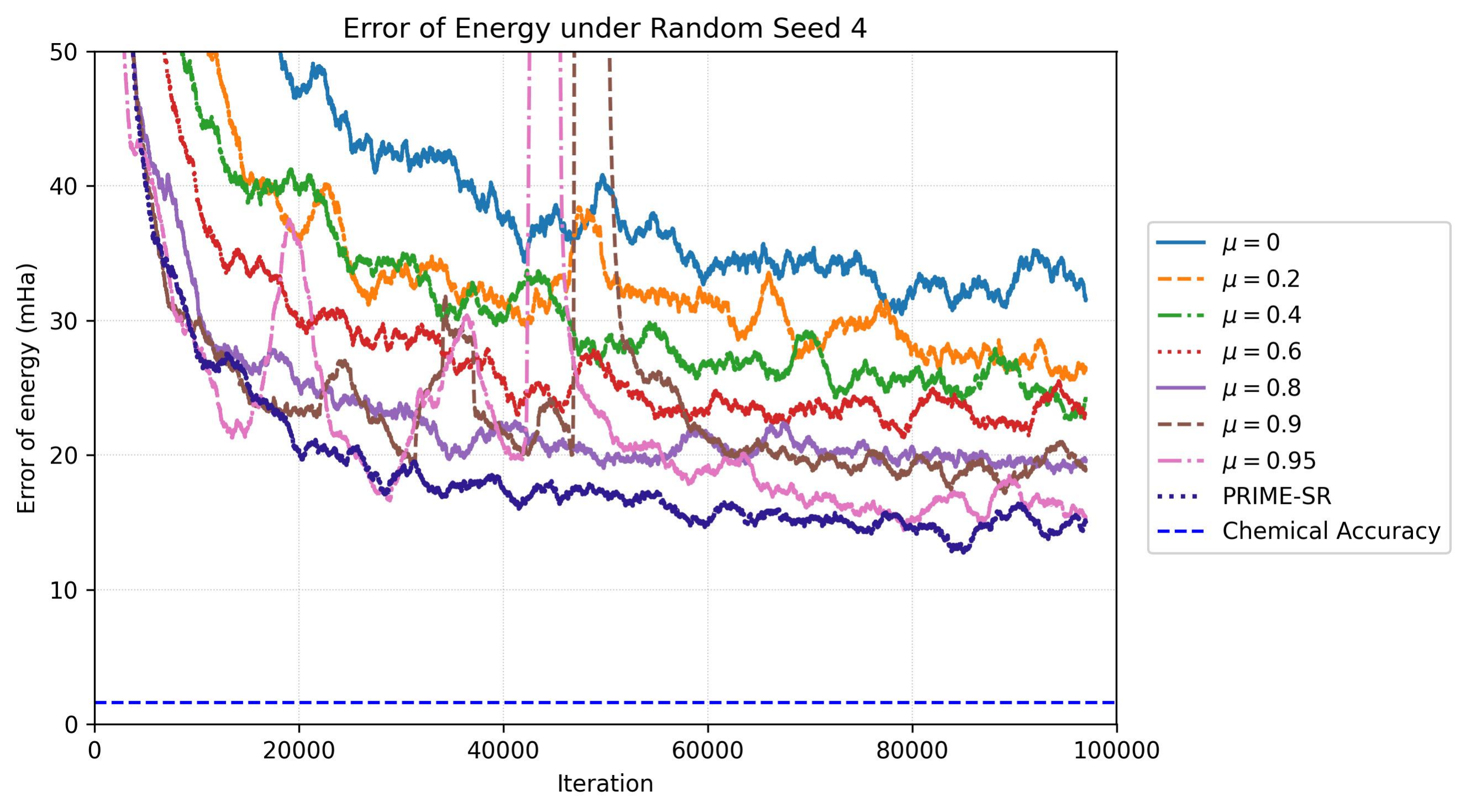}
    \hfill
    \includegraphics[width=0.42\linewidth]{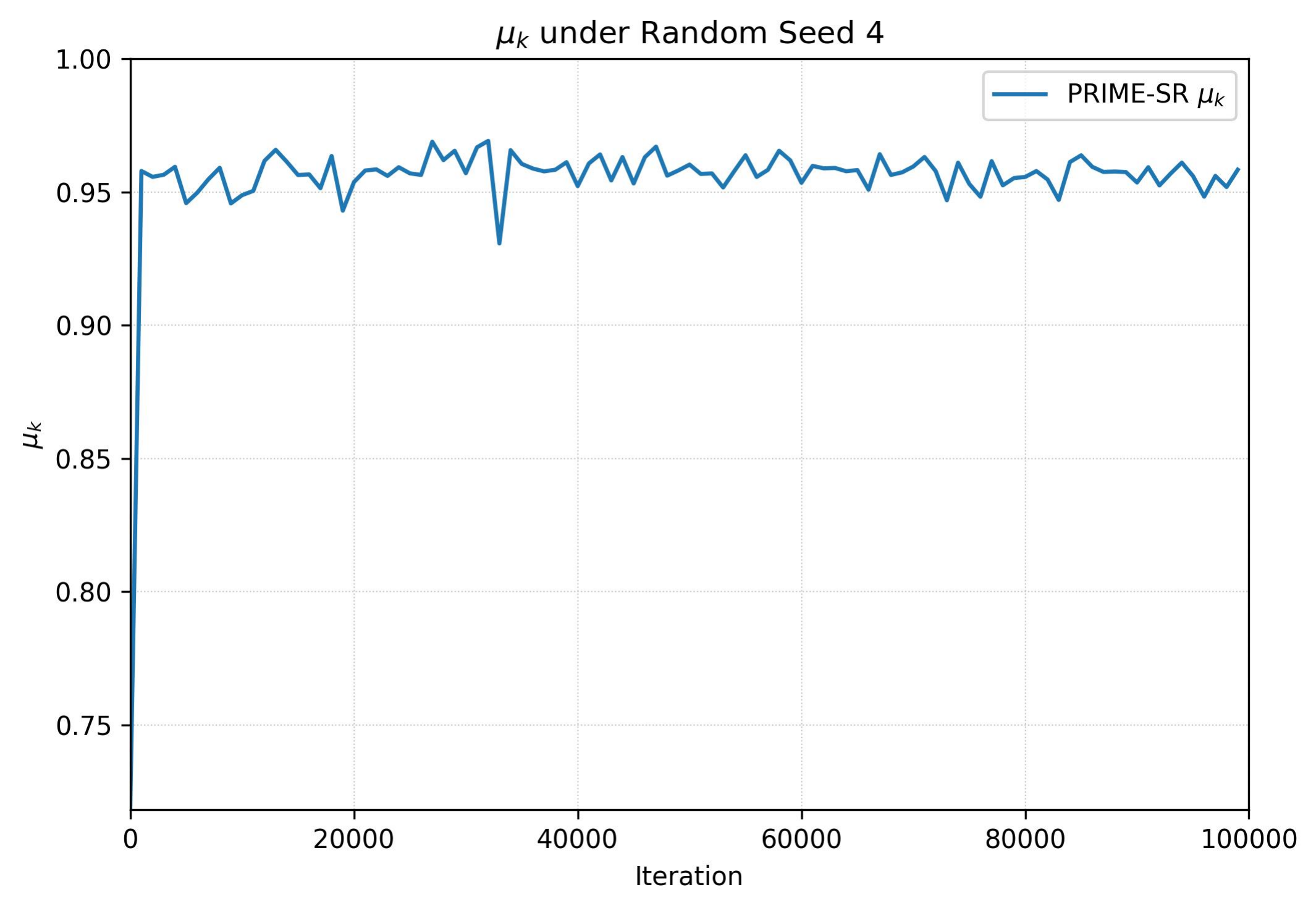}
    \caption{$\mathrm{CO}$ molecular. Left: relative energy error. Right: $\mu_k$.}
    \end{subfigure}
    
    \caption{Comparison of fixed-$\mu$ SPRING and PRIME-SR on $\mathrm{LiH}$, $\mathrm{N}_2$, and $\mathrm{CO}$ molecules for random seed 4.}
    
    \label{fig:compare_spring_mol_seed_4}

\end{figure}

\end{document}